\documentclass{article}

\usepackage{coulonbase}
\usepackage{coulonmath}
\usepackage{coulonstyle}
\usepackage{coulondrawing}
\usepackage{mathrsfs}
\usepackage{tikz-cd}

\tikzcdset{labels={font=\everymath\expandafter{\the\everymath\textstyle}}}

\setbftrue
\intvaldfrenchtrue
\distabstrue
\mcgfrenchtrue

\title{Equations in Burnside groups}
\author{Rémi Coulon, Zlil Sela}

\begin{document}

\maketitle

\begin{abstract}
This article is a first step in the study of equations in periodic groups.
As an application, we study the structure of periodic quotients of hyperbolic groups.
We investigate for instance the Hopf and co-Hopf properties, the isomorphism problem, the existence of free splittings, etc.
We also consider the automorphism group of such periodic groups.
\end{abstract}

\medskip
{


}

\tableofcontents

%
\section{Introduction}
%

Let $n \in \N$.
A group $\Gamma$ is \emph{periodic with exponent $n$} (or simply \emph{$n$-periodic}) if it satisfies the law $x^n = 1$, i.e. $\gamma^n = 1$, for every $\gamma \in \Gamma$.
The \emph{Burnside variety of exponent $n$}, denoted by $\mathfrak B_n$, is the class of all $n$-periodic groups.
Its study was initiated by a question of Burnside who asked whether every finitely generated group in $\mathfrak B_n$ is finite or not.
The first breakthrough was achieved by Novikov and Adian \cite{Novikov:1968tp}.
They prove that if $n$ is a sufficiently large odd exponent, then the free Burnside group of rank $r \geq 2$, defined by 
\begin{equation*}
	\burn rn = \group{a_1, \dots , a_r \mid x^n = 1, \ \forall x},
\end{equation*}
is infinite.
See also Ol'shanski\u\i\ \cite{Olshanskii:1982ts} or Delzant-Gromov \cite{Delzant:2008tu} for an alternative proof.
Free Burnside groups of sufficiently large even exponents are also infinite \cite{Ivanov:1994kh,Lysenok:1996kw,Coulon:2018vp}.
Nevertheless the structure of $\burn rn$ in this case is much more intricate.
In this article we restrict our study to \emph{odd} exponents.
It turns out that $\mathfrak B_n$ is a very rich class of groups \cite{Atabekyan:1987vz,Sonkin:2003il,Minasyan:2009jb,Coulon:2019ac}.
In particular, every non-elementary, hyperbolic group admits infinite periodic quotients \cite{Ivanov:1996va,Coulon:2018vp}.

This article is an initial step in the study of the first order logic of periodic groups.
We focus here on equations in periodic groups.
Let $F$ be the free group generated by $x_1, \dots, x_r$.
An \emph{equation} is an element of $w \in F$.
Given a tuple $(\gamma_1, \dots, \gamma_r)$ of elements in a group $\Gamma$, we write $w(\gamma_1, \dots, \gamma_r)$ for the element of $\Gamma$ obtained by replacing in $w$ each $x_i$ by $\gamma_i$.
Such a tuple is the \emph{solution} to a set of equations $W \subset F$, if $w(\gamma_1, \dots, \gamma_r) = 1$, for every $w \in W$.
The set of solutions to $W$ in $\Gamma$ is in bijection with the set ${\rm Hom}(G, \Gamma)$ of all homomorphisms $G \to \Gamma$, where $G$ is the group defined by the presentation
\begin{equation*}
	G = \group{x_1, \dots, x_r \mid W}.
\end{equation*}
For every integer $n \in \N$, we denote by $\Gamma^n$ the (normal) subgroup of $\Gamma$ generated by the $n$-th power of all its elements, so that the quotient $\Gamma/ \Gamma^n$ belongs to $\mathfrak B_n$.
We call it the \emph{$n$-periodic quotient} of $\Gamma$.
If $W \subset F$ is a set of equations, then every solution of $W$ in $\Gamma$ gives rise to a solution in $\Gamma / \Gamma^n$.
Said differently, the projection $\pi \colon \Gamma \onto \Gamma / \Gamma^n$ induces a map 
\begin{equation}
\label{eqn: morphism between Hom}
	\begin{array}{ccc}
		{\rm Hom}(G, \Gamma) & \to & {\rm Hom}(G, \Gamma/ \Gamma^n) \\
		\phi & \mapsto & \pi \circ \phi
	\end{array}.
\end{equation}

We would like to understand under which conditions the solutions in $\Gamma/ \Gamma^n$ of a set of equations $W$ all come from solutions in $\Gamma$.
Reformulated in terms of homomorphisms, our problem takes the following form.

\begin{ques}
\label{que: main question}
	Let $\Gamma$ be a non-elementary, torsion-free, hyperbolic group.
	Let $G$ be a finitely generated group.
	Does there exist a critical exponent $N \in \N$ such that for every odd integer $n \geq N$, every morphism $\phi \colon G \to \Gamma / \Gamma^n$ \emph{lifts}? 
	That is, can we find a morphism $\tilde \phi \colon G \to \Gamma$, such that $\phi = \pi \circ \tilde \phi$, where $\pi \colon \Gamma \onto \Gamma/ \Gamma^n$ stands for the canonical projection?
\end{ques}

\begin{rema*}
	Our approach is asymptotic in the exponent $n$.
	Indeed, we could have fixed $n$ and study all possible equations in $\Gamma / \Gamma^n$.
	Instead, we fix once and for all a set of equations, and investigate its solutions in $\Gamma / \Gamma^n$ for all sufficiently large (odd) exponents $n$.
	In particular, the critical exponent $N$ in \autoref{que: main question} does depend on the group $G$.
\end{rema*}

\autoref{que: main question} has various answers depending on the group $G$.
Let us describe first some obstructions to lifting morphisms.

\begin{exam}[Torsion abelian groups]
\label{exa: obstruction ab group}
	Consider the group
	\begin{equation*}
		G = \group{x \mid x^3 = 1}.
	\end{equation*}
	Let $r,n \in \N\setminus\{0\}$.
	Fix an element $\gamma \in \burn r{3n}$ of order exactly $3n$.
	The map $\phi \colon G \to \burn r{3n}$ sending $x$ to $\gamma^n$ does not lift to a morphism $G \to \free r$.
	This obstruction can be eventually fixed by restricting ourselves to periodic groups with (sufficiently large) prime exponents.
\end{exam}

\begin{exam}[Roots]
\label{exa: first example w/ root}
	Consider the group
	\begin{equation*}
		G = \left<x_1,x_2,y \mid x_1^2x_2^2 = y^2\right>.
	\end{equation*}
	Let $n = 2p-1$ be an odd integer.
	We denote by $a_1$ and $a_2$ the standard generators of $\burn 2n$ and define an epimorphism $\phi \colon G \onto \burn 2n$ by
	\begin{equation*}
		\phi(x_1) = a_1, \quad
		\phi(x_2) = a_2, \quad \text{and} \quad
		\phi(y) = \left(a_1^2 a_2^2\right)^p.
	\end{equation*}
	It follows from a result of Lyndon that the image of any morphism $G \to \free 2$ is abelian \cite{Lyndon:1959vy}.
	Hence, no matter how large $n$ is, the morphism $\phi$ does not lift to a morphism $G \to \free 2$.
\end{exam}

The problem in the latter example comes from the fact that it is possible to extract roots in a periodic group.
This can be formalized as follows.
Let $p \in \N$.
Let $H$ be a finitely generated group and $h \in H$.
Consider the group
\begin{equation*}
	G = \group{H, y \mid y^p = h}.
\end{equation*}
Assume now that $n$ is an integer co-prime with $p$ and $\Gamma$ an $n$-periodic group.
Then every morphism $H \to \Gamma$ uniquely extends to a morphism $G \to \Gamma$ (see \autoref{res: root splitting extending morphism}).

The previous example motivates the following definition.
A \emph{root splitting} of a group $G$ is a decomposition as a non-trivial amalgamated product $G = A \ast_C B$, where $B$ is an abelian group and $C$ a finite index subgroup of $B$.
Such a splitting is \emph{non-essential} if $C$ is a free factor of $A$ and abstractly isomorphic to $B$.
The general philosophy of our work is that every time we stay away from abelian groups with torsion as well as root splittings then \autoref{que: main question} is likely to have a positive answer.
For instance we have the following statement.

\begin{theo}[see \autoref{res: lifting quotient wo roots -  inf presented}]
\label{res: lift - intro}
	Let $\Gamma$ be a non-elementary, torsion-free, hyperbolic group.
	Let $G$ be a finitely generated group, none of whose quotients admit an essential root splitting.
	There exists a critical exponent $N \in \N$, such that for every odd integer $n \geq N$, for every morphism $\phi \colon G \to \Gamma / \Gamma^n$ whose image is not abelian, there is a morphism $\tilde \phi \colon G \to \Gamma$ lifting $\phi$.
\end{theo}

\begin{rema*}
	Examples of groups $G$ satisfying the assumptions of \autoref{res: lift - intro} are groups with Serre's property (FA) \cite{Serre:1977wy} or free products of perfect groups with Serre's property (FA).
	
	Recall that ${\rm Hom}(G, \Gamma)$ can be entirely parametrized in terms of Makanin-Razborov diagrams \cite{Sela:2009bh}.
	Hence in this case, we get for free a description of ${\rm Hom}(G, \Gamma / \Gamma^n)$.
\end{rema*}

\begin{rema}
	In general, the map ${\rm Hom}(G, \Gamma) \to {\rm Hom}(G, \Gamma/ \Gamma^n)$ given in (\ref{eqn: morphism between Hom}) is not one-to-one.
	The reason is the following.
	The subgroup $\Gamma^n$ is characteristic in $\Gamma$.
	Hence the projection $\Gamma \onto \Gamma/ \Gamma^n$ induces a morphism $\aut \Gamma \to \aut{\Gamma / \Gamma^n}$.
	The latter may not be one-to-one (for instance if $\Gamma$ is a surface group, then any Dehn twist induces a finite order automorphism of $\Gamma / \Gamma^n$).
	Consequently if $\phi$ belongs to ${\rm Hom}(G, \Gamma)$ and $\theta$ is an element in the kernel of $\aut \Gamma \to \aut{\Gamma / \Gamma^n}$, then $\phi$ and $\theta \circ \phi$ have the same image in ${\rm Hom}(G, \Gamma/ \Gamma^n)$.
	It follows that the lift $\tilde \phi$ provided by \autoref{res: lift - intro} is not necessarily unique.
	
	Observe also that \autoref{res: lift - intro} does not quite say that the map  ${\rm Hom}(G, \Gamma) \to {\rm Hom}(G, \Gamma/ \Gamma^n)$ is onto (provided $G$ satisfies the required hypothesis).
	Indeed we excluded morphisms with an abelian image to avoid the pathology such as the one described in \autoref{exa: obstruction ab group}.
\end{rema}

Using variations of \autoref{res: lift - intro}, we also explore topics such as free splittings in $\mathfrak B_n$, the Hopf and co-Hopf properties, the isomorphism problem, automorphism groups, etc.
Before detailing our general strategy let us survey these applications.

\paragraph{Free splitting in the Burnside variety.}
The Burnside variety $\mathfrak B_n$ comes with a notion of free product.
It is characterized by the universal property below.

\begin{prop}
\label{res: def free product burnside}
	Let $A_1, A_2 \in \mathfrak B_n$.
	There exist a group $G \in \mathfrak B_n$ and two monomorphisms $\iota_1 \colon A_1 \to G$ and $\iota_2 \colon A_2 \to G$ with the following property.
	For every group $H \in \mathfrak B_n$, for every morphisms $\phi_1 \colon A_1 \to H$ and $\phi_2 \colon A_2 \to H$, there is a unique morphism $\phi \colon G \to H$ such that the following diagram commutes.
	\begin{center}
	\begin{tikzpicture}
		\matrix (m) [matrix of math nodes, row sep=2em, column sep=2.5em, text height=1.5ex, text depth=0.25ex] 
		{ 
			G	& A_1 	\\
			A_2	& H	 \\
		}; 
		\draw[>=stealth, ->] (m-1-2) -- (m-1-1) node[pos=0.5, above]{$\iota_1$};
		\draw[>=stealth, ->] (m-2-1) -- (m-1-1) node[pos=0.5, left]{$\iota_2$};
		\draw[>=stealth, ->] (m-1-2) -- (m-2-2) node[pos=0.5, right]{$\phi_1$};
		\draw[>=stealth, ->] (m-2-1) -- (m-2-2) node[pos=0.5, below]{$\phi_2$};
		\draw[>=stealth, ->] (m-1-1) -- (m-2-2) node[pos=0.5, above right]{$\phi$};
	\end{tikzpicture}
	\end{center}
\end{prop}

\begin{proof}
	It suffices to take for $G$ the $n$-periodic quotient of the (regular) free product $A_1 \freep A_2$.
	Note that $G$ maps onto $A_1 \times A_2$. 
	Hence the morphism $\iota_i \colon A_i \to G$ is one-to-one.
\end{proof}

It follows from \autoref{res: def free product burnside} that $G$ is unique up to isomorphisms. 
We call $G$ the \emph{$n$-periodic free product} (or simply the \emph{free product} if there is no ambiguity) of $A_1$ and $A_2$ and denote it by $A_1\ast^n A_2$.
It turns out that the free product in the Burnside variety is also exact, associative and commutative, provided $n$ is a sufficiently large odd exponent \cite{Adian:1976vg}.
We say that a group $G \in \mathfrak B_n$ is \emph{freely decomposable in $\mathfrak B_n$}, if it is isomorphic to $A \ast ^n B$ for some non trivial groups $A, B \in \mathfrak B_n$.

\begin{theo}[see \autoref{res: no splitting}]
\label{res: no splitting - intro}
	Let $\Gamma$ be a torsion-free, hyperbolic group.
	If $\Gamma$ has no essential root splitting, then the following are equivalent.
	\begin{enumerate}
		\item \label{enu: no splitting - one-ended}
		The group $\Gamma$ is freely indecomposable.
		\item \label{enu: no splitting - no-splitting exist}
		For every $N \in \N$, there exists an odd integer $n \geq N$, such that the quotient $\Gamma/ \Gamma^n$ is freely indecomposable in $\mathfrak B_n$.
		\item \label{enu: no splitting - no-splitting forall}
		There exists $N \in \N$, such that for every odd exponent $n \geq N$, the quotient $\Gamma/ \Gamma^n$ is freely indecomposable in $\mathfrak B_n$.
	\end{enumerate}
\end{theo}

\begin{exam}
\label{exa: root source of troubles}
	The previous statement fails if $\Gamma$ admits a root-splitting.
	Indeed let $p \in \N \setminus \{0,1\}$ and define a group $\Gamma$ by the following presentation:
	\begin{equation*}
		\Gamma = \left<a,b,c \middle| c^p = [a,b]\right>.
	\end{equation*}
	It is a freely indecomposable, torsion-free, hyperbolic group. 
	We claim that for every odd integer $n \in \N$ which is co-prime with $p$, the $n$-periodic quotient of $\Gamma$ is isomorphic to $\burn 2n$.
	According to Bezout's identity, we can find $\alpha,\beta \in \Z$ such that $\alpha p - \beta n = 1$.
	In $\Gamma / \Gamma^n$ the relation $c^p = [a,b]$ is equivalent to 
	\begin{equation*}
		c = c c^{\beta n} = c^{\alpha p} = [a,b]^\alpha.
	\end{equation*}
	It follows that the generator $c$ and the relation $c^p = [a,b]$ can be removed from the presentation of $\Gamma/ \Gamma^n$, whence our claim.
	Consequently $\Gamma/\Gamma^n$ splits as a free product in $\mathfrak B_n$.
\end{exam}

\paragraph{Isomorphism problem between periodic groups.}
\autoref{res: no splitting - intro} states that the free splittings of $\Gamma/ \Gamma^n$ can be read from the free splittings of $\Gamma$.
The next result has a similar flavor. 
Roughly speaking, it says that if $\Gamma$ is freely indecomposable, then the isomorphism type of $\Gamma / \Gamma^n$ can be read from the one of $\Gamma$.

\begin{theo}[see \autoref{res: isom - particular}]
\label{res: isom - particular - intro}
	Let $\Gamma_1$ and $\Gamma_2$ be two torsion-free, hyperbolic groups with no essential root splitting.
	There exists $N \in \N$ such that the following are equivalent.
	\begin{enumerate}
		\item \label{enu: isom - particular - source}
		$\Gamma_1$ and $\Gamma_2$ are isomorphic.
		\item \label{enu: isom - particular - some}
		There exists an odd integer $n \geq N$, such that $\Gamma_1/ \Gamma_1^n$ and $\Gamma_2/\Gamma_2^n$ are isomorphic.
		\item \label{enu: isom - particular - any}
		For every integer $n \in \N$, the groups $\Gamma_1/ \Gamma_1^n$ and $\Gamma_2/\Gamma_2^n$ are isomorphic.
	\end{enumerate}
\end{theo}

\begin{exam}
	One more time, the assumption regarding the root splittings is crucial.
	Consider indeed the groups
	\begin{equation*}
		\Gamma_2 = \left<a,b,c \middle| c^2 = [a,b]\right> 
		\quad \text{and} \quad 
		\Gamma_3 = \left<a,b,c \middle| c^3 = [a,b]\right> 
	\end{equation*}
	Both are torsion-free, freely indecomposable, non-elementary hyperbolic groups.
	The groups  $\Gamma_2$ and $\Gamma_3$ are not isomorphic (the torsion parts of their abelianizations are distinct).
	However, we explained in \autoref{exa: root source of troubles} that for every exponent $n$ that is coprime with $6$, the groups $\Gamma_2 / \Gamma_2^n$ and $\Gamma_3 / \Gamma_3^n$ are both isomorphic to the free Burnside group $\burn 2n$.
\end{exam}

Recall that the isomorphism problem is solvable among hyperbolic groups, see Sela \cite{Sela:1995aa}, Dahmani-Groves \cite{Dahmani:2008uu} and Dahmani-Guirardel \cite{Dahmani:2011wo}.
It follows from the previous statement, that given two torsion-free, freely indecomposable, hyperbolic groups with no essential root splitting, there is an algorithm that can decide whether their periodic quotients (of sufficiently large exponent) are isomorphic.

\paragraph{Automorphism groups of periodic groups.}
Let $\Gamma$ be a group and $n$ an integer.
The subgroup $\Gamma^n \subset \Gamma$ is characteristic.
Hence the projection $\Gamma \onto \Gamma/\Gamma^n$ induces a homomorphism $\aut \Gamma \to \aut{\Gamma/\Gamma^n}$ at the level of automorphisms.
In general the latter map is neither one-to-one nor onto \cite{Coulon:2010ut}.
If $\Gamma$ is a free group, it is even not known whether the automorphism group $\aut{\Gamma/ \Gamma^n}$ or the outer automorphism group $\out{\Gamma / \Gamma^n}$ are finitely generated.
Nevertheless these problems can be solved when $\Gamma$ is a suitable freely indecomposable, hyperbolic group.

\begin{theo}[see \autoref{res: auto/co-hopf - particular}]
\label{res: auto one-ended -intro}
	Let $\Gamma$ be a non-elementary, torsion-free, freely indecomposable, hyperbolic group with no essential root splitting.
	There exists a critical exponent $N\in \N$ such that for every odd integer $n \geq N$, the map $\aut \Gamma \to \aut{\Gamma/\Gamma^n}$ is onto.
	In particular, $\aut{\Gamma/\Gamma^n}$ and $\out{\Gamma / \Gamma^n}$ are finitely generated.
\end{theo}

\begin{rema*}
	Let $\Sigma$ be a closed compact hyperbolic surface and $\Gamma$ its fundamental group.
	We denote by $\mcg \Sigma$ its mapping class group, that is the set of homeomorphisms of $\Sigma$ up to isotopy (we do not require the homeomorphisms to be orientation-preserving).
	The previous result has the following consequence.
	If $n$ is a sufficiently large odd exponent, then the projection $\Gamma \onto \Gamma / \Gamma^n$ induces an \emph{epimorphism} $\mcg \Sigma \onto \out{\Gamma / \Gamma^n}$.
	We suspect that the kernel of this map is exactly the normal subgroup generated by the $n$-th power of every Dehn twist.
\end{rema*}

\begin{exam}
\label{exa: lifting auto}
	\autoref{exa: root source of troubles} is again an example where \autoref{res: auto one-ended -intro} fails without the root splitting assumption.
	Indeed, if $n$ is an odd exponent, then $n$-periodic quotient of 
	\begin{equation*}
		\Gamma = \left<a,b,c \middle| c^2 = [a,b]\right>.
	\end{equation*}
	is the free Burnside group $\burn 2n$ generated by (the images of) $a$ and $b$.
	Consider now the automorphism $\phi$  of $\burn 2n$ given by $a \mapsto a^2$ and $b \mapsto b$.
	Assume that $\phi$ lifts to an automorphism $\tilde \phi$ of $\Gamma$.
	The abelianization of $\Gamma$ is $A = B\times T$ where $B = \Z^2$ (\resp $T = \Z/2\Z$) is generated by (the images) of $a$ and $b$ (\resp $c$).
	Moreover $A$ maps onto its torsion-free part $B \cong A / T$.
	Note that $\tilde \phi$ induces an automorphism of $B$, that we denote $\psi$.
	By looking at the $n$-periodic quotient of $B$, which is also the abelianization of $\Gamma / \Gamma^n$, i.e. $(\Z /n \Z)^2$, we observe that $\psi$ can be represented by a matrix of the form
	\begin{equation*}\left[
		\begin{array}{cc}
			2 & 0 \\
			0 & 1
		\end{array}
		\right] \mod  n.
	\end{equation*}
	However such a matrix is not invertible provided $n > 3$, which contradicts the fact that $\psi$ is an automorphism.
\end{exam}

\paragraph{Hopf/co-Hopf property.}
A group $G$ is \emph{Hopfian} if every surjective homomorphism from $G$ to itself is actually an isomorphism.
Hyperbolic groups are Hopfian, see Sela \cite{Sela:1999ha} and Weidmann-Reinfeldt \cite{Weidmann:2019ue}.
We prove that certain periodic quotients of hyperbolic groups are also Hopfian.

\begin{theo}[see \autoref{res: hopf - rigid}]
\label{res: hopf - rigid - intro}
	Assume that $\Gamma$ is a torsion-free, hyperbolic group, none of whose quotient admits an abelian splitting.
	Then there exists $N \in \N$, such that for every odd integer $n \geq N$, the periodic quotient $\Gamma / \Gamma^n$ is Hopfian.
\end{theo}

To the best of our knowledge it is still an open question whether \emph{free} Burnside groups of sufficiently large exponents are Hopfian.
It was proved by Ivanov and Storozhev that there exists a group variety whose free elements are not necessarily Hopfian \cite{Ivanov:2005du}.

Dually, a group $G$ is \emph{co-Hopfian} if every injective homomorphism from $G$ to itself is an isomorphism.
A hyperbolic group is co-Hopfian as soon as it does not split as an amalgamated product or an HNN-extension over a finite group, see Sela  \cite{Sela:1999ha} and Moioli \cite{Moioli:2013uk}.
We obtain here an analogue statement for periodic quotients of torsion-free, hyperbolic groups.

\begin{theo}[see \autoref{res: auto/co-hopf - particular}]
\label{res: co-hopf one-ended - intro}
	Let $\Gamma$ be a non-elementary, torsion-free, hyperbolic group with no essential root splitting.
	If $\Gamma$ is freely indecomposable, then there exists a critical exponent $N\in \N$, such that for every odd integer $n \geq N$, the quotient $\Gamma/\Gamma^n$ is co-Hopfian.
\end{theo}

\begin{rema*}
	The converse statement certainly holds  true, i.e. if $\Gamma/ \Gamma^n$ is co-Hopfian (for a sufficiently large odd exponent $n$) then $\Gamma$ is freely indecomposable.
	Since the article is already long, we did not include a proof of this fact.
	The strategy would be the following.
	Assume that $\Gamma$ splits as a free product $\Gamma = A \ast B$.
	Let $a \in A\setminus\{1\}$ and $b \in B\setminus\{1\}$.
	It is known that the morphism $\phi \colon \Gamma \to \Gamma$ fixing $A$ and conjugating $B$ by $\gamma = ba$ is one-to-one by not onto.
	Using a variation of \cite{Coulon:2018ac} in the context of free products, one can certainly prove that for sufficiently large odd exponents $n$, the map $\phi$ induces a one-to-one endomorphism of $\Gamma / \Gamma^n$ that is not onto.
	
	In particular if $\Gamma$ is one of the group with an essential root splitting studied in \autoref{exa: root source of troubles}, then $\Gamma / \Gamma^n$ may not be co-Hopfian.
\end{rema*}

\paragraph{Toward equational noetherianity.}
A group $H$ is \emph{equationally noetherian} if for every finitely generated free group $F$, for every epimorphism $\pi \colon F \onto G$, there is a finite subset $W \subset \ker \pi$ with the following property: every morphism $\phi \colon F \to H$ whose kernel contains $W$ factors through $\pi$. 
Hyperbolic groups are equationally noetherian, see Sela \cite{Sela:2009bh} and Weidmann-Reinfeldt \cite{Weidmann:2019ue}.
The definition of equational noetherianity can be extended to a class of groups as in Groves-Hull \cite{Groves:2019aa}.
As before we adopt an asymptotic point of view.

\begin{defi}
\label{def: asymp eq noetherian}
	Let $\Gamma$ be a finitely generated group.
	We say that the periodic quotients of $\Gamma$ are \emph{asymptotically equationally noetherian} if for every finitely generated free group $F$, for every epimorphism $\pi \colon F \onto G$, there exist a finite subset $W \subset \ker \pi$ and a critical exponent $N \in \N$ with the following property.
	Let $n \geq N$ be an (odd) integer and $\phi \colon F \to \Gamma / \Gamma^n$ be a morphism.
	If $W \subset \ker \phi$, then $\phi$ factors through $\pi$. 
\end{defi}

We are not able yet to prove that periodic quotients of torsion-free, hyperbolic groups are asymptotically equationally noetherian. 
Nevertheless we made a first step in this direction.

\begin{theo}[see \autoref{res: lifting quotient wo roots - noetherian}]
\label{res: asymp eq noetherian - intro}
	Let $\Gamma$ be a non-elementary, torsion-free, hyperbolic group.
	Let $F$ be a finitely generated free group and $\pi \colon F \onto G$ an epimorphism where no quotient of $G$ admits an essential root splitting.
	There exist a finite subset $W \subset \ker \pi$ and a critical exponent $N \in \N$ with the following property.
	Let $n \geq N$ be an odd integer and $\phi \colon F \to \Gamma / \Gamma^n$ be a morphism.
	If $W \subset \ker \phi$, then $\phi$ factors through $\pi$. 
\end{theo}

\begin{rema*}
	For some of the above statements (Theorems~\ref{res: no splitting - intro}, \ref{res: isom - particular - intro},  \ref{res: auto one-ended -intro}, and \ref{res: co-hopf one-ended - intro}) we illustrated with examples that it is crucial to assume that the group $\Gamma$ has no essential root splitting.
	We refer the reader to \autoref{sec: one-ended periodic groups} where we address this issue and provide more general results.
	The underlying idea is that the hypothesis that $\Gamma$ is freely indecomposable without root splitting can be replaced by asking that the periodic quotient $\Gamma/ \Gamma^n$ is freely indecomposable in $\mathfrak B_n$, see \autoref{res: no splitting - general}.
	This only works for suitable exponents though.
\end{rema*}

\paragraph{Strategy.}
If $\Gamma$ is a torsion-free hyperbolic group, the set of homomorphisms ${\rm Hom}(G,\Gamma)$ has been described by Sela in terms of Makanin-Razborov diagram \cite{Sela:2009bh}.
In our study of ${\rm Hom}(G, \Gamma / \Gamma^n)$ we borrow Sela's tools such as limit groups, actions on $\R$-trees, shortening arguments, etc.
This may sound odd at first.
Periodic groups are indeed very far from being hyperbolic: any action by isometries of a periodic group on a Gromov hyperbolic space is either elliptic or parabolic.
The ``trick'' is to replace $\Gamma / \Gamma^n$ by some hyperbolic group approximating it.
Let us briefly detail this idea.

Let $n \in \N$ be a large odd exponent.
All known strategies for studying $\Gamma / \Gamma^n$ start in the same way: one produces by induction an \emph{approximation sequence} of non-elementary, hyperbolic groups
\begin{equation}
\label{eqn: intro - sequence}
	\Gamma = \Gamma_0 \onto \Gamma_1 \onto \Gamma_2 \dots \onto \Gamma_j \onto \Gamma_{j+1} \onto \dots
\end{equation}
whose direct limit is exactly $\Gamma / \Gamma^n$.
At each step $\Gamma_{j+1}$ is obtained from $\Gamma_j$ by adding new relations of the form $\gamma^n = 1$, where $\gamma$ runs over a set of ``small'' loxodromic elements of $\Gamma_j$.

Fix now a finitely presented group $G = \group{U| R}$.
Any morphism $\phi \colon G \to \Gamma/ \Gamma^n$ factors through some $\Gamma_j$, that is there is morphism $\psi \colon G \to \Gamma_j$ such that $\phi = \pi_j \circ \psi$ where $\pi_j \colon \Gamma_j \onto \Gamma / \Gamma^n$ stands for the natural projection.
Hence it suffices to understand the morphisms from $G$ to $\Gamma_j$.
However the works of Champetier \cite{Champetier:2000jx} and Osin \cite{Osin:2009aa} show that it is not possible to give a useful ``uniform'' description for the set of morphisms from $G$ to any hyperbolic group.
Therefore we need to use the fact that the groups $\Gamma_j$ belong to a very specific subclass of hyperbolic groups, which we describe now.

Following Delzant-Gromov \cite{Delzant:2008tu} we attach to each group $\Gamma_j$ a hyperbolic space $X_j$ on which it acts by isometries.
This space is obtained by means of geometric small cancellation theory (see  \autoref{sec: appendix - sc}).
Its geometry is ``finer'' than the one of the Cayley graph of $\Gamma_j$.
In particular, we have a uniform control (i.e. independent of the exponent $n$ and the approximation step $j$) on its hyperbolicity constant, as well as on a weak form of acylindricity.
In addition, the projection $\Gamma_j$ is one-to-one when restricted to balls of radius $r(n)$ (for the metric on $\Gamma_j$ induced by the one on $X_j$) where $r(n)$ only depends on $n$ and diverges to infinity when $n$ tends to infinity.
All the needed properties of the spaces $X_j$ are wrapped up in a preferred class of group actions defined in \autoref{sec: a preferred class of groups}.
See in particular \autoref{res: approximating sequence}.
We now use $X_j$ to measure the energy of the lift $\psi \colon G \to \Gamma_j$ as follows (see \autoref{def: energy})
\begin{equation*}
	\lambda_\infty (\psi, U) = \inf_{x \in X_j} \max_{u \in U} \dist{\psi(u)x}x.
\end{equation*}

This is the place where our asymptotic approach jumps in.
Note indeed that all of our statements involve a critical exponent $N$ that depends on the group $G$.
Their proofs are usually by contradiction (see for instance Propositions~\ref{res: lifting morphism - rigid gal} and \ref{res: lifting morphism - killing elements}).
Assuming that a statement is false, we get a sequence of counter-examples, i.e. morphisms $\phi_k \colon G \to \Gamma / \Gamma^{n_k}$ that fail our theorem, where $n_k$ is a sequence of odd integers diverging to infinity.
For each $k \in \N$, we fix a lift $\psi_k \colon G \to \Gamma_{j_k, k}$ of $\phi_k$, where $\Gamma_{j_k, k}$ is an approximation of $\Gamma / \Gamma^{n_k}$ as above.
We choose it so that the approximation step $j_k$ is minimal.
We distinguish then two cases.
\begin{itemize}
	\item Suppose first that the energy $\lambda_\infty (\psi_k, U)$ is bounded.
	The projection between the approximation groups of $\Gamma / \Gamma^{n_k}$ are injective on larger and larger balls when $k$ tends to infinity.
	We use this fact to prove that the minimal lift $\psi_k$ is actually a lift from $G$ to $\Gamma_{0, k} = \Gamma$ (see \autoref{res: approx - energy minimal approx}).
	Hence we are back to the standard situation of understanding ${\rm Hom}(G, \Gamma)$ for the \emph{fixed} hyperbolic group $\Gamma$.
	
	\item Suppose now that (up to passing to a subsequence) $\lambda_\infty (\psi_k, U)$ diverges to infinity.
	Since the hyperbolicity constant of the spaces $X_{j_k, k}$ (on which $\Gamma_{j_k, k}$ acts) is uniformly bounded, the sequence of rescaled spaces
	\begin{equation*}
		\frac 1{\lambda_\infty (\psi_k, U)} X_{j_k, k}
	\end{equation*}
	converges to an $\R$-tree $T$ endowed with an action without global fixed point of the limit group $L$ associated the sequence of morphisms $(\psi_k)$ (\autoref{sec: building limit tree}).
	Our control of the action of $\Gamma_{j_k, k}$ on $X_{j_k, k}$ is enough to control the one of $L$ on $T$ (Sections~\ref{res: peripheral structure} and \ref{sec: action of the limit group}).
	In particular, if $L$ is freely indecomposable, we are back on track to exploit all the existing tools to analyze the action of a group on an $\R$-tree.
	We first prove that the action of $L$ on $T$ decomposes as a graph of actions whose components are either simplicial, axial or of Seifert type (\autoref{sec: prelim decomposition}).
	It provides in particular a graph of groups decomposition of $L$ whose edge groups are abelian (see \autoref{res: final decomposition tree}).
	Then, we shorten (when possible) the morphisms $\psi_k$ using modular automorphisms of this graph of groups decomposition (\autoref{sec: shortening}).
	This part of the argument involves a new kind of component in the decomposition of $T$ that is handled in \autoref{sec: shortening peripheral}.
\end{itemize}

In the latter case, we may face an obstruction to perform the shortening argument.
As we sketched, the action of $L$ on the $\R$-tree $T$ provides a graph of groups decomposition of $L$ whose modular group can be used to shorten the morphism $\psi_k$.
If this decomposition corresponds to a root splitting, then there may be no non-inner modular automorphism to perform the shortening.

For the moment we solve this problem, by imposing conditions on the group $G$ that will ensure that the limit group $L$ obtained via this process has no root splitting.
The existence of root splittings is a major difficulty to completely describe the solutions of an arbitrary system of equations.

\begin{rema*}
	In \cite{Groves:2019aa} Groves and Hull study equational noetherianity for a \emph{family} of groups $\mathcal G$.
	This is not unrelated to our approach as they need to understand morphisms $G \to \Gamma$ where $\Gamma$ can vary in $\mathcal G$.
	Nevertheless their work does not apply in our context.
	Indeed Groves and Hull focus on a class of \emph{uniformly} acylindrically hyperbolic groups.
	As the exponent $n$ varies, it turns out that the action of the group $\Gamma_j$ (approximating $\Gamma / \Gamma^n$) on the space $X_j$ is not uniformly acylindrical.
	Indeed its injectivity radius converges to zero as $n$ tends to infinity.
\end{rema*}

\paragraph{Outline of the article.}
The first five sections review various tools needed for our study. 
\autoref{sec: prelim aka hyp} deals with hyperbolic geometry.
\autoref{sec: approx periodic groups} recollects and formalizes all the properties of the hyperbolic groups approximating the periodic quotients $\Gamma / \Gamma^n$ (see in particular \autoref{res: approximating sequence} whose proof is given in \autoref{sec: approx}).
\autoref{sec: space marked groups} focuses on limit groups seen as limits in the compact space of marked groups.
\autoref{sec: splittings} reviews the vocabulary associated to splittings and graph of actions.
\autoref{sec: JSJ decomposition} recalls the construction of the JSJ decomposition of freely indecomposable, CSA groups.

The core of the article lies in \autoref{sec: action on limit tree}.
We study there the action of a limit group (over the periodic quotients of a hyperbolic group) on a limit $\R$-tree, and explain under which assumptions the shortening argument can be adapted to this context.
In \autoref{sec: lifting theorems}, we use this analysis to prove several lifting theorems.
\autoref{sec: applications} is devoted to the applications of this work.

\paragraph{Acknowledgment.}
The first author is grateful to the \emph{Centre Henri Lebesgue} ANR-11-LABX-0020-01 for creating an attractive mathematical environment.
His institute, the IMB receives support from the EIPHI Graduate School ANR-17-EURE-0002.
He acknowledges support from the Agence Nationale de la Recherche under Grants \emph{Dagger} ANR-16-CE40-0006-01 and \emph{GoFR} ANR-22-CE40-0004.
The second author is partially supported by an ISF fellowship.
This work was conducted partly during the MSRI Program \emph{Geometric Group Theory} in Fall 2016 and the HIM Program \emph{Logic and Algorithms in Group Theory} in Fall 2018.
The second author would like to add that although the work started as a joint project in MSRI in 2016, the entire work that is presented in the paper is due to the first author.
The authors thank the referee for their careful reading and helpful remarks.

%
\section{Preliminaries}
%
\label{sec: prelim aka hyp}

%
\subsection{Hyperbolic geometry}
%

In this section, we briefly recall the definitions and notations we use to deal with hyperbolic spaces in the sense of Gromov.
Unless mentioned otherwise all the paths we consider are parametrized by arc length.
For brevity, we did not elaborate on the proofs for statements which are standard exercises of hyperbolic geometry.
For more details, we refer the reader to Gromov's original article \cite{Gromov:1987tk} or the numerous literature on the subject, e.g. \cite{Coornaert:1990tj,Ghys:1990ki,Bowditch:1991wl,Bridson:1999ky}.

\paragraph{Four point inequality.}
Let $X$ be a length space.
For every $x,y \in X$, we write $\dist[X] xy$, or simply $\dist xy$, for the distance between $x$ and $y$.
If it exists, we write $\geo xy$ for a geodesic joining $x$ to $y$.
Given $x \in X$ and $r \in \R_+$, we denote by $B(x,r)$ the open ball of radius $r$ centered at $x$, i.e.
\begin{equation*}
	B(x,r) = \set{y \in X}{\dist xy < r}.
\end{equation*}
The \emph{punctured ball} of radius $r$ centered at $x$ is $\mathring B(x,r) = B(x,r) \setminus\{x\}$.
The Gromov product of three points $x,y,z \in X$ is
\begin{equation*}
	\gro xyz = \frac 12 \left( \dist xz + \dist yz - \dist yz\right).
\end{equation*}
Let $\delta \in \R_+^*$.
For the remainder of \autoref{sec: prelim aka hyp}, we always assume that $X$ is $\delta$-hyperbolic, i.e. for every $x,y,z,t \in X$,
\begin{equation}
\label{eqn: four point hyp}
	\min\left\{ \gro xyt, \gro yzt \right\} \leq \gro xzt + \delta.
\end{equation}
We denote by $\partial X$ its boundary at infinity.
The definition of the Gromov product $\gro xyz$ extends to the triple of points where $x,y \in X \cup \partial X$ and $z \in X$.

\paragraph{Quasi-geodesics.}
Let $\kappa \in \R_+^*$ and $\ell \in \R_+$.
A \emph{$(\kappa, \ell)$-quasi-isometric embedding} is a map $f \colon X_1 \to X_2$ between two metric spaces such that for every $x,x' \in X_1$,
\begin{equation*}
	\kappa^{-1}\dist x{x'} - \ell \leq \dist{f(x)}{f(x')} \leq \kappa \dist x{x'} + \ell.
\end{equation*}
A \emph{$(\kappa, \ell)$-quasi-geodesic} is a $(\kappa,\ell)$-quasi-isometric embedding $c \colon I \to X$  of an interval $I \subset \R$ into $X$.
A path is an \emph{$L$-local $(\kappa, \ell)$-quasi-geodesic} if its restriction to any subinterval of length $L$ is a $(\kappa, \ell)$-quasi-geodesic.
If $c \colon \R_+ \to X$ is $(\kappa, \ell)$-quasi-geodesic, then there exists a unique point $\xi \in \partial X$ such that $\lim_{t \to \infty} c(t) = \xi$.
We view $\xi$ as the endpoint at infinity of $c$ and write $\xi  = c(\infty)$.

\paragraph{Quasi-convex subsets.}
Let $\alpha \in \R_+$.
A subset $Y$ of $X$ is \emph{$\alpha$-quasi-convex} if $d(x,Y) \leq \gro y{y'}x + \alpha$, for every $x \in X$, and $y,y' \in Y$.
We denote by $\distV[Y]$ the length metric on $Y$ induced by the restriction of $\distV[X]$ to $Y$.
We say that $Y$ is \emph{strongly quasi-convex} if it is $2 \delta$-quasi-convex and for every $y,y' \in Y$,
\begin{equation}
\label{eqn: def strongly qc}
	\dist[X]y{y'} \leq \dist[Y]y{y'} \leq \dist[X]y{y'} + 8\delta,
\end{equation}
(in particular, it forces $Y$ to be connected by rectifiable paths).

We adopt the convention that the diameter of the empty set is zero, whereas the distance from a point to the empty set is infinite.
Let $x$ be a point of $X$.
A point $y \in Y$ is an \emph{$\eta$-projection} of $x$ on $Y$ if $\dist xy \leq d(x,Y) + \eta$.
A $0$-projection is simply called a \emph{projection}.

\paragraph{Isometries.}
An isometry $\gamma$ of $X$ is either \emph{elliptic} (its orbits are bounded) \emph{parabolic} (its orbits admit exactly one accumulation point in $\partial X$) or \emph{loxodromic} (its orbits admit exactly two accumulation points in $\partial X$).
In order to measure the action of $\gamma$ on $X$ we use the \emph{translation length} and the \emph{stable translation length} respectively defined by
\begin{equation*}	
	\norm[X] \gamma = \inf_{x \in X}\dist {\gamma x}x
	\quad \text{and} \quad
	\snorm[X] \gamma = \lim_{n \to \infty} \frac 1n \dist{\gamma^nx}x.
\end{equation*}
If there is no ambiguity, we will omit the space $X$ from the notations.
These lengths are related by
\begin{equation}
\label{eqn: regular vs stable length}
	\snorm \gamma\leq \norm \gamma \leq \snorm \gamma+ 8\delta. 
\end{equation}
In addition, $\gamma$ is loxodromic if and only if $\snorm \gamma > 0$.
In such a case, the accumulation points of $\gamma$ in $\partial X$ are
\begin{equation*}
	\gamma^- = \lim_{n \to \infty} \gamma^{-n}x
	\quad \text{and} \quad
	\gamma^+ = \lim_{n \to \infty} \gamma^nx.
\end{equation*}
They are the only points of $X\cup\partial X$, fixed by $\gamma$.
We write $U_\gamma$ for the union of all $L$-local $(1, \delta)$-quasi-geodesic joining $\gamma^-$ to $\gamma^+$ with $L > 12\delta$.
The \emph{cylinder} of $\gamma$ is the set
\begin{equation*}
	Y_\gamma = \set{x \in X}{ d(x, U_\gamma) < 20\delta}.
\end{equation*}
It is a strongly quasi-convex subset of $X$ which can be thought of as the ``smallest'' $\gamma$-invariant quasi-convex subset \cite[Section~2.3]{Coulon:2018vp}.
See also \cite[Lemma~2.32]{Coulon:2014fr}.

\begin{defi}
\label{def: fix}
	Let $U$ be a set of isometries of $X$.
	The \emph{energy} of $U$ is 
	\begin{equation*}
		\lambda(U) = \inf_{x \in X} \max_{\gamma \in U} \dist {\gamma x}x.
	\end{equation*}
	Given $d \in \R_+$, we define $\fix{U,d}$ as 
	\begin{equation}
	\label{eqn: def fix}
		\fix{U,d} = \set{x \in X}{\forall \gamma \in U,\ \dist{\gamma x}x \leq d}.
	\end{equation}
\end{defi}

If $d > \max\{\lambda(U), 5\delta\}$, then $\fix{U, d}$ is $8\delta$-quasi-convex.
Moreover for every $x \in X \setminus \fix{U,d}$, we have
\begin{equation}
\label{eqn: displacement outside fixed set}
	\sup_{\gamma \in U} \dist{\gamma x}x \geq 2d\left(x, \fix{U,d}\right) + d - 10\delta, 
\end{equation}
see \cite[Lemma~2.8]{Coulon:2018vp}.
Compare also with \cite[Proposition~2.24]{Coulon:2014fr}.
For simplicity, we write $\fix{U}$ for $\fix{U,0}$.
If the set $U = \{\gamma\}$ is reduced to a single isometry, then $\lambda(U) = \norm \gamma$ and we simply write $\fix{\gamma,d}$ for $\fix{U,d}$.

\begin{defi}
\label{def: thin isom}
	Let $\alpha \in \R_+$.
	\begin{itemize}
		\item An elliptic isometry $\gamma$ of $X$ is \emph{$\alpha$-thin at $x \in X$}, if for every $d \in \R_+$ the set $\fix{\gamma, d}$ is contained in $B(x, d/2 + \alpha)$.
		It is \emph{$\alpha$-thin}, if it is $\alpha$-thin at some point of $X$.
		\item A loxodromic isometry $\gamma$ of $X$ is \emph{$\alpha$-thin} if for every $d \in \R_+$, for every $y \in \fix{\gamma,d}$, we have
		\begin{equation*}
			\gro{\gamma^-}{\gamma^+}y \leq \frac 12 (d - \norm \gamma) + \alpha.
		\end{equation*}
	\end{itemize}
\end{defi}

Fix $\alpha \in \R_+^*$.
Examples of isometries of the hyperbolic plane $\mathbf H^2$ which are not $\alpha$-thin are rotations with a sufficiently small angle and loxodromic isometries with a very sufficiently translation length.

\begin{rema}
\label{rem: thin isom - loxo}
	It follows from the stability of quasi-geodesics that any loxodromic isometry $\gamma$ is $\alpha$-thin where 
	\begin{equation*}
		\alpha \leq 100\left(\frac{\delta}{\snorm \gamma} + 1\right)\delta. 
	\end{equation*}
	In our context the thinness provides a way to control the action of an isometry $\gamma$, when we do not have any lower bound for its stable norm $\snorm \gamma$ (see \autoref{sec: sc - approx properties}).
\end{rema}

The local-to-global statement below is a direct consequence of (\ref{eqn: displacement outside fixed set}).

\begin{lemm}
\label{res: local-to-global thinness}
	Let $\gamma$ be an isometry of $X$ and $\alpha \in \R_+$.
	\begin{enumerate}
		\item Assume that $\gamma$ is elliptic. 
		If $\fix{\gamma, 6\delta}$ is contained in the ball $B(x, \alpha)$ for some $x \in X$, then $\gamma$ is $\beta$-thin at $x$, where $\beta = \alpha + 2\delta$. 
 		\item Assume that $\gamma$ is loxodromic.
		If the set $\fix{\gamma, \norm \gamma + 5\delta}$ is contained in the $\alpha$-neighborhood of some $L$-local $(1, \delta)$-quasi-geodesic from $\gamma^-$ to $\gamma^+$  with $L > 12\delta$, then $\gamma$ is $\beta$-thin with $\beta = \alpha + 6\delta$. 
	\end{enumerate}
\end{lemm}

%
\subsection{Group actions}
%
\label{sec: invariants}

Let $\Gamma$ be a group acting by isometries on a $\delta$-hyperbolic length space $X$.
We denote by $\Lambda(\Gamma)$ its \emph{limit set}, i.e. the set of accumulation points in $\partial X$ of the $\Gamma$-orbit of a point $x \in X$.
We say that $\Gamma$ is \emph{non-elementary} (\emph{for its action on $X$}) if $\Lambda(\Gamma)$ contains at least three points.

Observe that we do not assume that the action of $\Gamma$ on $X$ is either proper or co-bounded.
In order to control this action, we use instead three invariants: the injectivity radius, the acylindricity parameter and the $\nu$-invariant.

\begin{defi}[Injectivity radius]
\label{def: inj rad}
	The \emph{injectivity radius} of $\Gamma$ on $X$ is the quantity
	\begin{equation*}
		\inj[X]\Gamma = \inf \set{\snorm[X]\gamma}{\gamma \in \Gamma \ \text{loxodromic}}
	\end{equation*}
\end{defi}

\begin{defi}
\label{def: acyl inv}
	Let $d \in \R_+$
	The \emph{acylindricity parameter at scale $d$}, denoted by $A(\Gamma,X,d)$ is defined as 
	\begin{equation*}
		 A(\Gamma,X,d) = \sup_{U \subset \Gamma} \diam \left(\fix{U,d}\right),
	\end{equation*}
	where $U$ runs over all subsets of $\Gamma$ generating a non-elementary subgroup.
\end{defi}

\begin{defi}[$\nu$-invariant]
\label{def: nu-inv}
	The $\nu$-invariant $\nu(\Gamma, X)$ is the smallest integer $m$ with the following property.
	For every $\gamma, \tau \in \Gamma$ with $\tau$ loxodromic, if $\gamma, \tau \gamma \tau^{-1}, \dots \tau^m \gamma \tau^{-m}$ generate an elementary subgroup, then so do $\gamma$ and $\tau$.
\end{defi}

The latter two quantities provide the following analogue of the Margulis lemma for manifolds with pinched negative curvature, \cite[Proposition~3.5]{Coulon:2018vp}.
See also \cite[Proposition 3.44]{Coulon:2016if}.

\begin{prop}
\label{res: margulis lemma}
	For every $d \in \R_+$, we have
	\begin{equation*}
		A(\Gamma,X, d) \leq \left[ \nu(\Gamma, X) + 3\right]d + A(\Gamma,X,400\delta) + 24\delta.
	\end{equation*}
\end{prop}

This motivates the next definition

\begin{defi}[Acylindricity]
	The \emph{(global) acylindricity parameter}, denoted by $A(\Gamma,X)$, is the smallest non-negative number such that for every $d \in \R_+$,
	\begin{equation*}
		A(\Gamma,X, d) \leq \left[ \nu(\Gamma, X) + 3\right]d + A(\Gamma,X).
	\end{equation*}
\end{defi}

\begin{rema*}
	The definition of $A(\Gamma, X)$ slightly differs from Delzant-Gromov \cite{Delzant:2008tu} and Coulon \cite{Coulon:2014fr,Coulon:2018vp}.
	The advantage of this new definition is that it does not involve the hyperbolicity constant $\delta$.
	Nevertheless it plays the exact same role.
	In particular, $A(\Gamma, X)$ satisfies the following homogeneity: if $\epsilon X$ stands for the space $X$ rescaled by a factor $\epsilon \in \R_+^*$, then $A(\Gamma, \epsilon X) = \epsilon A(\Gamma, X)$.
	Moreover $A(\Gamma, X) \leq A(\Gamma,X,400\delta) + 24\delta$.
\end{rema*}


%
\section{Approximation of periodic groups}
%
\label{sec: approx periodic groups}

Our goal is to study morphisms of the form $\phi \colon G \to \Gamma/ \Gamma^n$, where $\Gamma$ is a torsion-free hyperbolic group and $n$ a sufficiently large odd exponent.
As we explained in the introduction, one can produce a directed sequence of hyperbolic groups $(\Gamma_j)$ converging to $\Gamma/ \Gamma^n$.
Our strategy is to take advantage of this feature.
Indeed if $G$ is a finitely presented group, then every morphism $\phi \colon G \to \Gamma / \Gamma^n$ lifts to a morphism $\tilde \phi \colon G \to \Gamma_j$, provided $j$ is large enough.
In this way we reduce our problem to the study of morphisms from $G$ to a class of suitable hyperbolic groups.
The purpose of this section is to make precise the kind of negatively curved groups that appear in this context.

%
\subsection{A preferred class of groups.}
\label{sec: a preferred class of groups}
%

\begin{defi}
\label{def: csa}
	 A group $G$ is \emph{Conjugately Separated Abelian} (\emph{CSA}) if every maximal abelian subgroup of $G$ is malnormal.
\end{defi}

\paragraph{General settings.}
Let $\delta, \rho \in \R_+^*$.
In this section we consider triples $(\Gamma,X,\mathcal C)$ such that 
\begin{itemize}
	\item $X$ is a $\delta$-hyperbolic length space,
	\item $\Gamma$ is a CSA group acting by isometries on $X$,
	\item $\mathcal C$ is $\Gamma$-invariant, $2\rho$-separated subset of $X$ (that is $\dist c{c'} \geq 2\rho$ for every distinct $c,c' \in \mathcal C$).
\end{itemize}

\begin{voca}
	The elements of $\mathcal C$ are called \emph{apices} or \emph{cone points}.
	An element $\gamma \in \Gamma$ is \emph{conical} if it fixes a unique cone point.
	It is \emph{visible} if it is loxodromic or conical and \emph{elusive} otherwise.
	A subgroup $\Gamma_0 \subset \Gamma$ is \emph{elusive} if all its elements are elusive.	
	We chose the word ``elusive'' to stress that fact that, in practice, the action of $\Gamma_0$ on $X$ will not provide much information on the structure of $\Gamma_0$.
	Since $\mathcal C$ is $\Gamma$-invariant, the collection of visible elements is invariant under conjugation.
	Hence so it the collection of elusive subgroups.
	The latter is also stable under taking subgroups.

	We think of $\mathcal C$ as a thin-thick decomposition of $X$:
	the \emph{thin parts} are the balls $B(c, \rho)$, where $c$ runs over $\mathcal C$;
	the \emph{thick part}, that we denote by $X^+$, is 
	\begin{equation*}
		X^+ = X \setminus \bigcup_{c \in \mathcal C} B(c,\rho).
	\end{equation*}
\end{voca}

\begin{defi}
\label{def: preferred class}
We write $\mathfrak H_\delta(\rho)$ for the set of all triples $(\Gamma, X, \mathcal C)$ as above which satisfy the following additional properties.
\begin{labelledenu}[H]
	\item \label{enu: family axioms - acylindricity}
	$\nu(\Gamma, X) = 1$ and $A(\Gamma, X) \leq \delta$.
	
	\item \label{enu: family axioms - elem subgroups}
	Every elementary subgroup of $\Gamma$ is either elusive or cyclic (finite or infinite).
	\item \label{enu: family axioms - elusive}
	Every non-trivial elusive element of $\Gamma$ is elliptic and $\delta$-thin.
	
	\item \label{enu: family axioms - conical - 1}
	Every element $\gamma \in \Gamma \setminus\{1\}$ fixing a cone point $c \in \mathcal C$ is $(\rho+ \delta)$-thin at $c$.
	
	\item \label{enu: family axioms - conical - 2}
	For every $c \in \mathcal C$, there exists $\Theta \in [2\pi, \infty)$ and a $\stab c$-equivariant cocycle
	\begin{equation*}
		\theta \colon \mathring B(c,\rho)  \times \mathring B(c,\rho)\to \R/ \Theta\Z
	\end{equation*}
	such that for every $x,y \in \mathring B(c,\rho)$, we have $\abs{\dist xy - \ell} \leq \delta$, where
	\begin{equation}
	\label{eqn: family axioms - conical - 2}
		\cosh \ell = \cosh\dist cx\cosh\dist cy - \sinh\dist cx\sinh\dist cy \cos\left( \min \left\{ \pi, \tilde \theta(x,y) \right\}\right)
	\end{equation}
	and $\tilde \theta(x,y)$ is the unique representative of $\theta(x,y)$ in $(-\Theta/2, \Theta/2]$.
	Moreover, $\theta(\gamma x, x) \neq 0$, for every $\gamma \in \stab c\setminus \{1\}$, and $x \in \mathring B(c,\rho)$,
	
	\item \label{enu: family axioms - loxodromic}
	Every loxodromic element of $\Gamma$ is $\delta$-thin.
	\item \label{enu: family axioms - radial proj}
	For every $c \in \mathcal C$, for every $x \in B(c,\rho)$, there exists $y \in X^+$ such that $\dist xy = \rho - \dist xc$.
\end{labelledenu}
\end{defi}

\begin{nota*}
	For simplicity, we write $\mathfrak H$ for the union of all $\mathfrak H_\delta(\rho)$ where $\delta$ and $\rho$ run over $\R_+^*$.
	In practice we will consider situations where $\rho$ is very large compare to $\delta$.
\end{nota*}

\begin{rema}
Let us comment upon this set of axioms.
Note that we do not require $\Gamma$ to act properly on $X$.
All we will need is a weak form of acylindricity captured by \ref{enu: family axioms - acylindricity}.
This assumption is weaker than the usual acylindricity.
Indeed it does not provide any control on the injectivity radius for the action of $\Gamma$ on $X$.
Axiom~\ref{enu: family axioms - elem subgroups} refers to the algebraic structure of elementary subgroups of $\Gamma$.
Axioms~\ref{enu: family axioms - elusive}-\ref{enu: family axioms - loxodromic} give geometric informations on the action of $\Gamma$.
In particular, it follows from \ref{enu: family axioms - elusive} and \ref{enu: family axioms - conical - 1}, that $\Gamma$ has no parabolic subgroup.
In Axiom~\ref{enu: family axioms - conical - 2}, by cocycle we mean that for every $x,y,z \in \mathring B(c,\rho)$ we have $\theta(x,z) = \theta(x,y) + \theta(y,z)$.
The exact formula in (\ref{eqn: family axioms - conical - 2}) does not really matter. 
It provides a uniform control on the action of $\stab c$, i.e. which does not depend on the triple $(\Gamma, X, \mathcal C)$. 
It has the following interpretation:
the ball $B(c,\rho)$ is seen approximately as a cone of radius $\rho$ endowed with a hyperbolic metric whose total angle at the apex is $\Theta$.
The quantity $\theta(x,y)$ represents the angle at $c$ between $x$ and $y$.
With this idea in mind, one recognizes in (\ref{eqn: family axioms - conical - 2}) the law of cosines in $\H^2$.
\end{rema}

\begin{lemm}
\label{res; family - fully conical subgroup}
	Let $\delta, \rho \in \R_+^*$ with $\rho > \delta$.
	Let $(\Gamma, X, \mathcal C)$ in $\mathfrak H_\delta(\rho)$.
	Let $c \in \mathcal C$.
	Every non-trivial element in $\stab c$ is conical.
\end{lemm}

\begin{proof}
	Let $\gamma \in \stab c \setminus\{1\}$.
	Assume that contrary to our claim $\gamma$ is elusive.
	By \ref{enu: family axioms - elusive}, it is $\delta$-thin.
	Hence there exists $x \in X$ such that $\fix{\gamma}$ is contained in the ball $B(x, \delta)$.
	Since $\gamma$ fixes $c$, we have in particular $\dist xc \leq \delta$.
	Let $c' \in \mathcal C \setminus\{c\}$.
	It follows from the triangle inequality that
	\begin{equation*}
		\dist {c'} x \geq \dist {c'}c - \dist xc \geq 2 \rho - \delta. 
	\end{equation*}
	However by assumption, $\rho > \delta$.
	Consequently $c'$ does not belong to $B(x, \delta)$.
	Hence $\gamma$ cannot fix a cone point distinct from $c$, and therefore is visible, a contradiction.
\end{proof}

\begin{lemm}
	Let $\delta, \rho \in \R_+^*$ with $\rho > \delta$.
	\label{res: family - spreading visible}
	Let $(\Gamma, X, \mathcal C)$ in $\mathfrak H_\delta(\rho)$.
	Let $\gamma_1, \gamma_2 \in \Gamma\setminus\{1\}$ such that $[\gamma_1, \gamma_2] = 1$.
	If $\gamma_1$ is visible, then so is $\gamma_2$.
\end{lemm}

\begin{proof}
	By assumption $\gamma_1$ and $\gamma_2$ generate an elementary subgroup $\Gamma_0$ of $\Gamma$.
	Assume first that $\gamma_1$ is loxodromic.
	According to \ref{enu: family axioms - elem subgroups}, $\Gamma_0$ is cyclic, hence $\gamma_2$ is loxodromic as well.
	Suppose now that $\gamma_1$ is conical, fixing a cone point $c \in \mathcal C$ say.
	Since $\gamma_1$ commutes with $\gamma_2$, the point $\gamma_2c$ is also fixed by $\gamma_1$.
	Hence $\gamma_2c = c$.
	It follows from \autoref{res; family - fully conical subgroup} that $\gamma_2$ is conical.
\end{proof}

We sometimes make an abuse of notation and write $\mathfrak H_\delta(\rho)$ to designate the set of groups $\Gamma$ instead of the set of triples $(\Gamma, X, \mathcal C)$.
One should keep in mind that such a group always comes with a preferred action on a hyperbolic space.
This is useful for defining energies (see below).
The class $\mathfrak H_\delta(\rho)$ is stable under taking subgroup in the following sense: if $(\Gamma, X, \mathcal C)$ is an element of $\mathfrak H_\delta(\rho)$ and $\Gamma_0$ is a subgroup of $\Gamma$, then $(\Gamma_0, X, \mathcal C)$ also belongs to $\mathfrak H_\delta(\rho)$.

\paragraph{Energies.}

Let $G$ be a group and $U$ a finite subset of $G$.
We write $\card U$ for its cardinality.
We are going to investigate morphisms of the form $\phi \colon G \to \Gamma$, for some triple $(\Gamma, X, \mathcal C) \in \mathfrak H$.
If we need to emphasize the space $\Gamma$ is acting on, we write $\phi\colon G \to (\Gamma, X)$ or $\phi \colon G \to (\Gamma, X, \mathcal C)$.
Given $x \in X$, the $L^1$- and $L^\infty$-\emph{energies of $\phi$ at the point $x$}  (\emph{with respect to $U$}) are defined by 
\begin{align*}
	\lambda_1(\phi,U,x) & = \sum_{u \in U} \dist{\phi(u)x}x, \\
	\lambda_\infty(\phi,U,x) & = \sup_{u \in U} \dist{\phi(u)x}x.
\end{align*}

\begin{defi}[Energy]
\label{def: energy}
	Let $(\Gamma, X, \mathcal C) \in \mathfrak H$.
	Let $p \in \{1, \infty\}$.
	Given a homomorphism $\phi \colon G \to (\Gamma,X)$, the $L^p$-\emph{energy of $\phi$} (\emph{with respect to $U$}) is 
	\begin{equation*}
		\lambda_p(\phi,U) = \inf_{x \in X}\lambda_p(\phi,U,x) .
	\end{equation*}
\end{defi}

These energies are related as follows:
\begin{equation}
\label{eqn: comparing energies}
	 \lambda_\infty(\phi,U) \leq \lambda_1(\phi,U) \leq \card U\lambda_\infty(\phi,U).
\end{equation}
Later in the article, we will chose a point $o \in X$ minimizing one of these energies and treat is as a basepoint.
Nevertheless it will be convenient to keep the basepoint in the thick part $X^+$ of the space $X$.
This motivates the following variation of energy.

\begin{defi}[Restricted energy]
\label{def: restricted energy}
	Let $(\Gamma, X, \mathcal C) \in \mathfrak H$.
	Given a homomorphism $\phi \colon G \to (\Gamma,X)$, the \emph{restricted energy of $\phi$} (\emph{with respect to $U$}) is defined by
	\begin{equation*}
		\lambda^+_1(\phi,U) = \inf_{x \in X^+}\lambda_1(\phi,U,x).
	\end{equation*}
	(Note that $x$ runs over the thick part $X^+$ instead of $X$.)
\end{defi}

\begin{lemm}
\label{res: comparing energies}
	Let $\delta, \rho \in \R_+^*$ with $\rho > \delta$.
	Let $(\Gamma, X, \mathcal C) \in \mathfrak H_\delta(\rho)$. 
	Let $\phi \colon G \to (\Gamma, X)$ be a homomorphism and $U$ a finite subset of $G$.
	Assume that the image of $\phi$ is not abelian, then 
	\begin{equation*}
		\lambda_\infty(\phi,U) \leq \lambda_1(\phi,U) \leq \lambda_1^+(\phi,U) \leq 2 \card U \lambda_\infty(\phi,U).
	\end{equation*}
	Moreover, for every $\epsilon > 12\delta$, there exists $x \in X$ and $x^+ \in X^+$ such that 
	\begin{enumerate}
		\item $\lambda_\infty(\phi,U,x) \leq \lambda_\infty(\phi,U) + \epsilon$ and $\lambda_1(\phi,U,x^+) \leq \lambda^+_1(\phi,U) + \epsilon$,
		\item $\dist x{x^+} \leq \card U \lambda_\infty(\phi,U) + \epsilon$.
	\end{enumerate}
\end{lemm}

\begin{proof}
	The first inequalities $\lambda_\infty(\phi,U) \leq \lambda_1(\phi,U) \leq \lambda_1^+(\phi,U)$ follow from the definition of the various energies.
	Let $x \in X$.
	We claim that 
	\begin{equation*}
		2d(x, \mathcal C) \geq 2\rho - \lambda_\infty(\phi,U,x).
	\end{equation*}
	Indeed for every $u \in U$ we have $\dist{\phi(u)x}x \leq  \lambda_\infty(\phi,U,x)$.
	Hence, if our claim fails, it follows from the triangle inequality that there exists $c \in \mathcal C$ such that for every $u \in U$, we have $\dist{\phi(u)c}c < 2\rho$.
	Consequently the image of $\phi$ is contained in $\stab c$.
	By \autoref{res; family - fully conical subgroup}, if $\stab c$ is not trivial, then it cannot be elusive. 
	Axiom \ref{enu: family axioms - elem subgroups} forces $\stab c$, and thus the image of $\phi$, to be abelian.
	This contradicts our assumption and completes the proof of our claim.
	By \ref{enu: family axioms - radial proj}, there exists $y \in X^+$ such that $\dist xy \leq \lambda_\infty(\phi,U,x)/2$.
	It follows then from the triangle inequality that 
	\begin{equation*}
		\lambda_1^+(\phi,U) \leq \lambda_1(\phi,U,y) \leq \lambda_1(\phi,U,x) + 2\card U \dist xy \leq 2\card U \lambda_\infty(\phi,U,x).	
	\end{equation*}
	This inequality holds for every $x \in X$, therefore $\lambda_1^+(\phi,U) \leq 2 \card U \lambda_\infty(\phi,U)$.
	
	\medskip
	Let us focus now on the second part of the statement.
	Let $\epsilon > 12\delta$ and $x^+ \in X^+$ be a point in the thick part such that $\lambda_1(\phi,U,x^+) \leq \lambda_1^+(\phi,U) + \epsilon$.
	For simplicity we let 
	\begin{equation*}
		Y = \fix{\phi(U), \lambda_\infty(\phi,U) + \epsilon}.
	\end{equation*}
	If $x^+$ already belongs to $Y$, then we simply take $x = x^+$.
	Otherwise, we denote by $x$ a $\delta$-projection of $x^+$ on $Y$.
	Inequality (\ref{eqn: displacement outside fixed set}) yields
	\begin{equation*}
		\lambda_\infty(\phi,U,x^+) \geq 2 \dist {x^+}x  + \lambda_\infty(\phi,U).
	\end{equation*}
	In particular, 
	\begin{equation*}
		\dist x{x^+} \leq \frac 12\lambda_\infty(\phi,U,x^+)  \leq \frac 12 \left[ \lambda_1^+(\phi,U) + \epsilon\right]  \leq \card U \lambda_\infty(\phi,U) + \epsilon. \qedhere
	\end{equation*}
\end{proof} 

%
\subsection{Lifting morphisms}
%

Let us now explain how the class of groups $\mathfrak H$ introduced in the previous section naturally appears in the study of periodic groups.

\paragraph{Bootstrap and approximations.}
Recall that the acylindricity parameter, the injectivity radius and the $\nu$-invariant have been defined in \autoref{sec: invariants}.

\begin{defi}
\label{def: bootstrap}
	Let $\tau \in (0,1)$.
	A \emph{$\tau$-bootstrap} is a pair $(\Gamma, X)$ where $X$ is a $\delta$-hyperbolic length space for some $\delta \in \R_+^*$, endowed with a non-elementary action by isometries of a CSA group $\Gamma$ with the following properties.
	\begin{itemize}
		\item $\nu(\Gamma, X) = 1$, $A(\Gamma,X) \leq \delta$, and $\inj[X]{\Gamma} \geq \tau \delta$.
		\item Every elementary subgroup of $\Gamma$ which is not elliptic is loxodromic and cyclic.
		\item Every element of  $\Gamma$ (elliptic or loxodromic) is $\delta$-thin.
	\end{itemize}
	Let $n \in \N$.
	If in addition every elliptic element in $\Gamma$ has finite order dividing $n$, we say that $(\Gamma, X)$ is a \emph{$\tau$-bootstrap for the exponent $n$}.
\end{defi}

\begin{exam}
	For our applications, we have two main examples of bootstraps in mind.
	\begin{itemize}
		\item $\Gamma$ is a non-elementary, torsion-free, hyperbolic group acting on its Cayley graph $X$.
		\item $\Gamma$ is a non-trivial free product of CSA groups with no even torsion, acting on the corresponding Bass-Serre tree $X$.
	\end{itemize}
\end{exam}

Let $(\Gamma_1, X_1)$ and $(\Gamma_2, X_2)$ be two pairs where $X_i$ is a metric space and $\Gamma_i$ a group acting by isometries on $X_i$.
A \emph{morphism} $(\Gamma_1, X_1) \to (\Gamma_2, X_2)$ is a pair $(\pi, f)$ where $\pi$ is a homomorphism from $\Gamma_1$ to $\Gamma_2$ and $f$ a $\pi$-equivariant map from $X_1$ to $X_2$ (we do not require $f$ to be an isometry).
Such a morphism is called an \emph{epimorphism} if $\pi$ is onto (we do not require $f$ to be onto).
The statement below summarizes all the properties satisfied by our approximation groups.

\begin{theo}
\label{res: approximating sequence}
	There exist $\delta \in \R_+^*$ and a non-decreasing function $\rho \colon \N \to \R_+^*$ diverging to infinity with the following properties.
	For every $\tau \in (0,1)$, there is a critical exponent $N_\tau \in \N$ such that for every odd integer $n \geq N_\tau$, the following holds.
	
	Let $(\Gamma, X)$ be a $\tau$-bootstrap for the exponent $n$.
	The quotient $\Gamma / \Gamma^n$ is infinite.
	Moreover, there exists a sequence of epimorphisms
	\begin{equation*}
		(\Gamma_0, X_0) \onto (\Gamma_1, X_1) \onto \dots \onto (\Gamma_j, X_j) \onto (\Gamma_{j+1}, X_{j+1}) \onto \dots
	\end{equation*}
	with the following properties.
	\begin{enumerate}
		\item \label{enu: approximating sequence - init}
		$\Gamma_0 = \Gamma$ while $X_0$ is a rescaled version of $X$, that does not depend on $n$.
		\item \label{enu: approximating sequence - cvg}
		The direct limit of the sequence $(\Gamma_j)$ is isomorphic to $\Gamma/\Gamma^n$.
		Moreover, if every finite subgroup of $\Gamma$ is cyclic, then the same holds for $\Gamma / \Gamma^n$.
		\item \label{enu: approximating sequence - control}
		For every $j \in \N$, there is a subset $\mathcal C_j \subset X_j$ such that $(\Gamma_j, X_j, \mathcal C_j)$ belongs to the class $\mathfrak H_\delta(\rho(n))$.
		\item \label{enu: approximating sequence - metric}
		Let $j \in\N$.
		On the one hand, the map $X_j \to X_{j+1}$ is $\lambda(n)$-Lipschitz, where $\lambda(n) <1$ only depends on $n$, and converges to zero as $n$ tends to infinity.
		On the other hand, for every $x \in X_j$, the projection $\pi_j \colon \Gamma_j \onto \Gamma_{j+1}$ is one-to-one when restricted to the set
		\begin{equation*}
			\set{\gamma \in \Gamma_j}{\lambda(n)\dist{\gamma x}x \leq \frac{\rho(n)}{100}}.
		\end{equation*}
		\item \label{enu: approximating sequence - lifting}
		Let $j \in \N\setminus\{0\}$.
		Let $F$ be the free group generated by a finite set $U$.
		Let $\ell \in \N$ and $V$ be the set of elements of $F$ whose length (seen as words over $U$) is at most $\ell$.
		Let $\phi \colon F \to \Gamma_j$ be a homomorphism whose image does not fix a cone point in $\mathcal C_j$.
		Assume that  
		\begin{equation*}
			\lambda_\infty\left(\phi, U\right) < \frac {\rho(n)}{100\ell}
		\end{equation*}
		Then there exists a map $\tilde \phi \colon F \to \Gamma_{j-1}$ such that $\phi = \pi_{j-1} \circ \tilde \phi$ and $V \cap \ker {\phi} = V \cap \ker {\tilde \phi}$.
	\end{enumerate}
\end{theo}

The proof of \autoref{res: approximating sequence} relies on iterated geometric small cancellation theory.
It is postponed in \autoref{sec: approx}.

\begin{voca*}
	We call any sequence of triples $(\Gamma_j, X_j, \mathcal C_j)$ satisfying the conclusion of \autoref{res: approximating sequence} an \emph{approximation sequence of $\Gamma/ \Gamma^n$}.
\end{voca*}

\begin{rema}\
\label{rem: canonical proj asymp injective}
	\begin{itemize}
		\item 
		Being CSA is stable under direct limit, see for instance \cite{Myasnikov:1996aa}.
		Hence $\Gamma / \Gamma^n$ is CSA as well.
		\item 
		Assume that $\Gamma$ is a non-elementary, torsion-free, hyperbolic group.
		Using an induction on $j$ based on \autoref{res: approximating sequence}~\ref{enu: approximating sequence - metric} one proves the following fact:
		for every finite subset $U \subset \Gamma \setminus\{1\}$, there exists $N \in \N$ such that for every odd integer $n \geq N$, we have $U \cap \Gamma^n = \emptyset$.
		Although not directly stated in this form, this fact was already observed by Ol'shanski\u\i\ \cite{Olshanskii:1991tt}.
	\end{itemize}
\end{rema}

\paragraph{Lifting morphisms.}
Let $\tau \in (0,1)$.
We choose an odd integer $n \geq N_\tau$, where $N_\tau$ is the critical exponent given by \autoref{res: approximating sequence}.
Fix a $\tau$-bootstrap $(X, \Gamma)$ for the exponent $n$.
Let $(\Gamma_j, X_j, \mathcal C_j)$ be a sequence of approximations of $\Gamma/\Gamma^n$.
For every $j \in \N$, we write $\sigma_j$ for the natural epimorphism $\sigma_j \colon \Gamma_j \onto \Gamma/\Gamma^n$.
Similarly we write $\sigma$ for the projection $\sigma \colon \Gamma \to \Gamma/ \Gamma^n$.

\begin{defi}
\label{def: approx of morphism}
	Let $G$ be a discrete group, $V$ a subset of $G$, and $\phi \colon G \to \Gamma/\Gamma^n$ a morphism.
	An \emph{approximation of $\phi$ relative to $V$} is a morphism $\tilde \phi \colon G \to \Gamma_j$ for some $j \in \N$ such that $\phi = \sigma_j \circ \tilde \phi$ and $V \cap \ker \tilde \phi = V \cap \ker \phi$.
	Alternatively, we say that $\tilde \phi$ \emph{lifts $\phi$ relative to $V$}.
	Such an approximation is \emph{minimal} if $j$ is minimal for the above properties.
\end{defi}

Recall that $\Gamma/\Gamma^n$ is the direct limit of $\Gamma_j$.
Hence if $G$ is finitely presented and $V$ is finite, then any morphism $\phi \colon G \to \Gamma/ \Gamma^n$ admits a minimal approximation relative to $V$.

\begin{prop}
\label{res: approx - energy minimal approx}
	Let $G = \group{ U | R}$ be finitely presented group.
	Let $\ell \in \N$ bounded from below by the length of any relation in $R$ (seen as words over $U$).
	Let $V$ be the set of all elements of $G$ of length at most $\ell$ (for the word metric with respect to $U$).
	Let $\phi \colon G \to \Gamma/ \Gamma^n$ be a morphism whose image is not abelian.
	If $\phi_j \colon G \to \Gamma_j$ is a minimal approximation of $\phi$ relative to $V$, then either $j = 0$ or 
	\begin{equation*}
		\lambda_\infty(\phi_j, U) \geq \frac {\rho(n)}{100\ell}
	\end{equation*}
\end{prop}

\begin{proof}
	The proof is by contradiction. 
	Assume indeed that $j \geq 1$ and 
	\begin{equation*}
		\lambda_\infty(\phi_j, U) < \frac {\rho(n)}{100\ell}
	\end{equation*}
	Note that the image of $\phi_j$ does not fix a cone point $c \in \mathcal C_j$.
	Indeed, otherwise, it would be abelian by Axiom~\ref{enu: family axioms - elem subgroups} and thus the image of $\phi$ as well.
	The presentation of $G$ defines an epimorphism $\chi \colon F \onto G$, where $F$ is the free group generated by $U$.
	We denote by $\tilde V$ the set of all elements of $F$ of length at most $\ell$ so that $\chi$ maps $\tilde V$ onto $V$.
	It follows from our choice of $\ell$, that $R$ is contained in $\tilde V$.
	We apply \autoref{res: approximating sequence}~\ref{enu: approximating sequence - lifting}, to the morphism $\psi_j = \phi_j \circ \chi$.
	There exists a morphism $\psi_{j-1} \colon F \to \Gamma_{j-1}$ such that
	\begin{enumerate}
		\item $\phi \circ \chi = \sigma_{j-1} \circ \psi_{j-1}$, and 
		\item $\tilde V \cap \ker \psi_{j-1} = \tilde V \cap \ker\psi_j =\tilde  V \cap \ker \left(\phi \circ \chi\right)$.
	\end{enumerate}
	The last item tells us that $R$ lies in the kernel of $\psi_{j-1}$.
	Hence $\psi_{j-1}$ factors through $\chi$.
	The resulting morphism $\phi_{j-1} \colon G \to \Gamma_{j-1}$ is an approximation of $\phi$ relative to $V$, which contradicts the minimality of $j$.
\end{proof}

%
\section{The space of marked groups}
%
\label{sec: space marked groups}

Our study of ${\rm Hom}(G, \Gamma / \Gamma^n)$ make use of \emph{limit groups} (see \autoref{def: limit group - general}).
They can be seen as limit of periodic groups in a compact space of marked groups.
In this section we recall some features of this topological approach.
We follow \cite{Champetier:2000jx,Champetier:2005ic}.

%
\subsection{Generalities}
%

Let $G$ be a finitely generated group.
A \emph{group marked by $G$} is a pair $(H,\phi)$ where $H$ is a group and $\phi \colon G \onto H$ an epimorphism from $G$ onto $H$.
If there is no ambiguity we omit the group $G$ and simply say that $(H,\phi)$ is a \emph{marked group}.
Given two groups $(H_1,\phi_1)$ and $(H_2,\phi_2)$ marked by $G$, we say that $(H_1,\phi_1)$ is a \emph{cover} of $(H_2,\phi_2)$ or $(H_2, \phi_2)$ is \emph{quotient} of $(H_1, \phi_1)$ and write $(H_2,\phi_2) \prec (H_1,\phi_1)$ if there exists an epimorphism $\theta \colon H_1 \onto H_2$ such that $\theta\circ \phi_1 = \phi_2$.
Equivalently it means that $\ker \phi_1 \subset \ker \phi_2$.
The relation $\prec$ defines a pre-order on the set of marked groups.
Two marked groups $(H_1,\phi_1)$ and $(H_2,\phi_2)$ are equivalent, if $(H_1,\phi_1)$ is a cover of $(H_2,\phi_2)$ and conversely.
Equivalently it means that there exists an isomorphism $\theta \colon H_1 \to H_2$ such that $\theta \circ \phi_1 = \phi_2$.

\begin{defi}
\label{def: space of marked groups}
	The \emph{space of groups marked by $G$} (or simply the \emph{space of marked groups}) is the set of equivalence classes of groups marked by $G$.
We denote it by $\mathfrak G(G)$.
\end{defi}

\begin{nota*}
	In the remainder of the article we make the following abuse of notation: given a group $(H,\phi)$ marked by $G$, we still denote by $(H,\phi)$ its equivalence class in $\mathfrak G(G)$.

\end{nota*}

We endow the space $\mathfrak G(G)$ with a topology (called the \emph{topology of marked groups}) defined as follows.
Let $(H,\phi)$ be a group marked by $G$.
Let $U$ be a finite set of $G$.
We let 
\begin{equation*}
	\mathfrak V_U(H,\phi) = \set{(H',\phi') \in \mathfrak G(G)}{U \cap \ker \phi = U \cap \ker \phi'}.
\end{equation*}
When $U$ runs over all finite subsets of $G$, the collection $\{\mathfrak V_U(H,\phi)\}_{U \subset G}$ describes a basis of open neighborhoods of $(H,\phi)$.
The space $\mathfrak G(G)$ endowed with this topology is metrizable and compact.
\begin{rema*}
	Sometimes it is convenient to consider pairs $(H,\phi)$, where the morphism $\phi \colon G \to H$ is not necessarily onto.
	In this situation we make an abuse of notation and write $(H, \phi)$ to mean $(\phi(G), \phi)$.
	In particular, we say that a sequence $(H_k, \phi_k)$ of such pairs converges to $(L, \eta)$ and write 
	\begin{equation*}
		(L, \eta) = \lim_{k \to \infty} (H_k,\phi_k),
	\end{equation*}
	if $(\phi_k(G), \phi_k)$ converges to $(L, \eta)$ in $\mathfrak G(G)$.
	In such a situation $L$ is the quotient of $G$ by the \emph{stable kernel} of $(\phi_k)$, i.e. the normal subgroup
	\begin{equation}
	\label{eqn: stable kernel}
		K = \set{\gamma \in G}{\phi_k(\gamma) = 1,\ \text{for all but finitely many}\ k \in \N}.
	\end{equation}
\end{rema*}

By construction, the pre-order $\prec$ induces an order on $\mathfrak G(G)$.
It is compatible with the topology in the sense that the set 
\begin{equation*}
	\set{\left((H_1,\phi_1), (H_2, \phi_2)\right) \in \mathfrak G(G) \times\mathfrak G(G)}{ (H_1, \phi_1) \prec (H_2, \phi_2) }
\end{equation*}
is closed for the product topology.

\paragraph{Changing the marker.}
Given an epimorphism $\pi \colon G \onto G'$, we define a map $\pi_\ast \colon \mathfrak G(G') \to \mathfrak G(G)$ by sending $(H,\phi)$ to $(H,\phi \circ \pi)$.
This is an order-preserving homeomorphism from $\mathfrak G(G')$ onto its image.
In addition, if the kernel of $\pi \colon G \to G'$ is the normal closure of a \emph{finite} subset of $G$, then the image of $\pi_\ast$ is an open subset of $\mathfrak G(G)$.

%
\subsection{Limit groups.}
%

A \emph{filtration} of groups is a sequence $\mathfrak F =(\mathfrak F_k)$ of collections of groups that is non-increasing for the inclusion (i.e. $\mathfrak F_k \subset \mathfrak F_{k'}$ for every $k' \leq k$).
Given such a filtration we write $\mathfrak F_k(G)$ for the groups $(H, \phi)$ marked by $G$ such that $H$ belongs to $\mathfrak F_k$.

\begin{defi}
\label{def: limit group - general}
	Let $\mathfrak F =(\mathfrak F_k)$ be a filtration of groups.	
	A \emph{limit group over $\mathfrak F$} is a marked group in the set
	\begin{equation*}
		\bigcap_{k \in \N} \overline{\mathfrak F_k(G)}.
	\end{equation*}
	where $\overline{\mathfrak F_k(G)}$ stands for the closure of $\mathfrak F_k(G)$ in $\mathfrak G(G)$.
\end{defi}

\begin{voca*}
	A sequence $(H_k, \phi_k)$ of marked groups is an \emph{$\mathfrak F$-sequence} if $(H_k, \phi_k)$ belongs to $\mathfrak F_k(G)$ for every $k \in \N$.
	Since $(\mathfrak F_k)$ is non-increasing for the inclusion, any subsequence of an $\mathfrak F$-sequence is still an $\mathfrak F$-sequence.
	Concretely, $(L, \eta)$ is a limit group over $\mathfrak F$, if and only if it is the limit of an $\mathfrak F$-sequence.
\end{voca*}

We will be interested in the following particular situations.
\begin{itemize}
	\item A \emph{limit group over a group $\Gamma$} is a limit group over the constant filtration that consists of all subgroups of $\Gamma$.
	\item Let $\Gamma$ be a group.
	Given $k \in \N$, we write $\mathfrak F_k$ for the union over all $n \geq k$ of all subgroups of the periodic quotient $\Gamma/ \Gamma^n$.
	A \emph{limit group over the periodic quotients of $\Gamma$} is a limit group over the filtration $\mathfrak F = (\mathfrak F_k)$.
	\item Let $\delta \in \R_+^*$.
	Given $k \in \N$, denote by $\mathfrak F_k$,  the collection 
	\begin{equation*}
		\mathfrak F_k = \bigcup_{ \rho \geq k} \set{ \Gamma}{ (\Gamma, X, \mathcal C) \in \mathfrak H_\delta(\rho)}.
	\end{equation*}
	A \emph{limit group over $\mathfrak H_\delta$} is a limit group over the filtration $\mathfrak F = (\mathfrak F_k)$.
\end{itemize}
	
In the latter case, we borrow the  following terminology from Groves and Hull \cite{Groves:2019aa}.
\begin{defi}
\label{def: div limit group}
	Let $U$ be a finite generating set of $G$.
	A limit group $(L, \eta)$ over $\mathfrak H_\delta$ is \emph{divergent} if there exist a sequence $(\rho_k)$ diverging to infinity, a sequence of marked groups $(\Gamma_k, \phi_k)$, where $(\Gamma_k, X_k, \mathcal C_k) \in \mathfrak H_\delta(\rho_k)$ for every $k \in \N$, which converges to $(L, \eta)$ and such that the energy $\lambda_\infty(\phi_k, U)$ measured in $X_k$ diverges to infinity.
\end{defi}

\begin{rema}
	Note that the definition does not depend on the generating set $U$.
	Because of (\ref{eqn: comparing energies}), we could also have replaced the $L^\infty$-energy, by the $L^1$-energy.
	If $L$ is not abelian, it is also equivalent to the fact that the restricted energy $\lambda_1^+(\phi_k,U)$ diverges to infinity (\autoref{res: comparing energies}).
\end{rema}

We now prove a few statements which are not necessarily needed for our study, but should help to understand the relations between these various types of limit groups.
\autoref{res: trichotomy limit groups} can also be seen as a first warm up for the lifting statements proved in \autoref{sec: lifting theorems}.

\begin{prop}
	Let $\Gamma$ be a non-elementary, torsion-free, hyperbolic group.
	Let $G$ be a finitely generated group.
	Let $(L, \eta) \in \mathfrak G(G)$.
	If $(L, \eta)$ is a limit group over $\Gamma$ then it is also a limit group over the periodic quotients of $\Gamma$.
\end{prop}

\begin{proof}
	Let $\phi_k \colon G \to \Gamma$ be a sequence of morphisms so that $(\Gamma, \phi_k)$ converges to $(L, \eta)$.
	Let $(U_k)$ be a non-decreasing exhaustion of $G$ by finite subsets.
	For every $k \in \N$, there exists an odd integer $n_k \geq k$ such that the projection $\pi_k \colon \Gamma \onto \Gamma / \Gamma^{n_k}$ is one-to-one when restricted to $\phi_k(U_k)$ (\autoref{rem: canonical proj asymp injective}).
	It follows directly that $(\Gamma / \Gamma^{n_k}, \pi_k \circ \phi_k)$ converges to $(L, \eta)$ as well.
\end{proof}

\begin{prop}
\label{res: trichotomy limit groups}
	Let $\Gamma$ be a non-elementary, torsion-free, hyperbolic group.
	There exists $\delta \in \R_+^*$ with the following property.
	Let $G$ be a finitely presented group.
	Let $(L,\eta) \in \mathfrak G(G)$ be a limit group over the periodic quotients of $\Gamma$.
	Then one of the following holds
	\begin{itemize}
		\item $L$ is a subgroup of $\Gamma$,
		\item $L$ is an abelian group.
		\item $(L, \eta)$ is a divergent limit group over $\mathfrak H_\delta$,
	\end{itemize}
\end{prop}

\begin{proof}
	Let $\delta \in \R^*_+$ and $\rho \colon \N \to \R_+$ be the data given by \autoref{res: approximating sequence}.
	We assume that $L$ is not abelian.
	There exists a sequence of morphisms $\phi_k \colon G \to \Gamma / \Gamma^{n_k}$ where $(n_k)$ is a sequence of odd integers diverging to infinity and such that $(\Gamma / \Gamma^{n_k}, \phi_k)$ converges to $(L, \eta)$.
	In particular, $\rho(n_k)$ diverges to infinity.
	We choose a non-decreasing sequence of integers $(\ell_k)$ diverging to infinity and such that 
	\begin{equation*}
		\lim_{k \to \infty} \frac{\rho(n_k)}{\ell_k} = \infty.
	\end{equation*}
	Let $U$ be a finite generating set of $G$.
	For every $k \in \N$, we denote by $U_k$ the ball of radius $\ell_k$ in $G$ (for the word metric with respect to $U$).
	Let $X$ be a Cayley graph of $\Gamma$ so that $(\Gamma, X)$ is a $\tau$-bootstrap for some $\tau \in (0,1)$.
	For each $k \in \N$, \autoref{res: approximating sequence}, provides an approximating sequence 
	\begin{equation*}
		(\Gamma_{j,k}, X_{j,k}, \mathcal C_{j,k})_{j \in \N}
	\end{equation*}
	of $\Gamma_k/\Gamma_k^{n_k}$ whose elements belong to $\mathfrak H_\delta(\rho(n_k))$.
	We denote by $\psi_k \colon G \to \Gamma_{j_k,k}$ a minimal approximation of $\phi_k$ relative to $U_k$.
	Since
	\begin{equation*}
		U_k \cap \ker \psi_k = U_k \cap \ker \phi_k, \quad \forall k \in \N,
	\end{equation*}
	the sequence $(\Gamma_{j_k, k}, \psi_k)$ also converges to $(L, \eta)$.
	We now distinguish two cases.
	Suppose first that the energy $\lambda_\infty(\psi_k, U)$ is not bounded.
	Up to passing to a subsequence we can assume that $\lambda_\infty(\psi_k, U)$ diverges to infinity.
	Hence $(L, \eta)$ is a diverging limit group over $\mathfrak H_\delta$.
	
	Suppose now that $\lambda_\infty(\psi_k, U)$ is bounded.
	We have chosen $\ell_k$ such that $\rho(n_k) / \ell_k$ diverges to infinity.
	By construction $\psi_k$ is a minimal approximation of $\phi_k$ relative to $U_k$.
	Its image is not abelian (provided $k$ is large enough) otherwise $L$ would be abelian as well.
	Applying \autoref{res: approx - energy minimal approx} with the minimal approximation $\psi_k \colon G \to \Gamma_{j_k,k}$, we obtain that $j_k = 0$ for all but finitely many $k \in \N$.
	In other words, $\psi_k$ is a morphism from $G$ to $\Gamma$.
	Recall that $\lambda_\infty(\psi_k, U)$ is bounded.
	Up to passing to a subsequence there exists $\psi \colon G \to \Gamma$ such that $\psi_k$ is conjugated to $\psi$ for every $k \in \N$.
	In particular, $L \cong G / \ker \psi$ is a subgroup of $\Gamma$.
\end{proof}

%
\subsection{Infinite descending sequence}
%
\label{sec: infinite desc  seq}

For the remainder of \autoref{sec: infinite desc  seq}, we fix a filtration of groups $\mathfrak F =(\mathfrak F_k)$.
Let $G$ be a finitely generated group.
As before we write $\mathfrak F_k(G)$ for the groups $(H, \phi)$ marked by $G$ such that $H$ belongs to $\mathfrak F_k$.
In the remainder of this section, unless mentioned otherwise, all limit groups will be limit groups over $\mathfrak F$.
To lighten the statements, we simply refer to them as \emph{limit groups}.

One of the central tools in the study of limit groups is the \emph{shortening argument} of Rips and Sela \cite{Rips:1994jg}.
Given an $\mathfrak F$-sequence $(H_k, \phi_k)$ converging to a limit group $(L, \eta)$ it provides (under certain circumstances) a new $\mathfrak F$-sequence $(H'_k, \phi'_k)$ converging to a proper quotient $(L', \eta')$ of $(L, \eta)$.
This construction has numerous applications which are now standard in the area.
We postpone the proof of the shortening argument in \autoref{sec: limit group w/o root splitting aka shortening}.
Instead we provide in this section a general abstract framework to investigate its consequences.
We revisit some material from \cite{Sela:2001gb,Sela:2009bh,Groves:2005ht, Weidmann:2019ue,Groves:2019aa}.
In particular, we show that if every limit group is shortenable (see \autoref{def: shortenable}) then there is no infinite descending sequence of limit groups.

\paragraph{Factorization property.}
If $(L, \eta)$ is a limit group over the free group, it is known that $L$ is finitely presented \cite{Kharlampovich:1998vl,Sela:2001gb,Guirardel:2004aa}.
In particular, the image of $\mathfrak G(L)$ in $\mathfrak G(G)$ is open.
This is no more the case a priori if $(L, \eta)$ is a limit group over $\mathfrak F$.
This motivates the next definition.

\begin{defi}
\label{def: factorization property}
	Let $G$ be a finitely generated group.
	We say that $(Q, \pi) \in \mathfrak G(G)$ has the \emph{factorization property} if there exist a finite subset $W \subset \ker \pi$ and $k \in \N$, with the following property:
	for every marked group $(H, \phi) \in \mathfrak F_k(G)$, if $W \subset \ker \phi$, then $\phi$ factors through $\pi$.
\end{defi}

Note that if $(Q, \pi)$ is a finitely presented group, then it has the factorization property.

\begin{lemm}
\label{res: factorization prop implies almost open}
	Let $G$ be a finitely generated group.
	Let $(Q, \pi) \in \mathfrak G(G)$ with the factorization property.
	There exists a finite subset $W \subset \ker \pi$ such that for every limit group $(L, \eta) \in \mathfrak G(G)$, if $W \subset \ker \eta$, then $(L, \eta) \prec (Q, \pi)$.
\end{lemm}

\begin{proof}
	Let $W \subset \ker \pi$ and $k \in \N$ be the data provided by the definition of the factorization property.
	Let $(L, \eta)$ be a limit group such that $W \subset \ker \eta$.
	Let $g \in \ker \pi$ and set $U = W \cup \{g\}$.
	Since $(L, \eta)$ is a limit group, there exists $(H,\phi) \in \mathfrak F_k(G)$ such that $U \cap \ker \phi = U \cap \ker \eta$.
	In particular, $W$ is contained in $\ker \phi$.
	Thus $\phi$ factors through $\pi$ and $\phi(g) = 1$.
	Consequently $\eta(g) = 1$ as well.
	We proved that $\ker \pi \subset \ker \eta$, whence the result.
\end{proof}

\begin{defi}
\label{def: fully res F}
	A finitely generated group $L$ is fully residually $\mathfrak F$ if for every finite subset $U \subset L \setminus\{1\}$, for every $k \in \N$, there is a morphism $\phi \colon L \to H$ where $H \in \mathfrak F_k$ and $U \cap \ker \phi = \emptyset$.
\end{defi}

Let $(L, \eta) \in \mathfrak G(G)$.
If $L$ is fully residually $\mathfrak F$, then $(L, \eta)$ is a limit group.
The converse requires the factorization property.

\begin{lemm}
\label{res: full res F w/ factorization}
	Let $G$ be a finitely generated group.
	Let $(L, \eta) \in \mathfrak G(G)$ be a limit group.
	If $(L, \eta)$ has the factorization property, then $L$ is fully residually $\mathfrak F$.
\end{lemm}

\begin{proof}
	Let $W \subset \ker \eta$ and $k$ be the data provided by the factorization property.
	Let $U \subset L \setminus \{1\}$ be a finite subset.
	We fix a finite subset $V \subset G$ containing $W$ such that $\eta(V) = U$.
	Since $(L, \eta)$ is a limit group, there exists $(H, \phi) \in \mathfrak F_k(G)$ such that $V \cap \ker \phi = V \cap \ker \eta$.
	In particular, $W$ is contained in $\ker \phi$, thus $\phi$ factors through $\eta$.
	The resulting morphism $\psi \colon L \to H$ is such that $U \cap \ker \psi = \emptyset$.
	Hence $(L, \eta)$ is fully residually $\mathfrak F$.
\end{proof}

\paragraph{Shortenable limit groups.}
For the remainder of this section, we fix a property $\mathcal P$ of groups.

\begin{exam*}
	In practice we shall use the following one: a group $H$ has $\mathcal P$ if every abelian subgroup of $H$ is finitely generated.
	Nevertheless in other contexts the required property can be slightly different.
	For instance \cite{Weidmann:2019ue} uses the following property:  group $H$ has $\mathcal P$ if every almost abelian subgroup of $H$ is finitely generated.
	Since this section aims to be as general as possible, we will not use a specific property $\mathcal P$ for the moment.
\end{exam*}

\begin{defi}
\label{def: shortenable}
	Let $G$ be a finitely generated group.
	Let $(H_k, \phi_k) \in \mathfrak G(G)$ be an $\mathfrak F$-sequence converging to a limit group $(L, \eta)$.
	We say that the sequence $(H_k, \phi_k)$ is \emph{shortenable with respect to $\mathcal P$} if there is an $\mathfrak F$-sequence $(H'_k, \phi'_k) \in \mathfrak G(G)$ such that, up to passing to a subsequence, the following holds.
	\begin{enumerate}
		\item \label{enu: shortening - sbgp}
			 $H'_k=H_k$, for every $k \in \N$.
		\item \label{enu: shortening - cvg}
			$(H'_k, \phi'_k)$ converges to a limit group $(L', \eta') \prec (L, \eta)$.
		\item \label{enu: shortening - dichotomy}
			If $(L', \eta')$ is not a proper quotient of $(L, \eta)$, then $L'$ has $\mathcal P$.
			Moreover, all but finitely many $\phi'_k$ factor through $\eta'$.
		\item \label{enu: shortening - fact}
			If $(L', \eta')$ has $\mathcal P$, then so has $(L, \eta)$.
			If in addition all but finitely many $\phi'_k$ factor through $\eta'$, then all but finitely many $\phi_k$ factor through $\eta$.
	\end{enumerate}
	In this context, we say that $(H'_k, \phi'_k)$ is a \emph{shortening sequence} of $(H_k, \phi_k)$ and $(L', \eta')$ a \emph{shortening quotient} of $(L, \eta)$.
	
	A limit group $(L, \eta)$ is \emph{shortenable with respect to $\mathcal P$} if every $\mathfrak F$-sequence converging to $(L, \eta)$ is shortenable with respect to $\mathcal P$.
\end{defi}

\begin{rema}
\label{rem: fp group are shortenable}\
	\begin{itemize}
		\item 
		Note that in the definition of shortenable sequences, we require that the target $H'_k$ of shortening morphism $\phi'_k$ is coincides with the one $H_k$ of the original morphism $\phi_k$.
		This plays an important role later, see for instance the footnote in the proof of \autoref{res: shortenable stable under free product}.
		
		\item Note that if $(L, \eta)$ is a limit group satisfying $\mathcal P$ and the factorization property, then it is automatically shortenable.
		Indeed if $(H_k, \phi_k)$ is an $\mathfrak F$-sequence converging to $(L, \eta)$ we can simply take as shortening sequence $(H'_k, \phi'_k) = (H_k, \phi_k)$.
		In particular, finitely presented limit groups satisfying $\mathcal P$ are shortenable.
	\end{itemize}
\end{rema}

We start with some stabilities of the shortenable property.
If the context is clear, we may omit the property $\mathcal P$ and simply write that a sequence / a limit group is shortenable.
The first statement directly follows from the definition.

\begin{lemm}[Changing the marker]
\label{res: shortenable changing marker}
	Let $G$ be a finitely generated group.
	Let $W \subset G$ be a finite subset and $\pi \colon G \onto Q$ the projection onto $Q = G / \normal W$.
	Let $(H_k, \phi_k) \in \mathfrak G(Q)$ be an $\mathfrak F$-sequence converging to a limit group $(L, \eta)$.
	Then $(H_k, \phi_k)$ is shortenable with respect to $\mathcal P$, as a sequence in $\mathfrak G(Q)$, if and only if $(H_k, \phi_k \circ \pi)$ is shortenable with respect to $\mathcal P$ as a sequence in $\mathfrak G(G)$.
	
	In particular, $(L, \eta)$ is shortenable if and only if so is $(L, \eta \circ\pi)$.
\end{lemm}

\begin{prop}[Free product stability]
\label{res: shortenable stable under free product}
	Assume that $\mathcal P$ is stable under free products and taking subgroups.
	Let $G$ be a finitely generated group which splits as a free product $G = G_1 \ast G_2$.
	Let $(L, \eta) \in \mathfrak G(G)$ be a limit group which decomposes as $L = L_1 \ast L_2$ so that $\eta$ maps $G_1$ and $G_2$ onto $L_1$ and $L_2$ respectively.
	We denote by $\eta_i \colon G_i \onto L_i$ the restriction of $\eta$ to $G_i$.
	If for every $i \in \{1, 2\}$, the marked group $(L_i, \eta_i)$ is shortenable with respect to $\mathcal P$ -- viewed as an element of $\mathfrak G(G_i)$ -- then so is $(L, \eta)$.
\end{prop}

\begin{proof}
	Let $(H_k, \phi_k)$ be an $\mathfrak F$-sequence converging to $(L, \eta)$.
	Given $i \in \{1, 2\}$, we denote by $\phi_k^i$ the restriction of $\phi_k$ to $G_i$ and $H^i_k \subset H_k$ its image.
	Note that $H^1_k$ and $H^2_k$ generate $H_k$.
	By assumption, $(L_i, \eta_i)$ is shortenable, viewed as a limit group in $\mathfrak G(G_i)$.
	In particular, \emph{any} subsequence of $(H^i_k, \phi_k^i)$ is shortenable.
	We denote by $\psi_k^1 \colon G_1 \onto H^1_k$ and $\psi_k^2 \colon G_2 \onto H^2_k$ for the sequences obtained by shortening successively $(H^1_k, \phi_k^1)$ and $(H^2_k, \phi_k^2)$.
	We write $(L'_i, \eta'_i)$ for the corresponding shortening quotient of $(L_i, \eta_i)$.
	Let $\psi_k \colon G \onto H_k$ be the morphism whose restriction to $G_1$ (\resp $G_2$) is $\psi_k^1$ (\resp $\psi_k^2$).\footnote{
	We have used here the fact that the morphism $\phi^i_k$ and its shortening $\psi^i_k$ have the same image (which are also subgroups of $H_k$).
	It allows us to build a map $\psi_k$ whose target $H_k$ belongs to the filtration $\mathfrak F$.
	Without this property of shortenings an alternative strategy would have been to consider the map $G \onto H^1_k \ast H^2_k$ built out of $\psi^1_k$ and $\psi^2_k$.
	However there is a priori no reason that $H^1_k \ast H^2_k$ is an element of the filtration $\mathfrak F$.}
	We define $\eta' \colon G \onto L'_1 \ast L'_2$ in the same way.
	Up to passing to a subsequence, $(H_k, \psi_k)$ converges to a limit group $(Q, \nu) \in \mathfrak G(G)$.
	We are going to prove that $(H_k, \psi_k)$ is a shortening sequence of $(H_k, \phi_k)$.
	\begin{itemize}
		\item By construction, we have
		\begin{equation*}
			(Q, \nu) \prec (L'_1 \ast L'_2, \eta') \prec (L, \eta).
		\end{equation*}
		which proves Point~\ref{enu: shortening - cvg} of \autoref{def: shortenable}.
		
		\item Let us focus on Point~\ref{enu: shortening - dichotomy} of \autoref{def: shortenable}.
		Assume that $Q$ is not proper quotient of $L$.
		It follows that $(L'_i,\eta'_i) = (L_i, \eta_i)$ for every $i \in \{1, 2\}$.
		Recall that $(L'_i, \eta'_i)$ has been obtained from $(L_i, \eta_i)$ by the shortening procedure.
		According to \autoref{def: shortenable}~\ref{enu: shortening - dichotomy} both $(L_1, \eta_1)$ and $(L_2, \eta_2)$ satisfy $\mathcal P$.
		Since $\mathcal P$ is stable under taking free product, $(L, \eta)$ has $\mathcal P$.
		On the other hand, we get that all but finitely many $\psi_k^i$ factor through $\eta'_i$.
		Hence all but finitely many $\psi_k$ factor through $\eta'$.
		This proves Point~\ref{enu: shortening - dichotomy} of \autoref{def: shortenable}.

		\item We are left to prove Point~\ref{enu: shortening - fact} of \autoref{def: shortenable}.
		To that end we assume that $(Q, \nu)$ has $\mathcal P$.
		Note that the projection $L'_1 \ast L'_2 \onto Q$ induces an embedding from each $L'_i$ into $Q$.
		Since $\mathcal P$ is stable under taking subgroups, each $L'_i$ has $\mathcal P$.
		According to \autoref{def: shortenable}~\ref{enu: shortening - fact}, both $L_1$ and $L_2$ have $\mathcal P$, hence so has $L$ ($\mathcal P$ is stable under free products).
		Suppose in addition that all but finitely many $\psi_k$ factor through $\nu$.
		It follows from the construction that all but finitely many $\psi_k^i$ factor through $\eta'_i$ for every $i \in \{1, 2\}$.
		By \autoref{def: shortenable}~\ref{enu: shortening - fact}, all but finitely many $\phi_k^1$ and $\phi_k^2$ factor through $\eta_1$ and $\eta_2$.
		Hence all but finitely many $\phi_k$ factor through $\eta$.
		This proves Point~\ref{enu: shortening - fact} of \autoref{def: shortenable}. \qedhere
	\end{itemize}
\end{proof}

Given $(L, \eta) \in \mathfrak G(G)$, an \emph{infinite descending sequence starting at $(L, \eta)$} is a sequence of marked groups
\begin{equation*}
	(L, \eta) = (L_0, \eta_0) \succ (L_1, \eta_1) \succ (L_2, \eta_2) \succ \dots
\end{equation*}
such that $(L_{i+1}, \eta_{i+1})$ is a proper quotient of $(L_i, \eta_i)$, for every $i \in \N$.

\begin{prop}
\label{res: dichotomy idc vs factorization}
	Let $G$ be a finitely generated group.
	Let $(L, \eta) \in \mathfrak G(G)$ be a limit group.
	Assume that every limit group $(Q, \nu) \prec (L, \eta)$ is shortenable with respect to $\mathcal P$.
	Then one of the following holds.
	\begin{enumerate}
		\item \label{enu: dichotomy idc vs factorization - idc}
		There is an infinite descending sequence of limit groups starting at $(L, \eta)$.
		\item \label{enu: dichotomy idc vs factorization - fac}
		$(L, \eta)$ satisfies $\mathcal P$ and the factorization property.
	\end{enumerate}
\end{prop}

\begin{proof}
	We assume that \ref{enu: dichotomy idc vs factorization - fac} does hot hold.
	We are going to build by induction an infinite descending sequence of limit groups
	\begin{equation*}
		(L, \eta) = (L_0, \eta_0) \succ (L_1, \eta_1) \succ (L_2, \eta_2) \succ \dots
	\end{equation*}
	such that for every $i \in \N$, either $L_i$ does not satisfy $\mathcal P$ or $(L_i, \eta_i)$ has not the factorization property.
	We start with $(L_0, \eta_0) = (L, \eta)$.
	Let $i \in \N$ such that $(L_i, \eta_i)$ has been built.
	We are going to define an $\mathfrak F$-sequence converging to an auxiliary quotient $(Q, \nu)$ of $(L_i, \eta_i)$.
	To that end, we distinguish two cases.
	\begin{labelledenu}[C]
		\item \label{res: dichotomy idc vs factorization - P}
		Suppose that $(L_i, \eta_i)$ does not satisfies $\mathcal P$.
		We choose any $\mathfrak F$-sequence $(H_k, \phi_k)$ converging to $(L_i, \eta_i)$.
		In addition we let $(Q, \nu) = (L_i, \eta_i)$.
		\item \label{res: dichotomy idc vs factorization - fact}
		Suppose that $(L_i, \eta_i)$ does not have the factorization property.
		We choose a non-decreasing exhaustion $(W_k)$ of $\ker \eta_i$ by finite subsets.
		For every $k \in \N$, there exists $(H_k, \phi_k) \in \mathfrak F_k(G)$ such that $W_k \subset \ker \phi_k$, but $\phi_k$ does not factor through $\eta_i$.
		Up to passing to a subsequence, we may assume that $(H_k, \phi_k)$ converges to a limit group $(Q, \nu)$.
	\end{labelledenu}
	In both cases, $(Q, \nu)$ is a quotient of $(L, \eta)$ hence is shortenable.
	Let $(H'_k, \phi'_k)$ be a sequence obtained by shortening $(H_k, \phi_k)$.
	Up to passing to a subsequence, it converges to the shortening quotient $(L_{i+1}, \eta_{i+1})$ of $(Q, \nu)$.
	
	\begin{clai}
		Either $L_{i+1}$ has not $\mathcal P$ or $(L_{i+1}, \eta_{i+1})$ does not have the factorization property.
	\end{clai}
	
	Suppose on the contrary that $(L_{i+1}, \eta_{i+1})$ had both $\mathcal P$ and the factorization property.
	In particular, all but finitely many $\phi'_k$ factor through $\eta_{i+1}$.
	On the one hand, by \autoref{def: shortenable}~\ref{enu: shortening - fact} $(Q, \nu)$ has $\mathcal P$.
	Thus we cannot be in Case~\ref{res: dichotomy idc vs factorization - P}.
	On the other hand, $\phi_k$ factors through $\nu$ for all but finitely many $k \in \N$.
	However $(Q, \nu) \prec (L_i, \eta_i)$, hence $\phi_k$ factors through $\eta_i$, for all but finitely many $k \in \N$.
	This contradicts Case~\ref{res: dichotomy idc vs factorization - fact} and completes the proof of our claim.
	
	In view of our claim and  \autoref{def: shortenable}~\ref{enu: shortening - dichotomy}, the group $(L_{i+1}, \eta_{i+1})$ is a proper quotient of $(Q, \nu)$ and thus of $(L_i, \eta_i)$.
	Hence $(L_{i+1}, \eta_{i+1})$ satisfies all the announced properties.
\end{proof}

\begin{prop}
\label{res: dcc}
	Let $G$ be a finitely generated group.
	Let $(L, \eta) \in \mathfrak G(G)$ be a limit group.
	Assume that every limit group $(Q, \nu) \prec (L, \eta)$ is shortenable with respect to $\mathcal P$.
	Then there is no infinite descending sequence of limit groups starting at $(L, \eta)$.
\end{prop}

\begin{proof}
	We fix a non-decreasing exhaustion $(U_i)$ of $G$ by finite subsets.
	Assume that the statement does not hold.
	We build a preferred descending sequence of limit groups $(L_i, \eta_i)$ such that for every $i \in \N$, there is an infinite descending sequence of limit groups starting at $(L_i, \eta_i)$.
	We proceed as follows.
	We start with $(L_0, \eta_0) = (L, \eta)$.
	Assume now that $(L_{i-1}, \eta_{i-1})$ has been defined for some $i \in \N \setminus\{0\}$.
	Among all the limit groups $(Q, \nu) \prec (L_{i-1}, \eta_{i-1})$ we choose $(L_i, \eta_i)$ such that 
	\begin{itemize}
		\item $(L_i, \eta_i)$ is a proper quotient of $(L_{i-1}, \eta_{i-1})$;
		\item there is an infinite descending sequence of limit groups starting at $(L_i, \eta_i)$;
		\item the cardinality of $U_i \cap \ker \eta_i$ is maximal for the above property.
	\end{itemize}
	By construction the sequence $(L_i, \eta_i)$ converges in $\mathfrak G(G)$ to a group $(L_\infty, \eta_\infty)$ such that $(L_\infty, \eta_\infty) \prec (L_i, \eta_i)$, for every $i \in \N$.
	Since the set of limit groups is closed, $(L_\infty, \eta_\infty)$ is also a limit group.
	
	We claim that there is no infinite descending sequence starting at $(L_\infty, \eta_\infty)$.
	Assume on the contrary that such a sequence exists
	\begin{equation*}
		(L_\infty, \eta_\infty) \succ (R_1, \mu_1) \succ (R_2, \mu_2) \succ \dots
	\end{equation*}
	Hence, there is an element $g \in \ker \mu_1 \setminus \ker \eta_\infty$ which belongs to $U_i$ for some $i \in \N\setminus\{0\}$.
	In particular, $g \in \ker \mu_1 \setminus \ker \eta_i$.
	By construction 
	\begin{equation*}
		(L_{i-1},\eta_{i-1}) \succ (R_1, \mu_1) \succ (R_2, \mu_2) \succ \dots
	\end{equation*}
	is an infinite descending sequence.
	Moreover $(R_1, \mu_1)$ is a proper quotient of $(L_{i-1}, \eta_{i-1})$ and $U_i \cap \ker \eta_i$ is strictly contained in $U_i \cap \ker \mu_1$.
	This contradicts our definition of $(L_i, \eta_i)$ and completes the proof of our claim.
	
	By construction $(L_\infty, \eta_\infty)$ is a quotient of $(L, \eta)$.
	It follows from our assumption that every limit group $(Q, \nu) \prec (L_\infty, \eta_\infty)$ is shortenable.
	Consequently $(L_\infty, \eta_\infty)$ has the factorization property (\autoref{res: dichotomy idc vs factorization}).
	Hence by \autoref{res: factorization prop implies almost open} the sequence $(L_i, \eta_i)$ eventually stabilizes, a contradiction.
\end{proof}

Combined with \autoref{res: dichotomy idc vs factorization} we get the following corollary

\begin{coro}
\label{res: shortenable implies factorization}
	Let $G$ be a finitely generated group.
	Let $(L, \eta) \in \mathfrak G(G)$ be a limit group.
	If every limit group $(Q, \nu) \prec (L, \eta)$ is shortenable with respect to $\mathcal P$, then $(L, \eta)$ has the factorization property.
\end{coro}


%
\section{Splittings and graphs of actions}
%
\label{sec: splittings}

In this section we mostly set up notation and vocabulary around graphs of groups and graphs of actions.

%
\subsection{Splittings}
%
\label{sec: graph of groups}

A \emph{graph} is a pair $(V,E)$ of two sets: the \emph{vertex set} $V$ and the \emph{edge set} $E$ together with a fixed point free involution $E \to E$ that we denote by $e \to \bar e$, and two attachment maps $\iota \colon E \to V$ and $t \colon E \to V$ such that $\iota(\bar e) = t(e)$, for every $e \in E$.
The vertices $\iota(e)$ and $t(e)$ are respectively called the \emph{start} and the \emph{end} of $e$.
Given a vertex $v$, its \emph{link}, denoted by $\link v$, is the set of all edges starting at $v$.
A \emph{path} in this graph is a sequence $\nu = (e_1, e_2, \dots, e_m)$ of edges such that $\iota(e_{i+1}) = t(e_i)$, for every $i \in \intvald 1{m-1}$.
We usually denote it by $\nu = e_1e_2 \cdots e_m$.
Given $v \in V$, a \emph{circuit based at $v$} is a path which starts and ends at $v$.
A \emph{simplicial tree} is a connected graph without non-backtracking circuits.

\paragraph{$G$-trees.}
Let $G$ be a group.
A \emph{$G$-tree} is a simplicial tree $S$ endowed with a simplicial action of $G$ without inversion.
This action is \emph{trivial} if $G$ fixes a point.
It is \emph{minimal} if $S$ does not contain any proper $G$-invariant subtree.
Given a $G$-tree, one of the following holds: the action of $G$ is trivial, $G$ fixes an end, or $G$ contains loxodromic elements and there exists a minimal $G$-invariant subtree.
A \emph{splitting} of $G$ is a non-trivial minimal $G$-tree.
A \emph{one-edge splitting} of $G$ is a splitting with a single edge orbit.
Bass-Serre theory builds a correspondance between minimal $G$-trees and \emph{graph of groups decompositions of $G$} \cite{Serre:1980aa}.
In this article we favor the point of view of $G$-trees.

If $v$ (\resp $e$) is a vertex (\resp edge) of $S$, we denote by $G_v$ (\resp $G_e$) its stabilizer in $G$.

\paragraph{Maps between trees.}
Let $S, S'$ be two $G$-trees.
\begin{itemize}
	\item A \emph{morphism} is a $G$-equivariant graph morphism $f \colon S \to S'$.
	In particular, no edge of $S$ is collapsed to a point.
	\item A \emph{collapse map} is a $G$-equivariant map $f \colon S \to S'$ obtained by collapsing certain edges to points, followed by an isomorphism.
	In particular, it sends vertices to vertices (\resp edges to edges or vertices).
	In this setting we say that $S'$ is a \emph{collapse} of $S$ and $S$ \emph{refines} $S'$.
\end{itemize}
If the action of $G$ on $S'$ is minimal, then such maps are always onto.

\paragraph{Normal forms.}
Let $S$ be a minimal $G$-tree.
Our goal is to define a normal form to represent each element of $G$, which reflects the action on $S$.
The construction goes as follows.

We choose a subtree $D_0 \subset S$ lifting a maximal subtree in the quotient graph $S/G$.
In addition, we choose a subtree $D_1 \subset S$ containing $D_0$ such that every edge orbit in $S$ has exactly one representative in $D_1$.
We denote by $V$ the vertex set of $D_0$ and by $E$ the edge set of $D_1$.
Let $\mathbf F(E)$ be the free group
\begin{equation*}
	\mathbf F(E) = \group{E \mid e\bar e = 1, \ \forall e \in E},
\end{equation*}
and
\begin{equation*}
	\mathbf G = \mathbf F(E) \freep \left(\freep_{v \in V} G_v \right),
\end{equation*}

\begin{nota}
	In order to avoid any confusion, we always write in bold type the elements of $E$ and $G_v$ (for $v \in V$) when primarily seen as elements of $\mathbf G$.
\end{nota}

We now define a map $\sigma \colon \mathbf G \to G$ which extends the embeddings $G_v \into G$, for every $v \in V$.
The construction goes as follows.
For every edge $\mathbf e \in E$ we choose an element $s(\mathbf e) \in G$ so that $s(\mathbf e)t(\mathbf e)$ belongs to $V$.
Then we let 
\begin{equation*}
	\sigma(\mathbf e) = s(\bar {\mathbf e}) s(\mathbf e)^{-1}.
\end{equation*}
Note that $\sigma(\bar {\mathbf e}) = \sigma(\mathbf e)^{-1}$, for every $\mathbf e \in E$, hence $\sigma$ is well-defined.
It is known that $\sigma$ is onto.
Let us make this idea precise.
A \emph{decorated path} is an element $\mathbf g \in \mathbf G$, which can be written as follows
\begin{equation*}
	\mathbf g = \mathbf g_0 \mathbf e_1 \mathbf g_1 \mathbf e_2 \mathbf g_2 \dots \mathbf g_{m-1}\mathbf e_m\mathbf g_m, 
\end{equation*}
where 
\begin{enumerate}
	\item for every $i \in \intvald 1m$, the element $\mathbf e_i$ belongs to $E$; moreover the sequence $(\mathbf e_1,\dots, \mathbf e_m)$ maps to a path in the quotient graph $S/G$;
	\item for every $i \in \intvald 0m$, the element $\mathbf g_i$ fixes the (unique) vertex of $v_i$ of $D_0$ lying in the $G$-orbit of $t(\mathbf e_i)$ if $i > 0$ (\resp $\iota(\mathbf e_{i+1})$ if $i <m$)
\end{enumerate}
Such a decorated path is \emph{reduced} if
\begin{enumerate}
\setcounter{enumi}{2}
	\item for every $i \in \intvald 1{m-1}$, either $\mathbf e_{i+1} \neq \bar {\mathbf e}_i$, or $\mathbf g_i$ does not fix $s(\mathbf e_i)\mathbf e_i$.
\end{enumerate}
A \emph{decorated circuit based at a vertex $v \in V$} is a decorated path
\begin{equation*}
	\mathbf g = \mathbf g_0 \mathbf e_1 \mathbf g_1 \mathbf e_2 \mathbf g_2 \dots \mathbf g_{m-1}\mathbf e_m\mathbf g_m, 
\end{equation*}
such that $\iota(\mathbf e_1)$ and $t(\mathbf e_m)$ lie in the $G$-orbit of $v$.
It is known that every element $g \in G$ has a pre-image under $\sigma$ which is a reduced decorated circuit based at $v$
This is essentially a direct consequence of Serre's combinatorial definition of the fundamental group of a graph of groups.
See for instance Serre \cite[Chapitre~5.2, Exercice~(c)]{Serre:1977wy}.
We call this pre-image the \emph{normal form of $g$ based at $v$}.
Note that this normal form is not necessarily unique though.

A decorated path $\mathbf g = \mathbf g_0 \mathbf e_1 \mathbf g_1 \dots \mathbf e_m\mathbf g_m$ also gives rise to a (regular) path in $S$.
For every $i \in \intvald 1m$ define the edge $f_i$ of $S$ as $f_i = u_i \mathbf e_i$ where
\begin{equation*}
	u_i = \sigma(\mathbf g_0\mathbf e_1\dots \mathbf e_{i-1}\mathbf g_{i-1})s(\bar {\mathbf e}_i).
\end{equation*}
Then $\nu = f_1f_2 \dots f_m$ is a path in $S$ from $v$ to $gv'$ where $g = \sigma(\mathbf g)$ while $v$ and $v'$ are the vertices of $D_0$ in the $G$-orbits of $\iota(\mathbf e_1)$ and $t(\mathbf e_m)$ respectively.
It follows from the construction that $\nu$ has no backtracking if and only if $\mathbf g$ is reduced.

%
\subsection{Graph of actions}
%
\label{sec: graph of actions}

Graph of actions were formalized by Levitt to decompose the action of a given group on an $\R$-tree \cite{Levitt:1994aa}.
We follow here \cite[Definition~4.3]{Guirardel:2004aa}.

\begin{defi}
\label{def: graph of actions}
	A \emph{graph of actions on $\R$-trees} $\Lambda$ consists of the following data.
	\begin{enumerate}
		\item A group $G$ and a $G$-tree $S$ (called the \emph{skeleton} of the graph of actions).
		\item An $\R$-tree $Y_v$ (called \emph{vertex tree}) for each vertex $v$ of $S$.
		\item An \emph{attaching point} $p_e \in Y_{t(e)}$ for each oriented edge $e$ of $S$.
	\end{enumerate}
	Moreover all these data should be invariant under $G$, i.e. 
	\begin{itemize}
		\item $G$ acts on the disjoint union of the vertex trees so that the projection $Y_v \mapsto v$ is equivariant,
		\item for every edge $e$ of $S$, for every $g \in G$, we have $p_{ge}=gp_e$. 
	\end{itemize}
\end{defi}

\begin{rema*}
	If no confusion can arise, we write $\Lambda$ for both the graph of actions and its underlying skeleton.
	Similarly we speak of $\Lambda$ as a splitting of $G$.
\end{rema*}

Given a vertex $v$ of $S$, the action of $G_v$ on $Y_v$ is called a \emph{vertex action}.
To such a graph of actions $\Lambda$ one associates an $\R$-tree $T_\Lambda$ endowed with an action of $G$ by isometries.
Roughly speaking, $T_\Lambda$ is obtained from the disjoint union of the vertex trees by identifying the attaching points $p_e$ and $p_{\bar e}$ for each edge $e$ of $S$, see \cite[Section~4.3]{Guirardel:2004aa} for the precise construction.
We say that the action of $G$ on an $\R$-tree $T$ \emph{decomposes as a graph of actions}, if there exists a $G$-equivariant isometry from $T$ to $T_\Lambda$.

Let us now relate, for later use, the geometry of $T_\Lambda$ with the normal form introduced in the previous section.
The map $\sigma \colon \mathbf G\to G$, and all related objects are chosen as in \autoref{sec: graph of groups}.
Let $v \in V$ and $x$ be a point in the vertex tree $Y_v$.
Let $g \in G$ and 
\begin{equation*}
	\mathbf g = \mathbf g_0 \mathbf e_1 \mathbf g_1 \dots \mathbf g_{m-1}\mathbf e_m\mathbf g_m
\end{equation*}
a decorated circuit based at $v$ such that $\sigma(\mathbf g) = g$.
For every $i \in \intvald 0m$ we let 
\begin{equation*}
	u_i = \sigma\left(\mathbf g_0\mathbf e_1\dots \mathbf g_{i-1}\right)s(\bar {\mathbf e}_i)
	\quad \text{and} \quad
	f_i = u_i \mathbf e_i
\end{equation*}
According to our previous discussion $\nu = f_1f_2\dots f_m$ is a path in $S$ from $v$ to $gv$.
By construction the following concatenation is a path in $T_\Lambda$ between $x$ and $gx$:
\begin{equation*}
	\geo x{p_{f_1}} \cup \geo{p_{f_1}}{p_{f_2}} \cup \geo{p_{f_2}}{p_{f_3}}\cup  \dots \cup \geo {p_{f_m}}{gx}.
\end{equation*}
This path is a geodesic of $T$ whenever $\mathbf g$ is reduced.
In particular, we get
\begin{equation}
\label{eqn: graph of actions - dist}
	\dist {gx}x   
	\leq \dist{x}{\mathbf g_0s(\bar {\mathbf e}_1)p_{\bar {\mathbf e}_1}} 
	+ \sum_{i=1}^{m-1} \dist{s(\mathbf e_i)p_{\mathbf e_i}}{\mathbf g_is(\bar{\mathbf e}_{i+1})p_{\bar{\mathbf e}_{i+1}}}
	+ \dist{s(\mathbf e_m)p_{\mathbf e_m}}{\mathbf g_m x}
\end{equation}
with equality whenever $\mathbf g$ is reduced.
This decomposition of $\geo x{gx}$ will be useful when it comes to the shortening argument, see \autoref{sec: shortening}

%
\subsection{Root splittings}
%
\label{sec: root splittings}

Let us recall the definition of a root splitting, whose relevance in our context was highlighted in the introduction (see for instance Examples~\ref{exa: first example w/ root}  and \ref{exa: root source of troubles}).

\begin{defi}[Root splitting]
\label{def: root splitting}
	Let $G$ be a group. 
	A \emph{root splitting} of $G$ is a decomposition of $G$ as a non-trivial amalgamated product $G = A \freep_CB$, where $B$ is an abelian group and $C$ is a proper, finite index subgroup of $B$.
	We call this index the \emph{order} of the splitting.
	Alternatively we say that $G$ is \emph{obtained from $A$ by adjoining roots} and $B$ is the \emph{adjoint root factor} of the splitting.
	We say that the root splitting is \emph{non-essential} if $C$ is a proper, free factor of $A$ and (abstractly) isomorphic to $B$.
\end{defi}

Consider the free group $\free 2$ generated by two elements, say $a$ and $b$.
A typical example of a non-essential root splitting is
\begin{equation*}
	\free 2 = \group{a,b^2} *_{\group {b^2}} \group b.
\end{equation*}
Note that if $G$ is freely indecomposable, then every root splitting of $G$ is essential.
The next statements stresses the importance of root splittings in our study of periodic groups.

\begin{prop}
\label{res: root splitting extending morphism}
	Let $G$ be a group which admits a root splitting $G = A\freep_CB$ of order $p$ where $B$ is the adjoint root factor of the splitting.
	Let $n \in \N$ be an integer co-prime with $p$.
	Let $\Gamma$ be a periodic group of exponent $n$.
	Every morphism $\phi \colon A \to \Gamma$ uniquely extends to a morphism $\phi' \colon G \to \Gamma$.
	Moreover $\phi$ and $\phi'$ have the same image.
\end{prop}

\begin{proof}
	Since $n$ is co-prime with $p$, one checks that $B = CB^n$ and $C \cap B^n = C^n$.
	It means that the embedding $C \into B$ induces an isomorphism $\alpha \colon C/C^n \to B/B^n$.
	Let $\phi \colon A \to \Gamma$ be a morphism.
	Since $\Gamma$ has exponent $n$, its restriction to $C$ factors through the projection $C \onto C/C^n$.
	We write $\psi \colon C/C^n \to \Gamma$ for the resulting morphism.
	Extending $\phi$ to a morphism $\phi' \colon G \to \Gamma$ is equivalent to finding a morphism $\psi ' \colon B/B^n \to \Gamma$ such that $\psi' \circ \alpha = \psi$.
	Since $\alpha$ is an isomorphism, there is exactly one way to do so.
	Moreover the extension $\psi'$ has the same image as $\psi$.
\end{proof}

\begin{coro}
\label{res: root splitting isom burnside}
	Let $G$ be a group which admits a root splitting $G = A\freep_CB$ of order $p$ where $B$ is the adjoint root factor of the splitting.
	Then for every integer $n \in \N$, which is co-prime with $p$, the embedding $A \into G$ induces an isomorphism $A/A^n \to G/G^n$.
\end{coro}

\begin{proof}
	We denote by $\pi_A \colon A \onto A/A^n$ and $\pi_G \colon G \onto G/G^n$ the canonical projections.
	The embedding $\iota \colon A \into G$ induces a morphism $\iota_n \colon A / A^n \to G/G^n$ such that $\pi_G \circ \iota = \iota_n \circ \pi_A$.
	According to \autoref{res: root splitting extending morphism} there is an epimorphism $\phi \colon G \onto A/A^n$ extending $\pi_A$, i.e. $\phi \circ \iota = \pi_A$.
	Since $A/A^n$ has exponent $n$, the morphism $\phi$ factors through $\pi_G$. 
	We denote by $\psi \colon G/G^n \onto A/A^n$ the resulting epimorphism, so that $\psi \circ \pi_G = \phi$.
	One checks that 
	\begin{equation*}
		\iota_n \circ \psi \circ \pi_G \circ \iota 
		= \iota_n \circ \phi \circ \iota
		= \iota_n \circ \pi_A
		= \pi_G \circ \iota.
	\end{equation*}
	In other words $\iota_n \circ \psi \circ \pi_G$ extends $\pi_G \circ \iota$.
	The uniqueness part of \autoref{res: root splitting extending morphism}, tells us that $\iota_n \circ \psi \circ \pi_G = \pi_G$.
	Recall that $\pi_G$ is onto, hence $\iota_n \circ \psi = \id$.
	It follows that $\psi$, and thus $\iota_n$ is an isomorphism.
\end{proof}

\begin{defi}
\label{def: generalized root splitting}
	Let $G$ be a group.
	A \emph{generalized root splitting} of $G$ is a graph of groups decomposition of $G$ of the following form
	
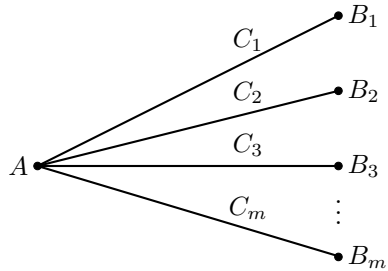
\begin{center}
		\begin{tikzpicture}
			\def \a{4}
			\def \b{1}
			\def \r{0.05}
						
			\draw[fill=black] (0,0) circle (\r) node[left, anchor=east]{$A$};
			\draw[fill=black] (\a,2*\b) circle (\r) node[right, anchor=west]{$B_1$};
			\draw[fill=black] (\a,\b) circle (\r) node[right, anchor=west]{$B_2$};
			\draw[fill=black] (\a,0) circle (\r) node[right, anchor=west]{$B_3$};
			\draw[fill=black] (\a,-0.5\b) node[anchor=center]{$\vdots$};
			\draw[fill=black] (\a,-1.2\b) circle (\r) node[right, anchor=west]{$B_m$};
			
			\draw[thick] (0,0) -- (\a, 2*\b) node[pos=0.7, above, anchor=south]{$C_1$};
			\draw[thick] (0,0) -- (\a, \b) node[pos=0.7, above, anchor=south]{$C_2$};
			\draw[thick] (0,0) -- (\a, 0) node[pos=0.7, above, anchor=south]{$C_3$};
			\draw[thick] (0,0) -- (\a, -1.2*\b) node[pos=0.7, above, anchor=south]{$C_m$};
		\end{tikzpicture}
	\end{center}
	where for every $i \in \intvald 1m$, the group $B_i$ is abelian and $C_i$ is a proper finite index subgroup of $B_i$.
\end{defi}

%
\subsection{Covers of marked group}
%
\label{sec: strong covering}

In general, limit groups need not be finitely presented.
In order to bypass this difficulty, it is common to replace them by finitely presented covers.
For our purpose we need these covers to reflect a given graph of groups decomposition.
This motivates the next definitions.

\begin{defi}[Strong cover]	
\label{def: strong cover}
	Let $G$ be a finitely generated group.
	Let $(H,\phi)$ and $(\hat H, \hat \phi)$ be two groups marked by $G$.
	Let $S$ (\resp $\hat S$) be a splitting of $H$ (\resp $\hat H$) over abelian groups.
	We say that the triple $(\hat H,\hat \phi,\hat S)$ is a \emph{strong cover} of $(H,\phi, S)$ if there exist an epimorphism $\zeta \colon \hat H \onto H$ and a 
	$\zeta$-equivariant map $f \colon \hat S \to S$ such that the following holds.
	\begin{enumerate}
		\item $\zeta \circ \hat \phi = \phi$. In particular, $(H, \phi) \prec (\hat H, \hat \phi)$.
		\item The map $f$ induces an isomorphism from $\hat S/\hat H$ onto $S/H$.
		\item The morphism $\zeta$ is one-to-one when restricted to any edge group of $\hat S$.
		\item Let $\hat v$ be a vertex of $\hat S$ and $v = f(\hat v)$ its image in $S$.
		If $H_v$ is finitely presented (\resp abelian), then $\zeta$ induces an isomorphism from $\hat H_{\hat v}$ onto $H_v$ (\resp an embedding from $\hat H_{\hat v}$ into $H_v$).
	\end{enumerate}
\end{defi}

\begin{defi}[Convergence of strong covers]
\label{def: conv strong cover}
	Let $G$ be a finitely generated group.
	Let $(H,\phi)$ be a group marked by $G$ and $S$ a splitting of $H$ over abelian groups.
	Let $(\hat H_i, \hat \phi_i, \hat S_i)$ be a sequence of strong covers of $(H,\phi, S)$.
	We write $\zeta_i \colon \hat H_i \to H$ and $f_i \colon \hat S_i \to S$ for the corresponding underlying maps.
	The sequence $(\hat H_i, \hat \phi_i, \hat S_i)$ \emph{converges to} $(H, \phi, S)$ if the following holds.
	\begin{itemize}
		\item $(H_i, \phi_i)$ converges to $(H, \phi)$ in the space of marked groups.
		\item Let $e$ be an edge of $S$.
		For every $i \in \N$, there exists a pre-image $\hat e_i$ of $e$ in $\hat S_i$ such that 
		\begin{equation*}
			H_e = \bigcup_{i \in \N} \zeta_i\left(\hat H_{i,\hat e_i}\right).
		\end{equation*}
		\item Let $v$ be a vertex of $S$.
		For every $i \in \N$, there exists a pre-image $\hat v_i$ of $v$ in $\hat S_i$ such that 
		\begin{equation*}
			H_v = \bigcup_{i \in \N} \zeta_i\left(\hat H_{i,\hat v_i}\right).
		\end{equation*}
	\end{itemize}
\end{defi}

The next lemma ensures the existence of converging sequences of strong covers.
The proof works verbatim as in \cite[Lemma~7.1]{Weidmann:2019ue}.

\begin{lemm}
\label{res: graph of groups cover}
	Let $G$ be a finitely presented group.
	Let $(H,\phi)$ be a group marked by $G$ and $S$ a splitting of $H$ over abelian groups.
	There exists a sequence of strong covers $(\hat H_i, \hat \phi_i, \hat S_i)$ which converges to $(H, \phi, S)$ such that for every $i \in \N$, the group $\hat H_i$ is finitely presented.
\end{lemm}

%
\section{JSJ decompositions}
%
\label{sec: JSJ decomposition}

Let $G$ be a (finitely generated) group.
Given a class $\mathcal A$ of subgroups of $G$, the JSJ decomposition of $G$ relative to $\mathcal A$ is a graph of groups decomposition which (when it exists) encodes all the possible splittings of $G$ over subgroups in $\mathcal A$.
For our purpose we need to understand the abelian splittings of $G$, provided $G$ is CSA (see \autoref{def: csa}).
In this setting, the existence of a JSJ decomposition is often stated in the literature under the assumption that $G$ is one-ended (which is equivalent to asking that $G$ is freely indecomposable, when $G$ is torsion-free), see for instance \cite[Theorem~9.5 and 9.14]{Guirardel:2017te}.
However in this article we will consider CSA groups which are freely indecomposable, but not necessarily torsion-free (and thus maybe not one-ended).
It turns out that the JSJ decomposition does exist in this case as well.
In this section we briefly recall the main arguments of its construction.
We use the language of Guirardel and Levitt \cite{Guirardel:2017te}.

%
\subsection{Generalities}
%

Let $G$ be a group.
Let $\mathcal A$ and $\mathcal H$ be two classes of subgroups of $G$.
We assume that $\mathcal A$ is closed under conjugation and taking subgroups.
In addition we suppose that $G$ is finitely generated relative to $\mathcal H$.
An $(\mathcal A,\mathcal H)$-tree is a $G$-tree such that each edge stabilizer belongs to $\mathcal A$ while the subgroups of $\mathcal H$ are elliptic.

A subgroup $E$ of $G$ is \emph{universally elliptic} if it is elliptic in any $(\mathcal A, \mathcal H)$-tree.
An $(\mathcal A, \mathcal H)$-tree is \emph{universally elliptic} if its edge stabilizers are universally elliptic.
Given two $(\mathcal A, \mathcal H)$-tree $T_1$ and $T_2$, we say that $T_1$ \emph{dominates} $T_2$ if there is a $G$-equivariant map $T_1 \to T_2$ or equivalently, if every vertex stabilizer of $T_1$ is elliptic in $T_2$.
\emph{Deformation spaces} are defined by saying that two $(\mathcal A,\mathcal H)$-trees belong to the same deformation space if they dominate each other.

\begin{defi}
\label{def: JSJ tree}
	The \emph{JSJ deformation space of $G$ over $\mathcal A$ relative to $\mathcal H$} is the collection of universally elliptic $(\mathcal A,\mathcal H)$-trees which dominate every other universally elliptic $(\mathcal A, \mathcal H)$-tree.
	A point of the JSJ deformation space is called a \emph{JSJ tree}.
	The corresponding graph of groups decomposition of $G$ is called a \emph{JSJ decomposition}.
	The vertex group of a JSJ tree is \emph{rigid} if it is universally elliptic and \emph{flexible} otherwise.
	The vertices of the JSJ tree are called \emph{rigid} and \emph{flexible} accordingly.
\end{defi}

\begin{defi}
	A subgroup $Q$ of $G$ is \emph{quadratically hanging}, or simply \emph{QH} (over $\mathcal A$ relative to $\mathcal H$) if 
	\begin{enumerate}
		\item $Q$ is the stabilizer of a vertex $v$ of an $(\mathcal A, \mathcal H)$-tree,
		\item $Q$ is an extension $1 \to F \to Q \to \pi_1(\Sigma)$  with $\Sigma$ a compact hyperbolic $2$-orbifold; we call $F$ the \emph{fiber}, and $\Sigma$ the \emph{underlying orbifold};
		\item each edge group incident to $v$ is an \emph{extended boundary subgroup}: by definition, this means that its image in $\pi_1(\Sigma)$ is finite or contained in a boundary subgroup $B$ of $\pi_1(\Sigma)$.
	\end{enumerate}
\end{defi}

Although it is not explicit in the terminology the fiber $F$ and the orbifold $\Sigma$ are part of the structure of a quadratically hanging subgroup.

%
\subsection{Existence}
%
\label{sec: existence JSJ}

In this section, $G$ is a CSA group.
We denote by $\mathcal A$ the collection all its abelian subgroups.
We write $\mathcal H$ for a family of subgroups of $G$ and assume that $G$ is finitely generated relative to $\mathcal H$.
As the notation suggests, we are interested in splittings of $G$ over $\mathcal A$ relative to $\mathcal H$.
In order to prove the existence of the JSJ deformation space Guirardel and Levitt use a stronger notion of domination, see \cite[Definition~7.10]{Guirardel:2017te}.

\begin{defi}
\label{def: small domination}
	Let $T,T'$ be $(\mathcal A, \mathcal H)$-trees.
	The tree $T$  \emph{$\mathcal A$-dominates} $T'$ if the following holds
	\begin{enumerate}
		\item $T$ dominates $T'$;
		\item every edge stabilizer of $T'$ is elliptic in $T$;
		\item any subgroup which is elliptic in $T'$ but not in $T$ belongs to $\mathcal A$.
	\end{enumerate}
\end{defi}

Let $\mathcal A^*$ be the set of all \emph{non-trivial} groups $A \in \mathcal A$.
We endow $\mathcal A^*$ with the following binary relation, $A_1 \sim A_2$ if $A_1$ and $A_2$ generate an abelian group.
Since $G$ is CSA, this is actually an equivalence relation.
Moreover it is \emph{admissible} in the sense of \cite[Definition~7.1]{Guirardel:2017te}, i.e.
\begin{labelledenu}[A]
	\item \label{enu: admissible - conj} 
		If $A_1 \sim A_2$, then $g A_1 g^{-1} \sim g A_2 g^{-1}$, for every $g \in G$.
	\item If $A_1 \subset A_2$ then $A_1 \sim A_2$.
	\item \label{enu: admissible - tree} 
		Let $T$ be a simplicial $(\mathcal A, \mathcal H)$-tree with non-trivial edge stabilizers.
		If $A_1 \sim A_2$ and $A_1$ and $A_2$ respectively fix vertices $x_1$ and $x_2$, then $G_e \sim A_i$, for every edge $e$ contained in $\geo{x_1}{x_2}$.
\end{labelledenu}
For every $A \in \mathcal A^*$, we denote by $[A]$ the equivalence class of $A$.
In follows from \ref{enu: admissible - tree}  that the action by conjugation induces an action of $G$ on the set of equivalence classes.
Since $G$ is CSA, the stabilizer of $[A]$ is the maximal abelian subgroup containing $A$.
In particular, it belongs to $\mathcal A^*$.

Consider now an $(\mathcal A^*, \mathcal H)$-tree $T$, i.e. an $(\mathcal A, \mathcal H)$-tree whose edge stabilizers are non-trivial.
We recall a construction called the \emph{tree of cylinders}, see Guirardel and Levitt \cite{Guirardel:2011wv}.
We declare two edges $e,e'$ of $T$ to be equivalent if  $G_e \sim G_{e'}$.
By \ref{enu: admissible - tree}, the union of all edges in a given equivalence class is a subtree called a \emph{cylinder}.
Two distinct cylinders meet in at most one vertex.
The \emph{tree of cylinders} $T_c$ is bipartite tree with vertex set $V = V_0 \sqcup V_1$ where $V_0$ consists of all vertices of $T$ that belong to at least two cylinders and $V_1$ is the set of cylinders.
There is an edge $\epsilon = (v, Y)$ between $v \in V_0$ and $Y \in V_1$ if $v$ belongs to $Y$.

\begin{rema*}
	In \cite{Guirardel:2017te} Guirardel and Levitt use this construction for $(\mathcal A, \mathcal H)$-trees with \emph{infinite} edge stabilizers.
	Nevertheless, as pointed out in \cite[Section~3.5]{Guirardel:2011wv}, if $G$ is CSA, the construction also works for $(\mathcal A, \mathcal H)$-trees with non-trivial edge stabilizers.
\end{rema*}

We denote by $\mathcal A_{\rm nc}$ the collection of all non-cyclic groups in $\mathcal A$.

\begin{lemm}
\label{res: tree of cylinders properties}
	Let $T$ be an $(\mathcal A^*, \mathcal H)$-tree.
	Its tree of cylinders $T_c$ has the following properties
	\begin{enumerate}
		\item $T_c$ is an $(\mathcal A^*, \mathcal H \cup \mathcal A_{\rm nc})$-tree;
		\item $T$ $\mathcal A$-dominate $T_c$;
		\item the action of $G$ on $T_c$ is $2$-acylindrical.
		\item If $\mathcal A_{\rm nc} \subset \mathcal H$, then $T$ and $T_c$ are in the same deformation space.
	\end{enumerate}
\end{lemm}

\begin{proof}
	Let $Y$ be a cylinder of $T$.
	One checks that the stabilizer $G_Y$ of $Y$ seen as a vertex of $T_c$ is the maximal abelian subgroup containing $G_e$ for some (hence any) edge $e$ of $Y$.
	Let $\epsilon = (v, Y)$ be an edge of $T_c$.
	We fix an edge $e$ of $Y$ starting at $v$.
	The stabilizer $G_\epsilon$ is contained in $G_Y$ hence is abelian.
	It also contains $G_e$, thus is non-trivial.
	Consequently every edge stabilizer of $T_c$ belongs to $\mathcal A^*$.

	Let $v$ be a vertex of $T$.
	If $v$ belongs to at least two cylinders, then it defines a vertex in $T_c$, hence $G_v$ is elliptic in $T_c$.
	If $v$ belongs to a single cylinder $Y$.
	Then $G_v$ necessarily preserves $Y$, hence fixes the corresponding vertex of $T_c$.
	This shows that $T$ dominates $T_c$.
	In particular, every subgroup $H \in \mathcal H$ is elliptic in $T_c$.
	By construction, edge stabilizers in $T_c$ always fix a vertex in $T$.
	Moreover the stabilizer of a vertex $v$ of $T_c$ is either elliptic in $T$ (if $v \in V_0$) or abelian (if $v \in V_1$).
	Thus $T$ $\mathcal A$-dominates $T$.
	
	Consider two cylinders $Y$ and $Y'$ of $T$.
	Write $G_Y$ and $G_{Y'}$ for their respective stabilizers (seen as vertices of $T_c$).
	Suppose that $G_Y \cap G_{Y'}  \neq \{1\}$.
	We choose $e$ and $e'$ two edges of $T$ lying in $Y$ and $Y'$ respectively.
	As we observed before $G_Y$ is the maximal abelian subgroup containing $G_e$.
	The same holds for $e'$.
	Since $G$ is CSA, it forces that $G_Y = G_{Y'}$, thus $e \sim e'$ consequently $Y = Y'$.
	We proved that if $Y$ and $Y'$ are distinct cylinders then their stabilizers have a trivial intersection.
	Since $T_c$ is bipartite, it implies that the action of $G$ on $T_c$ is $2$-acylindrical.	
	In particular, every non-elliptic abelian subgroup of $G$ is infinite cyclic.
	Hence $T_c$ is an $(\mathcal A, \mathcal H \cup \mathcal A_{\rm nc})$-tree.
	
	Assume now that $\mathcal A_{\rm nc} \subset \mathcal H$.
	We claim that $T$ and $T_c$ are in the same deformation space.
	Since $T$ dominates $T_c$, it suffices to prove that $T_c$ dominates $T$.
	Let $v$ be a vertex of $T_c$.
	We need to show that $G_v$ is elliptic in $T$.
	Since $T$ $\mathcal A$-dominates $T_c$, the group $G_v$ either is elliptic in $T$ or belongs to $\mathcal A$.
	Since $\mathcal A_{\rm nc} \subset \mathcal H$, the only case to consider is when $G_v$ is infinite cyclic.
	Let $e$ be an edge of $T_c$ adjacent to $v$.
	Since edge stabilizers are non-trivial, $G_e$ is necessarily a finite index subgroup of $G_v$.
	However $G_e$ is elliptic in $T$, hence the same holds for $G_v$.
\end{proof}

By definition the class $\mathcal A$ contains all cyclic subgroups of $G$.
Note that if $G$ is free indecomposable relative to $\mathcal H$, then any $(\mathcal A, \mathcal H)$-tree $T$ has non-trivial edge stabilizers, hence $\mathcal A$-dominates its tree of cylinders $T_c$ (\autoref{res: tree of cylinders properties})
The next statement is now an application of Guirardel and Levitt \cite[Theorem~8.7]{Guirardel:2017te} for $C = 1$.

\begin{theo}
\label{res: jsj - existence}	
	Assume that $G$ is freely indecomposable relative to $\mathcal H$.
	The JSJ deformation space of $G$ over $\mathcal A$ relative to $\mathcal H$ exists.
	Moreover every flexible vertex group is either abelian or QH with trivial fiber.
\end{theo}

%
\subsection{Compatibility}
%
\label{sec: compatibility JSJ}

We still denote by $G$ a CSA group that is finitely generated relative to $\mathcal H$.
Two $(\mathcal A, \mathcal H)$-trees $T_1$ and $T_2$ are \emph{compatible} if they a have a common refinement (i.e. there exists an $(\mathcal A, \mathcal H)$-tree $S$ and two collapse maps $S \to T_1$ and $S\to T_2)$.
The goal of this section is to prove that if $G$ does not split over the trivial group or $\Z/2$, then the tree of cylinders of its JSJ decomposition is compatible with any other $(\mathcal A, \mathcal H)$-tree.
The result is well-known if $G$ is one-ended, see \cite[Corollary~11.5]{Guirardel:2017te}
As before, we briefly show that the same arguments actually work in this more general setting.

We start by recalling a few additional properties of trees of cylinders.
First, one checks from the definition that if $T$ is an $(\mathcal A^*, \mathcal H)$-tree, then $(T_c)_c = T_c$.

\begin{prop}[Guirardel-Levitt {\cite[Corollary~4.10]{Guirardel:2011wv}}]
\label{res: tree of cylinders inv of deformation space}
	If $T$ and $T'$ are two minimal non-trivial $(\mathcal A^*, \mathcal H)$-trees belonging to the same deformation space, there is a canonical equivariant isomorphism between their trees of cylinders.
\end{prop}

More generally any $G$-equivariant map $f \colon T \to T'$ induces a cellular $G$-equivariant map $f_c \colon T_c \to T'_c$, which actually does not depend on $f$.
(A cellular map is a map which sends a vertex to a vertex, an edge to a vertex or an edge.)
This construction is functorial in the sense that $(f \circ f')_c = f_c \circ f'_c$ \cite[Proposition~4.11]{Guirardel:2011wv}.


\begin{prop}[Guirardel-Levitt {\cite[Proposition~8.1]{Guirardel:2011wv}}]
\label{res: tree of cylinders and compatibility}
	If $T$ and $T'$ are two minimal non-trivial $(\mathcal A^*, \mathcal H)$-trees.
	If $T$ dominates $T'$, then $T_c$ is compatible with $T'$ and $T'_c$.
\end{prop}

\begin{prop}
\label{res: tree of cylinders functorial collapse}
	Assume that $G$ does not split over the trivial group or $\Z/2$ relative to $\mathcal H$.
	Let $S$ and $T$ be two $(\mathcal A, \mathcal H)$-tree.
	Assume that every vertex stabilizer of $T$ either belongs to $\mathcal A$ or is QH with trivial fiber.
	If $S$ refines $T$, then $S_c$ refines $T_c$.
\end{prop}

\begin{proof}
	If $G$ is one-ended relative to $\mathcal H$, it is a particular case of \cite[Lemma~7.15]{Guirardel:2017te}.
	The proof follows the exact same steps.
	The goal is to show that the map $f_c \colon S_c \to T_c$ induced by a collapse map $f \colon S \to T$ is actually a collapse map as well.
	
	Trees in the same deformation space have the same tree of cylinders (\autoref{res: tree of cylinders inv of deformation space}).
	Thus, we can assume that no proper collapse of $S$ refining $T$ belongs to the same deformation space as $S$. 
	In particular, for every vertex $v$ in $T$, the pre-image $S_v \subset S$ of $v$ is minimal for the action of $G_v$.
	
	Let $v$ be a vertex of $T$.
	Let $S'$ be the $(\mathcal A, \mathcal H)$-tree obtained from $S$ by collapsing every edge that is mapped to a vertex that does not belong to the orbit of $v$.
	Note that $S_v$ embeds in $S'$, so that we can view $S_v$ as a subtree of $S'$.
	It suffices to prove the following facts.
	\begin{enumerate}
		\item \label{enu: tree of cylinders functorial collapse - stab}
			The stabilizer of any vertex of $S'$ either belongs to $\mathcal A$ or is QH with trivial fiber.
		\item $S'_c$ refines $T_c$.
	\end{enumerate}
	The result then follows from an induction argument.
	Without loss of generality we can assume that $S_v$ is not reduced to a point (otherwise $S' = T$).
	Note that the stabilizer of any vertex in $S'$ is either a vertex stabilizer for the action of $G$ on $T$ (and thus belongs to $\mathcal A$ or is QH with trivial fiber) or is conjugated to the stabilizer of a vertex in $S_v$.
	Hence to get \ref{enu: tree of cylinders functorial collapse - stab} it suffices to prove that every vertex stabilizer of $S_v$ either belongs to $\mathcal A$ or is QH with trivial fiber.
	We distinguish two cases.
	
	Assume first that $G_v \in \mathcal A$.
	The stabilizer of any vertex in $S_v$ is contained in $G_v$ and thus belongs to $\mathcal A$.
	Since $G$ does not split as a free product, the stabilizer $G_e$ of any edge $e$ in $S_v$ is non-trivial.
	Moreover it is contained in $G_v$.
	Hence $G_e \sim G_v$.
	It follows that all the edges of $S_v$ are contained in the same cylinder of $S'$.
	In particular, $G_v$ is elliptic for its action on $S'_c$.	
	Consequently we have domination maps 
	\begin{equation*}
		S' \to T \to S'_c.
	\end{equation*}
	They induce cellular maps
	\begin{equation*}
		S'_c \to T_c \to (S'_c)_c = S'_c.
	\end{equation*}
	which forces $T_c = S'_c$.
	
	Assume now that $v$ is QH with trivial fiber.
	We denote by $\Sigma$ the underlying orbifold.
	Every component of $\partial \Sigma$ is used, see \cite[Lemma~5.15]{Guirardel:2017te}.
	Indeed otherwise, by \cite[Lemma~5.16]{Guirardel:2017te} $G$ would split over the trivial group or $\Z/2$ relative to $\mathcal H$.
	Combining Lemma~5.18 and Proposition~5.21 of \cite{Guirardel:2017te} one observes that the splitting of $G_v$ induced by the action of $G_v$ on $S_v$ is dual to a family of essential simple closed geodesics $\mathcal L$ on $\Sigma$.
	Consequently every vertex stabilizer in $S_v$ is QH with trivial fiber.
	
	We now claim that every cylinder of $S'$ containing an edge in $S_v$ is actually entirely contained in $S_v$.
	Let $e$ be an edge of $S_v$.
	Let $e'$ be an equivalent edge of $S'$, i.e. such that $G_e \sim G_{e'}$.
	Suppose that contrary to our claim, $e'$ is not contained in $S_v$.
	Since cylinders are connected, we can assume that $e'$ has an endpoint in $S_v$ and thus is not collapsed in $T$.
	Recall that $G$ does not split over the trivial group of $\Z/2$ relative to $\mathcal H$.
	Hence the subgroup $G_{e'}$ is either finite (but distinct from $\{1\}$ and $\Z/2$) or has finite index in a boundary subgroup of $G_v = \pi_1(\Sigma)$. 
	According to our previous discussion, there is an essential simple closed curve $\gamma \in \mathcal L$ such that $G_e = \group \gamma$.
	Since $G_e \sim G_{e'}$, the groups $G_{e'}$ and $G_e$ commute.
	This is impossible and completes the proof of our claim.
	
	It follows from our claim that the map $S' \to T$ is either injective or constant when restricted to cylinders.
	It implies that $S'_c$ refines $T_c$, see \cite[Remark 4.13]{Guirardel:2011wv}
 \end{proof}

\begin{theo}
\label{res: compatible JSJ}
	Let $G$ be a CSA group.
	Let $\mathcal A$ be the collection of all its abelian subgroups.
	Let $\mathcal H$ be a class of subgroups of $G$.
	Assume that $G$ is finitely generated relative to $\mathcal H$ and does not split over the trivial group or $\Z/2$ relative to $\mathcal H$.
	Let $S$ be a JSJ tree of $G$ over $\mathcal A$ relative to $\mathcal H$.
	The tree of cylinders $S_c$ of $S$ is an $(\mathcal A, \mathcal H \cup \mathcal A_{\rm nc})$-tree which is compatible with any $(\mathcal A, \mathcal H)$-tree.
\end{theo}

\begin{proof}
	Consider an $(\mathcal A, \mathcal H)$-tree $T$.
	According to \cite[Lemma~2.8(1)]{Guirardel:2017te}, $S$ has a refinement $S'$ that dominates $T$.
	It follows from \autoref{res: tree of cylinders and compatibility} that $S'_c$ and $T$ have a common refinement, say $R$.
	However $S'_c$ also refines $S_c$ (\autoref{res: tree of cylinders functorial collapse}).
	Hence $R$ is a common refinement of $S_c$ and $T$.
\end{proof}

%
\subsection{Modular group}
%
\label{sec: modular group}

Let $G$ be a CSA group.
As before $\mathcal A$ stands for the class of all abelian subgroups of $G$, while $\mathcal H$ is an arbitrary class of subgroups of $\mathcal A$.
We assume that $G$ is finitely generated relative to $\mathcal H$ and does not split over the trivial group or $\Z/2$ relative to $\mathcal H$.

Consider a one edge splitting of $G$ of the form $G = A \ast_C B$ or $G = A \ast_C$.
Let $u \in G\setminus\{1\}$ be an element centralizing $C$.
A \emph{Dehn twist} by $u$ is an automorphism which fixes $A$ and 
\begin{itemize}
	\item conjugates $B$ by $u$, in the amalgamated product case;
	\item sends the stabler letter $t$ to $tu$, in the HNN case.
\end{itemize}
Let $T$ be an $(\mathcal A, \mathcal H)$-tree.
Given an edge $e$ of $T$, a \emph{Dehn twist over $e$} is a Dehn twist in the one edge splitting obtained by collapsing all the edges of $T$ which are not in the orbit of $e$.
Let $v$ be a vertex of $T$.
Any automorphism $\alpha_v$ of the vertex group $G_v$ which acts by conjugacy on the adjacent edge groups can be extended to an automorphism $\alpha$ of $G$.
We call $\alpha$ the \emph{natural extension} of $\alpha_v$.
Assume now that $G_v$ is abelian.
Let $E_v$ be the subgroup of $G_v$ generated by the stabilizer of every edge in $T$ starting at $v$.
The \emph{peripheral subgroup} of $v$ is the subgroup
\begin{equation*}
	K_v = \bigcap_{\phi} \ker \phi,
\end{equation*}
where $\phi$ runs over all morphisms $\phi \colon G_v \to \Z$ whose kernel contains $E_v$.

\begin{defi}
\label{def: modular group}
	Let $T$ be an $(\mathcal A, \mathcal H)$-tree.
	The \emph{modular group of $G$ associated to $T$}, denoted by $\mcg{G, T}$, is the subgroup of the automorphism group $\aut G$ generated by
	\begin{itemize}
		\item inner automorphisms,
		\item Dehn twists over edges of $T$,
		\item natural extensions of automorphisms of QH vertex groups (called \emph{surface type automorphisms}),
		\item natural extensions of automorphisms of abelian vertex groups $G_v$ that fix the peripheral subgroup $K_v$ (called \emph{generalized Dehn twists}).
	\end{itemize}
	If $S$ is a JSJ tree  of $G$ over $\mathcal A$ relative to $\mathcal H$ and $S_c$ its tree of cylinders, we simply write $\mcg G$ for $\mcg {G,S_c}$.
	We say that $\mcg G$ is the \emph{modular group} of $G$ (over $\mathcal A$ relative to $\mathcal H$).
\end{defi}

The next statement is a consequence of the compatibility of the JSJ tree  (\autoref{res: compatible JSJ}).

\begin{prop}
\label{res: compatible modular group.}
	Let $G$ be a CSA group.
	Let $\mathcal A$ be the collection of all its abelian subgroups.
	Let $\mathcal H$ be a class of subgroups of $G$.
	Assume that $G$ is finitely generated relative to $\mathcal H$ and does not split over the trivial group or $\Z/2$ relative to $\mathcal H$.
	For every $(\mathcal A, \mathcal H)$-tree $T$, we have $\mcg{G,T} \subset \mcg G$.
\end{prop}


%
\section{Action on a limit tree}
%
\label{sec: action on limit tree}

This section is the core of the article.
In view of \autoref{res: trichotomy limit groups}, we need to understand non-abelian, divergent limit groups over the class $\mathfrak H_\delta$ (for some fixed value of $\delta$).
In this section, we prove that such a limit group $(L,\eta)$ admits an action on an $\R$-tree.
Although this action is not super-stable, we have enough control to decompose it into a graph of actions whose vertex actions are either peripheral, simplicial, axial or of Seifert type (\autoref{res: final decomposition tree}).
Compared to similar results in the literature, the peripheral action is a new kind of vertex action that comes from the action of the groups $H \in \mathfrak H$ around cone points.
The decomposition into a graph of actions provides a splitting $S$ of $L$ over abelian subgroups.
We prove that if $S$ is not a generalized root splitting, then we can shorten the morphisms that were used to define $L$ (\autoref{res: shortening argument}).

%
\subsection{Building a limit tree}
%
\label{sec: building limit tree}

\paragraph{A limit of actions.}
Recall that energies and divergent limit groups have been defined in Definitions~\ref{def: energy} and \ref{def: div limit group} respectively.
Let $G$ be a group and $U$ a finite generating set of $G$.
Let $\delta \in \R_+^*$.
Let $(L, \eta)$ be a \emph{non-abelian}, \emph{divergent} limit group over $\mathfrak H_\delta$.
By definition, there exists a sequence of morphisms $\phi_k \colon G \onto (\Gamma_k, X_k, \mathcal C_k)$ converging to $(L,\eta)$, where the triple $(\Gamma_k, X_k, \mathcal C_k)$ belongs to $\mathfrak H_\delta(\tilde \rho_k)$, for every $k \in \N$, and such that the sequence $(\tilde \rho_k)$ as well as the energy $\lambda_\infty(\phi_k, U)$ diverge to infinity\footnote{We assumed here that $\tilde \rho_k$ diverges to infinity to be coherent with the configuration we will handle later. However for the remainder of \autoref{sec: action on limit tree} it would be enough to assume that $\rho_k \geq C \delta_k$ for some fixed sufficiently large $C$.}

Since $L$ is not abelian, we can assume without loss of generality that the image of $\phi_k$ is not abelian, for every $k \in \N$.
According to \autoref{res: comparing energies}, the energies of $\phi_k$ are related by 
\begin{equation*}
	\lambda_\infty(\phi_k, U) \leq \lambda^+_1(\phi_k, U) \leq 2 \card U \lambda_\infty(\phi_k, U).
\end{equation*}
In addition there exists a point $o_k \in X_k$ with the following properties.
\begin{itemize}
	\item  $\lambda_\infty(\phi_k, U, o_k) \leq \lambda_\infty(\phi_k, U) + 20\delta$,
	\item $o_k$ is at a distance at most $\card U \lambda_\infty(\phi_k, U) + 20\delta$ from a point (almost) minimizing the restricted energy $\lambda^+_1(\phi_k, U)$.
\end{itemize}
For the moment we prefer not to give a name to the latter point, as it may eventually change later (see \autoref{res: basepoint in minimal tree}).
For simplicity we let 
\begin{equation*}
	\epsilon_k = \frac 1 {\lambda_\infty\left(\phi_k, U\right)}.
\end{equation*}
It follows from our assumptions that $(\epsilon_k)$ converges to zero.
In the remainder of this section, unless mentioned otherwise, we work with the rescaled space $\epsilon_k X_k$, i.e. for every $x,x' \in X_k$
\begin{equation*}
	\dist[\epsilon_k X_k] x{x'} = \epsilon_k \dist[X_k] x{x'}.
\end{equation*}
We endow $\epsilon_k X_k$ with the action by isometries of $G$ induced by $\phi_k$.
This space is $\delta_k$-hyperbolic, where $\delta_k = \epsilon_k \delta$ converges to zero.
For simplicity we write $\rho_k = \epsilon_k \tilde \rho_k$ for the rescaled version of the parameter $\tilde \rho_k$, so that the ratio $\rho_k / \delta_k$ diverges to infinity.

\paragraph{Asymptotic properties.}
We fix a non-principal ultra-filter $\omega \colon \mathcal P(\N) \to \{0, 1\}$.
Recall that a property $P_k$ holds \oas if 
\begin{equation*}
	\omega(\set{k \in \N}{P_k\ \text{holds}}) = 1.
\end{equation*}
A real valued sequence $(u_k)$ is \oeb if there exists $M \in \R$, such that $\abs{u_k} \leq M$, \oas.
It \emph{converges to $\ell$ along $\omega$}, if for every $\epsilon \in \R_+^*$, we have $\abs{u_k - \ell} < \epsilon$, \oas.

The next statements and proofs involve several asymptotic comparisons.
To that end, we adopt Landau's notation.
Let $f,g \colon \N \to \R$ be two maps.
We write 
\begin{itemize}
	\item $f =o(g)$, if for every $\epsilon > 0$, we have $\abs{f(k)} \leq \epsilon \abs{g(k)}$ \oas;
	\item $f=O(g)$, if there exists $C > 0$ such that $\abs{f(k)} \leq C \abs{g(k)}$ \oas.
\end{itemize}
We say that $f$ and $g$ are \emph{equivalent} and write $f \sim g$, if $f = g + o(g)$.
If there is no ambiguity, we make the following abuse of notation and write
\begin{equation*}
	f(k) = o \left(g(k)\right), \quad
	f(k) = O\left(g(k)\right)
	\quad \text{and} \quad
	f(k) \sim g(k).
\end{equation*}

\paragraph{Limit tree.}
We consider now an ultra-limit of metric spaces. 
We refer the reader to Dru\c tu and Kapovich \cite[Chapter~10]{Drutu:2018aa} for a detailed exposition of this construction.
As the sequence $(\delta_k)$ converges to zero, the limit space
\begin{equation*}
	X_\omega = \limo \left(\epsilon_k X_k, o_k\right)
\end{equation*}
is an $\R$-tree.
The action of $G$ on $X_k$ induces an action without global fixed point of $G$ on $X_\omega$  \cite{Paulin:1991fx}.
Moreover the stable kernel of $(\phi_k)$ acts trivially on $X_\omega$.
Consequently, the limit group $L$ acts on $X_\omega$ without global fixed point.

We denote by $T$ the minimal $L$-invariant subtree of $X_\omega$.
Although the construction does depend on the sequence $(\Gamma_k, \phi_k)$ we call $T$ the \emph{limit tree of $(L, \eta)$}.

\begin{lemm}
\label{res: limit action - dist to cone point}
	The basepoint $o_k \in X_k$ is such that \oas
	\begin{equation*}
	\label{eqn: limit action - dist to cone point}
		2\rho_k \leq 2d(o_k, \mathcal C_k) +1 + 20\delta_k
	\end{equation*}
\end{lemm}

\begin{proof}
	Assume that contrary to our claim, there exists $c_k \in \mathcal C_k$ such that $2\rho_k > 2d(o_k, c_k) +1 +20\delta_k$, \oas.
	Since any generator $u\in U$ moves $o_k$ by at most $1 + 20\delta_k$, the triangle inequality yields $\dist{\phi_k(u)c_k}{c_k} < 2\rho_k$, \oas, for every $u \in U$.
	Nevertheless two distinct cone points in $\epsilon_k X_k$ are at a distance at least $2\rho_k$ far appart.
	Consequently $\phi_k(u)$ fixes $c_k$, for every $u \in U$, and thus $\phi_k(G)$ is contained in $\stab{c_k}$ \oas.
	Combining \autoref{res; family - fully conical subgroup} and \ref{enu: family axioms - elem subgroups} we observe that $\stab{c_k}$ is abelian.
	It follows that the image of $\phi_k$ is abelian \oas, which contradicts our assumption.
\end{proof}

\paragraph{Limit radius.}
Note that we have no control a priori on the behavior of $(\rho_k)$, beside the fact that $\rho_k/\delta_k$ diverges to infinity.
Nevertheless we define 
\begin{equation*}
	\rho = \limo \rho_k.
\end{equation*}
It belongs to $\R_+ \cup \{\infty\}$.

%
\subsection{Peripheral structure}
%
\label{res: peripheral structure}

We endow the tree $X_\omega$ with an additional structure that will keep track of the elliptic elements fixing apices from $\mathcal C_k$.
First let 
\begin{equation*}
	\Pi_\omega \mathcal C_k
	= \set{(c_k) \in \Pi_{\N} \mathcal C_k}{\dist {o_k}{c_k} - \rho_k\ \text{is \oeb}}.
\end{equation*}
We write $\mathcal C_\omega$ for the quotient of $\Pi_\omega \mathcal C_k$ by the equivalence relation which identifies two sequences $(c_k)$ and $(c'_k)$ if $c_k = c'_k$ \oas.
Given a sequence $(c_k) \in \Pi_\omega \mathcal C_k$, we write $c = [c_k]$ for its image in $\mathcal C_\omega$.
The action of $\Gamma_k$ on $X_k$ induces an action of $L$ on $\mathcal C_\omega$.

\paragraph{Busemann functions.}
Let $c = [c_k]$ be a point in $\mathcal C_\omega$.
We define a map $b_c \colon X_\omega \to \R$ by sending $x = \limo x_k$ to
\begin{equation*}
	b_c(x) = \limo \left( \dist{c_k}{x_k} - \rho_k \right).
\end{equation*}
It follows from our definition of $\mathcal C_\omega$ and the triangle inequality that the quantity $\dist {c_k}{x_k} - \rho_k$ is \oeb.
Hence $b_c$ is well defined.
Note also that the collection $(b_c)$ is $L$-invariant in the following sense: $b_{gc}(gx) = b_c(x)$, for all $c \in \mathcal C_\omega$, $x \in X_\omega$, and $g \in L$.
The map $b_c$ can be seen as the \emph{Busemann function} at a point of $X_\omega \cup \partial X_\omega$.
We distinguish two cases.

\begin{itemize}
	\item 
	Assume first that $\rho$ is \emph{finite}.
    By definition of $\Pi_\omega \mathcal C_k$ the distance $\dist{o_k}{c_k}$ is \oeb.
    Hence $\limo c_k$ is a well-defined point in $X_\omega$, which we still denote by $c$.
    One checks that $b_c(x) = \dist cx - \rho$, for every $x \in X_\omega$.
    \item
    Assume now that $\rho$ is \emph{infinite}.
    It follows from \autoref{res: limit action - dist to cone point} that $d_k = \dist{o_k}{c_k}$ diverges to infinity.
    For every $k \in \N$, we write $\sigma_k \colon [0, d_k] \to X_k$ for a $(1, \delta_k)$-quasi-geodesic joining $o_k$ to $c_k$.
	Passing to the limit, we get a geodesic ray $\sigma \colon \R_+ \to X_\omega$ starting at $o$, and sending $t$ to $\limo \sigma_k(t)$.
	We denote by $\xi \in \partial X_\omega$ the endpoint at infinity of $\sigma$.
	Let $x = \limo x_k$ be a point of $X_\omega$.
	It is a standard exercise of hyperbolic geometry to prove that for every $t \geq \dist{o_k}{x_k}$, we have
	\begin{equation*}
		\abs{ \left(\dist{c_k}{x_k} - \dist{c_k}{o_k} \right) - \left( \dist{\sigma_k(t)}{x_k} - \dist{\sigma_k(t)}{o_k}\right)} \leq 100\delta_k
	\end{equation*}
	Passing to the limit we get
	\begin{equation*}
		b_c(x) - b_c(o) = \lim_{t \to \infty} \left(\dist{\sigma(t)}x - \dist{\sigma(t)}o\right).
	\end{equation*}
	Consequently $b_c$ is the Busemann function at $\xi$.
\end{itemize}
As a Busemann function in a tree, $b_c$ is convex and $1$-Lipschitz.

\paragraph{Peripheral subtrees.}
Let $r \in \R_+$ and $c = [c_k]$ be a point in $\mathcal C_\omega$.
We define the \emph{open} and \emph{closed} \emph{peripheral subtrees} centered at $c$ by
\begin{equation*}
	P(c,r) = \set{x \in X_\omega}{b_c(x) < -r}
	\quad \text{and} \quad
	\bar P(c,r) = \set{x \in X_\omega}{b_c(x) \leq -r}.
\end{equation*}
For simplicity, we write $P(c)$ and $\bar P(c)$ for $P(c,0)$ and $\bar P(c,0)$ respectively.
Peripheral subtrees are \emph{strictly convex}, i.e. if $x$ and $y$ are two points in $\bar P(c,r)$, then the interior of the geodesic $\geo xy$ is contained in $P(c,r)$.
As $(b_c)$ is $L$-invariant, we get that $gP(c,r) = P(gc,r)$, for every $g \in L$.
The same statement holds for closed peripheral subtrees.

\begin{voca*}
	A subtree of $X_\omega$ is \emph{non-degenerate}, if it contains at least two distinct points.
	It is \emph{transverse} if its intersection with any peripheral subtree is degenerated.
\end{voca*}

Let us now give another description of peripheral subtrees.
Given a collection $(Y_k)$ where each $Y_k$ is a subset of $X_k$, we define $\limo Y_k$ as the subset of $X_\omega$ given by
\begin{equation*}
	\limo Y_k = \set{y = \limo y_k}{y_k \in Y_k\ \oas}.
\end{equation*}
\begin{lemm}
\label{res: balls as limit of balls}
	Let $c = [c_k]$ be a point in  $\mathcal C_\omega$.
	For every $r \in \R_+$, we have
	\begin{equation*}
		\bar P(c,r) = \limo B\left(c_k, \rho_k-r\right).
	\end{equation*}
\end{lemm}

\begin{proof}
	Let $r \in \R_+$.
	For simplicity, we denote by $B\subset X_\omega$ the limit of the balls $B(c_k, \rho_k-r)$.
	The inclusion $B \subset \bar P(c,r)$ is straightforward.
	Let us prove the converse inclusion.
	Let $x = \limo x_k$ be a point of $X_\omega$ such that $b_c(x) \leq -r$.
	In particular, there is a sequence of positive numbers $(t_k)$ that converges to zero and such that 
	\begin{equation*}
		\dist{c_k}{x_k} < \rho_k-r + t_k,\ \oas.
	\end{equation*}
	Recall that $X_k$ is a length space, for every $k \in \N$.
	Thus there exist points $x'_k \in B(c_k, \rho_k-r)$ such that $\dist{x_k}{x'_k} < t_k$ \oas.
	In particular, $x = \limo x'_k$.
	Thus $x$ belongs to $B$.
\end{proof}

We now interpret the distance in a peripheral subtree in terms of the asymptotic behavior of the angle cocycle given by Axiom~\ref{enu: family axioms - conical - 2}.

\begin{lemm}
\label{res: limit tree -  asymp angle}
	Let $c = [c_k]$ be a point in $\mathcal C_\omega$.
	For every $k \in \N$, we denote by 
	\begin{equation*}
		\theta_k \colon \mathring B(c_k, \rho_k) \times \mathring B(c_k, \rho_k) \to \R / \Theta_k\Z
	\end{equation*}
	the angle cocycle provided by Axiom~\ref{enu: family axioms - conical - 2}.
	In addition, we choose sequences $x_k, x'_k \in \mathring B(c_k, \rho_k)$ and let $d = \limo \dist {x_k}{x'_k}$. 
	Assume that there is $r \in \R_+$, such that 
	\begin{equation*}
		\limo \dist {c_k}{x_k} - \rho_k  = \limo \dist{c_k}{x'_k} - \rho_k = -r.
	\end{equation*}
	\begin{enumerate}
		\item \label{enu: limit tree -  asymp angle - direct} 
		If $d + 2r < 2 \rho$, then $\tilde \theta_k\left(x_k,x'_k\right)$ converges to zero.
		More precisely
		\begin{equation*}
			\ln  \abs{\tilde \theta_k\left(x_k,x'_k\right)}  \leq - \frac {\rho_k - r}{\epsilon_k}  + \frac d{2\epsilon_k} + o\left( \frac 1{\epsilon_k}\right),
		\end{equation*}
		with equality whenever $d\neq 0$.
		\item \label{enu: limit tree -  asymp angle - reverse} 
		Conversely, if there exists $\ell \in \R_+$ such that $\ell + 2r < 2\rho$ and 
		\begin{equation*}
			\ln  \abs{\tilde \theta_k\left(x_k,x'_k\right)}  \leq - \frac {\rho_k - r}{\epsilon_k}  + \frac\ell{2\epsilon_k} + o\left( \frac 1{\epsilon_k}\right),
		\end{equation*}
		then $d \leq \ell$.
	\end{enumerate}

	 \end{lemm}

\begin{rema}
	Let us make a few comments on this statement.
	We keep the same notation.
	\begin{itemize}
		\item Recall that $\tilde \theta_k(x_k, x'_k)$ stands for the unique representative of $\theta_k(x_k,x'_k)$ in $(-\Theta_k/2, \Theta_k/2]$.

		\item If $x = \limo x_k$ and $x' = \limo x_k$ are two well-defined points of $X_\omega$, then our assumption can be reformulated as
		\begin{equation*}
			b_c(x) = b_c(x')=-r
			\quad \text{and} \quad 
			\dist x{x'} + 2r < 2\rho.
		\end{equation*}
		Nevertheless this more general form will be useful later when the sequences $\dist{x_k}{o_k}$  and $\dist{x'_k}{o_k}$ may not be \oeb.
		
		\item It follows from the triangle inequality that $d + 2r \leq 2\rho$.
		The parameter $\rho = \limo \rho_k$ can be infinite, in which case the assumption $d +2r < 2\rho$ is automatically satisfied, provided $d$ is finite.
		
		\item By convention $\ln(0) = -\infty$, so that the inequalities hold even if $\tilde \theta_k\left(x_k,x'_k\right)$ vanishes.
	\end{itemize}
\end{rema}

\begin{proof}
	Recall that, in this section, all distances are measured with the \emph{rescaled} metric of the spaces $\epsilon_k X_k$.
	For simplicity we let 
	\begin{equation*}
		s_k = \dist{c_k}{x_k}
		\quad \text{and} \quad
		s'_k = \dist{c_k}{x'_k},
	\end{equation*}
	so that 
	\begin{equation}
	\label{eqn: limit tree -  asymp angle - s}
		s_k = \rho_k - r + o(1)
		\quad \text{and} \quad
		s'_k = \rho_k - r + o(1)
 	\end{equation}
	In particular, $\abs{s_k- s'_k} = o(1)$.
	In addition, we let
	\begin{equation*}
		\beta_k = \min \left\{ \pi, \tilde \theta_k(x_k, x'_k)\right\},
	\end{equation*}
	and define the quantity $d_k \in \R_+$ by the relation
	\begin{equation}
	\label{eqn: limit tree -  asymp angle - d original}
		\cosh\left(\frac {d_k}{\epsilon_k}\right) 
		= \cosh\left(\frac{s_k}{\epsilon_k}\right)\cosh\left(\frac{s'_k}{\epsilon_k}\right) - \sinh\left(\frac{s_k}{\epsilon_k}\right)\sinh\left(\frac{s'_k}{\epsilon_k}\right)\cos\beta_k.
	\end{equation}
	It follows from Axiom~\ref{enu: family axioms - conical - 2} that 
	\begin{equation*}
		d_k = \dist {x_k}{x'_k} + o(1) = d + o(1).
	\end{equation*}
	(Equation (\ref{eqn: limit tree -  asymp angle - d original}) is the law of cosine in $\H^2$.
	The triangle inequality in the hyperbolic plane implies in particular that $\abs{s_k-s'_k} \leq d_k$.)
	A trigonometric computation shows that (\ref{eqn: limit tree -  asymp angle - d original}) is equivalent to 
	\begin{equation}
	\label{eqn: limit tree -  asymp angle - d}
		\sinh\left(\frac{s_k}{\epsilon_k}\right)\sinh\left(\frac{s'_k}{\epsilon_k}\right)\sin^2\left(\frac{\beta_k}2\right)
		=  \sinh\left(\frac {d_k + \abs{s_k-s'_k}}{2\epsilon_k}\right)\sinh\left(\frac {d_k - \abs{s_k-s'_k}}{2\epsilon_k}\right).
	\end{equation}
	Note that $(s_k)$ and $(s'_k)$ converge to $\rho - r$, which is assume to be positive both  in \ref{enu: limit tree -  asymp angle - direct}  and \ref{enu: limit tree -  asymp angle - reverse}.
	Hence $(s_k / \epsilon_k)$ and $(s'_k / \epsilon_k)$ diverge to infinity.
	If $d > 0$, then both sequences
	\begin{equation*}
		\left(\frac {d_k + \abs{s_k-s'_k}}{2\epsilon_k}\right)
		\quad \text{and} \quad 
		\left(\frac {d_k - \abs{s_k-s'_k}}{2\epsilon_k}\right)
	\end{equation*}
	diverge to infinity.
	This follows indeed from the fact that $(\epsilon_k)$ converges to zero while $\abs{s_k - s'_k} = o(1)$.
	Using the fact that $t \mapsto \ln \sinh(t) - t$ is bounded in a neighborhood of infinity, (\ref{eqn: limit tree -  asymp angle - d}) becomes
	\begin{equation*}
		\ln \circ \sin \left(\frac{\abs{\beta_k}}2\right) = - \frac 1{2\epsilon_k} \left(s_k + s'_k - d_k \right) + O(1).
	\end{equation*}	
	Recall that $\ln \sinh(t) \leq t$ for every $t \in \R_+$.
	If $d = 0$, we can use this inequality to bound from above the right hand side of (\ref{eqn: limit tree -  asymp angle - d}).
	We get 
	\begin{equation*}
		\ln \circ \sin \left(\frac{\abs{\beta_k}}2\right) \leq  - \frac 1{2\epsilon_k} \left(s_k + s'_k - d_k \right) + O(1).
	\end{equation*}	
	Plugging in the asymptotic behaviors of $s_k$ and $s'_k$ given by (\ref{eqn: limit tree -  asymp angle - s})  as well as the one of $d_k$ we obtain
	\begin{equation}
	\label{eqn: limit tree -  asymp angle - almost there}
		\ln \circ \sin \left(\frac{\abs{\beta_k}}2\right) \leq - \frac {\rho_k - r}{\epsilon_k}  + \frac d{2\epsilon_k}  + o\left( \frac 1{\epsilon_k}\right),
	\end{equation}
	with equality whenever $d > 0$.
	
	Assume now that  $d+ 2r < 2\rho$.
	It follows that the right hand side of the previous inequality diverges to $- \infty$.
	In particular, $\beta_k$ and thus $\tilde \theta_k(x_k,x'_k)$ converges to zero, moreover
	\begin{equation}
	\label{eqn: limit tree -  asymp angle - sine}
		\ln \circ \sin \left(\frac{\abs{\beta_k}}2\right) = \ln\abs{\tilde \theta_k(x_k,x'_k)} + O(1).
	\end{equation}
	Combining the previous two identities leads to Point~\ref{enu: limit tree -  asymp angle - direct}.
	
	Assume now that there exists $\ell \geq 0$ such that $\ell + 2r < 2\rho$ and 
	\begin{equation*}
		\ln  \abs{\tilde \theta_k\left(x_k,x'_k\right)}  \leq - \frac {\rho_k - r}{\epsilon_k}  + \frac\ell{2\epsilon_k} + o\left( \frac 1{\epsilon_k}\right),
	\end{equation*}
	Without loss of generality we can assume that $d > 0$.
	Otherwise the result is straightforward.
	Since $\ell + 2r < 2\rho$, the right hand side of the previous inequality diverges to $- \infty$.
	In particular, $\tilde \theta_k(x_k,x'_k)$ converges to zero and (\ref{eqn: limit tree -  asymp angle - sine}) still holds.
	As $d > 0$, Inequality (\ref{eqn: limit tree -  asymp angle - almost there}) is actually an equality.
	Combining these observations we get
	\begin{equation*}
		- \frac {\rho_k - r}{\epsilon_k}  + \frac d{2\epsilon_k}  - o\left( \frac 1{\epsilon_k}\right)
		 \leq \ln \circ \sin \left(\frac{\abs{\beta_k}}2\right) 
		 \leq - \frac {\rho_k - r}{\epsilon_k}  + \frac\ell{2\epsilon_k} + o\left( \frac 1{\epsilon_k}\right).
	\end{equation*}
	Hence $d \leq \ell$, which completes the proof of Point~\ref{enu: limit tree -  asymp angle - reverse}.
\end{proof}

\begin{lemm}
\label{res: intersection pieces}
	Let $c, c' \in \mathcal C_\omega$.
	The following statements are equivalent.
	\begin{enumerate}
		\item \label{enu: intersection pieces - equality}
		$c = c'$
		\item \label{enu: intersection pieces - close}
		$\bar P(c)$ and $\bar P(c')$ have a non-degenerate intersection.
		\item \label{enu: intersection pieces - open}
		$P(c)$ and $P(c')$ have a non-empty intersection.
	\end{enumerate}
\end{lemm}

\begin{proof}
	The implication \ref{enu: intersection pieces - equality} $\Rightarrow$ \ref{enu: intersection pieces - close} is trivial.
	Assume that $\bar P(c)\cap \bar P(c')$ is non-degenerate.
	Since balls/horoballs are strictly quasi-convex, $P(c) \cap P(c')$ is non-empty hence  \ref{enu: intersection pieces - close} $\Rightarrow$ \ref{enu: intersection pieces - open} holds.
	Suppose now that $P(c) \cap P(c')$ is non-empty.
	We write $c =[c_k]$ and $c' = [c'_k]$ where $(c_k)$ and $(c'_k)$ are sequences of $\Pi_\omega \mathcal C_k$.
	Combining \autoref{res: balls as limit of balls} with the triangle inequality, one observes that $\dist{c_k}{c'_k} < 2 \rho_k$ \oas.
	Recall that any two distincts apices in $\mathcal C_k$ are at a distance at least $2\rho_k$ far apart.
	Thus $c_k = c'_k$ \oas, i.e. $c = c'$.
	\end{proof}

\begin{lemm}
\label{res: large intersection w/ pieces}
	Let $c \in \mathcal C_\omega$ and $r \in \R_+$.
	Let $x,x' \in X_\omega \setminus P(c)$.
	If $\geo x{x'}$ intersects $\bar P(c,r)$, then $\geo x{x'} \cap \bar P(c)$ is an geodesic of length at least $2r$.
\end{lemm}

\begin{proof}
	Since $\bar P(c)$ is closed subtree, the intersection $\geo x{x'} \cap \bar P(c)$ is a geodesic.
	Fix a point $z \in \geo x{x'}$ such that $b_c(z) \leq -r$.
	We denote by $y$ a point of $\geo zx$ such that $b_c(y) = 0$.
	Such a point exists as $b_c$ is continuous.
	The point $y'$ on $\geo z{x'}$ is defined in a similar way.
	By construction $\geo y{y'}$ is contained in $\bar P(c)$.
	Since $b_c$ is $1$-Lipschitz, we have $\dist zy \geq \abs{b_c(y) - b_c(z)}$, that is $\dist zy \geq r$.
	Similarly $\dist z{y'} \geq r$.
	Hence the result.
\end{proof}

%
\subsection{Action of the limit group}
%
\label{sec: action of the limit group}

In this section we study the action of $L$ on the $\R$-tree $X_\omega$ built above.
Recall that all the groups $\Gamma_k$ are CSA \cite{Myasnikov:1996aa}.
Since CSA is a closed property for the topology or marked groups, $L$ is CSA as well.

\begin{defi}
\label{def: jsj - subgroups}
	An element $g \in L$ is \emph{elusive} if given any pre-image $\tilde g \in G$ of $g$, the element $\phi_k(\tilde g) \in \Gamma_k$ is elusive \oas.
	An element of $L$ is \emph{visible} if it is not elusive.
	A subgroup $L$ is \emph{elusive} if all its elements are elusive.
	It is \emph{visible} otherwise.
\end{defi}

\paragraph{Elliptic elements.}

\begin{lemm}
\label{res: limit tree - elliptics become elliptic}
	Let $g \in G$.
	If $\phi_k(g)$ is elliptic (for its action on $X_k$) \oas, then $\eta(g)$ is elliptic (for its action on $X_\omega$).
\end{lemm}

\begin{proof}
	Let $x_k$ be a $\delta_k$-projection of $o_k$ onto $\fix{\phi_k(g), 6\delta_k}$.
	It follows from (\ref{eqn: displacement outside fixed set}) that 
	\begin{equation*}
		\dist {x_k}{o_k} \leq \frac 12 \dist{o_k}{\phi_k(g)o_k} + 3\delta_k.
	\end{equation*}
	Consequently $\dist {x_k}{o_k}$ is \oeb, thus $x = \limo x_k$ is a well-defined point of $X_\omega$.
	Moreover $\eta(g)$ fixes $x$.
\end{proof}

\paragraph{Elusive subgroups.}

\begin{lemm}
\label{res: limit tree - elusive sbgps are elliptic}
	Every non-trivial elusive element $g \in L$ fixes a single point in $X_\omega$.
	In particular, elusive subgroups are elliptic in $X_\omega$.
\end{lemm}

\begin{proof}
	We know by \autoref{res: limit tree - elliptics become elliptic} that $g$ is elliptic. 
	Let us prove that it has a unique fix point.
	Let $\tilde g \in G$ a pre-image of $g$.
	Let $y = \limo y_k$ and $z = \limo z_k$ be two points fixed by $g$.
	There exists a sequence $(d_k)$ converging to zero such that $y_k$ and $z_k$ belongs to $\fix{\phi_k(\tilde g), d_k}$ \oas.
	By Axiom~\ref{enu: family axioms - elusive} the diameter of $\fix{\phi_k(\tilde g), d_k}$ is bounded above by $ d_k + 2\delta_k$ \oas.
	Hence $y = z$, which proves the first part of our statement.
	
	Let $H$ be an elusive subgroup of $L$.
	If $H$ is not elliptic, then there exist $u_1, u_2 \in H$ whose unique fixed point in $X_\omega$ are distinct.
	It follows that $u_1u_2$ is loxodromic in $X_\omega$, hence cannot be elusive.
\end{proof}

\paragraph{Cone point stabilizers.}

For every $c \in \mathcal C_\omega$, we write $L_c$ for its stabilizer in $L$.

\begin{lemm}[Cone point stabilizers are radial]
\label{res: tree-graded - fixed point set}
	Let $c \in \mathcal C_\omega$.
	Every element $g \in L_c$ is elliptic.
	Moreover if $g$ is non trivial, then there exists a finite $r \in \intval 0\rho$ such that $\fix g = \bar P(c,r)$.
\end{lemm}

\begin{proof}
	We write $c =[c_k]$ where $(c_k)$ is a sequence of $\Pi_\omega \mathcal C_k$.
	Let $g \in L_c$ and $\tilde g \in G$ a pre-image of $g$.
	It follows from the definition of $\mathcal C_\omega$ that $\phi_k(\tilde g)$ belongs to $\stab {c_k}$ \oas.
	Hence $g$ is elliptic by \autoref{res: limit tree - elliptics become elliptic}.

	Assume now that $g$ is not trivial.
	We first prove that $\fix g$ is contained in $\bar P(c)$.
	Let $x = \limo x_k$ be a point in $\fix g$.
	By definition there exists a sequence $(d_k)$ converging to zero such that $x_k \in \fix{\phi_k(\tilde g), d_k}$ \oas.
	Since $g$ is non trivial, so is $\phi_k(\tilde g)$ \oas.
	In particular, $\phi_k(\tilde g)$ is $(\rho_k + \delta_k)$-thin at $c_k$ by Axiom~\ref{enu: family axioms - conical - 1}, thus 
	\begin{equation*}
		\dist{x_k}{c_k}\leq  \rho_k + d_k/2 + \delta_k,\ \oas.
	\end{equation*}
	Hence $b_c(x) \leq 0$ as announced.	
	We now let
	\begin{equation*}
		r = \inf \set{-b_c(x)}{gx=x,\ x \in X_\omega}
	\end{equation*}
	It follows from our previous discussion that $r$ is a non negative finite number.
	Note also that $\fix g$ is contained in $\bar P(c,r)$.
	We now focus on the other inclusion.
	We distinguish two cases.
	
	\paragraph{Case 1.} \emph{Assume that $r = \rho$}.
	In particular, $\rho$ is finite and $c = \limo c_k$ is a well-defined point of $X_\omega$, which is fixed by $g$.
	In addition $\bar P(c,\rho)$ is reduced to $\{c\}$, hence is contained in $\fix g$.
		
	\paragraph{Case 2.} \emph{Assume that $r \in [0,\rho)$.}
	For every $k \in \N$, we denote by 
	\begin{equation*}
		\theta_k \colon \mathring B(c_k, \rho_k) \times \mathring B(c_k, \rho_k) \to \R / \Theta_k\Z
	\end{equation*}
	the angle cocycle provided by Axiom~\ref{enu: family axioms - conical - 2}.
	Since $\theta_k$ is a $\stab{c_k}$-invariant cocycle, the quantity $\theta_k(\phi_k(\tilde g)x_k, x_k)$ does not depend on the point $x_k$.
	For simplicity we denote by $\beta_k$ its unique representative in $(-\Theta_k/2, \Theta_k/2]$.
	By definition of $r$, for every $s > r$, the element $g$ fixes a point $x \in X_\omega$, with $-b_c(x) \leq s$.
	Since $2r < 2\rho$, it follows from \autoref{res: limit tree -  asymp angle}\ref{enu: limit tree -  asymp angle - direct}  that the sequence
	\begin{equation*}
		\epsilon_k \ln \abs{\beta_k} + \rho_k
	\end{equation*}
	admits an $\omega$-limit in $[-\infty, r]$.
	Using \autoref{res: limit tree -  asymp angle}\ref{enu: limit tree -  asymp angle - reverse}, we get that $\bar P(c,r)$ is contained in $\fix g$.
\end{proof}

Combined with \autoref{res: intersection pieces}, it yields the following statement.

\begin{coro}
\label{res: intersection stab of cone pts}
	Let $c,c'\in \mathcal C_\omega$.
	If $c \neq c'$, then $L_c \cap L_{c'}$ is trivial.
\end{coro}

\begin{lemm}
\label{res: tree-graded - abelian piece stabilizer}
	Let $c \in \mathcal C_\omega$.
	The global stabilizer of $\bar P(c)$ coincides with $L_c$.
	It is a maximal abelian subgroup, provided it is not trivial.
\end{lemm}

\begin{proof}
	We write $c = [c_k]$ where $(c_k)$ is a sequence of $\Pi_\omega \mathcal C_k$.
	The first claim is a direct consequence of \autoref{res: intersection pieces}.
	Let $g_1, g_2 \in L_c$ and $\tilde g_1, \tilde g_2 \in G$ respective pre-images of $g_1$ and $g_2$.
	It follows from the definition of $\mathcal C_\omega$, that $\phi_k(g_i)c_k = c_k$ \oas.
	According to Axiom~\ref{enu: family axioms - elem subgroups}, $\stab{c_k}$ is abelian.
	Hence $\phi_k(\tilde g_1)$ and $\phi_k(\tilde g_2)$ commute \oas, thus so do $g_1$ and $g_2$.
	
	Assume now that $L_c$ is non trivial.
	Let $g_0 \in L_c \setminus\{1\}$.
	Let $g \in L$ which commutes with $g_0$.
	In particular, $g_0$ fixes $gc$.
	According to \autoref{res: intersection stab of cone pts}, $gc = c$.
	Hence every abelian subgroup containing $g_0$ is contained in $L_c$.
\end{proof}

\begin{prop}[Root stability]
\label{res: root stability}
	Let $g \in L \setminus\{1\}$.
	If $\fix{g}$ is non-degenerate, then for every $m \in \Z\setminus\{0\}$, we have $\fix g = \fix{g^m}$.
	Moreover, $g$ has infinite order.
\end{prop}

\begin{proof}
	We fix a pre-image $\tilde g \in G$ of $g$.
	The set $\fix g$ is contained in $\fix{g^m}$.
	Let $x = \limo {x_k}$ be a point in $\fix{g^m}$.
	In order to show that $x \in \fix g$, it suffices to establish that 
	\begin{equation}
	\label{eqn: root stability}
		\dist{\phi_k(\tilde g)x_k}{x_k} \leq \dist{\phi_k(\tilde g^m)x_k}{x_k} + 1000\delta_k,\ \oas.
	\end{equation}
	We distinguish two cases.
	Assume first that $\phi_k(\tilde g)$ is loxodromic \oas.
	Then (\ref{eqn: root stability}) is a standard exercise of hyperbolic geometry.
	Note also that $g$ cannot have finite order.
	Indeed, otherwise we could find $p \in \N \setminus\{0\}$ such that $\phi_k(\tilde g^p) = 1$ \oas, which contradicts the fact that $\phi_k(\tilde g)$ is loxodromic \oas.
	
	Assume now that $\phi_k(\tilde g)$ is elliptic \oas.
	Note that $\phi_k(\tilde g)$ cannot be elusive \oas, otherwise $\fix{g}$ would be reduced to a point (\autoref{res: limit tree - elusive sbgps are elliptic}).
	Hence $g$ fixes a cone point $c =[c_k]$ where $(c_k)$ is a sequence of $\Pi_\omega \mathcal C_k$.
	For every $k \in \N$, we denote by 
	\begin{equation*}
		\theta_k \colon \mathring B(c_k, \rho_k) \times \mathring B(c_k, \rho_k) \to \R / \Theta_k\Z
	\end{equation*}
	the angle cocycle provided by Axiom~\ref{enu: family axioms - conical - 2}.
	Since $\theta_k$ is a $\stab{c_k}$-invariant co-cycle, the quantity $\theta_k(\phi_k(\tilde g)x_k, x_k)$ does not depend on the point $x_k\in\mathring B(c_k, \rho_k)$.
	For simplicity we denote by $\beta_k$ its unique representative in $(-\Theta_k/2, \Theta_k/2]$.
	Since $g$ is non-trivial and $\fix g$ non-degenerate, the latter has the form $\fix g = \bar P(c, r)$ with $r \in [0,\rho)$ (\autoref{res: tree-graded - fixed point set}).
	It follows from \autoref{res: limit tree -  asymp angle} that $(\beta_k)$ converges to zero.
	Thus $m\beta_k$ is the unique representative in $(-\Theta_k/2, \Theta_k/2]$ of 
	\begin{equation*}
		\theta_k(\phi_k(\tilde g^m)x_k, x_k) = m \theta_k(\phi_k(\tilde g)x_k, x_k).
	\end{equation*}
	Since $\abs{\beta_k} \leq \abs{m \beta_k} \leq \pi$, Inequality (\ref{eqn: root stability}) follows from Axiom~\ref{enu: family axioms - conical - 2}.
	Assume now that $g$ has finite order.
	There exists $p \in \N \setminus\{0\}$ such that $\phi_k(\tilde g^p) = 1$ \oas.
	Using the cocycle $\theta_k$ we get
	\begin{equation*}
		0 = \theta_k(\phi_k(\tilde g^p)x_k, x_k) = p \theta_k(\phi_k(\tilde g)x_k, x_k).
	\end{equation*}
	Hence $\theta_k(\phi_k(\tilde g)x_k, x_k)=0$.
	It follows from \ref{enu: family axioms - conical - 2} that $\phi_k(\tilde g)=1$ \oas, which contradicts our assumption.
\end{proof}

\begin{coro}
\label{res: coro root stability}
	Let $c \in \mathcal C_\omega$ and $r,s \in [0, \rho)$.
	Let $A_r$ and $A_s$ be the respective pointwise stabilizers of $\bar P(c,r)$ and $\bar P(c,s)$.
	If $r < s$, then the quotient $A_s/A_r$ is torsion-free.
\end{coro}

\paragraph{Arcs and tripods.}
An \emph{arc} is a subset homeomorphic to $[0,1]$.
A subtree $Y$ of $X_\omega$ is \emph{stable}, if for every arc $I\subset Y$, the pointwise stabilizers of $I$ and $Y$ coincide.
The action of $G$ on $X_\omega$ is \emph{super-stable} if any arc with non-trivial stabilizer is stable.

\begin{lemm}
\label{res: abelian arc stabilizers}
	Let $[x,y]$ be an arc in $X_\omega$.
	Assume that its pointwise stabilizer $H$ is non-trivial. 
	Then $H$ is torsion-free and abelian.
	Moreover, all its non-trivial elements are visible.
\end{lemm}

\begin{proof}
	We write $x = \limo x_k$ and $y = \limo y_k$.
	It follows from \autoref{res: limit tree - elusive sbgps are elliptic} that every element in $H\setminus\{1\}$ is visible.
	Moreover $H$ is torsion-free (\autoref{res: root stability}).
	Let $g_1, g_2 \in H \setminus\{1\}$.
	Let $\tilde g_1, \tilde g_2 \in G$ be respective pre-images of $g_1$ and $g_2$.
	By definition there exists a sequence $(d_k)$ converging to zero such that $x_k,y_k \in \fix{W_k,d_k}$, where $W_k = \{ \phi_k(\tilde g_1), \phi_k(\tilde g_2)\}$.
	If $W_k$ does not generate an elementary subgroup, then \ref{enu: family axioms - acylindricity} tells us that 
	\begin{equation*}
		\dist {x_k}{y_k} \leq \diam \left( \fix{W_k, d_k} \right) \leq 4d_k + \delta_k.
	\end{equation*}
	Passing to the limit we get $x = y$ which contradicts our assumption.
	The set $W_k$ generates an elementary subgroup which is not elusive, hence abelian by \ref{enu: family axioms - elem subgroups}.
	Consequently $\phi_k(\tilde g_1)$ and $\phi_k(\tilde g_2)$ commute, hence so do $g_1$ and $g_2$.
\end{proof}

\begin{lemm}
\label{res: tree-graded - transverse tripod}
	Let $x$, $y$, $z$ be three points of $X_\omega$.
	We assume that they are not contained in a single peripheral subtree and they do not lie on a transverse geodesic.
	Then the pointwise stabilizer of the tripod $[x,y,z]$ is trivial.
\end{lemm}

\begin{proof}
	We write $x = \limo x_k$, $y = \limo y_k$ and $z = \limo z_k$.
	Assume that there exists $g \in L \setminus\{1\}$ fixing $x$, $y$ and $z$.
	Let $\tilde g \in G$ be a pre-image of $g$.
	As usual there exists a sequence $(d_k)$ converging to zero such that $x_k$, $y_k$ and $z_k$ belong to $\fix{\phi_k(\tilde g),d_k}$.
	We distinguish two cases.

	\begin{itemize}
		\item Assume first that $\phi_k(\tilde g)$ is elliptic \oas.
		As $g$ is not trivial, $\phi_k(\tilde g)$ is non-trivial \oas.
		If $\phi_k(\tilde g)$ is elusive \oas, then $\fix g$ is reduced to a single point (\autoref{res: limit tree - elusive sbgps are elliptic}).
		If $\phi_k(\tilde g)$ is conical \oas, then it fixes some $c_k \in \mathcal C_k$.
		Hence $g$ fixes $c = [c_k]$ and $\fix g$ is contained in $\bar P(c)$ (\autoref{res: tree-graded - fixed point set}).
		In both cases  $x$, $y$ and $z$ belong to a single peripheral subtree, or they lie on a (degenerated) transverse geodesic.
		\item Assume now that $\phi_k(\tilde g)$ is loxodromic \oas.
		We denote by $\xi_k^-, \xi_k^+ \in \partial X_k$ the repulsive and attractive points of $\phi_k(\tilde g)$.
		It follows from Axiom~\ref{enu: family axioms - loxodromic} that for every $t \in \{x_k, y_k, z_k\}$ we have $\gro{\xi_k^-}{\xi_k^+}t \leq d_k/2 + \delta_k$.
		Hence up to permuting $x$, $y$ and $z$, we observe that
		\begin{equation*}
			\gro{x_k}{z_k}{y_k} \leq 3d_k/2 + 100\delta_k,\ \oas.
		\end{equation*}
		Hence $y$ lies on the geodesic $\geo xz$.
		We now claim that this geodesic is transverse.
		Assume it is not.
		Since $g$ fixes pointwise this geodesic, it fixes a point in $P(c)$ for some $c \in \mathcal C_\omega$.
		Hence by \autoref{res: intersection pieces}, $g$ fixes $c$ and $\phi_k(\tilde g)$ is elliptic \oas, which contradicts our assumption. \qedhere
	\end{itemize}	
\end{proof} 

\begin{lemm}
\label{res: tree-graded - transverse stability}
	Let $\geo xy$ be an arc of $X_\omega$ which is not contained in a peripheral subtree.
	If its pointwise stabilizer is non trivial, then it is stable.
\end{lemm}

\begin{rema}
	It follows from the above statement that if $\rho = 0$, then the action of $L$ on $X_\omega$ is super-stable (i.e. any arc with non-trivial stabilizer is stable).
	Similarly, if $H$ is a subgroup of $L$ preserving a transverse subtree $Y\subset X_\omega$, then the action of $H$ on $Y$ is super-stable.
\end{rema}

\begin{proof}
	The proof goes as in \cite[Proposition~4.2]{Rips:1994jg}.
	Let $\geo {x'}{y'}$ be an arc contained in $\geo xy$.
	Without loss of generality we can assume that $x,x',y',y$ are ordered in this way on $\geo xy$.
	Let $H$ and $H'$ be the pointwise stabilizers of $\geo xy$ and $\geo{x'}{y'}$ respectively, so that $H$ is contained in $H'$.
	Let $g'$ be an element of $H'$.
	By assumption $H$ contains a non trivial element, say $g$.
	It follows from \autoref{res: abelian arc stabilizers} that $g$ and $g'$ commutes.
	Hence $g$ fixes $x$, $y$ and $g'y$.
	According to our assumption, $x$ and $y$ are not in the same peripheral subtree.
	Hence \autoref{res: tree-graded - transverse tripod}, tells us that $x$, $y$ and $g' y$ lie on a (transverse) geodesic, which forces $g' y = y$.
	We prove in the same way that $g' x = x$, hence $g'$ belongs to $H$.
\end{proof}

\paragraph{Small subgroups.}
We say that a subgroup $H$ of $L$ is \emph{small} (for its action on $X_\omega$)  if  it does not contain two elements acting loxodromically on $X_\omega$ whose corresponding axes have a bounded intersection.

\begin{prop}
\label{res: tree-graded - small subgroups}
	Let $H$ be a small subgroup of $L$.
	If $H$ is not elliptic (for its action on $X_\omega$), then $H$ is visible and abelian.
	In addition, if $H$ is not cyclic, then either there exists $c \in \mathcal C_\omega$ such that $H$ is contained in $L_c$ or $H$ preserves a transverse subtree.
\end{prop}

\begin{proof}
	According to the classification of actions on trees, either $H$ fixes a point in the boundary at infinity $\partial X_\omega$ of $X_\omega$ or $H$ preserves a bi-infinite geodesic.
	In both cases, it contains a non-trivial subgroup $H^+$ of index at most $2$ which fixes a point in $\xi \in \partial X_\omega$.
	Moreover, if $[H:H^+] = 2$, then $H$ cannot be abelian.
	Let $\sigma \colon \R_+ \to X_\omega$ be a geodesic ray pointing to $\xi$.
	If $\beta_\xi \colon X_\omega \to \R$, is the Busemann cocycle at $\xi$, then the map $\psi \colon H^+ \to \R$ sending $g$ to $\beta_\xi(go,o)$ is a homomorphism whose kernel  $H_0^+$ is exactly the set of elliptic elements of $H^+$.

	Let us first show that $H$ is abelian and visible.
	We focus first on $H_0^+$.
	For every $g \in H_0^+$, there exists $t \in \R_+$ such that $g$ fixes pointwise $\sigma$ restricted to $[t, \infty)$.
	It follows from \autoref{res: abelian arc stabilizers} that non-trivial elements of $H_0^+$ are visible and pairwise commutes.
	Suppose now that $H_0^+$ is not trivial.
	Note that $H$ normalizes $H^+_0$.
	Since $L$ is CSA, it follows that $H$ is abelian and visible.
	If $H_0^+$ is trivial, then $H^+$ is isomorphic to its image under $\psi$, thus is abelian.
	In addition, it cannot be elusive, otherwise $H$ would be elliptic, thus $H^+$ is abelian and visible.
	Note that $H^+$ is a normal subgroup of $H$.
	Using again the fact that $L$ is CSA, we conclude that $H$ is abelian and visible.
	As we observe earlier it implies that $H = H^+$, hence $H$ fixes $\xi$.

	Suppose that there exists $c \in \mathcal C_\omega$ such that $H \cap L_c$ is not trivial.
	Since $L_c$ is a maximal abelian subgroup (\autoref{res: tree-graded - abelian piece stabilizer}) $H$ is contained in $L_c$.
	From now on we assume that $H \cap L_c$ is trivial for every $c \in \mathcal C_\omega$.
	We are going to prove that either $H$ preserves a transverse subtree or $H$ is cyclic.
	We distinguish two cases.
	Suppose that there exists $t \in \R$ such that $\sigma$ restricted to $[t, \infty)$ is contained in a transverse subtree, say $Y$.
	Since $H$ fixes $\xi$, for every $g \in H$, $g \sigma \cap \sigma$ is an infinity ray.
	It follows that $H$ preserves $Y$.
	
	Suppose now that for every $t \in \R$, the ray $\sigma$ restricted to $[t, \infty)$ is not transverse.
	We claim that $H_0^+$ is trivial.
	Indeed, otherwise we could find $c \in \mathcal C_\omega$ and $g \in H_0^+ \setminus \{1\}$ such that $g$ fixes point in $\sigma \cap P(c)$.
	In particular, $g$ would belong to $H \cap L_c$ (\autoref{res: intersection pieces}) which contradicts our assumption.
	It follows from our claim that all elements of $H$ are loxodromic.
	Since $H$ is abelian, it necessarily preserves a bi-infinite geodesic $\mu \colon \R \to X$.
	Note that $\sigma \cap \mu$ is a geodesic ray whose endpoint at infinity is $\xi$.
	In particular, $\mu$ is not transverse either.
	There exists a cone point $c\in \mathcal C_\omega$ such that $\mu \cap \bar P(c)$ contains a bounded arc, say $\geo xy$.
	Since $H \cap L_c$ is trivial, $g \geo xy\cap\geo xy$ is degenerated, for every $g \in H\setminus\{1\}$.
	It follows that $\norm g \geq \dist xy > 0$.
	As all the translation lengths of the non-trivial elements of $H$ are bounded from below, $H$ is cyclic.
\end{proof}


%
\subsection{Decomposition of the action}
%
\label{sec: prelim decomposition}
Recall that $T$ is the minimal $L$-invariant subtree of $X_\omega$.
The goal of this section is to decompose the action of $L$ on $T$ into well understood ``elementary bricks'', see \autoref{res: final decomposition tree}.
We denote by $\mathcal A$ (\resp $\mathcal H$) the class of all abelian (\resp elusive) subgroups of $L$.
In addition $\mathcal A_{\rm nc}$ stands for the collection of all non-cyclic abelian subgroups of $L$.
From now on, we make the following additional assumption
\begin{assu}
\label{ass: splitting - freely indecomposable}
	The group $L$ is freely indecomposable relative to $\mathcal H$.
\end{assu}

It is a standard fact that the basepoint $o = \limo o_k$ belongs to the minimal subtree $T$.
However if $(x_k)$ is a sequence of points minimizing the restricted energy $\lambda^+_1(\phi_k,U)$ there is a priori no reason, first that $x = \limo x_k$ is a well-defined point of $X_\omega$, and second that it belongs to $T$.
This is the purpose of the next statement.

\begin{lemm}
\label{res: basepoint in minimal tree}
	For every $k \in \N$, there exists a point $o^+_k \in X^+_k$ such that
	\begin{enumerate}
		\item the limit $o^+ = \limo o^+_k$ is a well defined point of $X_\omega$ which belongs to $T$.
		\item $\displaystyle \lambda_1(\phi_k, U, o^+_k) = \lambda^+_1(\phi_k,U) + o\left( \frac 1{\epsilon_k}\right)$.
	\end{enumerate}
\end{lemm}

\begin{rema}
	Unlike in the rest of this section, the distance used to compute the energies $ \lambda_1(\phi_k, U, o^+_k)$ and $\lambda^+_1(\phi_k,U)$ is the one of $X_k$ and not its rescaled version.
	Exceptionally, the proof will also take place in the non-rescaled space.
\end{rema}

\begin{proof}
	We chose the basepoint $o_k$ using \autoref{res: comparing energies}.
	In particular, for every $k \in \N$, there exists a point $x_k \in X^+_k$ such that $\lambda_1(\phi_k, U, x_k) \leq \lambda^+_1(\phi_k, U) + 20\delta$ and $\dist{o_k}{x_k} \leq \card U\lambda_\infty(\phi_k,U) + 20\delta$.
	In particular, $x = \limo x_k$ is a well-defined point of $X_\omega$.
	We denote by $p = \limo p_k$ its projection onto $T$.
	
	We claim that there exists a point $o^+ = \limo o^+_k$ in $T$ such that $o^+_k \in X^+_k$ \oas\ and $\dist {o^+}p \leq \dist xp$.
	If $p_k$ belongs to the thick part $X^+_k$ \oas\ then we choose $o^+_k = p_k$.
	Let us assume now that there exists $c_k \in \mathcal C_k$ such that $\dist{p_k}{c_k} < \rho_k / \epsilon_k$ \oas\ (recall that the set of cone points $\mathcal C_k$ is $2\rho_k$-separated in the \emph{rescaled} space $\epsilon_k X_k$).
	It follows that $c = [c_k]$ belongs to $\mathcal C_\omega$ and $p \in \bar P(c)$.
	Since $T$ is the minimal $L$-invariant subtree, the point $p$ lies on a bi-infinite geodesic $\sigma$ of $T$.
	We denote by $o^+ = \limo o^+_k$ the point on $\sigma$ that is the closest to $p$ such that $b_c(o^+) = 0$.
	In particular, one can choose $o^+_k$ in the thick part $X^+_k$ for every $k \in \N$.
	Note also that $\dist p{o^+} = -b_c(p)$.
	Since $x_k$ lies in the thick part $X^+_k$ \oas, we necessarily have $b_c(x) \geq 0$.
	As $b_c$ is $1$-Lipschitz we get
	\begin{equation*}
		\dist p{o^+}   \leq -b_c(p) \leq - b_c(p) + b_c(x) \leq \abs{b_c(p) - b_c(x)} \leq \dist px,
	\end{equation*}
	which completes the proof of our claim.
	
	Let $g \in L$.
	Recall that the point $p$ is a projection of $x$ onto $T$. Hence $\dist {gx}x = \dist {gp}p + 2\dist xp$.
	Combining our previous claim with the triangle inequality, we get
	\begin{equation*}
		\dist {go^+}{o^+}  \leq  \dist {gp}p + 2\dist {o^+}p \leq  \dist {gp}p + 2\dist xp \leq \dist {gx}x.
	\end{equation*}
	We sum this inequality for $g = \eta(u)$ when $u$ runs over $U$.
	The result we get exactly says that 
	\begin{equation*}
		\epsilon_k\lambda_1(\phi_k, U,o^+_k) \leq \epsilon_k\lambda_1(\phi_k, U,x_k) + o(1).
	\end{equation*}
	The point $x_k$ has been chosen to (almost) minimize the restricted energy.
	Thus
	\begin{equation*}
		\epsilon_k\lambda^+_1(\phi_k, U) 
		\leq \epsilon_k\lambda_1(\phi_k, U,o^+_k) 
		\leq \epsilon_k\lambda^+_1(\phi_k, U) + o(1). \qedhere
	\end{equation*}
\end{proof}

The next definition is borrowed from Guirardel \cite[Definition~1.4]{Guirardel:2008ik}.

\begin{defi}
\label{def: transverse covering}
	A \emph{transverse covering} of an $\R$-tree $T$ is a collection $\mathcal Y$ of closed subtrees of $T$ with the following properties.
	\begin{enumerate}
		\item \label{enu: transverse covering - intersection}
		Every two distinct elements of $\mathcal Y$ have a degenerated intersection
		\item \label{enu: transverse covering - cover}
		Any bounded arc of $T$ can be covered by finitely many elements of $\mathcal Y$.
	\end{enumerate}
\end{defi}

\paragraph{A family of transverse coverings.}
We fix a parameter $r \in [0, \rho)$ (its value will vary later).
Let $\mathcal C(r)$ be the set of cone points $c\in \mathcal C_\omega$ such that $T \cap \bar P(c,r)$ is non-degenerate.
We write $\mathcal Y(r)$ for the collection of closed subtrees of $T$ which consists of 
\begin{itemize}
	\item the peripheral subtrees $T \cap \bar P(c,r)$, where $c$ runs over $\mathcal C(r)$, which we call \emph{$r$-peripheral components}, and 
	\item the connected components of 
	\begin{equation*}
		T \setminus \bigcup_{c \in \mathcal C(r)} P(c,r)
	\end{equation*}
	which we call \emph{$r$-transverse components}.
\end{itemize}
The action of $L$ on $X_\omega$ induces an action on $\mathcal Y(r)$ which preserves the set of $r$-peripheral components (\resp $r$-transversal components).

\begin{assu}
	The definition above makes sense for any $r \in [0, \rho)$.
	Our goal is to prove that $\mathcal Y(0)$ is a transverse covering, see \autoref{res: 0-transverse covering}.
	Nevertheless for the moment we suppose that $r > 0$.
\end{assu}

\begin{lemm}
\label{res: valence two vertices in transverse covering}
	A point of $T$ belongs to at most two subtrees in $\mathcal Y(r)$: one $r$-transverse component and/or one $r$-peripheral component.
\end{lemm}

\begin{proof}
	Since $r > 0$, any two distinct $r$-peripheral components are disjoint, see \autoref{res: intersection pieces}.
	Similarly any two distinct $r$-transverse components are disjoint.
	Whence the result.
\end{proof}

\begin{prop}
\label{res: r-transverse covering}
	The family $\mathcal Y(r)$ is an $L$-invariant \emph{transverse covering}.
\end{prop}

\begin{proof}
	As we explained above, if two elements of $\mathcal Y(r)$ have a non-empty intersection, one is an $r$-transverse component, say $Y$, while the other one is an $r$-transverse component, say $T \cap \bar P(c,r)$ for some $c \in \mathcal C(r)$.
	Recall that peripheral subtrees are strictly convex, hence $Y \cap \bar P(c,r)$ is degenerated.
	Indeed otherwise $Y$ would contain a point in $P(c,r)$ which contradicts its definition.
	
	Let us now prove that any bounded arc $\geo xy$ of $T$ is covered by finitely many elements of $\mathcal Y(r)$.
	Let $c \in \mathcal C(r)$ such that $\geo xy \cap P(c,r)$ is non-empty.
	If $x$ and $y$ do not belong to $P(c)$ then $\geo xy \cap \bar P(c)$ is an arc of length at least $2r$ (\autoref{res: large intersection w/ pieces}).
	On the other hand $x$ (\resp $y$) can only belong to a single open peripheral subtree $P(c)$.
	Thus there are only finitely many $c \in \mathcal C(r)$ such that $\geo xy \cap P(c,r)$ is non-empty. 
	In addition, 
	\begin{equation*}
		\geo xy \setminus \bigcup_{c \in \mathcal C(r)} P(c,r)
	\end{equation*}
	consists of a finite number of arcs each of which is contained in an $r$-transverse component.
\end{proof}

We use this transverse covering to decompose the action of $L$ on $T$ as a graph of actions $\Lambda(r)$.
The construction goes as follows -- compare with Guirardel \cite[Section~4.5]{Guirardel:2004aa}.
We first define its skeleton $S(r)$.
Its vertex set is $\mathcal Y(r)$.
Two subtrees $Y_1$ and $Y_2$ seen as vertices in $S(r)$ are connected by an edge if $Y_1 \cap Y_2 \neq \emptyset$.
The action of $L$ on $T$ induces an action by isometries on $S(r)$.
Observe that $S(r)$ is bipartite.
The vertices are indeed of two types: the $r$-transverse components and the $r$-peripheral components.
We call them respectively \emph{transverse} and \emph{peripheral vertices} of $S(r)$.
The vertex trees and the attaching points of $\Lambda(r)$ are the natural ones.
Since the action of $L$ on $T$ is minimal, so is the one of $L$ on $S(r)$.
Let us now describe a bit more its properties.

\begin{lemm}
\label{res: dual tree - ab stab}
	Let $c \in \mathcal C(r)$.
	The stabilizer of $T \cap \bar P(c,r)$ seen as a vertex of $S(r)$ is $L_c$.
\end{lemm}

\begin{proof}
	This is a direct consequence of \autoref{res: intersection pieces}.
\end{proof}

\begin{lemm}
\label{res: dual tree - transverse stab}
	Let $v$ be a transverse vertex of $S(r)$.
	Then $L_v$ is non-abelian.
\end{lemm}

\begin{proof}
	Suppose that contrary to our claim $L_v$ is abelian.
	Since $S(r)$ is minimal, there are two distinct edges $e$ and $e'$ of $S(r)$ starting at $v$.
	The peripheral endpoints of $e$ and $e'$ correspond to $r$-peripheral subtrees $T \cap \bar P(c,r)$ and $T \cap \bar P(c',r)$ for distinct $c,c' \in \mathcal C(r)$.
	In particular, $L_e = L_v \cap L_c$ and $L_{e'} = L_v \cap L_{c'}$.
	According \autoref{ass: splitting - freely indecomposable}, $L_e$ and $L_{e'}$ are non trivial.
	Since $L$ is CSA, the groups $L_c$ and $L_{c'}$ commute.
	However $L_c$ and $L_{c'}$ are maximal abelian subgroup (\autoref{res: tree-graded - abelian piece stabilizer}).
	Hence $L_c = L_{c'}$ which contradicts \autoref{res: intersection stab of cone pts}.
\end{proof}

\begin{lemm}
\label{res: dual tree - edge stab}
	Let $e$ be an edge of $S(r)$.
	Let $c \in \mathcal C(r)$ such that one of the vertices of $e$ corresponds to the $r$-peripheral component $T\cap\bar P(c,r)$.
	The stabilizer in $S(r)$ of $e$ coincide with the pointwise stabilizer of $\bar P(c,r)$.
	In particular, it is torsion-free, abelian.
\end{lemm}

\begin{proof}
	Let $Y$ be the $r$-transverse component corresponding to the other vertex of $e$ and write $p$ for the unique point in $Y \cap \bar P(c,r)$.
	It follows from the construction that the edge stabilizer $L_e$ is exactly the point stabilizer $L_p$ of $p \in T$.
	Since $p$ belongs to $Y$, it does not lie in $P(c,r)$.
	Hence $p \in \bar P(c,r) \setminus P(c,r)$, i.e. $b_c(p) = -r$.
	It follows from \autoref{res: tree-graded - fixed point set} that $L_p$ is the pointwise stabilizer of $\bar P(c,r)$.
\end{proof}

\begin{lemm}
\label{res: ab-tree rel to nc ab}
	Every non-cyclic abelian subgroup of $L$ fixes a vertex in $S(r)$.
	Hence $S(r)$ is an $(\mathcal A, \mathcal H \cup \mathcal A_{\rm nc})$-tree.
\end{lemm}

\begin{proof}
	Let $A$ be a non-cyclic abelian subgroup of $L$.
	If $A$ is elliptic (for its action on $T$) then it automatically fixes a point in $S(r)$.
	Otherwise, we distinguish two cases, following \autoref{res: tree-graded - small subgroups}.
	If $A$ preserves a transverse subtree, say $Y$, then it fixes the vertex of $S(r)$ associated to the unique $r$-transverse component containing $Y$.
	Assume now that there exists $c \in \mathcal C_\omega$ such that $A$ is contained in $L_c$.
	If $c \in \mathcal C(r)$, then $A$ fixes the vertex of $S(r)$ associated to the $r$-peripheral tree $T \cap \bar P(c,r)$.
	If $c$ does not belong to $\mathcal C(r)$, then by definition $T \cap \bar P(c,r)$ is degenerated.
	It follows that the closest point projection of $\bar P(c,r)$ onto $T$ is reduced to a single point say $y$, which is fixed by $A$.
	Note that $y$ cannot lie in an $r$-peripheral subtree $T \cap \bar P(c', r)$ for some $c' \in \mathcal C(r)$.
	Indeed otherwise $A$ would fix both $c$ and $c'$ and thus be trivial (\autoref{res: intersection stab of cone pts}).
	Hence $y$ belongs to an $r$-transverse subtree $Y$, which is $A$-invariant.
	In particular, $A$ fixes the vertex of $S(r)$ corresponding to $Y$.
\end{proof}

Recall that the action of a group on a simplicial tree is \emph{$k$-acylindrical}, if the pointwise stabilizer of any segment of length greater than $k$ is trivial.

\begin{lemm}
\label{res: acyl}
	The action of $L$ on $S(r)$ is $2$-acylindrical.
\end{lemm}

\begin{proof}
	Any segment $I$ of length at least $3$ in $S(r)$ contains two vertices of the form $T \cap \bar P(c,r)$ and $T \cap \bar P(c',r)$ for some distinct $c,c' \in \mathcal C(r)$.
	It follows from \autoref{res: dual tree - ab stab} that the pointwise stabilizer of $I$ is contained in $L_c \cap L_{c'}$  which is trivial by \autoref{res: intersection stab of cone pts}.
\end{proof}

An $L$-tree $S$ is \emph{reduced} if it is minimal, and for every vertex $v \in S$ whose image in $S/L$ has valence two, for every $e$ starting at $v$, the edge group $L_e$ is properly contained in $L_v$, see \cite{Bestvina:1991jh}.
Although $S(r)$ may not be reduced, it is not ``far'' from being so.

\begin{lemm}
\label{res: dual tree - reduced collapse}
	There exists an $L$-tree $\bar S(r)$ and an $L$-equivariant collapse map $S(r) \onto \bar S(r)$ with the following properties.
	\begin{enumerate}
		\item \label{res: dual tree - reduced collapse - acylindrical}
		The action of $L$ on $\bar S(r)$ is minimal and $2$-acylindrical
		\item \label{res: dual tree - reduced collapse - reduced}
		For every vertex $v$ of $\bar S(r)$, for every edge $e$ starting at $v$, the group $L_{e}$ is properly contained in $L_{v}$.
		In particular, $\bar S(r)$ is reduced.
		\item \label{res: dual tree - reduced collapse - edges}
		If $E$ and $\bar E$ stands for the edge sets of $S(r) /L$ and $\bar S(r) / L$ respectively, then $\card{E} \leq 2 \card{\bar E}$.
	\end{enumerate}
\end{lemm}

\begin{proof}
	For simplicity we write $V_0$ and $V_1$ for the set of peripheral and transverse vertices of $S(r)$.
	Let $v_1, \dots, v_m$ be a set of representatives for the $L$-orbite of vertices in $V_0$.
	In addition, we write $I$ for the set of indices $i \in \intvald 1m$ for which there is an edge, say $e_i$, adjacent to $v_i$ such that $L_{e_i} = L_{v_i}$.
	For every $i \in \intvald 1m$, denote by $E_i$ the set of edges of $S(r)$ one of whose endpoint lies in the orbit of $v_i$.
	Since $S(r)$ is bi-partite, the collection $E_1, \dots, E_m$ forms an $L$-invariant partition of the edges of $S(r)$.
	Moreover, since $S(r)$ is minimal, $E_i$ contains at least two distinct $L$-orbits whenever $i \in I$.
	
	Denote by $\bar S(r)$ the tree obtained from $S(r)$ by collapsing the $L$-orbit of $e_i$ for every $i \in I$.
	By construction the collapse map $f \colon S(r) \onto \bar S(r)$ is $L$-equivariant.
	In each class $E_i$ we have collapsed at most one edge orbit.
	According to our previous observation, we have collapsed in total at most half of the orbits of $S(r)$, whence \ref{res: dual tree - reduced collapse - edges}.
	Since the action of $L$ on $S(r)$ is $2$-acylindrical, so is the one on $\bar S(r)$, hence \ref{res: dual tree - reduced collapse - acylindrical} holds.
	
	We are left to prove \ref{res: dual tree - reduced collapse - reduced}.
	Assume that contrary to our claim, there is an edge $e$ of $\bar S(r)$ starting at a vertex $v$ such that $L_e = L_v$.
	Denote by $\tilde e$ the pre-image of $e$ in $S(r)$ and $\tilde v$ the endpoint of $\tilde e$ lifting $v$.
	By construction $L_e = L_{\tilde e}$ is abelian, hence so is $L_v$.
	Note that $L_{\tilde v}$ is contained in $L_v$, hence $\tilde v$ cannot belongs to $V_1$.
	Indeed it would contradict the fact that transverse vertex stabilizers are non-abelian.
	Therefore there is $g \in L$ and $i \in \intvald 1m$ such that $\tilde v = gv_i$.
	Note that $i$ does not belong to $I$.
	Indeed otherwise the non-abelian subgroup fixing the other end of $ge_i$ would also be contained in $L_v$.
	Since $i \notin I$, no edges in the link of $v_i$ has been collapsed, hence $L_v = L_{\tilde v}$ properly contains $L_e = L_{\tilde e}$.
	Contradiction
\end{proof}

\paragraph{First consequences.}
We exploit now the collection of acylindrical splittings $S(r)$ to get informations on the group $L$ and its action on $T$.
For simplicity we write $\mathcal C$ for $\mathcal C(0)$.
Observe that 
\begin{equation*}
	\mathcal C = \bigcup_{r \in (0,\rho)} \mathcal C(r).
\end{equation*}

\begin{lemm}
\label{res: conical subgroup elliptic}
	For every $c \in \mathcal C$, the group $L_c$ is elliptic, for its action on $T$.
\end{lemm}

\begin{proof}
	Let $c \in \mathcal C$.
	As we observed, there exists $r \in (0,\rho)$ such that $c$ belongs to $\mathcal C(r)$.
	Let $v$ be the vertex of $S(r)$ associated to $T \cap \bar P(c,r)$.
	Let $e_1, \dots, e_m$ be a set of $L_c$-representatives of the edges of $S(r)$ starting at $v$.
	Since $L$ is finitely generated, $L_c$ is finitely generated relative to $L_{e_1}, \dots, L_{e_m}$.
	However by \autoref{res: dual tree - edge stab}, those subgroups are contained in the pointwise stabilizer of $\bar P(c,r)$, that we denote by $N_c(r)$.
	Hence the quotient $Q_c = L_c / N_c(r)$ is finitely generated. 
	The action of $L_c$ on $T$ induces an action of $Q_c$ on $\bar P(c,r)$.
	Every element of $Q_c$ is elliptic by \autoref{res: tree-graded - fixed point set}.
	It follows from Serre's Lemma that $Q_c$ and thus $L_c$ is elliptic.
\end{proof}

\begin{prop}
\label{res: b_v achieves its min on T0}
	Let $c \in \mathcal C$.
	The map $b_c \colon T \to \R$ achieves its minimum at a unique point $z_c$.
	Moreover $z_c$ belongs to $T \cap \bar P(c)$ and is fixed by $L_c$.
\end{prop}

\begin{proof}
	Recall that balls/horoballs are strictly convex, hence the point $z_c$, if it exists, is necessarily unique, and therefore is fixed by $L_c$.

	According to \autoref{res: conical subgroup elliptic}, $L_c$ is elliptic.
	Hence by \autoref{res: tree-graded - fixed point set} there exists a finite number $r \in [0,\rho]$ such that $T\cap \bar P(c,r)$ is non-empty and pointwise fixed by $L_c$ (recall that $\rho$ can be finite or infinite).
	Since $\mathcal C$ is the increasing union of all $\mathcal C(s)$ for $s > 0$, we can assume without loss of generality that $r > 0$.
	We distinguish two cases.
	Suppose first that $T\cap \bar P(c,r)$ is reduced to a point, say $z_c$.
	One checks that $b_c(z_c) = -r$, while $b_c(x) \geq -r$, for every $x \in T$.
	Hence $z_c$ satisfies all the required properties.
	
	Suppose now that $T\cap \bar P(c,r)$ is non-degenerate, so that $c$ belongs to $\mathcal C(r)$.
	The peripheral subtree $T \cap \bar P(c,r)$ corresponds to a vertex $v$ of the tree $S(r)$ whose stabilizer is $L_c$.
	We denote by $e_1, \dots, e_m$ a set of $L_c$-representatives of the edges of $S(r)$ starting at $v$ and by $x_1, \dots, x_m$ the corresponding attaching points of the graph of actions $\Lambda(r)$.
	Since the action of $L$ on $T$ is minimal, $T \cap \bar P(c, r)$ is the convex hull of $L_c \cdot \{x_1, \dots, x_m\}$.
	Each $x_i$ is fixed by $L_c$, hence  $T \cap \bar P(c, r)$ is actually the convex hull of $\{x_1, \dots, x_m\}$.
	In particular, it is compact.
	Since $b_c$ is continuous, it achieves its minimum at some point $z_c$ in $T \cap \bar P(c, r)$.
\end{proof}

\paragraph{Comparing the splittings.}
Let $r,r' \in (0, \rho)$ with $r' < r$.
In order to compare the splittings $S(r)$ and $S(r')$, we build a map $f \colon S(r') \to S(r)$.
First we define $f$ on the set of vertices.
Let $v'$ be a vertex of $S(r')$.
\begin{itemize}
	\item Suppose that $v'$ corresponds to an $r'$-transverse component, say $Y'$, then $f(v')$ is the unique $r$-transverse component of $T$ containing $Y'$.
	\item Suppose that $v'$ corresponds to a peripheral subtree $T \cap \bar P(c, r')$ for some cone point $c \in \mathcal C(r')$, then we distinguish two cases.
	\begin{enumerate}
		\item If $c \in \mathcal C(r)$, that is $T \cap \bar P(c,r)$ is non-degenerate, then we define $f(v')$ to be the $T \cap \bar P(c,r)$.
		\item Otherwise, $f(v')$ is the unique $r$-transverse component of $T$ containing $T\cap \bar P(c,r')$.
	\end{enumerate}
\end{itemize}
Second, we extend the definition of $f$ to the edge set of $S(r')$.
By definition an edge $e'$ of $S(r')$ consists of a pair of two subtrees in $\mathcal Y(r')$: an $r$-transverse component, say $Y'$ and an $r$-peripheral component, say $T \cap \bar P(c, r')$ where $c \in \mathcal C(r')$.
We denote by $v'_1$ and $v'_2$ the corresponding vertices of $S(r')$.
We also write for $Y$ the (unique) $r$-transverse component containing $Y'$.
\begin{enumerate}
	\item Assume that $c \in \mathcal C(r)$.
	We claim that $Y \cap \bar P(c,r)$ is non-empty.
	Indeed choose a point in $T \cap \bar P(c,r)$, which is non-empty by definition of $\mathcal C(r)$.
	Join this point to $Y'$ by an a minimal arc $I$.
	By construction $(Y' \cup I)\setminus P(c,r)$ is contained in $Y$ and intersects $\bar P(c,r)$.
	Consequently $f(v'_1)$ and $f(v'_2)$ are joining by an edge $e$ in $S(r)$.
	We let $f(e') = e$.
	\item Assume now that $c \notin \mathcal C(r)$. 
	Then $Y$ also contains $T \cap \bar P(c,r')$, that is $f(v'_1) = f(v'_2)$.
	In this case we collapse the edge $e'$ to the vertex $f(v'_i)$.
\end{enumerate}

Observe that $f$ is the composition of a collapse map and a folding map, in the sense of \cite{Stallings:1983uz,Bestvina:1991jh}.
By construction, $f$ is $L$-equivariant.
Since $S(r)$ is minimal, $f$ is also onto.

\begin{prop}
\label{res: splitting acyl access}
	There exists $r \in (0,\rho)$ such that for every $r' \in (0, r)$ the map $f \colon S(r') \to S(r)$ induces an isomorphism from $S(r')/L$ onto $S(r)/L$.
\end{prop}

\begin{proof}
	For every $r \in (0, \rho)$ we write $\bar S(r)$ for the reduced $L$-tree obtained from $S(r)$ by \autoref{res: dual tree - reduced collapse}.
	The action of $L$ on $\bar S(r)$ is $2$-acylindrical.
	Since $L$ is finitely generated, it follows from acylindrical accessibility -- see for instance \cite{Sela:1997gh,Weidmann:2012aa} -- that the number of edges in $\bar S(r)/L$ is uniformly bounded from above.
	Hence the same holds of the number of edges in $S(r) / L$.
	However if given $r,r' \in (0,\rho)$ with $r' < r$, the $L$-equivariant map $f \colon S(r') \to S(r)$ induces an epimorphism from $S(r')/L$ onto $S(r)/L$.
	Therefore if $r$ is sufficiently small, this map need be an isomorphism.
\end{proof}

\paragraph{Passing to the limit.}
Recall that $L$ is CSA.
In addition we assumed that $L$ is freely indecomposable relative to $\mathcal H$.
It follows from \autoref{res: jsj - existence} that the JSJ deformation space $\mathcal D$ of $L$ over $\mathcal A$ relative to $\mathcal H \cup \mathcal A_{\rm nc}$ exists.
We denote by $S_{\rm JSJ}$ the tree of cylinders of an element in $\mathcal D$.
According to \autoref{res: tree of cylinders properties}, $S_{\rm JSJ}$ also belongs to $\mathcal D$.

\begin{prop}
\label{res: tree-graded - non-ab big elliptic}
	Let $c \in \mathcal C$.
	Let $i \in \intvald 1m$ and $H$ a non-abelian vertex group of $S_{\rm JSJ}$.
	If $L_c \cap H$ is not virtually cyclic, then it pointwise fixes $\bar P(c)$.
\end{prop}

\begin{proof}
	We claim that $H$ is rigid in $(\mathcal A, \mathcal H \cup \mathcal A_{\rm nc})$-trees.
	Assume indeed that it is not the case.
	Since $H$ is not abelian, it is quadratically hanging (see \autoref{res: jsj - existence}).
	Recall that if $\Sigma$ is a hyperbolic orbifold, the abelian subgroups of its fundamental group are all virtually cyclic.
	In particular, $L_c \cap H$ is virtually cyclic, which contradicts our assumption.

	Let $r \in (0,\rho)$.
	We assume that $r$ is sufficiently small so that $c$ belongs to $\mathcal C(r)$.
	According to \autoref{res: ab-tree rel to nc ab}, $S(r)$ is an $(\mathcal A, \mathcal H \cup \mathcal A_{\rm nc})$-tree.
	Since $H$ is rigid, it fixes a vertex in $S(r)$.
	Such a vertex necessarily corresponds to an $r$-transverse component $Y$, otherwise $H$ would be abelian.
	On the other hand $L_c$ fixes the vertex of $S(r)$ associated to $T \cap \bar P(c,r)$.
	Hence $L_c \cap H$ fixes the first edge on the geodesic in $S(r)$ from $T \cap \bar P(c,r)$ to $Y$.
	It follows from \autoref{res: dual tree - edge stab} that $L_c \cap H$ pointwise fixes $\bar P(c,r)$ for every sufficiently small $r \in (0, \rho)$.
	Hence $L_c \cap H$ fixes $\bar P(c) = \bar P(c,0)$.
\end{proof}

\begin{coro}
\label{res: tree-graded - stab v fg rel fix}
	Let $c \in \mathcal C$.
	The group $L_c$ is finitely generated relative to the pointwise stabilizer of $\bar P(c)$.
\end{coro}

\begin{proof}
	Without loss of generality we can assume that $L_c$ is non cyclic.
	It follows that $L_c$ fixes a vertex in $S_{\rm JSJ}$.
	We distinguish two cases.
	Assume first that $L_c$ is contained in a non-abelian vertex group $H$ of  $S_{\rm JSJ}$.
	Applying \autoref{res: tree-graded - non-ab big elliptic}, we see that $L_c$ pointwise fixes $\bar P(c)$ hence the result holds.
	Assume now that $L_c$ fixes an abelian vertex $v$ of  $S_{\rm JSJ}$.
	Since $L_c$ is a maximal abelian subgroup of $L$ (\autoref{res: tree-graded - abelian piece stabilizer}) the stabilizer of $v$ is exactly $L_c$.
 	We write $e_1, \dots, e_p$ for a set of representatives of $L_c$-orbits of edges from $S_{\rm JSJ}$ starting at $v$.
	We know that $L_c$ is finitely generated with respect to the collection $\{L_{e_j}\}$.
	Hence it suffices to prove that  for every $j \in \intvald 1p$, either $L_{e_j}$ is finitely generated or fixes $\bar P(c)$.
	Recall that $S_{\rm JSJ}$ is a tree of cylinders. 
	In particular, it is a bipartite tree where half of the vertices are non-abelian.
	Thus $e_j$ is adjacent to a non-abelian vertex group.
	The statement now follows from \autoref{res: tree-graded - non-ab big elliptic}.
\end{proof}

\begin{prop}
\label{res: finite filtration stab v}
	Let $c \in \mathcal C$ and set 
	\begin{equation*}
		r = - \min_{x \in T} b_c(x).
	\end{equation*}
	There exists a finite partition  $r_0 < r_1 < \dots < r_m$ of $\intval 0r$ with the following property.
	For every $i \in \intvald 0{m-1}$, for every $s \in [r_i, r_{i+1})$ the pointwise stabilizers of $\bar P(c,s)$ and $\bar P(c,r_i)$ coincide.
\end{prop}

\begin{proof}
	For every $s \in [0, r]$, we denote by $A(s)$ the pointwise stabilizer of $\bar P(c,s)$.
	According to \autoref{res: tree-graded - stab v fg rel fix}, the quotient $Q_c = L_c/A(0)$ is a finitely generated abelian group.
	Let $p$ be the rank of its maximal free abelian group.
	We claim that the collection 
	\begin{equation*}
		\set{A(s)}{ 0 \leq s \leq r}
	\end{equation*}
	contains at most $p+2$ distinct pairwise distinct elements.
	Assume on the contrary that is not the case.
	There exist $r_0 < r_1 < \dots r_{p+2}$ in $[0,r]$, such that $A(r_i)$ is strictly increasing sequence.
	Let $i \in \intvald 0p$.
	By \autoref{res: coro root stability}, $A(r_{i+1})/A(r_i)$ is a free abelian group of rank at least $1$ -- note that $A(r_{p+2})/A(r_{p+1})$ may have torsion if $r_{p+2} = r = \rho$.
	It follows that $L_c/A(0)$ contains a free abelian group of rank at least $p+1$, which contradicts the definition of $p$.
\end{proof}

\begin{coro}
\label{res: splittings isom}
	There exists $r \in (0, \rho)$ such that for every $r' \in (0, r)$ the simplicial map $f \colon S(r') \to S(r)$ is an isometry.
\end{coro}

\begin{proof}
	In view of Propositions~\ref{res: splitting acyl access} and \ref{res: finite filtration stab v}, there exists $r \in (0,\rho)$ such that 
	\begin{enumerate}
		\item \label{enu: splittings isom - graph}
		for every $r' \in (0,r)$ the map $S(r') \to S(r)$ induces an isomorphism from $S(r')/L$ onto $S(r)/L$;
		\item \label{enu: splittings isom - stab}
		for every $c \in \mathcal C$ the pointwise stabilizers of $\bar P(c,r)$ and $\bar P(c)$ coincide.
	\end{enumerate}
	Let $r' \in (0,r)$.
	Assume that the map $f \colon S(r') \to S(r)$ (which is onto) is not an isometry.
	Observe that $f$ is one-to-one when restricted to the set of peripheral vertices.
	It follows that there exist two distinct edges $e'_1$ and $e'_2$ in $S(r')$ starting from the same peripheral vertex $v'$ such that $f(e'_1) = f(e'_2)$.
	Since $S(r')/L \to S(r)/L$ is a bijection, $e'_1$ and $e'_2$ are in the same orbit.
	Hence there exists $g \in L$ such that $ge'_1 = e'_2$.
	Recall that $S(r)$ is bipartite. 
	Thus $g$ necessarily fixes the peripheral vertex $v'$.
	According to \autoref{res: dual tree - edge stab}, there is $c \in \mathcal C$ such that the stabilizers of $e'_1$ and $f(e'_1)$ are the pointwise stabilizers of $\bar P(c,r')$ and $\bar P(c,r)$ respectively.
	By construction $g$ fixes $f(e'_1)$ but not $e'_1$, hence pointwise fixes $\bar P(c,r)$ but not $\bar P(c,r')$.
	This contradicts \ref{enu: splittings isom - stab}.
	Hence $f \colon S(r') \to S(r)$ is an isometry.
\end{proof}

We now investigate the consequences of the above discussion for the family $\mathcal Y(0)$.

\begin{lemm}
\label{res: pre 0-transverse covering}
	There exists $r \in (0, \rho)$ such that for every $c \in \mathcal C$, for every $x_1,x_2 \in T \setminus P(c)$, the intersection $\geo {x_1}{x_2} \cap \bar P(c)$ is either degenerated or contains an arc of length $2r$.
\end{lemm}

\begin{proof}
	Let $r \in (0, \rho)$ be the parameter given by \autoref{res: splittings isom}.
	Note that it suffices to prove the result when $x_1$ and $x_2$ lie in $\bar P(c) \setminus P(c)$, i.e. when $b_c(x_1) = b_c(x_2) = 0$.
	Assume that $\geo {x_1}{x_2} \cap \bar P(c)$ is an arc whose length is less than $2r$.
	It follows from \autoref{res: large intersection w/ pieces}, that $\geo {x_1}{x_2}$ does not intersect $P(c,r)$.
	Hence $x_1$ and $x_2$ lie in the same $r$-transverse component, say $Y$.
	On the other hand, since $\geo {x_1}{x_2} \cap P(c)$ is non-empty, there exists $r' \in (0,r)$ such that $x_1$ and $x_2$ lie in distinct $r'$-transverse component, say $Y'_1$ and $Y'_2$.
	Observe that $Y'_1$ and $Y'_2$ correspond to two distinct transverse vertices of $S(r')$ which are mapped by $f \colon S(r') \to S(r)$ to the same vertex, namely $Y$.
	This contradicts the fact that $f$ is an isometry.
	Hence $\geo {x_1}{x_2} \cap \bar P(c)$ is either degenerated or has length at least $2r$ as announced.
\end{proof}

\begin{prop}
\label{res: 0-transverse covering}
	The family $\mathcal Y(0)$ is a transverse covering.
\end{prop}

\begin{proof}
	The proof goes verbatim as the one of \autoref{res: r-transverse covering}, using \autoref{res: pre 0-transverse covering} instead of \autoref{res: large intersection w/ pieces}.
\end{proof}

Note that this time, two elements of $\mathcal Y(0)$ of the form $\bar P(c)$, $\bar P(c')$ for distinct $c, c'\in \mathcal C$ may have a non-empty intersection.
A transverse component may be reduced to a single point.

As in the previous section we define a bipartite tree $S(0)$.
Its vertex set is $\mathcal Y(0)$.
There are two type of vertices. 
The \emph{peripheral vertices} correspond to the subtrees of the form $T \cap \bar P(c)$ for some $c \in \mathcal C$.
The \emph{transverse vertices} correspond to the transverse components of $T$.
A peripheral vertex and a transverse vertex are joined by an edge if the subtrees of $T$ that they represent have a non-empty intersection.
One checks using \autoref{res: 0-transverse covering}, that this construction defines indeed a simplicial bipartite tree endowed with a minimal action by isometries of $L$.

Proceeding as above, one defines for every $r \in (0, \rho)$ an $L$-equivariant simplicial map $f \colon S(0) \to S(r)$.
It follows from \autoref{res: splittings isom} that there exists $r \in (0, \rho)$ such that the map $f \colon S(0) \to S(r)$ is an isometry.
In particular, $S(0)$ is an $(\mathcal A, \mathcal H \cup \mathcal A_{\rm nc})$-tree.
The action of $L$ on $S(0)$ is $2$-acylindrical.

The tree $S(0)$ is the skeleton of a graph of actions $\Lambda(0)$ (see \autoref{def: graph of actions}).
If $v$ is a vertex of $S(0)$, the associated vertex tree is the subtree of $T$ represented by $v$.
If $e$ is an edge of $S(0)$ joining a peripheral vertex $v_1$ to a transverse vertex $v_2$, the associated attaching point is the unique intersection point of the subtrees represented by $v_1$ and $v_2$.

\paragraph{Further splitting of $T$.}
Our next task is to further refine the graph of actions $\Lambda(0)$ defined above, into well-understood ``elementary bricks''.
Before doing so, let us recall the kind of actions we are interested in.

\begin{defi}
\label{def: action type on R-tree}
	Let $G$ be a group acting by isometries on an $\R$-tree $Y$.
	We say that the action is 
	\begin{enumerate}
		\item \emph{simplicial} if $Y$ is simplicial and the action of $G$ on $Y$ is simplicial;
		\item \emph{axial} if $Y$ is a line and the image of $G$ in $\isom Y$ is a finitely generated group acting with dense orbits on $Y$;
		\item of \emph{Seifert type} if the action has a kernel $N$ and the faithful action of $G/N$ on $Y$ is dual to an arational measured foliation on a closed $2$-orbifold with boundary.
		\end{enumerate}
\end{defi}

The next theorem is a relative version of Sela \cite[Theorem~3.1]{Sela:1997gh}.

\begin{theo}[{\cite[Theorem~5.1]{Guirardel:2008ik}}]
\label{res: splitting relative version}
	Let $G$ be a group and $\mathcal H$ a finite collection of subgroups of $G$ such that $G$, is finitely generated relative to $\mathcal H$.
	Let $Y$ be an $\R$-tree endowed with a minimal, super-stable action of $G$ so that each subgroup in $\mathcal H$ is elliptic.
	Then one of the following holds
	\begin{itemize}
		\item The group $G$ splits as a free product relative to $\mathcal H$.
		\item The action of $G$ on $Y$ splits as a graph of actions where each vertex action is either simplicial, axial or of Seifert type.
	\end{itemize}	
\end{theo}

Let us now move back to the study of the action of $L$ on $T$.
Recall that $\mathcal H$ stands for the collection of elusive subgroups of $L$.

\begin{theo}
\label{res: final decomposition tree}
	Assume that $L$ does not split as a free product relative to $\mathcal H$.
	Then the action of $L$ on $T$ decomposes as a graph of actions over abelian groups $\Lambda$ whose skeleton $S_\Lambda$ refines $S(0)$.
	Each vertex action of $\Lambda$ is either peripheral, simplicial, axial or of Seifert type.
	Moreover,
	\begin{enumerate}
		\item if $L_v$ is a Seifert type vertex group of $\Lambda$, then $L_v$ is quadratically hanging with trivial fiber and the corresponding vertex tree is transverse in $T$;
		\item if $L_v$ is an axial vertex group of $\Lambda$, then the corresponding vertex tree is transverse in $T$.
	\end{enumerate}
\end{theo}

\begin{rema*}
	One can prove that the peripheral components are also simplicial.
	Nevertheless we distinguish them as they cannot always be shorten.
\end{rema*}

\begin{proof}
	We start with the graph of actions $\Lambda(0)$ defined above.
	The strategy is to blow up each transverse vertex of $S(0)$ to further refine this decomposition.
	The first step is to replace $\mathcal H$ be a \emph{finite} collection of subgroups.
	According to \cite[Lemma~8.2]{Guirardel:2017te}, there exists a finite collection $\mathcal H' = \{H_1, \dots, H_m\}$ where each $H_i$ is a finitely generated subgroup of an element in $\mathcal H$ such that $L$ does not split as a free product relative to $\mathcal H'$.
	Let $v$ be a transverse vertex of $S(0)$ and $Y_v$ the corresponding vertex tree.
	Let $e_1, \dots, e_m$ a set of $L_v$-representatives of the edges of $S(0)$ starting at $v$.
	For simplicity, we let
	\begin{equation*}
		\mathcal E_v = \{ L_{e_1}, \dots, L_{e_m}\}.
	\end{equation*}
	Consider now all the conjugates (in $L$) of groups in $\mathcal H'$ which fix $v$ an no other vertex in $S(0)$.
	We define $\mathcal H'_v$ by choosing a representative for each $L_v$-conjugacy class in this family.
	Since $L$ is finitely generated, $\mathcal E_v \cup \mathcal H'_v$ is finite and $L_v$ is finitely generated relative to $\mathcal E_v \cup \mathcal H'_v$, see for instance \cite[Lemma~1.12]{Guirardel:2008ik}.
	It follows from our assumption that $L_v$ does not split as a free product relative to $\mathcal E_v \cup \mathcal H'_v$.
	Moreover each group in $\mathcal E_v \cup \mathcal H'_v$ is elliptic for its action on $Y_v$.
	
	Let $Y'_v$ be the minimal $L_v$-invariant subtree of  $T$.		
	Since $L_v$ preserves $Y_v$, the latter contains $Y'_v$.
	We claim that $Y'_v = Y_v$.
	If $e$ is an edge of $S(0)$ starting at $v$, we denote by $x_e \in T$ the associated attaching point in the graph of actions $\Lambda(0)$ and by $y_e$ its projection on $Y'_v$.
	Since the action of $L$ on $T$ is minimal, $Y_v$ is exactly
	\begin{equation*}
		Y_v = Y'_v \cup \left(\bigcup_{e \in \link v} \geo{y_e}{x_e}\right)
	\end{equation*}
	Let $e \in \link v$.
	The group $L_e$ is the pointwise stabilizer of $\bar P(c)$ for some $c \in \mathcal C$.
	Since $L$ is freely indecomposable relative to $\mathcal H$, the group $L_e$ is not trivial.
	According to \autoref{res: tree-graded - fixed point set}, any point of $T$ fixed by $L_e$ belongs to $\bar P(c)$.
	By construction the geodesic $\geo{y_e}{x_e}$ is pointwise fixed by $L_e$, but belongs to the transverse subtree $Y_v$.
	Hence it is necessary degenerated.
	Consequently $Y'_v = Y_v$, which completes our claim.
	According to \autoref{res: tree-graded - transverse stability} the action of $L_v$ on $Y_v$ is super-stable.
	Consequently we can apply \autoref{res: splitting relative version} to the action of $L_v$ on $Y_v$ (relative to $\mathcal E_v \cup \mathcal H'_v$).
	This provides a decomposition of $Y_v$ as a graph of actions whose vertex actions fall in one of the three categories of the theorem.
		
	We obtain the graph of actions $\Lambda$ by assembling together all the decomposition of the vertex actions of $S(0)$.
	Let $S_\Lambda$ be its underlying skeleton.
	By construction $S_\Lambda$ refines $S(0)$ and every vertex action of $\Lambda$ is either peripheral, simplicial, axial or of Seifert type.
	Moreover axial and Seifert type vertex trees are transverse in $T$.
	Note that we can assume that transverse simplicial vertex trees are maximal, that is if $v$ and $v'$ are two vertices in $S_\Lambda$ such that the corresponding vertex trees $Y_v$ and $Y_{v'}$ are simplicial and the images of $v$ and $v'$ in $S(0)$ are transverse vertices, then either $Y_v \cap Y_{v'} = \emptyset$ or $v = v'$.
	
	Let us now prove the additional properties of $\Lambda$.
	Let $v$ be a Seifert type vertex of $S_\Lambda$ and $Y_v \subset T$ the associated vertex tree.
	We already mentioned that $Y_v$ is transverse.
	The kernel $N_v$ of the action of $L_v$ on $Y_v$ fixes a non-degenerate tripod, hence is trivial by \autoref{res: tree-graded - transverse tripod}.
	In particular, $L_v$ is isomorphic to the fundamental group of a closed $2$-orbifold $\Sigma$ with boundary.
	Consider a subgroup $H$ which is conjugated to an element of $\mathcal H$ so that $H \cap L_v$ fixes a point in $Y_v$.
	Since $Y_v$ is dual to an arational foliation, either $H\cap L_v$ is contained in a cone point of $\Sigma$ or in a boundary subgroup of $\pi_1(\Sigma)$.
	Consequently $H \cap L_v$ is an extended boundary subgroup.
	We prove in the same way that every edge group of $S_\Lambda$ incident to $v$ is an extended boundary subgroup.
	It follows from this discussion that $L_v$ is $(\mathcal A, \mathcal H)$-quadratically-hanging.
	
	We are left to prove that edge groups of $S_\Lambda$ are abelian.
	Let $e$ be an edge of $S_\Lambda$.
	If $e$ comes from an edge of $S(0)$, then $L_e$ is abelian.
	Otherwise, since transverse simplicial trees are maximal, one of the end point $v$ of $e$ is either an axial vertex or a Seifert type vertex.
	If $v$ is an axial vertex, then $L_v$ is abelian, thus so is $L_e$.
	If $v$ is a Seifert type vertex, then $L_e$ is an extended boundary subgroup of the underlying manifold.
	Since $L$ is CSA, $L_e$ is abelian.
\end{proof}

\begin{defi}
\label{def: complete graph of action}
	We call the graph of actions $\Lambda$ given by \autoref{res: final decomposition tree} the \emph{complete decomposition of $T$} (as a graph of actions).
	Let $S_\Lambda$ be its underlying skeleton.
	We define a new splitting of $L$ refining $S_\Lambda$ as follows.
	For each vertex $v$ of $S_\Lambda$ whose corresponding vertex tree $Y_v \subset T$ is simplicial and transverse, we blow up the vertex $v$ and replace it by a copy of $Y_v$.
	We call the resulting tree $S$ the \emph{complete splitting of $L$ with respect to $T$}.
\end{defi}

\begin{rema*}
	Note that both $\Lambda$ and $S$ depend on the sequence $(\Gamma_k, \phi_k)$ defining $L$ and $T$.
	However the sequence is normally clear from the context. 
	Hence we can safely ignore it in the terminology.
	
	The tree $S$ is not necessarily the underlying skeleton of a decomposition of $T$ as a graph of actions.
	Note that $\mcg{L, S_\Lambda} \subset \mcg{L, S}$.
	This larger modular group $\mcg{L, S}$ has the advantage of containing all the automorphisms needed for the shortening argument in the next section.
\end{rema*}


%
\subsection{Shortening argument}
%
\label{sec: shortening}

Let $(L, \eta)$ be a divergent limit group as before.
We now explain how the action of $L$ on its limit tree $T$ can be used to (eventually) shorten the morphism $\phi_k \colon G \to \Gamma_k$, that where used to define $L$.
Unlike in the case of limit groups over a free group, a shortening does not always exist.
This comes from the fact that $L$ may have no modular automorphism beside the inner ones.
Such a situation occurs for instance if $L$ is obtained by adjoining a root to a rigid group.

Recall that strong covers and limit of strong covers have been defined in \autoref{sec: strong covering}.
This section is dedicated to the proof of the following statement.

\begin{theo}[Shortening argument]
\label{res: shortening argument}
	Let $\delta \in \R_+^*$.
	Let $G$ be a group and $U$ a finite generating set of $G$.
	Let $\phi_k \colon G \to \Gamma_k$ be a sequence of marked groups where $(\Gamma_k, X_k, \mathcal C_k) \in \mathfrak H_\delta(\tilde \rho_k)$ and such that both $(\tilde \rho_k)$ and $\lambda_\infty(\phi_k, U)$ diverge to infinity.
	We make the following assumptions.
	\begin{itemize}
		\item The sequence $(\Gamma_k, \phi_k)$ converges to a limit group $(L, \eta)$, which is non-abelian and freely indecomposable relative to its collection $\mathcal H$ of elusive subgroups.
		\item The complete splitting $S$ of $L$ with respect to its limit tree $T$ is not a generalized root splitting.
	\end{itemize}
	Let $(\hat L_k, \hat \eta_k, \hat S_k)$ be a sequence of strong covers of $(L, \eta, S)$ converging to $(L, \eta, S)$.
	Let $\zeta_k \colon \hat L_k \onto L$ be the covering map associated to $\hat L_k$.
	Assume that for every $k \in \N$, the morphism $\phi_k \colon G \to \Gamma_k$ factors through $\hat \eta_k \colon G \to \hat L_k$ and write $\hat \phi_k \colon \hat L_k \to \Gamma_k$ for the resulting morphism, so that $\phi_k = \hat \phi_k \circ \hat \eta_k$.
	
	There exists $\kappa > 0$ with the following property: for infinitely many $k \in \N$, there exist automorphisms $\alpha_k \in \mcg{L, S}$ and $\hat \alpha_k \in \aut{\hat L_k}$ preserving $\hat S_k$ such that $\zeta_k \circ \hat \alpha_k = \alpha_k \circ \zeta_k$ and 
	\begin{equation*}
		\lambda^+_1\left(\hat \phi_k \circ \hat \alpha_k \circ \hat \eta_k, U\right) \leq (1- \kappa) \lambda^+_1\left(\phi_k,U\right)  \ \oas.
	\end{equation*}
\end{theo}

\begin{rema*}
	In the statement, the energies are computed in the \emph{non-rescaled} space $X_k$.
\end{rema*}

The remainder of this section is dedicated to the proof of the theorem.
The material below comes from \cite{Rips:1994jg}.
We follow the exposition given in \cite{Perin:2008aa,Weidmann:2019ue}.
Let us recall the notation from Sections~\ref{sec: building limit tree} - \ref{sec: prelim decomposition}.

We denote by $\mathcal A$ the collection of all abelian subgroups of $L$ and by $\mathcal H$ the collection of elusive subgroups of $L$.
For every $k \in \N$, we let
\begin{equation*}
	\epsilon_k = \frac 1{\lambda_\infty\left(\phi_k, U\right)}.
\end{equation*}
The group $L$ acts on an $\R$-tree obtained $X_\omega$ as the $\omega$-limit the rescaled spaces $\epsilon_kX_k$, where $\omega$ is a non-principal ultra-filter fixed 
as in the beginning of \autoref{sec: building limit tree}.
As before we let $\rho_k = \epsilon_k \tilde \rho_k$ for every $k \in \N$.

Let $T$ be the minimal $L$-invariant subtree of $X_\omega$.
By \autoref{res: basepoint in minimal tree}, there exists a sequence of point $o^+_k \in X^+_k$ such that $o^+ = \limo o^+_k$ belongs to $T$ and
\begin{equation}
\label{eqn: basepoint oplus}
	\displaystyle \lambda_1(\phi_k, U, o^+_k) = \lambda^+_1(\phi_k,U) + o\left( \frac 1{\epsilon_k}\right).
\end{equation}
The action of $L$ on $T$ splits as a graph of actions $\Lambda$ whose vertex actions are either peripheral, simplicial, axial, or Seifert type (\autoref{res: final decomposition tree}).
We call a vertex $v$ of its skeleton $S_\Lambda$ \emph{peripheral}, \emph{simplicial}, \emph{axial}, or \emph{Seifert type} depending whether the corresponding component $Y_c \subset T$ is peripheral, simplicial, axial, or Seifert type.
We denote by $S$ the complete splitting of $L$ with respect to $T$, which refines the skeleton $S_\Lambda$ of $\Lambda$.
We denote by $\mcg[0]{L, S}$ and $\mcg[0]{L, \Lambda}$ the respective subgroups of $\mcg{L, S}$ and $\mcg{S, \Lambda}$ generated by the following elements
\begin{itemize}
	\item inner automorphisms of $L$,
	\item Dehn twists over an edge $e$ of $S$ (\resp $S_\Lambda$) by an element $u \in L_e$,
	\item surface type automorphisms,
	\item generalized Dehn twists,
\end{itemize}
(see \autoref{sec: modular group} for the definitions).
The only difference between with regular modular groups concerns the (regular) Dehn twists: we require that the twister $u$ belongs to the edge group $L_e$ (and not its centralizer).
Note that $\mcg[0]{L, \Lambda}$ is contained in $\mcg[0]{L, S}$ as the edge set of $S$ is larger than the one of $S_\Lambda$.

\begin{lemm}
\label{res: lifting mcg to cover}
	Let $\alpha \in \mcg[0]{L, S}$.
	There exist $C \in \R_+$ and automorphisms $\hat \alpha_k \in \aut{\hat L_k}$ preserving $\hat S_k$, such that for all but finitely many $k \in \N$, we have $\zeta_k \circ \hat \alpha_k = \alpha \circ \zeta_k$ and $\lambda_\infty(\hat \alpha_k \circ \hat \eta_k, U) \leq C$.
\end{lemm}

\begin{rema*}
	In the above statement, the energy $\lambda_\infty(\hat \alpha_k \circ \hat \eta_k, U)$ is measured in the Cayley graph of $\hat L_k$ with respect to the generating set $\hat \eta_k (U)$.
\end{rema*}

\begin{proof}
	It suffices to prove the statement for every element in the generating set of $\mcg[0]{L, S}$.
	It is obvious for inner automorphisms.
	
	Let $e$ be an edge of $S$.
	Let $u$ be an element of $L_e$ and $\alpha$ the Dehn twist of $L$ over $e$ by $u$.
	Let $\tilde u \in G$ be a pre-image of $u$.
	If $k$ is sufficiently large, then $\hat u_k = \hat \eta_k(\tilde u)$ fixes a pre-image $\hat e_k$ in $\hat S_k$ of $e$.
	We choose for $\hat \alpha_k$ the Dehn twist of $\hat L_k$ over $\hat e_k$ by $\hat u_k$. 
	
	Let $v$ be a quadratically hanging vertex of $S$ and $\Sigma$ the underlying $2$-orbifold.
	Let $\alpha$ be the natural extension of a local automorphism $\beta$ induced by a homeomorphism of $\Sigma$ fixing the boundary.
	If $k$ is sufficiently large, there exists a pre-image $\hat v_k$ of $v$ in $\hat S_k$ such that $\zeta_k \colon \hat L_k \onto L$ induces an isomorphism from $\hat L_{k,\hat v_k}$ onto $L_v$.
	Hence the local automorphism $\beta$ of $L_v$ we started with lifts to a local automorphism $\hat \beta_k$ of $\hat L_{k,\hat v_k}$.
	We choose for $\hat \alpha_k$ the natural extension of $\hat \beta_k$.
	
	Let $v$ be a vertex of $S$ whose vertex group $L_v$ is abelian.
	We denote by $K_v$ the peripheral subgroup of $v$ (see \autoref{sec: modular group} for the definition).
	Let $\alpha$ be the natural extension of a local automorphism $\beta$ of $L_v$ fixing $K_v$.
	Unlike in the previous case, if $L_v$ is a priori not finitely generated, thus one cannot find a vertex $\hat v_k$ in $\hat S_k$ such that $\hat L_{k,\hat v_k}$ and $L_v$ are isomorphic.
	Nevertheless, $L_v$ is finitely generated relative to $K_v$.
	Consequently, if $k$ is sufficiently large, there exists a pre-image $\hat v_k$ of $v$ in $\hat S_k$ with the following properties.
	Let $\hat K_{k,\hat v_k}$ be the peripheral subgroup of $\hat v_k$.
	The map $\zeta_k \colon \hat L_k \onto L$ sends $\hat K_{k, \hat v_k}$ into $K_v$ and induces an isomorphism from $\hat L_{k,\hat v_k}/\hat K_{k,\hat v_k}$ onto $L_v/K_v$.
	By definition, $\beta$ fixes pointwise $K_v$.
	By construction, $\hat L_{k, \hat v_k}$ splits as $\Z^p \oplus \hat K_{k,\hat v_k}$, where the abelian free factor $\Z^p$ is isomorphic to $\hat L_{k,\hat v_k}/\hat K_{k,\hat v_k}$ and thus to $L_v / K_v$.
	Hence one can find an automorphism $\hat \beta_k$ of $\hat L_{k,\hat v_k}$ that lifts $\beta$ and fixes $\hat K_{k,\hat v_k}$.
	We choose for $\hat \alpha_k$ the natural extension of $\hat \beta_k$.
\end{proof}

%
\subsubsection{Axial and Seifert type component}
%
\label{sec: shortening - axial/surface}

We handle first the case where $\Lambda$ has a vertex action which is axial or Seifert type.
To that end, we some recall results of Rips-Sela \cite{Rips:1994jg} that can be used to shorten an action in this situation.

\begin{theo}[{}]
\label{res: shortening action - surface}
	Let $L$ be a finitely generated group acting on an $\R$-tree $T$ which admits a decomposition as a graph of actions $\Lambda$.
	Let $x \in T$.
	Let $U$ be a finite subset of $L$.
	There exists an element $\alpha \in \mcg[0]{L, \Lambda}$ such that for every element $u \in U$,
	\begin{itemize}
		\item if the geodesics $\geo x{ux}$ has a non-degenerate intersection with a Seifert type component of $\Lambda$, then $\dist{\alpha(u)x}x < \dist {ux}x$,
		\item otherwise $\alpha(u) = u$.
	\end{itemize}
\end{theo}

\begin{proof}
	The proof is detailed in Perin's thesis \cite[Theorem~5.12]{Perin:2008aa}.
	Note that Perin only asserts that  $\alpha$ belongs to $\mcg{L,\Lambda}$ which is a priori bigger that $\mcg[0]{L, \Lambda}$.
	Nevertheless one checks from the proof that $\alpha$ is a product of surface type automorphisms, whence the result.
\end{proof}

\begin{theo}
\label{res: shortening action - axial}
	Let $L$ be a finitely generated group. 
	Suppose that $L$ acts on a real tree $T$ with abelian arc stabilizers and that $T$ admits a decomposition as a graph of actions $\Lambda$.
	In addition, we assume that every solvable subgroup of $L$ is either elliptic (for its action on $T$) or abelian.
	Let $x \in T$.
	Let $U$ be a finite subset of $L$.
	There exists an element $\alpha \in \mcg[0]{L, \Lambda}$ such that for every element $u \in U$,
	\begin{itemize}
		\item if the geodesics $\geo x{ux}$ has a non-degenerate intersection with an axial type component of $\Lambda$, then $\dist{\alpha(u)x}x < \dist {ux}x$,
		\item otherwise $\alpha(u) = u$.
	\end{itemize}
\end{theo}

\begin{proof}
	Let $v$ be vertex of $\Lambda$ whose corresponding component $Y_v$ is axial.
	We claim that $L_v$ splits as a direct product $L_v = N_v \oplus A_v$, where $N_v$ is the kernel of the map $L_v \to \isom{Y_v}$.
	Since pointwise arc stabilizers are abelian, the group $L_v$ is the extension of the abelian group $N_v$ by a subgroup of $\isom{Y_v}$, say $Q_v$, which is solvable.
	In particular, $L_v$ is solvable, but not elliptic, hence abelian.
	Moreover $Y_v$ is a line on which $Q_v$ acts with dense orbits.
	Hence $Q_v$ is necessarily free abelian.
	Since $L_v$ is abelian, the projection $L_v \onto Q_v$ admits a section $\sigma \colon Q_v \to L_v$.
	In particular, $L_v$ splits as a semi-direct product $L_v = N_v \rtimes A_v$ where $A_v$ is the image of $\sigma$ (which is isomorphic to $Q_v$).
	As $L_v$ is abelian, the action by conjugation on $N_v$ of any element of $L_v$ is trivial. 
	Hence the above semi-direct product is actually a direct product, which completes the proof of our claim.
	The same argument also shows that $N_v$ is the peripheral subgroup of $v$, as well as the stabilizer of any edge of $S$ adjacent to $v$.
	
	One can now follow the proof given by Perin in \cite[Theorem~5.17]{Perin:2008aa}.
	Indeed the statement given there is almost the same as our theorem.
	The main difference is that Perin assumes that every solvable subgroup of $L$ is free abelian.
	Nevertheless this assumption is only used to prove that the vertex groups of $L$ associated to an axial component splits as a direct sum as we described above.
	Therefore the rest of the proof can be carried verbatim in our settings.
	
	Perrin's statement asserts that $\alpha$ belongs to $\mcg{L,\Lambda}$.
	As for \autoref{res: shortening action - surface}, one can check from the proof that $\alpha$ is actually a product of generalized Dehn twists, hence belongs to $\mcg[0]{L, \Lambda}$.
\end{proof}

Recall that for every $k \in \N$, the morphism $\hat \phi_k \colon \hat L_k \to \Gamma_k$ is obtained by factoring $\phi_k \colon G \to \Gamma_k$ through the cover $\hat \eta_k \colon G \to \hat L_k$.

\begin{prop}
\label{res: shortening morphism - axial / surface}
	Assume that the graph of actions $\Lambda$ contains an axial or a Seifert type vertex action.
	There exist $\ell \in \R_+^*$ and $\alpha \in \mcg[0]{L, \Lambda}$ with the following property.
	For every sufficiently large $k \in \N$, there exists  $\hat \alpha_k \in \aut{\hat L_k}$ preserving $\hat S_k$ such that $\zeta_k \circ \hat \alpha_k = \alpha \circ \zeta_k$ and 
	\begin{equation*}
		\lambda^+_1\left(\hat \phi_k \circ \hat \alpha_k \circ \hat \eta_k, U\right) \leq \lambda^+_1\left(\phi_k,U\right) - \frac \ell{\epsilon_k} \ \oas,
	\end{equation*}
\end{prop}

\begin{proof}
	Recall that $o^+ = \limo o^+_k$ is a point of $T$, where $o_k^+ \in X^+_k$ (almost) minimizes the restricted energy of $\phi_k$.
	We write $U_0$ for the set of all elements $u \in U$ such that $\geo {o^+}{\eta(u)o^+}$ has a non-degenerate intersection with an axial or Seifert type component of $\Lambda$.
	Observe that $U_0$ is non-empty.
	Indeed otherwise 
	\begin{equation*}
		L \cdot \left( \bigcup_{u \in U}\geo {o^+}{\eta(u)o^+}\right)
	\end{equation*}
	would be an $L$-invariant subtree of $T$ covered only by peripheral and simplicial subtrees.
	By minimality it should equal $T$, which contradicts our assumption.
	Since $L$ is CSA, every solvable subgroup of $L$ is abelian, see \cite[Lemma~1.4]{Sela:2001gb}.
	According to Theorems~\ref{res: shortening action - surface} and \ref{res: shortening action - axial}, there exist $\ell \in \R_+^*$ and $\alpha$ in  $\mcg[0]{L,\Lambda}\subset \mcg[0]{L,S}$ such that 
	\begin{enumerate}
		\item \label{eqn: shortening morphism - axial / surface}
		$\dist{\alpha \circ \eta(u)o^+}{o^+} \leq \dist {\eta(u)o^+}{o^+} - \ell$, for all $u \in U_0$
		\item $\alpha \circ \eta(u) = \eta(u)$, for all $u \in U \setminus U_0$.
	\end{enumerate}
	Let us translate this fact to the action of $\Gamma_k$ on $X_k$.
	According to \autoref{res: lifting mcg to cover}, there exists $k_0 \in \N$ such that for every integer $k \geq k_0$, there is an automorphism $\hat \alpha_k \in \aut{\hat L_k}$ preserving $\hat S_k$ that lifts $\alpha$, while $\lambda_\infty(\hat \alpha_k \circ \hat \eta_k, U)$ is uniformly bounded.
	Up to increasing the value of $k_0$, we can therefore assume that for every $k\geq k_0$ and $u \in U \setminus U_0$, we have $\hat \alpha_k \circ \hat\eta_k(u)  =\hat\eta_k(u)$, hence
	\begin{equation*}
		\hat \phi_k \circ \hat \alpha_k \circ \hat\eta_k(u) = \hat \phi_k \circ \hat\eta_k(u) = \phi_k(u).
	\end{equation*}
	On the other hand, for every $u \in U_0$, Inequality \ref{eqn: shortening morphism - axial / surface} yields
	\begin{equation*}
		\dist{\hat \phi_k\circ \hat \alpha_k \circ \hat\eta_k(u)o^+_k}{o^+_k} \leq \dist {\phi_k (u)o^+_k}{o^+_k} - \frac \ell{\epsilon_k} + o\left(\frac 1{\epsilon_k}\right),\ \oas.
	\end{equation*}
	Since $U$ is finite, we have
	\begin{align*}
		\lambda^+_1\left(\hat \phi_k\circ \hat \alpha_k \circ \hat \eta_k, U \right)
		& \leq \lambda^+_1\left(\hat \phi_k\circ \hat \alpha_k \circ \hat \eta_k, U, o^+_k\right) \\
		& \leq \lambda^+_1\left(\phi_k, U, o^+_k\right)- \frac{\ell |U_0|}{\epsilon_k} + o\left(\frac 1{\epsilon_k}\right).
	\end{align*}
	Recall that $o^+_k$ (almost) minimizes the restricted energy -- see Inequality (\ref{eqn: basepoint oplus}). Hence we get
	\begin{equation*}
		\lambda^+_1\left(\hat \phi_k\circ\hat  \alpha_k \circ \hat \eta_k, U\right)
		\leq \lambda^+_1\left(\phi_k, U\right) - \frac{\ell |U_0|}{\epsilon_k} + o\left(\frac 1{\epsilon_k}\right).
	\end{equation*}
	Recall that $\card{U_0} \geq 1$.
	Hence the result follows from the previous inequality.
\end{proof}

%
\subsubsection{Peripheral component}
%
\label{sec: shortening peripheral}

We now focus on the peripheral components.
This case is new compared to the situations handled in the literature.
Our goal in this section is to prove the following statement.

\begin{theo}
\label{res: shortening peripheral - main theo}
	Suppose that $\rho > 0$.
	Let $c \in \mathcal C$.
	Let $v_c$ be the vertex of $S_\Lambda$ corresponding to $T \cap \bar P(c)$
	Assume that one of the following holds:
	\begin{nenumerate}[label=(\Alph*)]
		\item \label{enu: shortening peripheral - orbit hyp}
		either the link of $v_c$ contains at least two $L_c$-orbits, or
		\item \label{enu: shortening peripheral - index hyp}
		the pointwise stabilizer of $\bar P(c)$ has infinite index in $L_c$,
	\end{nenumerate}
	Then there is $\ell \in \R_+^*$ such that for every sufficiently large $k \in \N$, there exist $\alpha_k \in \mcg[0]{L,\Lambda}$ and $\hat \alpha_k \in \aut{\hat L_k}$ preserving $\hat S_k$ such that $\zeta_k \circ \hat \alpha_k = \alpha_k \circ \zeta_k$ and 
		\begin{equation*}
			\lambda^+_1\left(\hat \phi_k \circ \hat \alpha_k \circ \hat \eta_k, U\right) \leq \lambda^+_1\left(\phi_k,U\right) - \frac \ell{\epsilon_k} \ \oas
		\end{equation*}
	Moreover, in Case (A), $\alpha_k$ is a Dehn twist over an edge of $S$ adjacent to $v$, while in Case (B), $\alpha_k$ is a generalized Dehn twist.
\end{theo}

The proof of the statement is quite long and somewhat technical.
Therefore we break it in several parts. 
A summary of the proof is given at the end of this section.
For this proof we actually do not need to work with the complete splitting $S$ of $L$ with respect to $T$.
We will only consider the skeleton $S_\Lambda$ of the graph of action decomposition $\Lambda$ of $T$ given by \autoref{res: final decomposition tree}.
Let us start with some notations.

Recall that $o^+ = \limo o^+_k$ is a base point in $T$ such that $o^+_k$ lies in the thick part of $X^+_k$ of $X_k$ as (almost) minimizes the restricted energy $\lambda_+(\phi_k,U)$ of $\phi_k$.
Consequently $o^+$ belongs to a transverse component of $T$.
In particular there is a vertex $v^+$ in $S_\Lambda$ which does not lie in the $L$-orbit of a peripheral vertex of $S_\Lambda$ and whose corresponding component $Y_{v^+} \subset T$ contains $o^+$.

\begin{rema*}
	This property is precisely the reason that motivated the definition of the restricted energy.
	It will greatly simplify the arguments to have the point $o^+$ outside of the peripheral subtree $P(c)$.
\end{rema*}

We start with a lemma to provide more informations about the action of peripheral subgroups.

\begin{lemm}
\label{res: shortening peripheral - shortening pairs}
	Suppose that $\rho > 0$.
	Let $c \in \mathcal C_\omega$.
	Let $x = \limo x_k$ and $x' = \limo x'_k$ be two points in $X_\omega$ such that $b_c(x) = b_c(x') = 0$.	
	Let $g \in G$ be such that $\eta(g)$ is a non-trivial element fixing $\bar P(c)$ pointwise.
	Let $n \in \N\setminus\{0\}$.
	There exists a sequence $(g_k)$ of elements of $\group g$ with the following properties.
	\begin{enumerate}
		\item \label{enu: shortening peripheral - shortening pairs - fix point}
		$\eta(g_k)$ pointwise fixes $\bar P(c)$ for every $k \in \N$.
		\item \label{enu: shortening peripheral - shortening pairs - shortening}
		$\limo\dist{\phi_k(g^n_k)x_k}{x'_k} = 0$.
		\item \label{enu: shortening peripheral - fixing deep points}
		If $y = \limo y_k$ is a point of $X_\omega$ such that $2b_c(y) \leq  - \dist x{x'}$, then for every $q \in \Z$,
		\begin{equation*}
			\limo \dist{\phi_k(g^q_k)y_k}{y_k} =0.
		\end{equation*}
		\item \label{enu: shortening peripheral - shortening pairs - non-decreasing}
		If $z = \limo z_k$ and $z' = \limo z'_k$ are two points of $X_\omega$ with $\dist z{z'} \geq \dist x{x'}$ and $b_c(z) = b_c(z') = 0$, then for every $q \in \Z$,
		\begin{equation*}
			\limo \dist{\phi_k(g^q_k)z_k}{z'_k} \leq \dist z{z'}.
		\end{equation*}
	\end{enumerate}
\end{lemm}

\begin{proof}
	By definition $c = [c_k]$ for some sequence $(c_k) \in \Pi_\omega \mathcal C_k$ where $\phi_k(g)$ fixes $c_k$, \oas.
	Without loss of generality, we can assume that $x_k, x'_k \in \mathring B(c_k, \rho_k)$ \oas.
	For every $k \in \N$, we denote by 
	\begin{equation*}
		\theta_k \colon \mathring B(c_k, \rho_k) \times \mathring B(c_k, \rho_k) \to \R / \Theta_k\Z
	\end{equation*}
	the angle cocycle provided by Axiom~\ref{enu: family axioms - conical - 2}.
	Since $\theta_k$ is a $\stab{c_k}$-invariant cocycle, the quantity $\theta_k(\phi_k(g)x_k, x_k)$ does not depend on the point $x_k$.
	For simplicity we write 
	\begin{equation*}
		\alpha_k = \tilde\theta_k\left(\phi_k(g)x_k, x_k\right).
	\end{equation*}
	for the unique representative of $\theta_k\left(\phi_k(g)x_k, x_k\right)$ in $(-\Theta_k/2, \Theta_k/2]$.
	By assumption $\eta(g) \neq 1$, hence $\phi_k(g)$ is non-trivial \oas.
	Consequently $\alpha_k \neq 0$ \oas.
	On the other hand, $\eta(g)$ fixes pointwise $\bar P(c)$.
	It follows from \autoref{res: limit tree -  asymp angle} that $(\alpha_k)$ converges to zero.
	More precisely we have
	\begin{equation*}
		\ln \abs{\alpha_k} \leq - \frac {\rho_k}{\epsilon_k} + o \left( \frac 1{\epsilon_k} \right).
	\end{equation*}
	Hence there exists a sequence of integers $(m_k)$ such that 
	\begin{equation}
	\label{eqn: shortening peripheral - shortening pairs}
		\abs{\tilde \theta_k\left(\phi_k\left(g^{nm_k}\right)x_k, x'_k\right)} \leq \frac n2 \abs{\alpha_k}, \ \text{\oas}.
	\end{equation}
	We let $g_k = g^{m_k}$, for every $k \in \N$.
	Let us now check the properties of $(g_k)$.
	Since $\eta(g)$ pointwise fixes $\bar P(c)$ so does $\eta(g_k)$, hence Point~\ref{enu: shortening peripheral - shortening pairs - fix point} holds.
	It follows from our choice of $(m_k)$ that 
	\begin{equation*}
		\ln\abs{\tilde \theta_k\left(\phi_k\left(g_k^n\right)x_k, x'_k\right)}
		\leq \ln\abs{\alpha_k} + \ln \left( \frac n2\right)
		\leq - \frac {\rho_k}{\epsilon_k} + o \left( \frac 1{\epsilon_k} \right).
	\end{equation*}
	Point~\ref{enu: shortening peripheral - shortening pairs - shortening} now follows from \autoref{res: limit tree -  asymp angle}.
	For the rest of the proof we will assume that $m_k \neq 0$ \oas.
	Indeed otherwise the last two points are straightforward.
			
	Consider now a point $y = \limo y_k$ and let $b_c(y) = -s$.
	If $s = \rho$, then $\rho$ is finite and $y$ coincide with $c = \limo c_k$.
	In this case Point~\ref{enu: shortening peripheral - fixing deep points} follows from the fact that $g_k$ fixes $c_k$.
	Suppose now that $s < \rho$.
	Without loss of generality we can assume that $y_k \in \mathring B(c_k,\rho_k)$ \oas.
	It follows from our assumption that $ \dist x{x'} \leq 2s < 2\rho$.
	Hence \autoref{res: limit tree -  asymp angle} tells us that $\tilde \theta_k\left(x_k, x'_k\right)$ converges to zero and 
	\begin{equation*}
		\ln\abs{\tilde \theta_k\left(x_k, x'_k\right)}
		\leq - \frac {\rho_k}{\epsilon_k} + \frac {\dist x{x'}}{2\epsilon_k} +  o \left( \frac 1{\epsilon_k} \right).
	\end{equation*}
	Note that 
	\begin{equation*}
		n\, \theta_k\left(\phi_k(g_k)y_k, y_k\right)
		= \theta_k\left(\phi_k(g^{nm_k})x_k, x'_k\right) +  \theta_k\left(x'_k, x_k\right)
	\end{equation*}
	It follows from our choice of $(m_k)$ that representatives in $(-\Theta_k/2, \Theta_k/2]$ of both terms on the right hand side of the previous equation converge to zero.
	Consequently the representative of $ \theta_k\left(\phi_k(g_k)y_k, y_k\right)$ in the same interval also converges to zero and satisfies
	\begin{align*}
		\abs{\tilde \theta_k\left(\phi_k(g_k)y_k, y_k\right)} 
		& \leq \frac 1n  \abs{\tilde \theta_k\left(\phi_k(g^{nm_k})x_k, x'_k\right) } + \frac 1n\abs {\tilde\theta_k\left(x'_k, x_k\right)} \\
		& \leq  \frac 12 \abs{\alpha_k}   + \frac 1n \abs {\tilde\theta_k\left(x'_k, x_k\right)}.
	\end{align*}
	Recall that if $a, b \in \R_+^*$, then $\ln(a+b) \leq \max \{\ln a, \ln b\} + \ln 2$.
	It follows from our previous observations that for every $q \in \Z$,
	\begin{equation*}
		\ln\abs{\tilde \theta_k\left(\phi_k(g^q_k)y_k, y_k\right)}
		\leq - \frac {\rho_k}{\epsilon_k} + \frac {\dist x{x'}}{2\epsilon_k} +  o \left( \frac 1{\epsilon_k} \right)
		\leq - \frac {\rho_k-s}{\epsilon_k} + o \left( \frac 1{\epsilon_k} \right)
	\end{equation*}
	Using again \autoref{res: limit tree -  asymp angle} we get that $\limo \dist{\phi_k(g^q_k)y_k}{y_k} =0$.
	
	Let $z = \limo z_k$ and $z' = \limo z_k$ in $X_\omega$ such that $\dist z{z'} \geq \dist x{x'}$ and $b_c(z) = b_c(z') = 0$.
	Let $q \in \Z$.
	We denote by $y = \limo y_k$ the midpoint of $\geo z{z'}$.
	Note that 
	\begin{equation*}
		2b_c(y) \leq - \dist z{z'} \leq  - \dist x{x'}.
	\end{equation*}
	It follows from the previous point that $\limo \dist{\phi_k(g^q_k)y_k}{y_k} =0$.
	Applying the triangle inequality we get
	\begin{align*}
		\dist{\phi_k(g^q_k)z_k}{z'_k} & \leq \dist{\phi_k(g^q_k)z_k}{\phi_k(g^q_k)y_k} +  \dist{\phi_k(g^q_k)y_k}{y_k}  + \dist{y_k}{z'_k} \\
		& \leq \dist{z_k}{y_k} +  \dist{\phi_k(g^q_k)y_k}{y_k}  + \dist{y_k}{z'_k}.
	\end{align*}
	Passing to the limit we get $\limo \dist{\phi_k(g^q_k)z_k}{z'_k} \leq \dist zy + \dist y{z'} \leq \dist z{z'}$, which completes the proof of Point~\ref{enu: shortening peripheral - shortening pairs - non-decreasing}.
\end{proof}

For the remainder of this section, we assume that $\rho > 0$.
We also fix a cone point $c \in \mathcal C$ and denote by $v_c$ the vertex of $S_\Lambda$ corresponding to the peripheral component $\bar P(c)$.
\begin{nota}
\label{not: peripheral subgroup of cone point}
	We write $N_c$ for the pointwise stabilizer of $\bar P(c)$.
	It is also the stabilizer of every edge of $S$ adjacent to $v_c$.
	In addition, we denote by $K_c$ the peripheral subgroup of $v_c$ (see \autoref{sec: modular group}).
	Concretely it is the pre-image of the torsion part of the finitely generated abelian group $Q_c = L_c / N_c$.
\end{nota}
	
\paragraph{Normal forms.}
We now define a suitable normal form for the elements of $L$ (\resp $\hat L_k$) which reflects their excursions in the peripheral subtree $\bar P(c)$.
Let $\bar S$ be the tree obtained from $S_\Lambda$ by collapsing every edge which is not adjacent to a vertex in the $L$-orbit of $v_c$.
In particular, $\bar S$ is a collapse of $S(0)$.
We write $\bar v^+$ and $\bar v_c$ for the respective images of $v^+$ and $v_c$ in $\bar S$.
Similarly we define $\bar S_k$ as the tree obtained from $\hat S_k$ by collapsing every edge whose image in $S$ is collapsed by the map $S \to S_\Lambda \to \bar S$.
Note that $\hat S_k \to S$ yields an equivariant map $\bar S_k \to \bar S$ such that the induced map $\bar S_k / \hat L_k \to \bar S/L$ is an isomorphism.	

We now proceed as in \autoref{sec: graph of groups}.
We choose a lift $D_0$ of a maximal subtree in $\bar S/L$ which contains $\bar v^+$.
Up to replacing $c$ by a translate, we may assume that $\bar v_c$ belongs to $D_0$.
We now fix a subtree $D_1$ of $\bar S$ such that $D_0$ is contained in $D_1$ and the $L$-orbit of every edge in $\bar S$ contains exactly one element is $D_1$.
We write $V$ (\resp $E$) for the vertex set of $D_0$ (\resp the edge set of $D_1$).
We write $\mathbf F(E)$ for the free group
\begin{equation*}
	\mathbf F(E) = \group{E \mid e\bar e = 1, \ \forall e \in E}
\end{equation*}
and let
\begin{equation*}
	\mathbf L = \mathbf F(E) \freep \left( \freep_{v \in V} L_v\right).
\end{equation*}
Recall that for every $k \in \N$, the quotient $\bar S/L$ and $\bar S_k / \hat L_k$ are isomorphic.
Hence we can choose for each $k \in \N$ two subtrees $D_{k,0}$ and $D_{k,1}$ such that for every $k' > k$, for every $i \in \{0,1\}$, the map $\bar S_k \to \bar S_{k'}$ (\resp $\bar S_k \to \bar S$) induces an isomorphism from $D_{k,i}$ to $D_{k',i}$ (\resp $D_{k,i}$ to $D_i$).
We write $V_k$ (\resp $E_k$) for the vertex set of $D_{k,0}$ (\resp the edge set of $D_{k,1}$) and let
\begin{equation*}
	\mathbf L_k = \mathbf F(E_k) \freep \left( \freep_{v \in V_k} L_{k,v}\right),
\end{equation*}
where $\mathbf F(E_k)$ is defined in the same way as $\mathbf F(E)$.
Observe that the maps $(\hat L_k, \bar S_k) \to (\hat L_{k'}, \bar S_{k'})$ and $(\hat L_k, \bar S_k) \to (L,\bar S)$ induce morphisms $\mathbf L_k \to \mathbf L_{k'}$ and $\mathbf L_k \to \mathbf L$.
For every $\mathbf e \in E_0$ we fix an element $s(\mathbf e) \in L_0$ such that $s(\mathbf e)t(\mathbf e)$ belongs to $V_0$.
Such an element always exists as $D_{0,0} \subset \bar S_0$ is a lift of a maximal subtree of $\bar S_0/L_0$.x
If $\mathbf e$ is an edge in $E_k$ (\resp $E$), we now denote by $s(\mathbf e)$ the image in $\hat L_k$ (\resp $L$) of $s(\mathbf e_0)$ where $\mathbf e_0$ is the (unique) pre-image of $\mathbf e$ in $E_0$.
With these data we can define a map $\sigma_k \colon \mathbf L_k \to \hat L_k$ which extends the embedding of the vertex groups $\hat L_{k,v} \into \hat L_k$, for every $v \in V_k$, and such that $\sigma_k(\mathbf e) = s(\bar{\mathbf e})s(\mathbf e)^{-1}$, for every $\mathbf e \in E_k$.
We define $\sigma \colon \mathbf L \to L$ in the exact same way so that for all integers $k' > k$, the following diagram commutes.
\begin{center}
\begin{tikzpicture}
	\matrix (m) [matrix of math nodes, row sep=2em, column sep=2.5em, text height=1.5ex, text depth=0.25ex] 
	{ 
		\mathbf L_k	& \mathbf L_{k'} & \mathbf L \\
		\hat L_k	& \hat L_{k'} & L \\
	}; 
	\draw[>=stealth, ->] (m-1-1) -- (m-1-2);
	\draw[>=stealth, ->] (m-1-2) -- (m-1-3);
	
	\draw[>=stealth, ->] (m-2-1) -- (m-2-2);
	\draw[>=stealth, ->] (m-2-2) -- (m-2-3);
	
	\draw[>=stealth, ->] (m-1-1) -- (m-2-1) node[pos=0.5, left, anchor=east]{$\sigma_k$};
	\draw[>=stealth, ->] (m-1-2) -- (m-2-2) node[pos=0.5, left, anchor=east]{$\sigma_{k'}$};
	\draw[>=stealth, ->] (m-1-3) -- (m-2-3) node[pos=0.5, right, anchor=west]{$\sigma$};	
\end{tikzpicture}
\end{center}
As in \autoref{sec: graph of groups}, these maps $\sigma_k$ and $\sigma$ provide a normal form at $\bar v_+$ to write the elements of $L_k$ and $L$ respectively.

\paragraph{Turn at $c$.}
Given $g \in L$, the next step is to analyze the non-degenerate excursions of $\geo o{go}$ in the peripheral subtree $\bar P(c)$ and its translates.
To read this information on the normal form of $g$, we define a notion of turn.

\begin{defi}
\label{def: turn}
	A \emph{turn at $c$} is a triple $(\mathbf e, \mathbf g, \mathbf e')$ where $\mathbf g \in L_{c}$ while $\mathbf e$ and $\mathbf e'$ are edges in  $E$ such that $t(\mathbf e)$ and $\iota(\mathbf e')$ are in the $L$-orbit of $\bar v_c$.
	Such a turn is \emph{reduced} if the decorated path $\mathbf e\mathbf g\mathbf e'$ is reduced.
\end{defi}

In practice, we sometimes blur the distinction between a turn $(\mathbf e, \mathbf g, \mathbf e')$ and the corresponding word $\mathbf e\mathbf g\mathbf e'$ of $\mathbf L$.
Each edge $e$ in $\bar S$ has unique pre-image in $S_\Lambda$.
For simplicity we still denote by $p_e$ the corresponding attaching point of the graph of actions $\Lambda$.

\begin{defi}
\label{def: realization turn}
	The \emph{realization} of a turn $(\mathbf e, \mathbf g, \mathbf e')$ is the pair of points 
	\begin{equation*}
		\left(s(\mathbf e)p_{\mathbf e}, \mathbf gs(\bar {\mathbf e}')p_{\bar{\mathbf e}'}\right).
	\end{equation*}
\end{defi}

\begin{rema*}
	Recall that $s(\mathbf e)$ maps the endpoint $t(\mathbf e)$ to a vertex in $V$, here $\bar v_c$.
	Same with $s(\bar {\mathbf e}')$.
	It follows that the realization $(z,z')$ of $(\mathbf e, \mathbf g, \mathbf e')$ is an element of $b_c^{-1}(0) \times b_c^{-1}(0)$.
	In particular, $\dist z{z'} \leq 2\rho$.
\end{rema*}

\begin{lemm}
\label{res: characterization reduced turn}
	Let $(\mathbf e, \mathbf g, \mathbf e')$ be a turn at $c$ and $(z,z')$ its realization.
	The following are equivalent
	\begin{enumerate}
		\item \label{enu: characterization reduced turn - def}
		$(\mathbf e, \mathbf g, \mathbf e')$ is not reduced;
		\item \label{enu: characterization reduced turn - nc}
		$\mathbf e' = \bar {\mathbf e}$ and $\mathbf g$ belongs to $N_c$;
		\item \label{enu: characterization reduced turn - realization}
		$z = z'$.
	\end{enumerate}
\end{lemm}

\begin{proof}
	By construction $s(\mathbf e)\mathbf e$ is an edge of $\bar S$ adjacent to $\bar v_c$.
	Hence its stabilizer is $N_c$.
	In particular, $N_c$ fixes $s(\mathbf e)p_{\mathbf e}$.
	The implications \ref{enu: characterization reduced turn - def} $\Rightarrow$ \ref{enu: characterization reduced turn - nc} and \ref{enu: characterization reduced turn - nc} $\Rightarrow$ \ref{enu: characterization reduced turn - realization} follow from these observation.
	Assume now that $z = z'$.
	We built $\bar S$ by collapsing certain edges of $S_\Lambda$.
	Alternatively it can be obtained from $S(0)$ by collapsing every edge which is not adjacent to a vertex in the orbit of the peripheral component $T \cap \bar P(c)$.
	We make an abuse of notation and still denote by $\mathbf e$ and $\mathbf e'$ their (unique) pre-images in $S(0)$.
	Let $v$ (\resp $v'$) be the start (\resp end) of $\mathbf e$ (\resp $\mathbf e'$) in $S(0)$.
	Note that $v$ and $v'$ are transverse vertices.
	The attaching point $z = s(\mathbf e)p_{\mathbf e}$ and $z' = \mathbf g s(\bar {\mathbf e}')p_{\bar{\mathbf e}'}$ belongs to the transverse subtrees $s(\mathbf e)Y_v$ and $\mathbf g s(\bar {\mathbf e}')Y_{v'}$ of the graph of actions $\Lambda(0)$.
	Since $z = z'$, the subtrees $s(\mathbf e)Y_v$ and $\mathbf g s(\bar {\mathbf e}')Y_{v'}$ are equal, i.e. $s(\mathbf e)v = \mathbf g s(\bar {\mathbf e}')v'$.
	By definition of a turn, $\mathbf g$ belongs to $L_c$.
	Consequently $s(\mathbf e)\mathbf e = \mathbf g s(\bar {\mathbf e}')\bar {\mathbf e}'$.
	It follows that $\mathbf e' = \bar {\mathbf e}$ and $\mathbf g$ fixes $s(\mathbf e)\mathbf e$.
	In other words $(\mathbf e, \mathbf g, \mathbf e')$ is not reduced, which proves \ref{enu: characterization reduced turn - realization} $\Rightarrow$ \ref{enu: characterization reduced turn - def}.
\end{proof}

We now define a finite set of turns $\mathscr T$ as follows.
First, we choose for each generator $u \in U$ a preferred normal form $\mathbf u$ based at $\bar v^+$ of $\eta(u)$:
\begin{equation*}
	\mathbf u = \mathbf g_0 \mathbf e_0 \mathbf g_1 \mathbf e_1 \dots \mathbf g_{m-1}\mathbf e_m \mathbf g_m
\end{equation*}
We write $\mathbf U$ for the set of normal forms obtained in this way.
Let $\mathscr T$ be the set of all turns $(\mathbf e, \mathbf g, \mathbf e')$ at $c$ such that $\mathbf e\mathbf g\mathbf e'$ is a subword of some $\mathbf u \in \mathbf U$.
We denote by $F$ the set of all realizations of the turns in $\mathscr T$ and let 
\begin{equation*}
	F_0 = \set{(y,y') \in F}{\dist y{y'} \leq \dist z{z'}, \ \forall (z,z') \in F}
\end{equation*}
which is the set of all ``shortest'' realizations.

\begin{defi}
\label{def: minimal turn}
	A turn  $(\mathbf e, \mathbf g, \mathbf e') \in \mathscr T$ is \emph{minimal} if its realization belongs to $F_0$.
	We denote by $\mathscr T_0$ the set of all minimal turns.
\end{defi}

Recall that $N_c$ stands for the pointwise stabilizer of $\bar P(c)$ while $K_c$ is the peripheral subgroup of $v_c$ (see \autoref{not: peripheral subgroup of cone point})

\begin{prop}
\label{res: shortening - dichotomy turn}
	Suppose that for every turn $(\mathbf e,\mathbf g,\mathbf e') \in \mathscr T_0$, the element $\mathbf g$ belongs $K_c$ and $\mathbf e' = \bar {\mathbf e}$.
	Then the link of $v_c$ in $S_\Lambda$ consists of a single $L_c$-orbit. 
	Moreover the pointwise stabilizer $N_c$ of $\bar P(c)$ has finite index in $L_c$.
\end{prop}

\begin{proof}
	The first step of the proof is to prove that $\mathscr T = \mathscr T_0$. 
	Let $(\mathbf e,\mathbf g,\mathbf e') \in \mathscr T_0$ be a turn at $c$.
	By assumption, we have $\mathbf e' = \bar {\mathbf e}$.
	Nevertheless, $\mathbf e \mathbf g \mathbf e'$ is a subword of a normal form $\mathbf u \in \mathbf U$.
	In particular, $(\mathbf e,\mathbf g,\mathbf e')$ is not reduced, hence $\mathbf g$ does not belong to $N_c$.
	By \autoref{res: tree-graded - fixed point set}, there exists a finite number $r \in (0, \rho]$ such that $\fix{\mathbf g} = \bar P(c,r)$.
	We claim that $r = \rho$ (in particular $\rho$ is finite).
	Assume on the contrary that $r < \rho$.
	In view of \autoref{res: coro root stability}, the image of $\mathbf g$ in $Q_c = L_c / N_c$ has infinite order.
	Recall that $K_c$ is the pre-image of the torsion part of $Q_c$.
	Hence $\mathbf g$ does not belongs to $K_c$.
	This contradicts our assumption and completes the proof of our first claim.
	
	We now claim that $\dist y{y'} = 2\rho$, where $(y,y')$ stands for the realization of $(\mathbf e,\mathbf g,\mathbf e')$.
	Since $\mathbf e' = \bar {\mathbf e}$, we have $p_{\mathbf e} = p_{\mathbf e'}$.
	According to our previous claim $\fix{\mathbf g}$ is exactly $\bar P(c,\rho)$, which is reduced to a single point $z_c$ satisfying $\dist {z_c}{s(\mathbf e)p_{\mathbf e}} = \rho$.
	Hence 
	\begin{equation*}
		\dist {s(\mathbf e)p_{\mathbf e}}{\mathbf g s(\mathbf e)p_{\mathbf e}} = 2\rho,
	\end{equation*}
	which completes the proof of our second claim.
	As we noticed earlier, $\dist z{z'} \leq 2 \rho$, for every realization $(z,z') \in F$.
	We get from our second claim that $F_0 = F$ and thus $\mathscr T_0 = \mathscr T$, as we announced.
	
	According to our assumption, for every turn $(\mathbf e,\mathbf g,\mathbf e') \in\mathscr T$ we have $\mathbf e' = \bar {\mathbf e}$ and $\mathbf g \in K_c$.
	Recall that $\bar v^+$ is not a translate of $\bar v_c$.
	Hence every element $h\in L$ can be represented by a (non necessarily reduced) decorated circuit at $\bar v^+$
	\begin{equation}
	\label{eqn: shortening - dichotomy turn}
		\mathbf h = \mathbf g_0 \mathbf e_1 \mathbf g_1 \mathbf e_2 \dots \mathbf g_{m-1}\mathbf e_m \mathbf g_m
	\end{equation}
	such that every subword $\mathbf e_i \mathbf g_i \mathbf e_{i+1}$ corresponding to a turn at $c$ belongs to $\mathscr T$. 
	Those turns always satisfy $\mathbf e_{i+1} = \bar {\mathbf e}_i$, which forces the link of $\bar v_c$ in $\bar S$ to be reduced to a single $L_c$-orbit.
	Hence so is the link of $v_c$ in $S_\Lambda$.
	Recall that the stabilizer of every edge of $\bar S$ adjacent to $\bar v_c$ is $N_c$.
	Consequently there is an  epimorphism from $L$ onto $Q_c = L_c/N_c$ obtained by ``killing'' every vertex group distinct from $\bar v_c$, and every stable letter in the graph of groups decomposition induced by $\bar S$.
	However for every turn $(\mathbf e,\mathbf g,\mathbf e') \in\mathscr T$, the element $\mathbf g$ belongs to $K_c$.
	It follows from the decomposition (\ref{eqn: shortening - dichotomy turn}) that the morphism $L \onto Q_c$ maps $L$ onto the torsion part of $Q_c$.
	Since $Q_c = L_c/N_c$ is finitely generated, $Q_c$ is finite.
	Consequently $N_c$ has finite index in $L_c$.
\end{proof}

Recall that \autoref{res: shortening peripheral - main theo} assumes that either the link of $v_c$ contains two $L_c$-orbits or $N_c$ has infinite index in $L_c$.
In view of \autoref{res: shortening - dichotomy turn}, there is a turn $(\mathbf e,\mathbf g,\mathbf e') \in \mathscr T_0$ such that either  $\mathbf e' \neq \bar {\mathbf e}$ or  $\mathbf g \notin K_c$.
\autoref{res: shortening peripheral - dehn twist} and \autoref{res: shortening peripheral - peripheral auto} handle these two cases respectively.


\begin{prop}[Shortening by Dehn twist]
\label{res: shortening peripheral - dehn twist}
	Assume that  there exists a turn $(\mathbf e,\mathbf g,\mathbf e') \in \mathscr T_0$  such that $\mathbf e' \neq \bar {\mathbf e}$.
	Then there is $\ell \in \R_+^*$ with the following property.
	For every sufficiently large $k \in \N$, there is a Dehn twist $\alpha_k \in \mcg[0]{L,\Lambda}$ over an edge of $S_\Lambda$ adjacent to $v$, and $\hat \alpha_k \in \aut{\hat L_k}$ preserving $\hat S_k$ such that $\zeta_k \circ \hat \alpha_k = \alpha_k \circ \zeta_k$ and 
		\begin{equation*}
			\lambda^+_1\left(\hat \phi_k \circ \hat \alpha_k \circ \hat \eta_k, U\right) \leq \lambda^+_1\left(\phi_k,U\right) - \frac \ell{\epsilon_k} \ \oas
		\end{equation*}
\end{prop}

\begin{proof}
	We denote by $(y,y')$ the realization of the turn  $(\mathbf e,\mathbf g,\mathbf e') $ and write $y = \limo y_k$ and $y' = \limo y'_k$.
	Fix $b \in G$ such that $\eta(b)$ is a non-trivial element of $N_c$.
	According to \autoref{res: shortening peripheral - shortening pairs}, there exists a sequence $(b_k)$ of elements of $\group b$ with the following properties:
	\begin{enumerate}
		\item $\eta(b_k)$ belongs to $N_c$, for every $k \in \N$.
		\item $\limo\dist{y_k} {\phi_k(b_k)y'_k}= 0$.
		\item if $z = \limo z_k$ and $z' = \limo z'_k$ are two points of $T$ such that $(z,z') \in F$, then for every $q \in \Z$,
		\begin{equation*}
			\limo \dist{z_k} {\phi_k(b_k)^qz'_k}\leq \dist z{z'}.
		\end{equation*}
	\end{enumerate}
	Recall that $V_k$ and $E_k$ are respectively the vertex set of $D_{k,0}$ and the edge set of $D_{k,1}$.
	We make an abuse of notation and still denote by $\mathbf e$ and $\mathbf e'$ the pre-images of $\mathbf e$ and $\mathbf e'$ in $E_k$.
	Similarly we still write $\bar v_c$ for the pre-image of $\bar v_c$ in $V_k$ and  $\hat L_{k,c}$ for the vertex group $\hat L_{k, \bar v_c}$.
	It follows from the definition of strong covers that for every sufficiently large $k \in \N$, the element $\hat \eta_k(b)$ belongs to 
	\begin{equation*}
		s\left(\mathbf e\right) \hat L_{k, \mathbf e} s\left(\mathbf e\right)^{-1} \cap s\left(\bar {\mathbf e}'\right) \hat L_{k, \bar {\mathbf e}'} s\left(\bar {\mathbf e}'\right)^{-1}
		\subset \hat L_{k,c}.
	\end{equation*}
	Hence so does $\hat \eta_k(b_k)$.
	For such an integer $k$,  we define an isomorphism $\mathbf L_k \to \mathbf L_k$, which fixes every edge in $E_k \setminus\{ \mathbf e, \bar{\mathbf e}\}$ as well as the factors $\hat L_{k,\bar v}$ for every $\bar v \in V_k$ and maps $\mathbf e$ to $\mathbf e\hat \eta_k(b_k)$.
	This morphisms induces a Dehn twist $\hat \alpha_k \in \aut{\hat L_k}$.
	The Dehn twists $\alpha_k \in \mcg[0]{L,\Lambda}$ are defined in a similar way so that the following diagram commutes.
	\begin{center}
	\begin{tikzpicture}
		\matrix (m) [matrix of math nodes, row sep=2em, column sep=2.5em, text height=1.5ex, text depth=0.25ex] 
		{ 
			\hat L_k	 & L \\
			\hat L_k	 & L \\
		}; 
		\draw[>=stealth, ->] (m-1-1) -- (m-1-2);
		\draw[>=stealth, ->] (m-2-1) -- (m-2-2);
		
		\draw[>=stealth, ->] (m-1-1) -- (m-2-1) node[pos=0.5, left, anchor=east]{$\hat \alpha_k$};
		\draw[>=stealth, ->] (m-1-2) -- (m-2-2) node[pos=0.5, right, anchor=west]{$\alpha_k$};
	\end{tikzpicture}
	\end{center}

	We now study the effect of this Dehn twist on the generating set $\eta(U)$ of $L$.
	Let $u \in U$.
	Recall that $\bar v^+$ is not in the orbit of $\bar v_c$.
	Hence, the normal form $\mathbf u \in \mathbf U$ of $\eta(u)$ can be written
	\begin{equation*}
		\mathbf u = \mathbf h_0 (\mathbf e_1 \mathbf g_1 \mathbf e'_1) \mathbf h_1 (\mathbf e_2 \mathbf g_2 \mathbf e'_2) \mathbf h_2 \dots(\mathbf e_m \mathbf g_m \mathbf e'_m) \mathbf h_m
	\end{equation*}
	where $(\mathbf e_i, \mathbf g_i, \mathbf e'_i)$ belongs to $\mathscr T$, each $\mathbf h_i$ fixes a vertex in $V$ different from $\bar v_c$.
	For each index $i$, we write $p_i$ (\resp $p'_i$) for the attaching point in $T$ of the edge $\mathbf e_i$ (\resp $\mathbf e'_i$) in the graph of actions $\Lambda$.
	In addition we let 
	\begin{equation*}
		z_i = s(\mathbf e_i)p_i,\quad
		z'_i = s(\bar {\mathbf e}'_i)p'_i,\quad
		x_i = s(\bar{\mathbf e}_i)p_i,\quad \text{and}\quad
		x'_i = s(\mathbf e'_i)p'_i.
	\end{equation*}
	In particular, $(z_i, \mathbf g_i z'_i)$ is the realization of $(\mathbf e_i, \mathbf g_i, \mathbf e'_i)$.
	The normal form of $\eta(u)$ can be used to decompose the geodesic $\geo o{\eta(u)o} \subset T$.
	More precisely applying (\ref{eqn: graph of actions - dist}) we get
	\begin{equation}
	\label{eqn: dehn twist - limit split - dehn twist}
		\dist{\eta(u)o^+}{o^+} = 
		\dist {o^+}{\mathbf h_0 x_0}
		+ \sum_{i=1}^m \dist{z_i}{\mathbf g_i z'_i}  
		+ \sum_{i=1}^{m-1} \dist{x'_i}{\mathbf h_i x_{i+1}}
		+ \dist{x'_m}{\mathbf h_mo^+}.
	\end{equation}
	Let us now apply the modular automorphism $\alpha_k$.	
	For every $i \in \intvald 0m$ we choose a pre-image $h_i$ of $\mathbf h_i$ in $G$.
	By construction $\mathbf h_i$ belongs to the vertex group $L_{\bar w_i}$ for some $\bar w_i \in V\setminus\{\bar v_c\}$.
	If $k$ is a sufficiently large integer, then 
	\begin{equation*}
		\hat \eta_k(h_i) \in \hat L_{k, \bar w_i}
	\end{equation*}
	where we have identified $\bar w_i \in V$ with its unique pre-image in $V_k$.
	Similarly we choose for each $i \in \intvald 1m$ a pre-image $g_i$ of $\mathbf g_i$ in $G$ such that $\hat \eta_k(g_i)$ belongs to $\hat L_{k,c}$, for all but finitely many $k\in \N$.
	Consequently, 
	\begin{equation*}
		\mathbf u_k = \hat  \eta_k (h_0) 
		\left[\mathbf e_1 \hat \eta_k(g_1) \mathbf e'_1\right] \hat \eta_k(h_1) \cdots 
		\left[\mathbf e_m \hat \eta_k(g_m) \mathbf e'_m\right] \hat \eta_k(h_m) 
	\end{equation*}
	is a normal form in $\mathbf L_k$ of $\hat \eta_k(u)$, provided $k$ is large enough (as above, we identified the edges of $E$ with their pre-images in $E_k$).
	Consequently a pre-image in $\mathbf L_k$ of $\hat \alpha_k \circ \hat \eta_k(u)$ is 
	\begin{equation*}
		\mathbf u'_k =\hat  \eta_k (h_0) 
		\left[\mathbf e_1  \hat \eta_k(b_k)^{q_1} \eta_k(g_1)  \mathbf e'_1\right] \hat \eta_k(h_1)
		 \cdots \\  
		\left[\mathbf e_m \hat \eta_k(b_k)^{q_m} \hat \eta_k(g_m) \mathbf e'_m\right] \hat \eta_k(h_m),
	\end{equation*}
	where the value of $q_i \in \{-1, 0,1\}$ depends on the edges of the turn  $(\mathbf e_i, \mathbf g_i, \mathbf e'_i)$.
	Since $p_i$ and $p'_i$ are points of $T$, we can write them as
	\begin{equation*}
		p_i = \limo p_{k,i} 
		\quad \text{and} \quad
		p'_i = \limo p'_{k,i} 
	\end{equation*}
	where $p_{k,i}, p'_{k,i} \in X_k$.
	In addition we let
	\begin{align*}
		z_{k,i}  &= \hat \phi_k \circ s(\mathbf e_i)p_{k,i}, & 
		z'_{k,i} &= \hat \phi_k \circ s(\bar {\mathbf e}'_i)p'_{k,i}, \\
		x_{k,i} &= \hat \phi_k \circ s(\bar{\mathbf e}_i)p_{k,i},&
		x'_{k,i} &= \hat \phi_k \circ s(\mathbf e'_i)p'_{k,i}.
	\end{align*}
	so that 
	\begin{equation*}
		z_i = \limo z_{k,i}, \quad
		z'_i = \limo z'_{k,i}, \quad
		x_i = \limo x_{k,i}, \quad \text{and} \quad
		x'_i = \limo x'_{k,i}.
	\end{equation*}
	Observe that our choices are consistent with the definition of $(y_k)$ and $(y'_k)$. 
	That is for instance, if $z_i = y$ (\resp $z_i = y'$), then $z_{k,i} = y_k$ (\resp $z_{k,i} = y'_k$) for all $k \in \N$. 
	Same with $z'_i$, $x_i$, $x'_i$.
	
	Applying the triangle inequality in $X_k$ as we did to establish (\ref{eqn: graph of actions - dist}) we observe that 
	\begin{multline}
	\label{eqn: shortening peripheral - apply trig inequality Xk - dehn twist}
			\dist{\hat \phi_k \circ \hat \alpha_k\circ \hat \eta_k(u)o^+_k}{o^+_k} \leq  
		 \dist {o^+_k}{\phi_k (h_0) x_{k,0}} \\
		 + \sum_{i=1}^m \dist{z_{k,i}}{\phi_k(b_k)^{q_i}\phi_k (g_i)z'_{k,i}}
		+ \sum_{i=1}^{m-1} \dist{x'_{k,i}}{\phi_k (h_i)x_{k,i+1}} \\
		 + \dist{x'_{k,m}}{\phi_k (h_m)o_k}
	\end{multline}
	We now focus on each term in the sum and estimate its asymptotic behavior as $k$ diverges to infinity.
	\begin{labelledenu}[T]
		\item 
		Observe first that for every $i \in \intvald 1{m-1}$,
		\begin{equation*}
			\limo \dist{x'_{k,i}}{\phi_k (h_i)x_{k,i+1}}
			= \dist{x'_i}{\mathbf h_i x_{i+1}}.
		\end{equation*}
		Similarly,
		\begin{equation*}
			\limo   \dist {o^+_k}{\phi_k (h_0) x_{k,0}} = \dist {o^+}{\mathbf h_0 x_0}
		\end{equation*}
		and
		\begin{equation*}
			\limo  \dist{x'_{k,m}}{\phi_k (h_m)o_k} = \dist{x'_m}{\mathbf h_mo^+}.
		\end{equation*}

		\item \label{enu: t2 dehn twist}
		Let $i \in \intvald 1m$ and consider the turn $(\mathbf e_i, \mathbf g_i, \mathbf e'_i)$.
		Recall that the pair $(z_i, \mathbf g_i z'_i)$ is the realization of $(\mathbf e_i, \mathbf g_i, \mathbf e'_i)$ which belongs to $\mathscr T$.
		Therefore it belongs to $F$.
		It follows then from our choice of $b_k$ that
		\begin{equation*}
			\limo  \dist{z_{k,i}}{\phi_k(b_k)^{q_i}\phi_k (g_i)z'_{k,i}} 
			 \leq \dist{z_i}{\mathbf g_i z'_i}.
		\end{equation*}
	\end{labelledenu}
	Combining these observations with (\ref{eqn: dehn twist - limit split - dehn twist}) and (\ref{eqn: shortening peripheral - apply trig inequality Xk - dehn twist}) we get that for every $u \in U$, 
	\begin{equation}
	\label{eqn: dehn twist - pre sum generic - dehn twist}
		\limo  \dist{\hat \phi_k \circ \hat \alpha_k\circ \hat \eta_k(u)o^+_k}{o^+_k} \leq \dist{\eta(u)o^+}{o^+}.
	\end{equation}
	Let us now come back to the terms studied in \ref{enu: t2 dehn twist}.
	By assumption there is at least one element $u \in U$ and one turn $(\mathbf e_i, \mathbf g_i, \mathbf e'_i)$ in its normal form that coincides with the turn $(\mathbf e, \mathbf g, \mathbf e')$ fixed at the beginning.
	For this turn, our choice of $b_k$ gives
	\begin{equation*}
		\limo  \dist{z_{k,i}}{\phi_k(b_k)\phi_k (g_i)z'_{k,i}} 
		= \limo \dist{y_k}{\phi_k (b_k)y'_k}
		= 0
	\end{equation*}
	Hence this specific element $u$ satisfies
	\begin{equation}
	\label{eqn: dehn twist - pre sum specific - dehn twist}
		\limo  \dist{\hat \phi_k \circ \hat \alpha_k\circ \hat \eta_k(u)o^+_k}{o^+_k} \leq \dist{\eta(u)o^+}{o^+} - \dist y{y'}.
	\end{equation}
	Consequently, when we sum (\ref{eqn: dehn twist - pre sum generic - dehn twist}) and (\ref{eqn: dehn twist - pre sum specific - dehn twist}) to estimate the energy of $\hat \phi_k \circ \hat \alpha_k \circ \hat \eta_k$ (in the non-rescaled space $X_k$) we get
	\begin{align*}
		\limo  \epsilon_k\lambda^+_1(\hat \phi_k \circ \hat \alpha_k \circ \hat \eta_k, U)
		& \leq \limo \epsilon_k\lambda_1(\hat \phi_k \circ \hat \alpha_k\circ \hat \eta_k, U, o^+_k) \\
		& \leq \limo \epsilon_k\lambda_1(\phi_k, U, o^+_k) - \dist y{y'}
	\end{align*}
	Recall that $o^+_k$ has been chosen to (almost) minimize the restricted energy of $\phi_k$.	Consequently we get
	\begin{equation*}
		\lambda^+_1(\hat \phi_k \circ \hat \alpha_k\circ \hat \eta_k, U)
		\leq \limo \lambda_1(\phi_k, U) - \frac{\dist y{y'}}{\epsilon_k} + o \left( \frac 1{\epsilon_k}\right). \qedhere
	\end{equation*}
\end{proof}

\begin{prop}[Shortening by peripheral automorphism]
\label{res: shortening peripheral - peripheral auto}
	Assume that  there exists a turn $(\mathbf e, \mathbf g, \mathbf e') \in \mathscr T_0$ and such that $\mathbf g \notin K_c$
	Then there is $\ell \in \R_+^*$ with the following property.
	For every sufficiently large $k \in \N$, there is a generalized Dehn twist $\alpha_k \in \mcg[0]{L,\Lambda}$, and $\hat \alpha_k \in \aut{\hat L_k}$ preserving $\hat S_k$ such that $\zeta_k \circ \hat \alpha_k = \alpha_k \circ \zeta_k$ and 
		\begin{equation*}
			\lambda^+_1\left(\hat \phi_k \circ \hat \alpha_k \circ \hat \eta_k, U\right) \leq \lambda^+_1\left(\phi_k,U\right) - \frac \ell{\epsilon_k} \ \oas
		\end{equation*}
\end{prop}

\begin{proof}
	Recall that $K_c$ is the pre-image of the torsion part of $Q_c = L_c / N_c$.
	Hence $L_c$ decomposes as $L_c = A_c \oplus K_c$ where $A_c$ is a non-trivial, free abelian group.
	We make an abuse of notation and still write $\bar v_c$ for the pre-image of $\bar v_c$ in $V_k$.
	Similarly, we write $\hat L_{k,c}$ for the vertex group $\hat L_{k, \bar v_c}$ associated to the vertex $\bar v_c \in \bar S_k$.
	Recall that $L_c$ is finitely generated relative to $N_c$ (\autoref{res: tree-graded - stab v fg rel fix}).
	In particular, $A_c$ is finitely generated.
	We fix a basis $a_0, \dots, a_r$ of $A_c$, and write $\tilde a_0, \dots, \tilde a_r$ for pre-images of these elements in $G$.
	Without loss of generality we can choose our basis so that $\mathbf g$ decomposes as $\mathbf g = ma_0^n$ where $m \in K_c$ and  $n \in \N \setminus\{0\}$.	
	Fix $b \in G$ such that $\eta(b)$ is a non-trivial element of $N_c$.
	According to \autoref{res: shortening peripheral - shortening pairs}, there exists a sequence $(b_k)$ of elements of $\group b$ with the following properties:
	\begin{enumerate}
		\item $\eta(b_k)$ belongs to $N_c$, for every $k \in \N$.
		\item $\limo\dist{y_k} {\phi_k(b^n_k)y'_k}= 0$.
		\item if $z = \limo z_k$ and $z' = \limo z'_k$ are two points of $T$ such that $(z,z') \in F$, then for every $q \in \Z$, 
		\begin{equation*}
			\limo \dist{z_k} {\phi_k(b^q_k)z'_k}\leq \dist z{z'}.
		\end{equation*}
	\end{enumerate}

	It follows from the definition of strong covers that for every sufficiently large $k \in \N$, the (abelian) group $\hat L_{k, c}$ admits a non trivial splitting $\hat L_{k,c} =A_{k,c} \oplus  K_{k,c}$ where
	\begin{itemize}
		\item the map $\zeta_k \colon \hat L_k \to L$ induces a isomorphism from $A_{k,c}$ onto $A_c$ and maps $K_{k,c}$ into $K_c$;
		\item the pre-image $\hat \eta_k(b_k) \in L_k$ of $\eta(b_k) \in L$ belongs to $K_{k,c}$.
	\end{itemize}
	We denote by $a_{j,k} = \hat \eta_k(\tilde a_j)$ for every $j \in \intvald 0r$.
	Note that $a_{0,k}, \dots, a_{r,k}$ is a free basis of $A_{k,c}$ \oas.
	We define an automorphism $\hat \alpha_k$ of $\hat L_{k,c}$ as follows: it fixes $K_{k,c}$ as well as $a_{1,k}, \dots, a_{r,k}$ and maps $a_{0,k}$ to $a_{0,k}\hat \eta_k(b_k)$.
	Observe that $\hat \alpha_k$ is trivial when restricted to the stabilizer $\hat L_{k,e}$ of any edge of $\bar S_k$ adjacent to $\bar v_c$.
	Indeed by definition of strong cover the map $\zeta_k\colon \hat L_k \to L$ maps $\hat L_{k,e}$ into the stabilizer of some edge of $\bar S$ adjacent to $\bar v_c$, that is into $N_c$.
	However $N_c$ is contained in the factor $K_c$, hence $\hat L_{k,e}$ necessarily belongs to $K_{k,c}$.
	Consequently $\hat \alpha_k$ is trivial when restricted to $\hat L_{k,e}$.
	It follows that $\hat \alpha_k$ extends to an automorphism of $\hat L_k$, that we still denote by $\hat \alpha_k$, that preserves the graph group decomposition of $\bar S_k$.
	
	The automorphisms $\alpha_k \in \mcg[0]{L,\Lambda}$ are defined in a similar way so that the following diagram commutes.
	\begin{center}
	\begin{tikzpicture}
		\matrix (m) [matrix of math nodes, row sep=2em, column sep=2.5em, text height=1.5ex, text depth=0.25ex] 
		{ 
			\hat L_k	 & L \\
			\hat L_k	 & L \\
		}; 
		\draw[>=stealth, ->] (m-1-1) -- (m-1-2);
		\draw[>=stealth, ->] (m-2-1) -- (m-2-2);
		
		\draw[>=stealth, ->] (m-1-1) -- (m-2-1) node[pos=0.5, left, anchor=east]{$\hat \alpha_k$};
		\draw[>=stealth, ->] (m-1-2) -- (m-2-2) node[pos=0.5, right, anchor=west]{$\alpha_k$};
	\end{tikzpicture}
	\end{center}
	
	We now study the effect of this automorphism on the generating set of $L$.
	We follow the same strategy as in the proof of the Dehn twist case.
	Let $u \in U$ and $\mathbf u \in \mathbf U$ its normal form that we write
	\begin{equation*}
		\mathbf u = \mathbf h_0 (\mathbf e_1 \mathbf g_1 \mathbf e'_1) \mathbf h_1 (\mathbf e_2 \mathbf g_2 \mathbf e'_2) \mathbf h_2 \dots(\mathbf e_m \mathbf g_m \mathbf e'_m) \mathbf h_m
	\end{equation*}
	where $(\mathbf e_i, \mathbf g_i, \mathbf e'_i)$ belongs to $\mathscr T$, each $\mathbf h_i$ fixes a vertex in $V$ different from $\bar v_c$.
	For each index $i$, we write $p_i$ (\resp $p'_i$) for the attaching point in $T$ of the edge $\mathbf e_i$ (\resp $\mathbf e'_i$).
	In addition we let 
	\begin{equation*}
		z_i = s(\mathbf e_i)p_i,\quad
		z'_i = s(\bar {\mathbf e}'_i)p'_i,\quad
		x_i = s(\bar{\mathbf e}_i)p_i,\quad \text{and}\quad
		x'_i = s(\mathbf e'_i)p'_i.
	\end{equation*}
	Just as before we have
	\begin{equation}
	\label{eqn: dehn twist - limit split - auto}
		\dist{\eta(u)o^+}{o^+} = 
		\dist {o^+}{\mathbf h_0 x_0}
		+ \sum_{i=1}^m \dist{z_i}{\mathbf g_i z'_i}  
		+ \sum_{i=1}^{m-1} \dist{x'_i}{\mathbf h_i x_{i+1}}
		+ \dist{x'_m}{\mathbf h_mo^+}.
	\end{equation}
	Let us now apply the modular automorphism $\alpha_k$.	
	For every $i \in \intvald 0m$ we choose a pre-image $h_i$ of $\mathbf h_i$ in $G$.
	By construction $\mathbf h_i$ belongs to the vertex group $L_{\bar w_i}$ for some $\bar w_i \in V\setminus\{\bar v_c\}$.
	If $k$ is a sufficiently large integer, then 
	\begin{equation*}
		\hat \eta_k(h_i) \in \hat L_{k, \bar w_i}
	\end{equation*}
	where we identified $\bar w_i \in V$ with its unique pre-image in $V_k$.
	Similarly for every $i \in \intvald 1m$, we chose a pre-image $g_i \in G$ of $\mathbf g_i$ such that $\hat \eta_k(g_i)$ belongs to $\hat L_{k,c}$ \oas.
	\begin{equation*}
		\mathbf u_k = \hat  \eta_k (h_0) 
		\left[\mathbf e_1 \hat \eta_k(g_1) \mathbf e'_1\right] \hat \eta_k(h_1) \cdots 
		\left[\mathbf e_m \hat \eta_k(g_m) \mathbf e'_m\right] \hat \eta_k(h_m) 
	\end{equation*}
	is a normal form in $\mathbf L_k$ of $\hat \eta_k(u)$ \oas.
	The element $\mathbf g_i$ can be decomposed $\mathbf g_i = m_ia_0^{q_i}$ where $q_i \in \Z$ and $m_i$ belongs to $\mathbf Z a_1 \oplus \dots \oplus \mathbf Z a_r \oplus K_c$.
	In particular, $m_i$ is fixed by $\alpha_k$, while, $a_0^{q_i}$ is sent to $a_0^{q_i}\eta(b_k)^{q_i}$.
	It follows that $\alpha_k(\mathbf g_i) =\eta(b_k)^{q_i} \mathbf g_i $.
	Hence $\hat \alpha_k \circ \hat \eta_k(g_i) =  \hat \eta_k(b_k)^{q_i}\hat \eta_k(g_i)$ \oas.
	Consequently a pre-image in $\mathbf L_k$ of $\hat \alpha_k \circ \hat \eta_k(u)$ is 
	\begin{equation*}
		\mathbf u'_k = \hat  \eta_k (h_0) 
		\left[\mathbf e_1 \hat \eta_k(b_k)^{q_1}\hat \eta_k(g_1) \mathbf e'_1\right] \hat \eta_k(h_1)
		 \cdots \\  
		\left[\mathbf e_m \hat \eta_k(b_k)^{q_m} \hat \eta_k(g_m)\mathbf e'_m\right] \hat \eta_k(h_m),
	\end{equation*}
	Since $p_i$ and $p'_i$ are points of $T$, we can write them as
	\begin{equation*}
		p_i = \limo p_{k,i} 
		\quad \text{and} \quad
		p'_i = \limo p'_{k,i} 
	\end{equation*}
	where $p_{k,i}, p'_{k,i} \in X_k$.
	In addition we let
	\begin{align*}
		z_{k,i}  &= \hat \phi_k \circ s(\mathbf e_i)p_{k,i}, & 
		z'_{k,i} &= \hat \phi_k \circ s(\bar {\mathbf e}'_i)p'_{k,i}, \\
		x_{k,i} &= \hat \phi_k \circ s(\bar{\mathbf e}_i)p_{k,i},&
		x'_{k,i} &= \hat \phi_k \circ s(\mathbf e'_i)p'_{k,i}.
	\end{align*}
	so that 
	\begin{equation*}
		z_i = \limo z_{k,i}, \quad
		z'_i = \limo z'_{k,i}, \quad
		x_i = \limo x_{k,i}, \quad \text{and} \quad
		x'_i = \limo x'_{k,i}.
	\end{equation*}	
	As previously we have the following estimate,
	\begin{multline}
	\label{eqn: shortening peripheral - apply trig inequality Xk - auto}
			\dist{\hat \phi_k \circ \hat \alpha_k\circ \hat \eta_k(u)o^+_k}{o^+_k} \leq  
		 \dist {o^+_k}{\phi_k (h_0) x_{k,0}} \\
		 + \sum_{i=1}^m \dist{z_{k,i}}{\phi_k(b_k)^{q_i}\phi_k (g_i)z'_{k,i}}
		+ \sum_{i=1}^{m-1} \dist{x'_{k,i}}{\phi_k (h_i)x_{k,i+1}} \\
		 + \dist{x'_{k,m}}{\phi_k (h_m)o_k}
	\end{multline}
	Moreover we prove that for every $i \in \intvald 1{m-1}$,
	\begin{equation*}
		\limo \dist{x'_{k,i}}{\phi_k (h_i)x_{k,i+1}}
		= \dist{x'_i}{\mathbf h_i x_{i+1}}.
	\end{equation*}
	while,
	\begin{equation*}
		\limo   \dist {o^+_k}{\phi_k (h_0) x_{k,0}} = \dist {o^+}{\mathbf h_0 x_0}
	\end{equation*}
	and
	\begin{equation*}
		\limo  \dist{x'_{k,m}}{\phi_k (h_m)o_k} = \dist{x'_m}{\mathbf h_mo^+}.
	\end{equation*}

	Let $i \in \intvald 1m$ and consider the turn $(\mathbf e_i, \mathbf g_i, \mathbf e'_i)$.
	Recall that the pair $(z_i, \mathbf g_i z'_i)$ is the realization of $(\mathbf e_i, \mathbf g_i, \mathbf e'_i)$ and thus belongs to $F$.		
	It follows then from our choice of $b_k$ that
	\begin{equation*}
		\limo  \dist{z_{k,i}}{\phi_k(b_k)^{q_i}\phi_k (g_i)z'_{k,i}} 
		 \leq \dist{z_i}{\mathbf g_i z'_i}.
	\end{equation*}
	Combining these observations with (\ref{eqn: dehn twist - limit split - auto}) and (\ref{eqn: shortening peripheral - apply trig inequality Xk - auto}) we get that for every $u \in U$, 
	\begin{equation}
	\label{eqn: dehn twist - pre sum generic- auto}
		\limo  \dist{\hat \phi_k \circ \hat \alpha_k\circ \hat \eta_k(u)o^+_k}{o^+_k} \leq \dist{\eta(u)o^+}{o^+}.
	\end{equation}
	
	The previous computation holds for every $u \in U$.
	However, by assumption there is at least one element $u \in U$ and one turn $(\mathbf e_i, \mathbf g_i, \mathbf e'_i)$ in its normal form that coincides with the turn $(\mathbf e, \mathbf g, \mathbf e')$ fixed at the beginning.
	By construction it satisfies $q_i = n$ (this is the way we defined $n$ at the beginning of the proof).
	For this turn, our choice of $b_k$ gives
	\begin{equation*}
		\limo  \dist{z_{k,i}}{\phi_k(b_k)^n\phi_k (g_i)z'_{k,i}} 
		= \limo \dist{y_k}{\phi_k (b_k)^ny'_k}
		= 0.
	\end{equation*}
	Hence this specific element $u$ satisfies
	\begin{equation}
	\label{eqn: dehn twist - pre sum specific- auto}
		\limo  \dist{\hat \phi_k \circ \hat \alpha_k\circ \hat \eta_k(u)o^+_k}{o^+_k} \leq \dist{\eta(u)o^+}{o^+} - \dist y{y'}.
	\end{equation}
	The rest of the proof goes as in the proof of \autoref{res: shortening peripheral - dehn twist}.
	We conclude that
	\begin{equation*}
		\lambda^+_1(\hat \phi_k \circ \hat \alpha_k \circ \hat \eta_k, U)
		\leq \limo \lambda_1(\phi_k, U) - \frac{\dist y{y'}}{\epsilon_k} + o \left( \frac 1{\epsilon_k}\right). \qedhere
	\end{equation*}
\end{proof}

\begin{proof}[Proof of \autoref{res: shortening peripheral - main theo}]
	We now summarize the above study.
	Fix a peripheral vertex $v \in S_\Lambda$.
	If we collapse all edges of $S_\Lambda$ which are not adjacent to a vertex in the $L$-orbit of $v_c$, we get a new tree $\bar S$ which is the ``smallest'' $L$-tree compatible is compatible with $S_\Lambda$ and ``sees'' $L_c$ as a vertex group.
	This tree $\bar S$ can also bee seen as a graph of group decomposition of $L$, providing a normal form to express the elements of $L$ (as explained in the beginning of this section).
	A turn is now a subword of the form $\mathbf e\mathbf g\mathbf e'$, where $\mathbf g$ belongs to $L_c$, which appears in the normal form an element of $L$ (see \autoref{def: turn}).
	Among all possible turns showing up in the normal form of the generators $u \in U$ of $L$, we pay a particular assumption to the subset $\mathscr T_0$ which consists of all turns whose ``angle at $c$'' is roughly speaking minimal (see \autoref{def: minimal turn}).
	In view of Assumptions \ref{enu: shortening peripheral - orbit hyp} and \ref{enu: shortening peripheral - index hyp}, \autoref{res: shortening - dichotomy turn} tells us that there is a turn  $\mathbf e\mathbf g\mathbf e'$ in $\mathscr T_0$ such that one of the following holds:
	\begin{enumerate}
		\item either $\mathbf e' \neq \bar{\mathbf e}$, or
		\item $\mathbf g$ does not belong to $K_c$.
	\end{enumerate}
	In the former case, we can use \autoref{res: shortening peripheral - dehn twist} to shorten the sequence $(\phi_k)$ with a sequence $(\alpha_k)$ of Dehn twists around an edge adjacent to $v$.
	In the latter case, the shortening is obtained by \autoref{res: shortening peripheral - peripheral auto} using generalized Dehn twists.
\end{proof}

%
\subsubsection{Simplicial transverse component}
%
\label{sec: shortening - simplicial transverse}

\begin{lemm}
\label{res: shortening - transverse simplicial arc}
	Let $x = \limo x_k$ and $y = \limo y_k$ be two points of $X_\omega$.
	Assume that $\geo xy$ is a transverse arc.
	Let $g \in G$ such that $\eta(g)$ is a non-trivial element fixing pointwise $[x,y]$.
	There exists a sequence $(g_k)$ of elements in $\group g$, such that  
	\begin{enumerate}
		\item $\eta(g_k)$ pointwise fixes $\geo xy$ \oas;
		\item $\limo\dist{x_k}{\phi_k(g_k)y_k} = 0$.
	\end{enumerate}
\end{lemm}

\begin{proof}
	Since $\geo xy$ is non-degenerate, $\eta(g)$ is not elusive (\autoref{res: limit tree - elusive sbgps are elliptic}).
	On the other hand, as $\geo xy$ is transverse, $\eta(g)$ does not belong to $L_c$ for some $c \in \mathcal C_\omega$ (\autoref{res: tree-graded - fixed point set}).
	It follows that $\phi_k(g)$ is loxodromic \oas.
	There exists a sequence $(d_k)$ converging to zero such that $x_k, y_k \in \fix{ \phi_k(g),d_k}$ \oas.
	In particular, $\norm{\phi_k(g)} \leq d_k$.
	Since $\phi_k(g)$ is $\delta_k$-thin, $x_k$ and $y_k$ lie in the $(d_k/2 + 50\delta_k)$-neighborhood of any $100\delta_k$-local $(1, \delta_k)$-quasi-geodesic of $\epsilon_kX_k$ joining the attractive and repulsive points of $\phi_k(g)$.
	Recall that $\phi_k(g)$ approximately acts on this geodesic by translation of length $\norm{\phi_k(g)}$.
	Consequently there exists $g_k \in \group {g}$ such that
	\begin{equation*}
		\dist{x_k}{\phi_k\left(g_k\right)y_k} \leq 2 d_k +1000 \delta_k\quad  \oas.
	\end{equation*}
	In particular, $\dist{x_k}{\phi_k(g_k)y_k}$ converges to zero.
	Since $\eta(g_k)$ belongs to $\group{\eta(g)}$, it pointwise fixes $\geo xy$, whence the result.
\end{proof}

\begin{prop}
\label{res: shortening morphism - transverse arc}
	Assume that $T$ contains a transverse simplicial arc.
	Then there exists $\ell > 0$ such that for every sufficiently large $k \in \N$, there are $\alpha_k \in \mcg[0]{L, S}$ and $\hat \alpha_k \in \aut{\hat L_k}$ preserving $\hat S_k$ such that $\zeta_k \circ \hat \alpha_k = \alpha_k \circ \zeta_k$ and 
		\begin{equation*}
			\lambda_1\left(\hat \phi_k \circ \hat \alpha_k \circ \hat \eta_k, U\right) < \lambda_1\left(\phi_k,U\right) - \frac \ell{\epsilon_k}\ \oas.
		\end{equation*}
\end{prop}

\begin{proof}
	The strategy in this situation has been detailed many times in the literature.
	We refer the reader to Rips and Sela~\cite[Section~6]{Rips:1994jg}, Perin~\cite[Section~5.6]{Perin:2008aa}, or Weidmann-Reinfeldt~\cite[Section~5.2.3]{Weidmann:2019ue}  for a carefully written proof.
	Let us just highlight its main steps.
	Let $x = \limo x_k$ and $y = \limo y_k$ such that  $\geo xy$ is a maximal simplicial arc of $T$ contained in a transverse subtree.
	Recall that $S$ is the simplicial tree obtained by refining the skeleton of the graph of actions $\Lambda$.
	We denote by $\bar S$ the one-edge splitting obtained from $S$ be collapsing every edge orbit but the edge $e$ corresponding to $\geo xy$.
	Note that the subgroup $L_e$ is non-trivial otherwise $L$ would be freely decomposable relative to $\mathcal H$.
	Using \autoref{res: shortening - transverse simplicial arc} we build a sequence $(g_k)$ of elements of $G$ such that 
	\begin{enumerate}
		\item $\eta(g_k)$ pointwise fixes $\geo xy$ \oas, in particular $\eta(g_k)$ belongs to $L_e$;
		\item $\limo\dist{x_k}{\phi_k(g_k)y_k} = 0$.
	\end{enumerate}
	For every $k \in \N$, we define a Dehn twist $\alpha_k \in \mcg[0]{L, S}$ over $e$ by mapping the edge $e$ to $e\eta(g_k)$.
	We write $\hat \alpha_k$ for the automorphism of $\hat L_k$ lifting $\alpha_k$ (see \autoref{res: lifting mcg to cover}).
	Proceeding as in the case of peripheral component (see \autoref{sec: shortening peripheral}) we prove that for every $u \in U$,
	\begin{equation*}
		\limo \dist{\hat\phi_k \circ \hat \alpha_k \circ \hat  \eta_k(u)o^+_k}{o^+_k}
		\leq \dist{\eta(u)o^+}{o^+}
	\end{equation*}
	However, since $U$ generates $L$, there is at least one element $u \in U$ such that the geodesic $\geo{o^+}{uo^+}$ contains a translated copy of $\geo xy$.
	For such an element we get
	\begin{equation*}
		\limo \dist{\hat\phi_k \circ \hat \alpha_k \circ \hat  \eta_k(u)o^+_k}{o^+_k}
		\leq \dist{\eta(u)o^+}{o^+} - \dist xy.
	\end{equation*}
	It follows from there that
	\begin{equation*}
		\lambda^+_1\left(\hat \phi_k \circ \hat \alpha_k \circ \eta_k, U\right) \leq \lambda^+_1\left(\phi_k,U\right) - \frac {\dist xy}{\epsilon_k}+ o \left( \frac 1{\epsilon_k}\right). \qedhere
	\end{equation*}
\end{proof}

\begin{rema*}
	We keep the notation of the proof.
	If the point $o^+$ lies on the arc $\geo xy$, we need to subdivide this arc first, but it does not change the strategy.
	
	In \cite{Rips:1994jg} or \cite{Weidmann:2019ue} the authors only consider limit groups over a fixed hyperbolic group.
	The proof of Perin \cite{Perin:2008aa} does not make this assumption. 
	However it requires uniform control on $L_e$.
	More precisely she asks that there exists $g \in G$ in the pre-image of $L_e$ such that $\norm {\phi_k(g)} > 12\delta_k$ \oas.
	This assumption is used to prove the analogue of our \autoref{res: shortening - transverse simplicial arc} see \cite[Lemma~5.25]{Perin:2008aa}.
	In our context, given a triple $(X,\Gamma, \mathcal C) \in \mathfrak H_\delta(\tilde \rho_k)$, we have no control on the injectivity radius of the action of $\Gamma$ on $X$.
	Consequently such a uniform control does not hold in general.
	However, as we have seen in the proof of \autoref{res: shortening - transverse simplicial arc}, it can be advantageously replaced by the fact that hyperbolic isometries of $\Gamma$ are uniformly thin.
\end{rema*}

%
\subsubsection{Proof of the shortening argument}
%

We are now in position to complete the proof of the shortening argument.

\begin{proof}[End of the proof of \autoref{res: shortening argument}]
	By assumption the image of $\phi_k \colon G \to \Gamma_k$ is not abelian \oas\ (otherwise so would be $L$).
	Applying \autoref{res: comparing energies}, we get
	\begin{equation*}
		\frac 1{2\card U} \lambda_1^+(\phi_k, U) \leq  \lambda_\infty(\phi_k, U) \leq \frac 1{\epsilon_k}.
	\end{equation*}
	Hence it suffices to prove that there exist $\ell > 0$ and two collections of automorphisms $\alpha_k \in \mcg[0]{L, S}$ and $\hat \alpha_k \in \aut{\hat L_k}$ preserving $\hat S_k$ such that $\zeta_k \circ \hat \alpha_k = \alpha_k \circ \zeta_k$ and satisfying
	\begin{equation*}
		\lambda^+_1\left(\hat \phi_k \circ \hat \alpha_k \circ \hat \eta_k, U\right) \leq \lambda^+_1\left(\phi_k,U\right) - \frac \ell{\epsilon_k} \ \oas.
	\end{equation*}
	Assume that such a statement does not hold.
	In view of Propositions~\ref{res: shortening morphism - axial / surface} and \ref{res: shortening morphism - transverse arc}, every transverse component of $T$ is reduced to a point.
	Said differently, $T$ is covered by its peripheral components.
	It follows from \autoref{res: shortening peripheral - main theo} that every vertex of $S/L$ corresponding to a peripheral subtree has valence one.
	In other words the graph of groups decomposition associated to $S$ has the following form
	
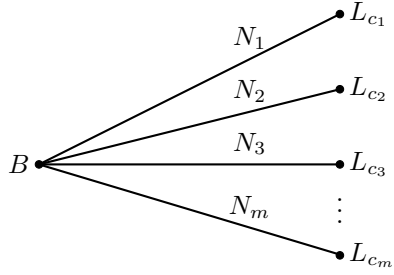
\begin{center}
		\begin{tikzpicture}
			\def \a{4}
			\def \b{1}
			\def \r{0.05}
						
			\draw[fill=black] (0,0) circle (\r) node[left, anchor=east]{$B$};
			\draw[fill=black] (\a,2*\b) circle (\r) node[right, anchor=west]{$L_{c_1}$};
			\draw[fill=black] (\a,\b) circle (\r) node[right, anchor=west]{$L_{c_2}$};
			\draw[fill=black] (\a,0) circle (\r) node[right, anchor=west]{$L_{c_3}$};
			\draw[fill=black] (\a,-0.5\b) node[anchor=center]{$\vdots$};
			\draw[fill=black] (\a,-1.2\b) circle (\r) node[right, anchor=west]{$L_{c_m}$};
			
			\draw[thick] (0,0) -- (\a, 2*\b) node[pos=0.7, above, anchor=south]{$N_1$};
			\draw[thick] (0,0) -- (\a, \b) node[pos=0.7, above, anchor=south]{$N_2$};
			\draw[thick] (0,0) -- (\a, 0) node[pos=0.7, above, anchor=south]{$N_3$};
			\draw[thick] (0,0) -- (\a, -1.2*\b) node[pos=0.7, above, anchor=south]{$N_m$};
		\end{tikzpicture}
	\end{center}
	where $B$ is the stabilizer of a transverse component, $c_1, \dots, c_m \in \mathcal C$ is a collection of $L$-representatives of $\mathcal C$, and $N_i$ is the pointwise stabilizer of $\bar P(c_i)$ for every $i \in \intvald 1m$.
	According to \autoref{res: shortening peripheral - main theo}, $N_i$ has finite index in $L_{c_i}$.
	This means that $S$ is a generalized root splitting, which contradicts our assumption.
\end{proof}


%
\section{Lifting theorems}
%
\label{sec: lifting theorems}

In this section we study when a morphism $\phi \colon G \to \Gamma / \Gamma^n$ can be lifted to a morphism $\tilde \phi \colon G \to  \Gamma$.
As we noticed in the introduction, this property heavily depends on the group $G$.
We prove three lifting statements that can be used in different contexts.
The first (and easiest) one applies when no quotient of $G$ has an abelian splitting (\autoref{res: lifting morphism - rigid}). 
It serves mostly as a warm up for the next results.
The second proposition deals with groups $G$ without essential root-splitting (\autoref{res: lifting morphism - killing element for hyperbolic groups}) and the last one pushes, when it is possible, all the standard arguments on limit groups (\autoref{res: lifting quotient wo roots -  inf presented}).

%
\paragraph{Optimal approximation.}
%
Our first task is to define a suitable notion of short morphisms for the study of limit groups over the periodic quotients of some hyperbolic group.

From now on the parameter $\delta \in \R_+^*$ and the map $\rho \colon \N \to \R_+^*$ are fixed as in \autoref{res: approximating sequence}.
Let $\tau \in (0,1)$.
We write $N_\tau$ for the critical exponent provided by \autoref{res: approximating sequence}.
Let $n \geq N_\tau$ be an odd integer.
Recall that a $\tau$-bootstrap for the exponent $n$ -- introduced in \autoref{def: bootstrap} -- is a pair $(\Gamma, X)$ designed to serve as the initial step to approximate the $n$-periodic quotients of the group $\Gamma$.
We fix such a bootstrap, say $(\Gamma, X)$.
We denote by $(\Gamma_j, X_j, \mathcal C_j)$ the approximating sequence given by \autoref{res: approximating sequence} to study $\Gamma / \Gamma^n$.
By construction all these triples belong to the class $\mathfrak H_\delta(\rho(n))$ provided by \autoref{def: preferred class}.

Let $G$ be a group and $U$ a finite subset of $G$.
Let $\phi \colon G \to \Gamma / \Gamma^n$.
Recall that an approximation of $\phi$ relative to $U$ is a morphism $\tilde\phi \colon G \to \Gamma_j$ lifting $\phi$ such that $U \cap \ker \tilde\phi  = U \cap \ker \phi$ (see \autoref{def: approx of morphism}).
Such an approximation is minimal if $j$ is minimal for the above property.

\begin{defi}
\label{def: short morphism}
	Let $\epsilon \in \R_+^*$.
	Let $G$ be a group generated by a finite set $U$ and $V$ a subset of $G$.
	Let $\mathcal P \subset {\rm Hom}(G,\Gamma/\Gamma^n)$ be a set of morphisms from $G$ to $\Gamma/\Gamma^n$.
	We say that a morphism $\phi \in \mathcal P$ is \emph{$\epsilon$-short among $\mathcal P$ relative to $V$} if it admits a minimal approximation $\phi_i \colon G \to \Gamma_i$ relative to $V$ such that for every $\psi \in \mathcal P$, for every minimal approximation $\psi_j \colon G \to \Gamma_j$ of $\psi$ relative to $V$, 
	\begin{itemize}
		\item either $i<j$,
		\item or $i=j$ and $\lambda_1^+(\phi_i, U) \leq \lambda_1^+(\psi_j,U) + \epsilon$.
	\end{itemize}
	In this context $\phi_i$ is called an \emph{optimal approximation} of $\phi$.
\end{defi}

\begin{rema}
	For our purpose we are interested in limit groups marked by a finitely generated group $G$.
	If $F$ is a finitely generated free group, recall that any projection $\pi \colon F \onto G$ induces an embedding $\mathfrak G(G) \to \mathfrak G(F)$.
	We often prefer to work in $\mathfrak G(F)$ rather than in $\mathfrak G(G)$.
	The reason is that $G$ may not be finitely presented.
	In such a situation finding approximations of a morphism $G \to \Gamma / \Gamma^n$ is not always automatic, unlike as in \autoref{res: approx - energy minimal approx}.
\end{rema}

%
\subsection{Rigid groups}
%

\begin{prop}
\label{res: lifting morphism - rigid gal}
	Let $F$ be a finitely generated free group and $U$ be a free basis of $F$.
	Let $\pi \colon F \onto G$ be an epimorphism where $G$ is a group none of whose quotients admit an abelian splitting.
	Let $\tau \in (0,1)$.
	There exist $\lambda \in \R_+$ and a finite subset $W \subset \ker \pi$, such that for every finite subset $V \subset F$, there is a critical exponent $N \in \N$, with the following property.
	
	Let $n \geq N$ be an odd integer.
	Let $(\Gamma, X)$ be a $\tau$-bootstrap for the exponent $n$.
	For every morphism $\phi \colon F \to \Gamma/ \Gamma^n$ whose image is not abelian, if $W \subset \ker \phi$, then there is a morphism $\tilde \phi \colon F \to \Gamma$ lifting $\phi$ relative to $V$ and such that $\lambda_\infty(\tilde \phi, U) \leq \lambda$.
\end{prop}

\begin{proof}
	We fix a non-decreasing exhaustion $(W_k)$ of $\ker \pi$ by finite subsets.
	Assume on the contrary that the statement is false.
	For every $k \in \N$, there exists a finite subset $V_k \subset F$ such that for every $N \in \N$, there are an odd integer $n \geq N$, a $\tau$-bootstrap $(\Gamma, X)$ for the exponent $n$, and a morphism $\phi \colon F \to \Gamma / \Gamma^n$ whose image is not abelian, such that $W_k \subset \ker \phi$, but every lift $\tilde \phi \colon F \to \Gamma$ of $\phi$ relative to $V_k$ satisfies $\lambda_\infty(\tilde \phi, U) \geq k$. 
	
	We write $\ell_k$ for the length of the longest element in $U_k = V_k \cup W_k$ (seen as a word over $U$).
	The map $\rho \colon \N \to \R_+$ diverges to infinity, hence there exists $M_k \in \N$ such that for every integer $n \geq M_k$ we have
	\begin{equation*}
		\frac {\rho(n)}{100\ell_k} \geq k.
	\end{equation*}
	In addition, we write $N_\tau$ for the critical exponent given by \autoref{res: approximating sequence}.

	By definition of $(V_k)$,  for every $k \in \N$, there exist $n_k \geq \max \{N_\tau, M_k\}$, a $\tau$-bootstrap $(\Gamma_k, X_k)$ for the exponent $n_k$ and a morphism $\phi_k \colon F \to \Gamma_k / \Gamma_k^{n_k}$ whose image is not abelian, such that $W_k \subset \ker \phi_k$, but every lift  $\tilde \phi_k \colon F \to \Gamma_k$ of $\phi_k$ relative to $V_k$ satisfies $\lambda_\infty(\tilde \phi_k, U) \geq k$. 
	For each $k \in \N$, \autoref{res: approximating sequence}, provides an approximating sequence 
	\begin{equation*}
		(\Gamma_{j,k}, X_{j,k}, \mathcal C_{j,k})_{j \in \N}
	\end{equation*}
	of $\Gamma_k/\Gamma_k^{n_k}$.
	We denote by $\psi_k \colon F \to \Gamma_{j_k, k}$ a minimal approximation of $\phi_k$ relative to $U_k$  (see \autoref{def: approx of morphism}).
	Recall that $V_k$ is contained in $U_k$.
	If $j_k = 0$, then $\psi_k \colon F \to \Gamma$ is a fortiori a lift of $\phi_k$ relative to $V_k$.
	Hence  $\lambda_\infty(\psi_k,U) \geq k$.
	If $j_k \geq 1$, combining \autoref{res: approx - energy minimal approx} with our choice of $M_k$, we get
	\begin{equation*}
		\lambda_\infty(\psi_k, U) \geq \frac {\rho(n_k)}{100\ell_k} \geq k
	\end{equation*}
	Consequently, up to passing to a subsequence, $(\Gamma_{j_k, k}, \psi_k)$ converges to a non-abelian, divergent limit group over $\mathfrak H_\delta$, say $(L, \eta)$ .
	By construction for every $k \in \N$, 
	\begin{equation*}
		W_k \cap \ker \psi_k = W_k \cap \ker \phi_k = W_k.
	\end{equation*}
	Consequently $L$ is necessarily a quotient of $G$.
	Proceeding as in \autoref{sec: action on limit tree} we build an action of $L$ on a limit $\R$-tree $T$.
	This action decomposes as a graph of actions (\autoref{res: final decomposition tree}).
	Since $L$ is non-abelian, it necessarily has a non-trivial abelian splitting, which contradicts our assumption.
\end{proof}

\begin{nota*}
	In the next statement, $\iota_\gamma$ stands for the conjugation by $\gamma \in \Gamma$.
\end{nota*}

\begin{coro}
\label{res: lifting morphism - rigid}
	Let $G$ be a finitely generated group none of whose quotient admits an abelian splitting.
	Let $\Gamma$ be a non-elementary, torsion-free, hyperbolic group.	
	There exist $N \in \N$ and finite collection $\{\psi_1, \dots, \psi_p\}$ of morphisms from $G$ to $\Gamma$, such that for every odd integer $n \geq N$, the following holds.
	
	For every morphism $\phi \colon G \to \Gamma/ \Gamma^n$ whose image is not abelian, there exist $i \in \intvald 1p$ and $\gamma \in \Gamma$, such that $\iota_\gamma \circ \psi_i$ lifts $\phi$.
\end{coro}

\begin{proof}
	Let $X$ be Cayley graph of $\Gamma$.
	It is a $\tau$-bootstrap from some $\tau \in (0,1)$ (and for every exponent $n \in \N$).
	Let $F$ be the free group generated by a finite generating set $U$ of $G$ and $\pi \colon F \onto G$ the corresponding epimorphism.
	Being torsion-free hyperbolic, $\Gamma$ is equationally Noetherian \cite[Theorem~1.22]{Sela:2009bh}.
	Consequently there exists a finite subset $V \subset \ker \pi$ with the following property:
	for every morphism $\phi \colon F \to \Gamma$, if $V \subset\ker \phi$, then $\phi$ factors through $\pi$.
	Let $\lambda \in \R_+$, $W \subset \ker \pi$ and $N \in \N$ be the parameters given by \autoref{res: lifting morphism - rigid gal}.
	Since the action of $\Gamma$ on $X$ is proper and co-compact, there are finitely many morphisms $\psi_1, \dots, \psi_p$ from $G$ to $\Gamma$ such that every morphism $\psi \colon G \to \Gamma$ satisfying $\lambda_\infty(\psi, U) \leq \lambda$ is conjugated to one of the $\psi_j$.
	
	Consider now an odd integer $n \geq N$ and a morphism $\phi \colon G \to \Gamma / \Gamma^n$ whose image is not abelian.
	By construction $W$ is contained in $\ker \phi \circ \pi$.
	According to \autoref{res: lifting morphism - rigid gal}, there exists a morphism $\tilde \mu \colon F \to \Gamma$ lifting $\mu = \phi \circ \pi$ relative to $V$ such that $\lambda_\infty(\tilde \mu, U) \leq \lambda$.
	Recall that $V$ is contained in $\ker \pi$, hence $V$ is also contained in $\ker \tilde \mu$.
	Consequently $\tilde \mu$ factors through $\pi$.
	The resulting morphism $\tilde \phi \colon G \to \Gamma$ is a lift of $\phi$ such that  $\lambda_\infty(\tilde \phi, U) \leq \lambda$.
	In particular, $\tilde \phi$ is conjugated to one of the $\psi_i$.
\end{proof}


%
\subsection{(Almost) injective morphisms}
%

\begin{prop}
\label{res: lifting morphism - killing elements}
	Let $F$ be a finitely generated free group and $U$ be a free basis of $F$.
	Let $\pi \colon F \onto G$ be an epimorphism where $G$ is a  freely indecomposable group with no essential root splitting.
	Let $\tau \in (0,1)$.
	There exist a finite subset $W \subset G \setminus\{1\}$ and  $\lambda \in \R_+$, such that for every finite subset $V \subset F$, there is a critical exponent $N \in \N$ with the following properties.
	
	Let $n \geq N$ be an odd integer.
	Let $(\Gamma, X)$ be a $\tau$-bootstrap for the exponent $n$.
	For every morphism $\phi \colon G \to \Gamma / \Gamma^n$ whose image is not abelian, there exists $\alpha \in \aut G$ such that one of the following holds:
	\begin{itemize}
		\item $W \cap \ker(\phi \circ \alpha) \neq \emptyset$;
		\item there is a morphism $\psi \colon F \to \Gamma$ lifting $\phi \circ \alpha \circ \pi$ relative to $V$ such that $\lambda_\infty(\psi, U) \leq \lambda$.
	\end{itemize}
\end{prop}

\begin{proof}
	Let $\tau \in (0,1)$.
	We fix a non-decreasing exhaustion $(W_k)$ of $G\setminus\{1\}$ by finite subsets.
	We assume that the statement is false.
	In particular, the following fact holds
	\begin{fact}
	\label{res: lifting morphism - killing elements - negation}
		For every $k \in \N$, there is a finite subset $V_k \subset F$ such that for every $N \in \N$, there exist an odd integer $n \geq N$, a $\tau$-bootstrap $(\Gamma, X)$ for the exponent $n$, and a morphism $\phi \colon G \to \Gamma / \Gamma^n$ whose image is not abelian such that for every $\alpha \in \aut G$, the following holds.
		\begin{enumerate}
			\item \label{res: lifting morphism - stably one-to-one - kernel}
			$W_k \cap \ker(\phi \circ \alpha) = \emptyset$
			\item \label{res: lifting morphism - stably one-to-one - lift}
			if $\phi \circ \alpha \circ \pi$ admits a lift $\psi \colon F \to \Gamma_k$ relative to $V_k$, then $\lambda_\infty(\psi, U) \geq k$.
		\end{enumerate}
	\end{fact}

	We fix a non-decreasing exhaustion $(U_k)$ of $F$ by finite subsets such that $V_k \subset U_k$ for every $k \in \N$.
	We write $\ell_k$ for the length of the longest element in $U_k$ (seen as a word over $U$).
	Since the map  $\rho \colon \N \to \R_+$ diverges to infinity, for every $k \in \N$, there exists $M_k$ such that for every $n \geq M_k$, 
	\begin{equation*}
		\frac {\rho(n)}{100\ell_k} \geq k.
	\end{equation*}

	Let $N_\tau$ be the critical exponent given by \autoref{res: approximating sequence}.
	For every $k \in \N$, there is an odd integer $n_k \geq \max \{N_\tau, M_k\}$, a $\tau$-bootstrap $(\Gamma_k, X_k)$ for the exponent $n_k$ and a morphism $\phi_k \colon G \to \Gamma_k / \Gamma_k^{n_k}$ as in \autoref{res: lifting morphism - killing elements - negation}.
	For each $k \in \N$, \autoref{res: approximating sequence}, provides an approximating sequence 
	\begin{equation*}
		(\Gamma_{j,k}, X_{j,k}, \mathcal C_{j,k})_{j \in \N}
	\end{equation*}
	of $\Gamma_k/\Gamma_k^{n_k}$.
	Let $\epsilon \in \R_+^*$.
	Up to pre-composing $\phi_k$ with an element of $\aut G$, we can assume that $\phi_k \circ \pi$ is $\epsilon$-short among $\mathcal P_k = \set{\phi_k \circ \alpha \circ \pi}{\alpha \in \aut G}$ relative to $U_k$. 
	We denote by $\psi_k \colon F \to \Gamma_{j_k, k}$ an optimal approximation of $\phi_k \circ \pi$ relative to $U_k$.	
	
	By construction $\ker \psi_k$ is contained in $\ker(\phi_k \circ \pi)$.
	On the one hand, by \ref{res: lifting morphism - stably one-to-one - kernel} the stable kernel of $(\phi_k \circ \pi)$ -- see (\ref{eqn: stable kernel}) for the definition -- is exactly $\ker \pi$.
	On the other hand, $U_k \cap \ker \pi$ is contained in $U_k \cap \ker (\phi_k \circ \pi)$ and therefore in $\ker \psi_k$ for every $k \in \N$.
	If follows that $(\Gamma_{j_k,k}, \psi_k)$ converges to $(G, \pi)$ in the space of marked groups $\mathfrak G(F)$.
	We next claim that for every $k \in \N$,
	\begin{equation*}
		\lambda_\infty(\psi_k, U) \geq k
	\end{equation*}
	Suppose first that $j_k = 0$.
	Since $V_k$ is contained in $U_k$, the morphism $\psi_k$ is a fortiori a lift of $\phi_k \circ \pi$ relative to $V_k$.
	Thus the conclusion follows from \ref{res: lifting morphism - stably one-to-one - lift}.
	Suppose now that $j_k \geq 1$.
	Observe that the image of $\psi_k$ is not abelian. 
	Indeed otherwise so would be the one of $\phi_k$.
	Recall that $\ell_k$ is the length of the longest element of $U_k$, while $\psi_k$ is a minimal lift of $\phi_k \circ \pi$ relative to $U_k$.
	By \autoref{res: approx - energy minimal approx}, we have $\lambda_\infty(\psi_k, U) \geq \rho(n_k) / 100\ell_k$ and the conclusion follows from our choice of $M_k$.
	
	It follows from our claim that $\lambda_\infty(\psi_k, U)$ diverges to infinity, that is $(G, \pi)$ is a divergent limit group.
	Proceeding as in \autoref{sec: action on limit tree}, we build an action of $G$ on a limit $\R$-tree $T$.
	This action decomposes as a graph of action, where each vertex action is either peripheral, simplicial, axial or has Seifert type (\autoref{res: final decomposition tree}).
	We denote by $S$ the complete splitting of $G$ with respect to $T$, see \autoref{def: complete graph of action}.
	We write $(\hat G_k, \hat \pi_k, \hat S_k)$ for a sequence of finitely presented strong covers of $(G, \pi, S)$ that converges to $(G, \pi, S)$, see \autoref{def: conv strong cover}.
	For every $k \in \N$, we denote by $\zeta_k \colon \hat G_k \onto G$ the associated projection.
	Up to replacing the original sequence $(\phi_k)$ by a subsequence, we can always assume that $\ker \hat \pi_k$ is generated (as a normal subgroup) by $U_k \cap \ker \hat \pi_k$.
	By construction, 
	\begin{equation*}
		U_k \cap \ker \hat \pi_k \subset U_k \cap \ker \pi \subset U_k \cap \ker (\phi_k \circ \pi) \subset U_k \cap \ker \psi_k,
	\end{equation*}
	thus $\ker \hat \pi_k$ lies in the kernel of $\psi_k$.
	Hence $\psi_k \colon F \to \Gamma_{j_k, k}$ factors through $\hat \pi_k \colon F \onto \hat G_k$.
	We write $\hat \psi_k \colon \hat G_k \to \Gamma_{j_k, k}$ for the resulting morphism.
	
	According to our assumption $G$ is freely indecomposable and does not admit any root splitting.
	Hence we can perform the shortening argument (\autoref{sec: shortening}) with $(\hat G_k, \hat \pi_k, \hat S_k)$ as the sequence of strong covers.
	By \autoref{res: shortening argument}, there exists $\kappa > 0$, as well as automorphisms $\alpha_k$ in $\mcg {G,S}$ and $\hat \alpha_k \in \aut{\hat G_k}$ such that for infinitely many $k \in \N$, we have $\alpha_k \circ \zeta_k = \zeta_k \circ \hat \alpha_k$ and 
	\begin{equation*}
		\lambda_1^+(\hat \psi_k \circ \hat \alpha_k \circ \hat \pi_k, U) < (1-\kappa)\lambda_1^+(\psi_k, U) 
	\end{equation*}
	Observe that $\hat \psi_k \circ \hat \alpha_k \circ \hat \pi_k$ is an approximation of $\phi_k \circ \alpha_k \circ \pi$.
	Since $\phi_k \circ \pi$ is $\epsilon$-short, this approximation is necessarily minimal.
	Moreover it satisfies
	\begin{equation*}
		\lambda_1^+(\psi_k, U) \leq \lambda_1^+(\hat \psi_k \circ \hat \alpha_k \circ \hat \pi_k, U) + \epsilon \leq (1-\kappa)\lambda_1^+(\psi_k,U)  + \epsilon.
	\end{equation*}
	Recall that $\lambda_\infty(\psi_k, U)$ diverges to infinity.
	This leads to a contradiction if $k$ is sufficiently large, whence the result.
\end{proof}

\begin{coro}
\label{res: lifting morphism - killing element for hyperbolic groups}
	Let $G$ be a finitely generated, freely indecomposable group with no essential root splitting.
	Let $\Gamma$ be a non-elementary, torsion-free, hyperbolic group.
	There exist a finite subset $W \subset G \setminus \{1\}$ and a critical exponent $N \in \N$, with the following properties.
	For every odd integer $n \geq N$, for every morphism $\phi \colon G \to \Gamma / \Gamma^n$ whose image is not abelian, one of the following holds:
	\begin{itemize}
		\item there is an automorphism $\alpha \in \aut G$ such that $W \cap \ker(\phi \circ \alpha) \neq \emptyset$,
		\item there is a monomorphism $\tilde \phi \colon G \to \Gamma$ lifting $\phi$.
	\end{itemize}
\end{coro}

\begin{proof}
	Let us first introduce all the necessary objects.
	The Cayley graph $X$ of $\Gamma$ is a $\tau$-bootstrap for some $\tau \in (0,1)$ (and for every exponent $n \in \N$).
	Let $U$ be a finite generating set of $G$ and $F$ the free group generated by $U$.
	Let $\pi \colon F \onto G$ be the corresponding epimorphism.
	Recall that torsion-free hyperbolic groups are equationally noetherian.
	Hence there exists a finite subset $V \subset \ker \pi$ such that for every morphism $\phi \colon F \to \Gamma$, if $V$ is contained in $\ker \phi$, then $\phi$ factors through $\pi$.
	We write $W \subset G \setminus\{1\}$, $\lambda \in \R_+$, and $N \in \N$ for the objects provided by \autoref{res: lifting morphism - killing elements}.
	In this proof the group $\Gamma$ and the space $X$ are fixed.
	Hence, up to conjugacy,  there are only finitely morphisms $\psi \colon G \to \Gamma$ such that $\lambda_\infty(\psi, U) \leq \lambda$.
	Up to replacing $W$ by a larger subset of $G \setminus\{1\}$, we can assume that for every non-injective such morphism $\psi \colon G \to \Gamma$, we have $W \cap \ker \psi \neq \emptyset$.
	
	Let $n \geq N$ be an odd integer.
	Let $\phi \colon G \to \Gamma / \Gamma^n$ be a morphism.
	Assume that $W \cap \ker (\phi \circ \alpha)$ is empty for every $\alpha \in \aut G$.
	We are going to prove that $\phi$ admits an injective lift.
	According to \autoref{res: lifting morphism - killing elements}, there exist $\alpha \in \aut G$ and a morphism $\psi \colon F \to \Gamma$ lifting $\phi \circ \alpha \circ \pi$ relative to $V$ and such that $\lambda_\infty(\psi, U) \leq \lambda$.
	It follows from our choice of $V$ that $\psi$ factors through $\pi$.
	In particular, the resulting morphism $\psi_0 \colon G \to \Gamma$ lifts $\phi \circ \alpha$.
	Hence $\psi_0 \circ \alpha^{-1}$ lifts $\phi$.
	In order to complete our proof it suffices to show that $\psi_0$ is one-to-one.
	By construction $\lambda_\infty(\psi_0, U) \leq \lambda$.
	Note that $W \cap \ker \psi_0 \subset W \cap \ker(\phi \circ \alpha)$.
	According to our assumption, $W \cap \ker \psi_0$ is empty.
	Thanks to the adjustment we made on the set $W$ we can conclude that $\psi_0$ is one-to-one.
\end{proof}

\begin{prop}
\label{res: morphism to free product}
	Let $F$ be a finitely generated free group and $U$ be a free basis of $F$.
	Let $\pi \colon F \onto G$ be an epimorphism onto a freely indecomposable group distinct from $\Z$.
	For every $\lambda \in \R_+$, there is a finite subset $V \subset F$ with the following property.
	Let $A$ and $B$ be two non-trivial groups and $X$ the Bass-Serre tree associated to the free product $\Gamma = A \ast B$.
	Let $\phi \colon F \to \Gamma$ be a morphism. 
	If $\lambda_\infty(\phi, U) \leq \lambda$ (where the energy is measured in $X$) and $V \cap \ker \pi = V \cap \ker \phi$, then the image of $\phi$ is contained in a conjugate of $A$ or $B$.
\end{prop}

\begin{proof}
	Let $\lambda \in \R_+^*$.
	We fix a non-decreasing exhaustion $(V_k)$ of $F$ by finite subsets.
	Assume that the statement is false.
	For every $k \in \N$, there exist two non trivial groups $A_k$ and $B_k$ and a morphism $\phi_k \colon F \to A_k \ast B_k$ such that if $X_k$ stands for the Bass-Serre tree associated to $\Gamma_k = A_k \ast B_k$, then the following holds
	\begin{enumerate}
		\item \label{enu: morphism to free product - energy}
		$\lambda(\phi_k, U) \leq \lambda$, (where the energy is measured in $X_k$);
		\item \label{enu: morphism to free product - cvg}
		$V_k \cap \ker \pi = V_k \cap \ker \phi_k$;
		\item \label{enu: morphism to free product - factor}
		the image of $\phi_k$ is not contained in a conjugate of $A_k$ of $B_k$.
	\end{enumerate}
	Let $(\epsilon_k)$ be a sequence converging to zero.
	For every $k \in \N$, we fix a basepoint $o_k \in X_k$ such that $\lambda_\infty(\phi_k, U, o_k) \leq \lambda_\infty(\phi_k, U) + \epsilon_k$.
	In addition we choose a non-principal ultra-filter $\omega$.
	We denote by $X_\omega$ the $\omega$-limit of the sequence of pointed space $(X_k, o_k)$.
	Note that we do not rescale the space $X_k$ here.
	By construction, $X_k$ is a simplicial tree for every $k \in \N$, hence so it $X_\omega$.
	For every $k \in \N$, the morphism $\phi_k$ defines an action of $F$ on $X_k$.
	Since its energy is uniformly bounded, it induces a simplicial action of $F$ on $X_\omega$.
	We get from \ref{enu: morphism to free product - cvg} that $(\Gamma_k, \phi_k)$ converges to $(G, \pi)$ for the topology of marked groups.
	Consequently $\ker \pi$ acts trivially on $X_\omega$.
	In particular, the previous action induces an action of $G$ on $X_\omega$.
	One checks that for every edge $e$ of $X_\omega$ its stabilizer (in $G$) is trivial.
	However $G$ is freely indecomposable and distinct from $\Z$.
	Hence it necessarily fixes a vertex $v$ of $X_\omega$.
	In particular, every generator of $U$ fixes $v$.
	We write $v$ as the limit $v = \limo v_k$ of a sequence of vertices $v_k \in X_k$.
	Since $U$ is finite, for every $u \in U$, the element $\phi_k(u)$ fixes $v_k$ \oas.
	Hence the image of $\phi_k$ fixes $v_k$ \oas, or in other words the image of $\phi_k$ is contained in a conjugate of $A_k$ or $B_k$.
	This contradicts our assumptions and completes the proof.
\end{proof}

\begin{coro}
\label{res: lifting morphism - killing element free product}
	Let $G$ be a finitely generated, freely indecomposable group with no essential root splitting.
	There exist a finite subset $W \subset G \setminus \{1\}$ and a critical exponent $N \in \N$, with the following properties.
	
	Let $n \geq N$ be an odd integer.
	Let $A,B \in \mathfrak B_n$ be two non-trivial CSA groups.
	For every morphism $\phi \colon G \to A \ast^n B$ whose image is not abelian, one of the following holds:
	\begin{itemize}
		\item there exists an automorphism $\alpha \in \aut G$ such that $W \cap \ker(\phi \circ \alpha) \neq \emptyset$;
		\item the image of $\phi$ is contained in a conjugate of $A$ or $B$.
	\end{itemize}\end{coro}

\begin{proof}
	Let $U$ be a finite generating set of $G$ and $F$ the free group generated by $U$.
	Let $\pi \colon F \onto G$ be the corresponding epimorphism.
	We denote by $W \subset G \setminus\{1\}$ and $\lambda \in \R_+^*$ the first two parameters given by \autoref{res: lifting morphism - killing elements} to lift morphisms from $G$ to periodic groups.
	Now that $\lambda$ is fixed, \autoref{res: morphism to free product} gives us a finite subset $V \subset F$ to analyze morphisms from $G$ to a free product.
	According to \autoref{res: lifting morphism - killing elements}, there exists $N \in \N$, such that for every odd integer $n \geq N$, the following holds.
	Let $(\Gamma, X)$ be a $1$-bootstrap for the exponent $n$.
	For every morphism $\phi \colon G \to \Gamma / \Gamma^n$ whose image is not abelian, there exists $\alpha \in \aut G$ such that one of the following holds:
	\begin{itemize}
		\item $W \cap \ker(\phi \circ \alpha) \neq \emptyset$;
		\item there is a morphism $\psi \colon F \to \Gamma$ lifting $\phi \circ \alpha \circ \pi$ relative to $V$ such that $\lambda_\infty(\psi, U) \leq \lambda$.
	\end{itemize}
	Note that we can always replace $W$ by a larger subset so that it also contains $\pi (V) \setminus \{1\}$.

	Let $n \geq N$ be an odd integer.
	Let $A,B \in \mathfrak B_n$ be two non-trivial CSA groups.
	Let $\phi \colon G \to A \ast ^n B$ be a morphism whose image is not abelian.
	Assume that $W \cap \ker (\phi \circ \alpha)$ is empty for every $\alpha \in \aut G$.
	We denote by $\Gamma$ the (regular) free product $\Gamma = A \ast B$ and let it act on its Bass-Serre tree $X$.
	Hence $(\Gamma, X)$ is a $1$-bootstrap for the exponent $n$.
	According to our choice of $N$, there exists $\alpha \in \aut G$ and a morphism $\psi \colon F \to \Gamma$ lifting $\phi \circ \alpha \circ \pi$ relative to $V$ such that $\lambda_\infty(\psi, U) \leq \lambda$.
	We claim that $V \cap \ker \psi = V \cap \ker \pi$.
	Since $\psi$ is a lift of $\phi \circ \alpha \circ \pi$ relative to $V$, it suffices to show that $V \cap \ker (\phi \circ \alpha \circ \pi)  = V \cap \ker \pi$.
	Note that $\ker \pi$ is contained in $\ker (\phi \circ \alpha \circ \pi)$.
	By construction $\pi(V) \setminus \{1\}$ is contained in $W$ while $W \cap \ker(\phi \circ \alpha)  = \emptyset$.
	Hence $V \cap \ker (\phi \circ \alpha \circ \pi)  \subset V \cap \ker \pi$, which completes the proof of our claim.
	
	Combining this claim with \autoref{res: morphism to free product}, we obtain that the image of $\psi$ is contained in a conjugate of $A$ or $B$ (seen as a subgroup of $A \ast B$).
	Recall that $\psi$ is a lift of $\phi \circ \alpha \circ \pi$.
	Consequently the image of $\phi$ is contained in a conjugate of $A$ or $B$ (seen this time as a subgroup of $A \ast^n B$).
\end{proof}

%
\subsection{Shortening limit groups with no root splitting}
%
\label{sec: limit group w/o root splitting aka shortening}
In this section, we fix once and for all a non-elementary hyperbolic group $\Gamma$ acting properly co-compactly on a hyperbolic length space $X$, so that $(\Gamma, X)$ is a $\tau$-bootstrap for some $\tau \in (0,1)$.
The goal is to prove the existence of shortening quotients that will allow us to apply the material of \autoref{sec: infinite desc  seq}.

Recall that the definition of a shortening limit group (\autoref{def: shortenable}) relies on a group property $\mathcal P$.
The one we use here is the following: a group $H$ has $\mathcal P$ if all its abelian subgroups are finitely generated.
This property being fixed for the remainder of this section, we make an abuse of vocabulary and say that a sequence / limit group is \emph{shortenable} if it is shortenable with respect to this property.

\begin{prop}
\label{res: shortening group w/o root - freely indecomposable}
	Let $G$ be a finitely presented group.
	Let $(L, \eta) \in \mathfrak G(G)$.
	If $L$ is freely indecomposable with no essential root splitting, then $(L, \eta)$ is shortenable.
\end{prop}

\begin{proof}	
	We fix once and for all a finite generating set $U$ of $G$.
	Note that if $L$ is abelian, then the conclusion directly follows from \autoref{rem: fp group are shortenable}.
	Therefore, from now on we assume that $L$ is non-abelian.
	Consider a sequence $(n_k)$ of odd integers diverging to infinity and a sequence of morphisms $\phi_k \colon G \to \Gamma / \Gamma^{n_k}$ such that $(\Gamma / \Gamma^{n_k}, \phi_k)$ converges to $(L, \eta)$.
	
	We denote by $S_{\rm JSJ}$ the tree of cylinders of a JSJ tree of $L$ (over the abelian subgroups of $L$).
	In addition we fix a sequence $(\hat L_k, \hat \eta_k, \hat S_k)$ of finitely presented strong covers of $(L, \eta, S_{\rm JSJ})$ converging to $(L, \eta, S_{\rm JSJ})$.
	We denote by $\zeta_k \colon \hat L_k \onto L_k$ the projection satisfying $\zeta_k \circ \hat \eta_k = \eta$.
	By construction there is a finite subset $W_k \subset G$ which generates $\ker \eta_k$ as a normal subgroup of $G$.
	We denote by $\ell_k$ the length of the longest element in $W_k$, seen as a word over $U$.
	Up to replacing $(\phi_k)$ by a subsequence, we can assume that 
	\begin{equation*}
		\lim_{k \to \infty} \frac {\rho(n_k)}{\ell_k} = \infty,
	\end{equation*}
	while every $\phi_k$ factors through $\hat \eta_k$.	
	We write $\hat \phi_k \colon \hat L_k \to \Gamma / \Gamma^{n_k}$ for the resulting morphism.
	Let $\mathcal M_k$ be the set of all automorphisms $\hat \alpha \in \aut{\hat L_k}$ lifting a modular automorphism $\alpha \in \mcg L$, i.e. such that $\zeta_k \circ \hat \alpha = \alpha \circ \zeta_k$.
	In addition we let 
	\begin{equation*}
		\mathcal P_k = \set{ \hat \phi_k \circ \hat \alpha \circ \hat \eta_k}{ \hat \alpha \in \mathcal M_k}.
	\end{equation*}
	For each $k \in \N$, we fix an approximation sequence 
	\begin{equation*}
		(\Gamma_{j,k}, X_{j,k}, \mathcal C_{j,k})_{j \in \N}
	\end{equation*}
	of $\Gamma / \Gamma^{n_k}$, as given by \autoref{res: approximating sequence}.
	Let $\epsilon \in \R_+^*$.
	For every $k \in \N$, we chose $\hat \alpha_k \in \mathcal M_k$ such that the morphism $\phi'_k  =  \hat \phi_k \circ \hat \alpha_k \circ \hat \eta_k$ is $\epsilon$-short among $\mathcal P_k$ relative to $W_k$, in the sense of \autoref{def: short morphism}.
	Up to passing to a subsequence $(\Gamma/ \Gamma^{n_k}, \phi'_k)$ converges to a limit group that we denote by $(L', \eta')$.
	We are going to prove that $(\Gamma/ \Gamma^{n_k}, \phi'_k)$ is a shortening of $(\Gamma / \Gamma^{n_k}, \phi_k)$.
	
	Recall that $\hat \eta_k$ is onto while $\hat \alpha_k$ is an automorphism of $\hat L_k$.
	Consequently $\phi_k$ and $\phi'_k$ have the same image of every $k \in \N$, proving Point~\ref{enu: shortening - sbgp} of \autoref{def: shortenable}.
	By construction $(L', \eta')$ is a quotient of $(L, \eta)$, which corresponds to Point~\ref{enu: shortening - cvg} of \autoref{def: shortenable}.
	We now focus on Point~\ref{enu: shortening - dichotomy} of this definition.
	Assume that $(L', \eta') = (L, \eta)$.
	
	\begin{clai}
		There exists $\lambda \in \R_+$ such that for all but finitely many $k \in \N$, there is a morphism $\tilde \phi'_k \colon G \to  \Gamma$ lifting $\phi'_k$ relative to $W_k$, such that $\lambda_\infty(\phi'_k, U) \leq \lambda$.
	\end{clai}

	\begin{proof}[Proof of the claim]
		For every $k \in \N$, we denote by $\psi_k \colon G \to \Gamma_{j_k, k}$ an optimal approximation of $\phi'_k$ relative to $W_k$ (see \autoref{def: short morphism}).
		By construction 
		\begin{equation*}
			W_k \subset \ker \hat \eta_k \subset \ker \phi'_k
		\end{equation*}
		Hence $W_k$ is also contained in $\ker \psi_k$.
		We chose $W_k$ in such a way that it generates $\ker \hat \eta_k$ as a normal subgroup.
		Consequently $\psi_k$ factors through $\hat \eta_k$.
		The resulting morphism $\hat \psi_k \colon \hat L_k \to \Gamma_{j_k, k}$ is an approximation of $\hat \phi_k \circ \hat \alpha_k$.
		It follows in particular that $(\Gamma_{j_k,k}, \psi_k)$ converges to $(L', \eta') = (L, \eta)$.
		
		Assume now that our claim fails, that is up to passing to a subsequence, any lift $\tilde \phi'_k \colon G \to \Gamma$ of $\phi'_k$, if it exists, satisfies $\lambda_\infty(\tilde \phi'_k, U) \geq k$.
		We are going to prove that for every $k \in \N$,
		\begin{equation*}
			\lambda_\infty(\psi_k,U) \geq \min\left\{ \frac {\rho(n_k)}{100\ell_k}, k\right\}
		\end{equation*}
		Let $k \in \N$.
		If $j_k \geq 1$, then the result follows from \autoref{res: approx - energy minimal approx}.
		If on the contrary $j_k = 0$, that is $\Gamma_{j_k, k} = \Gamma$, then $\psi_k \colon G \to \Gamma$ is a lift of $\phi'_k$ relative to $W_k$, and the statement follows from our assumption.
		It follows that $\lambda_\infty(\psi_k,U)$ diverges to infinity, that is $(L', \eta')$ is a divergent limit group.
		
		Proceeding as in \autoref{sec: action on limit tree}, we build an action of $L'$ on a limit $\R$-tree $T$.
		This action decomposes as a graph of actions, where each vertex action is either peripheral, simplicial, axial or has Seifert type (\autoref{res: final decomposition tree}).
		We denote by $S'$ the complete splitting of $L$ with respect to $T$ (see \autoref{def: complete graph of action}).
		
		Recall that $L$ is freely indecomposable, without even torsion.
		According to \autoref{res: compatible JSJ}, the trees $S_{\rm JSJ}$ and $S'$ are compatible.
		We write $R$ for a common refinement.
		By definition of the JSJ decomposition, $R$ is obtained from $S_{\rm JSJ}$ by splitting some of its flexible vertices, which are either surface groups or abelian groups.	
		Let $k \in \N$.
		It follows from the definition of strong cover that the $\zeta_k$-equivariant map $\hat S_k \onto S_{\rm JSJ}$ induces a bijection from $\hat S_k / \hat L_k$ to $S_{\rm JSJ}/ L$.
		Let $\hat v$ be a vertex of $\hat S_k$ and $v$ its image in $S_{\rm JSJ}$.
		If $L_v$ is a surface group (\resp abelian group) then the projection $\zeta_k$ induces an isomorphism from $\hat L_{k, \hat v}$ onto $L_v$ (\resp an embedding of  $\hat L_{k, \hat v}$ onto $L_v$).
		Splitting some vertices of $\hat S_k$, as we did with $S_{\rm JSJ}$ to obtain $R$, we build a refinement $\hat R_k$ of $\hat S_k$ such that $(\hat L_k, \hat \eta_k, \hat R_k)$ is a strong cover of $(L, \eta, R)$.
		We now collapse in $\hat R_k$ edges corresponding to edges of $R$ which are collapsed in $S'$.
		We obtained is this way a new $\hat L_k$-tree $\hat S'_k$ such that $(\hat L_k, \hat \eta_k, \hat S'_k)$ is a strong cover of $(L, \eta, S')$, provided $k$ is sufficiently large.
		It follows from the construction that the following equivariant diagram commute.
		\begin{center}
			\begin{tikzpicture}
				\matrix (m) [matrix of math nodes, row sep=2em, column sep=2.5em, text height=1.5ex, text depth=0.25ex] 
				{ 
					\hat S'_k & \hat R_k	 & \hat S_k \\
					S' & R	 & S_{\rm JSJ} \\
				}; 
				\draw[>=stealth, ->] (m-1-2) -- (m-1-1);
				\draw[>=stealth, ->] (m-2-2) -- (m-2-1);
				\draw[>=stealth, ->] (m-1-2) -- (m-1-3);
				\draw[>=stealth, ->] (m-2-2) -- (m-2-3);
				
				\draw[>=stealth, ->] (m-1-1) -- (m-2-1);
				\draw[>=stealth, ->] (m-1-2) -- (m-2-2);
				\draw[>=stealth, ->] (m-1-3) -- (m-2-3);
			\end{tikzpicture}
		\end{center}
		One checks that $(\hat L_k, \hat \eta_k, \hat S'_k)$  converges to $(L, \eta, S')$ in the sense of  \autoref{def: conv strong cover}.

		According to our assumption $L' = L$ is freely indecomposable and does not admit any essential root splitting.
		Hence we can perform the shortening argument (\autoref{sec: shortening}) with $(\hat L_k,\hat \eta_k,\hat S'_k)$ as the sequence of strong covers.		
		By \autoref{res: shortening argument}, there exists $\kappa > 0$, as well as automorphisms $\beta_k$ in $\mcg {L, S'}$ and $\hat \beta_k \in \aut{\hat L_k}$ such that for infinitely many $k \in \N$, we have $\beta_k \circ \zeta_k = \zeta_k \circ \hat \beta_k$ and 
		\begin{equation*}
			\lambda_1^+(\hat \psi_k \circ \hat \beta_k \circ \hat \eta_k, U) < (1-\kappa)\lambda_1^+(\psi_k, U) 
		\end{equation*}
		By compatibility $\beta_k$ (and thus $\alpha_k \circ \beta_k$) is also an element of $\mcg L$ (\autoref{res: compatible modular group.}).
		Hence $\hat \alpha_k \circ \hat  \beta_k$ belongs to $\mathcal M_k$.
		Recall that $\psi_k$ is an optimal approximation of $\phi'_k = \hat \phi_k \circ \hat \alpha_k \circ \hat \eta_k$.
		Observe that $\hat \psi_k \circ \hat \beta_k \circ \hat \eta_k$ is an approximation of $\hat \phi_k \circ (\hat \alpha_k \circ \hat \beta_k) \circ \hat \eta_k$ relative to $W_k$.
		Since $\phi'_k$ is $\epsilon$-short among $\mathcal P_k$ relative to $W_k$, this approximation is necessarily minimal.
		Moreover it satisfies
		\begin{equation*}
			\lambda_1^+(\psi_k, U) \leq \lambda_1^+(\hat \psi_k \circ \hat \beta_k \circ \hat \eta_k, U) + \epsilon \leq (1-\kappa)\lambda_1^+(\psi_k,U)  + \epsilon.
		\end{equation*}
		Recall that $\lambda_\infty(\psi_k, U)$ diverges to infinity.
		This leads to a contradiction if $k$ is sufficiently large, and completes the proof of our claim.
	\end{proof}

	We observe as above that $W_k \subset \ker \tilde \phi'_k$, for every $k \in \N$, so that $(\Gamma, \tilde \phi'_k)$ converges to $(L', \eta')$ as well.
	Up to conjugacy there are only finitely many morphisms $G \to \Gamma$ whose energy is at most $\lambda$.
	Hence, up to passing to a subsequence, we can assume that $\ker \tilde \phi'_k$ is constant, equal to $K$ say.
	In particular, $L = G/K$ is a subgroup of $\Gamma$, hence every abelian subgroup of $L$ is finitely generated.
	In addition all but finitely many $\tilde \phi'_k$ factor through $\eta'$.
	Thus all but finitely many $\phi'_k$ factor through $\eta'$.
	This completes the proof of Point~\ref{enu: shortening - dichotomy} of \autoref{def: shortenable}.
	
	\medskip
	The proof of Point~\ref{enu: shortening - fact} in \autoref{def: shortenable} follows the standard strategy, see for instance \cite{Sela:2001gb, Sela:2009bh,Weidmann:2019ue}.
	We quickly go through the main arguments.
	Assume that every abelian subgroup of $L'$ is finitely generated.
	Given $k \in \N$, we partition the vertices of $\hat S_k$ as follows: a vertex $\hat v$ belongs to $V_k^0$ (\resp $V_k^1$) if the stabilizer of its image $v$ in $S_{\rm JSJ}$ is a surface group or an abelian group (\resp a rigid group).
	If $\hat v$ belongs to $V_k^1$ then $\hat \alpha_k$ acts on the vertex group $\hat L_{k, \hat v}$ by conjugation.
	It follows that the projection $L \onto L'$ is one-to-one when restricted to every rigid vertex group of $L$.

	We now claim that the edge groups of $S_{\rm JSJ}$ are finitely generated.
	Consider an edge $e$ of $S_{\rm JSJ}$.
	Since $L$ is CSA, it cannot connect two vertices whose stabilizers are both abelian.
	Consequently $L_e$ is either contained in a surface group or a rigid group.
	If $L_e$ is contained in a surface subgroup, then $L_e$ is virtually cyclic.
	If $L_e$ is contained in a rigid vertex group, then according to our previous discussion, $L_e$ embeds in $L'$.
	It follows from our assumption that $L_e$ is finitely generated, which completes the proof of our claim.

	We now show that every abelian subgroup $A$ of $L$ is finitely generated.
	Without loss of generality we can assume that $A$ is not virtually cyclic.
	Since the action of $L$ on $S_{\rm JSJ}$ is acylindrical, $A$ fixes a vertex $v$.
	This vertex cannot be a surface group, otherwise $A$ would be virtually cyclic.
	Assume that the stabilizer of $v$ is abelian.
	Since $L$ is finitely generated, $L_v$ is finitely generated relative to the stabilizers of the edges in the link of $v$.
	We just proved that edge groups are finitely generated.
	Consequently $L_v$ and thus $A$ are finitely generated.
	Suppose now that the stabilizer of $v$ is rigid.
	As previously $A$ embeds in $L'$, and thus is finitely generated.
	
	We proved that every edge group and every abelian vertex group of $S_{\rm JSJ}$ is finitely generated.
	Consequently if $k$ is sufficiently large, the strong cover $(\hat L_k, \hat \eta_k, \hat S_k)$ satisfies the additional following properties.
	\begin{itemize}
		\item Let $\hat v$ be a vertex of $V_k^0$ and $v$ its image in $S_{\rm JSJ}$.
		The map $\zeta_k \colon \hat L_k \onto L$ induces an isomorphism from $\hat L_{k, \hat v}$ onto $L_v$.
		\item Let $\hat e$ be a an edge of $\hat S_k$ and $e$ its image in $S_{\rm JSJ}$.
		The map $\zeta_k \colon \hat L_k \onto L$ induces an isomorphism from $\hat L_{k, \hat e}$ onto $L_e$.
	\end{itemize}
	Hence the kernel of $\zeta_k$ is generated (as a normal subgroup) by a subset $U_k$ with the following property: for every $g \in U_k$, there exists a vertex $\hat v$ of $V_k^1$ such that $g$ belongs to $\hat L_{k, \hat v}$.	
	
	Suppose now that all but finitely many $\phi'_k$ factor through $\eta'$.
	Since $(L', \eta') \prec (L, \eta)$ those morphisms also factor through $\eta$.	
	Consequently if $k$ is sufficiently large $U_k$ is contained in $\ker \hat \phi_k \circ \hat \alpha_k$.
	As we observed before, for every $\hat v \in V_k^1$, the automorphism $\hat \alpha_k$ acts on $\hat L_{k, \hat v}$ by conjugation.
	Consequently, $U_k$ is also contained in $\ker \hat \phi_k$.
	However $U_k$ generates the kernel of $\zeta_k$ (as a normal subgroup).
	It follows that $\hat \phi_k$ factors through $\zeta_k$.
	Thus $\phi_k$ factors through $\eta$.
	This complete the proof of Point~\ref{enu: shortening - fact} of \autoref{def: shortenable}.
\end{proof}

The next statement extends the previous result, when $L$ is not necessarily freely indecomposable.

\begin{coro}
\label{res: shortening group w/o root - general}
	Let $G$ be a finitely presented group.
	Let $(L, \eta) \in \mathfrak G(G)$.
	If $L$ has no essential root splitting, then $(L, \eta)$ is shortenable.
\end{coro}

\begin{proof}
	In view of \autoref{res: shortening group w/o root - freely indecomposable}, we can assume that $L$ is freely decomposable.
	Consider the Grushko decomposition $L$,
	\begin{equation*}
		L = L_0 \ast L_1 \ast \dots \ast L_m \ast F
	\end{equation*}
	where $L_i$ is a non-trivial, freely indecomposable subgroup and $F$ a finitely generated free group.
	We fix a finitely presented cover $(\hat L, \hat\eta)$ of $(L, \eta)$ with the following properties
	\begin{itemize}
		\item $\hat L$ splits as $\hat L = \hat L_0 \ast \dots \ast \hat L_m \ast \hat F$.
		\item The morphism $\zeta \colon \hat L \onto L$ maps each $\hat L_i$ onto $L_i$.
		Moreover it induces an isomorphism from $\hat F$ onto $F$.
	\end{itemize}
	For every $i \in \intvald 0m$, we write $\zeta_i \colon \hat L_i \onto L_i$ for the restriction of $\zeta$ to $\hat L_i$.
	Similarly we write $\zeta_F \colon \hat F \onto F$ for the isomorphism induced by $\zeta$.
	Note that no $L_i$ admits an essential root splitting, otherwise so would $L$.
	Hence by \autoref{res: shortening group w/o root - freely indecomposable}, the group $(L_i, \zeta_i) \in \mathfrak G(\hat L_i)$ is shortenable.
	Since $F$ is free, it has the factorization property, hence $(F, \zeta_F)$ is shortenable (see \autoref{rem: fp group are shortenable}).	
	According to \autoref{res: shortenable stable under free product}, $(L, \zeta) \in \mathfrak G(\hat L)$ is shortenable.
	It follows then from \autoref{res: shortenable changing marker} that $(L, \eta) \in \mathfrak G(G)$ is shortenable.
\end{proof}

In this context, \autoref{res: shortenable implies factorization} yields the following statement.

\begin{prop}
\label{res: quotient wo root give factorization}
	Let $G$ be a finitely presented group.
	Let $(L, \eta) \in \mathfrak G(G)$ be a limit group.
	If no quotient of $L$ has an essential root splitting, then $L$ satisfies the factorization property.
\end{prop}

We now explore the consequences of this fact for the lifting problem tackled in the introduction.
We start with the case of morphisms with abelian image.

\begin{lemm}
\label{res: lifting abelian groups}
	Let $G$ be a finitely generated abelian group.
	Let $p$ the least common multiple of the orders of its elements.
	There exists $N \in \N$ such that for every odd exponent $n \geq N$ co-prime with $p$, every morphism $\phi \colon G \to \Gamma / \Gamma^n$ lifts to a morphism $\tilde \phi \colon G \to \Gamma$.
\end{lemm}

\begin{proof}
	There exists $N \in \N$, such that for every odd integer $n \geq N$, every abelian subgroup of $\Gamma / \Gamma^n$ is cyclic, see \autoref{res: approximating sequence}~\ref{enu: approximating sequence - cvg}.
	Consider a morphism $\phi \colon G \to \Gamma / \Gamma^n$ where $n \geq N$ is an odd exponent co-prime with $p$.
	We decompose $G$ as $G = G_0 \oplus G_1$ where $G_0$ is free abelian and $G_1$ finite.
	The order of any element of $G_1$ is co-prime with $n$.
	Hence $\phi$ maps $G_1$ to the trivial group.
	Consequently, it suffices to lift $\phi$ restricted to $G_0$.
	This is always possible since every abelian subgroup of $\Gamma / \Gamma^n$ is cyclic.
\end{proof}

\begin{theo}
\label{res: lifting quotient wo roots}
	Let $F$ be a finitely generated free group and $\pi \colon F \onto G$ an epimorphism where no quotient of $G$ admits an essential root splitting.
	For every finite subset $V \subset F$ there are $N \in \N$ and a finite subset $W \subset \ker \pi$ with the following property.
	Let $n \geq N$ be an odd integer and $\phi \colon F \to \Gamma/ \Gamma^n$ be a morphism whose image is not abelian.
	If $W \subset  \ker \phi$,  then there is a morphism $\tilde \phi \colon F \to \Gamma$ such that $\tilde \phi$ lifts $\phi$ relative to $V$.
\end{theo}

\begin{proof}
	Let $V \subset F$ be a finite subset.
	We fix a non-decreasing exhaustion $(W_k)$ of $\ker \pi$ by finite subsets.
	Assume that the statement is false.
	There is a sequence of morphisms $\phi_k \colon F \to \Gamma / \Gamma^{n_k}$ with the following properties:
	\begin{itemize}
		\item $(n_k)$ is a sequence of odd integers diverging to infinity,
		\item $W_k \subset \ker \phi_k$, for all $k \in \N$,
		\item the image of every $\phi_k$ is non abelian,
		\item no $\phi_k$ can be lifted relative to $V$.
	\end{itemize}
	Up to passing to a subsequence, we can assume that $(\Gamma / \Gamma^{n_k}, \phi_k)$ converges to a limit group $(L, \eta)$.
	Since $W_k \subset \ker \phi_k$, for every $k \in \N$, we observe that $(L, \eta) \prec (G, \pi)$.
	In particular, no quotient of $L$ admits a root splitting.
	Hence $(L, \eta)$ has the factorization property (\autoref{res: quotient wo root give factorization}).
	Up to passing again to a subsequence, every $\phi_k$ factors through $\eta$.
	We write $\mu_k \colon L \to \Gamma / \Gamma^{n_k}$ for the resulting morphism.
	
	Note that the image of $\mu_k$ is non abelian (since the one of $\phi_k$ is non abelian either).
	We claim that if $k$ is sufficiently large, then $\mu_k$ cannot be lifted.
	Assume on the contrary that there is a morphism $\tilde \mu_k \colon L \to \Gamma$ lifting $\mu_k$.
	Then the map $\tilde \phi_k  = \tilde \mu_k \circ \eta$ lifts $\phi_k$ so that $V \cap \ker \tilde \phi_k \subset V \cap \ker \phi_k$.
	In addition if $k$ is sufficiently large, then $V \cap \ker \phi_k = V \cap \ker \eta$.
	Hence $V \cap \ker \phi_k \subset V \cap \ker \tilde \phi_k$.
	In other words, $\tilde \phi_k$ lifts $\phi_k$ relative to $V$.
	It contradicts the definition of $\phi_k$ and completes the proof of our claim.
	Up to passing to a subsequence, we now assume that none of the $\mu_k$ can be lifted.
	
	We are now going to produce an infinite descending sequence of limit groups starting at $(L, \eta)$
	\begin{equation*}
		(L, \eta) = (L_0, \eta_0) \succ (L_1, \eta_1) \succ (L_2, \eta_2) \succ \dots
	\end{equation*}
	and for every $i \in \N$, a sequence of morphisms $\mu_k^i \colon L_i \to \Gamma / \Gamma^{n_k}$ so that, up to passing to a subsequence, for every $k \in \N$, the image of $\mu_k^i$ is non abelian, and $\mu_k^i$ does not lift.
	We begin with $(L_0, \eta_0) = (L, \eta)$ and let $\mu_k^0 = \mu_k$, for every $k \in \N$.
	
	Let $i \in \N$ for which all the above objects have beed already defined.
	If needed, we pass to a subsequence so that for every $k \in \N$, the image of $\mu_k^i$ is non abelian and $\mu_k^i$ does not lift to a morphism $\tilde \mu_k^i \colon L_i \to \Gamma$.
	\paragraph{Case 1.}
	Assume first that $L_i$ is freely indecomposable.
	If follows from \autoref{res: lifting morphism - killing element for hyperbolic groups} that up to passing to a subsequence, there is an element $g_i \in L_i \setminus\{1\}$ and for every $k \in \N$, an automorphism $\alpha_k^i \in \aut {L_i}$ such that $\mu_k^i \circ \alpha_k^i (g_i) = 1$.
	For every $k \in \N$, we let 
	\begin{equation*}
		\phi_k^{i+1} = \mu_k^i \circ \alpha_k^i \circ \eta_i.
	\end{equation*}
	Up to passing to a subsequence we can assume that $(\Gamma / \Gamma^{n_k}, \phi_k^{i+1})$ converges to a limit group $(L_{i+1}, \eta_{i+1}) \prec (L_i, \eta_i)$.
	Note that $L_{i+1}$ is a proper quotient of $L_i$.
	Indeed by construction, the element $g_i$ lies in the kernel of the projection $\zeta_i \colon L_i \onto L_{i+1}$.
	Since $L_{i+1}$ is a quotient of $L$, no quotient of $L_{i+1}$ admits an essential root splitting.
	Consequently $(L_{i+1}, \eta_{i+1})$ has the factorization property (\autoref{res: quotient wo root give factorization}).
	Up to passing to a subsequence, every $\phi_k^{i+1}$ factors through $\eta_{i+1}$.
	By construction the resulting morphism $\mu_k^{i+1} \colon L_{i+1} \to \Gamma / \Gamma^{n_k}$ satisfies 
	\begin{equation*}
		\mu_k^i  = \mu_k^{i+1} \circ \zeta_i \circ \left( \alpha_k^i\right)^{-1}, \quad \forall k \in \N.
	\end{equation*}
	Consequently the image of $\mu_k^{i+1}$ is non abelian.
	Moreover $\mu_k^{i+1}$ cannot lift to a morphism $\tilde \mu_k^{i+1} \colon L_{i+1} \to \Gamma$.
	Otherwise it would contradict the properties of $\mu_k^i$.
		
	\paragraph{Case 2.}
	Assume now that $L_i$ is freely decomposable.
	Since $\mu_k^i$ does not lift, its restrictions to the free factors of a Grushko decomposition of $L_i$ cannot lift all simultaneously.
	We can decompose $L_i$ as 
	\begin{equation*}
		L_i = M \ast  L_{i+1}
	\end{equation*}
	where $L_{i+1}$ is a freely indecomposable free factor of $L_i$ and $M$ a non-trivial subgroup of $L_i$.
	Moreover, up to passing to a subsequence, the restriction of $\mu_k^i$ to $L_{i+1}$, that we denote by $\mu_k^{i+1}$, does not lift to a morphism $L_{i+1} \to \Gamma$.
	We denote by $\eta_{i+1}$ the morphism $\eta_i$ composed with projection $L_i \onto L_{i+1}$ (obtained by ``killing'' the factor $M$) so that $(L_{i+1}, \eta_{i+1})$ is a limit group and a proper quotient of $(L_i, \eta_i)$.
	We are left to prove that (up to passing to a subsequence) the image of $\mu_k^{i+1}$ is not abelian, for infinitely many $k \in \N$.
	Assume that it is not the case.
	Up to passing to a subsequence, we can assume that $(\Gamma/ \Gamma^{n_k} , \mu_k^{i+1} \circ \eta_{i+1})$ converges to an abelian limit group $(B, \beta) \prec (L_{i+1}, \eta_{i+1})$.
	We decompose $B$ as $B = B_0 \oplus B_1$ where $B_0$ is a free abelian and $B_1$ finite.
	Note that $B_1$ is actually trivial. 
	Indeed, otherwise the group $Q = M \ast B$, which is a quotient of $L$, admits the following non-essential root splitting
	\begin{equation*}
		Q = (M \ast B_0) \ast _{B_0}B,
	\end{equation*}
	which contradicts our assumption.
	In other words $B$ is free abelian.
	Up to passing to a subsequence, every $\mu_k^{i+1}$ factors through $B$ and thus can be lifted (\autoref{res: lifting abelian groups}), which contradicts our construction.
	
	In both cases, we have built all the announced objects for the index $i+1$.
	Hence the induction is complete.
	
	No quotient of $L$ admits a root splitting.
	It follows from \autoref{res: shortening group w/o root - general} that every limit group $(Q, \nu) \prec (L, \eta)$ is shortenable.
	Hence there is no infinite descending sequence of limit groups starting at $(L, \eta)$ (\autoref{res: dcc}).
	This contradicts our previous construction and completes the proof of the result.
\end{proof}

\begin{coro}
\label{res: lifting quotient wo roots -  inf presented}
	Let $G$ be a finitely generated group, none of whose quotients admits an essential root splitting.
	For every finite subset $V \subset G$, there is a critical exponent $N \in \N$, such that for every odd integer $n \geq N$, the following holds.
	For every homomorphism $\phi \colon G \to \Gamma / \Gamma^n$ whose image is not abelian, there exists a morphism $\tilde \phi \colon G \to \Gamma$ lifting $\phi$ relative to $V$.
\end{coro}

\begin{proof}
	We fix a projection $\pi \colon F \onto G$, where $F$ is a finitely generated free group.
	Being hyperbolic, $\Gamma$ is equationally noetherian, thus there is a finite subset $\tilde V_0 \subset \ker \pi$ such that every morphism $\phi \colon F \to \Gamma$ whose kernel contains $\tilde V_0$ factors through $\pi$.
	
	Let $V$ be a finite subset of $G$. 
	We fix a finite subset $\tilde V \subset F$ containing $\tilde V_0$ such that $\pi(\tilde V) = V$.
	We write $W \subset \ker \pi$ and $N \in \N$ for the data given by \autoref{res: lifting quotient wo roots} applied with $\tilde V$.
	Let $n \geq N$ be an odd integer and $\phi \colon G \to \Gamma / \Gamma^n$ a homomorphism.
	Note that $W$ is contained in the kernel of $\phi \circ \pi$.
	According to \autoref{res: lifting quotient wo roots}, there is a morphism $\psi \colon F \to \Gamma$ lifting $\phi \circ \pi$ relative to $\tilde V$.
	In particular, $\tilde V_0$ is contained in the kernel of $\psi$, and thus $\psi$ factors through $\pi$.
	We write $\tilde \phi \colon F \to \Gamma$ for the resulting morphism.
	It follows from the construction that $\tilde \phi$ lifts $\phi$ relative to $V$.
\end{proof}

The next two corollaries are obtained in the exact same way, using the fact that $\Gamma$ is equationally noetherian.
Their proofs are left to the reader.

\begin{coro}
\label{res: lifting quotient wo roots - noetherian}
	Let $F$ be a finitely generated group and $\pi \colon F \onto G$ a projection so that no quotient of $G$ admits an essential root splitting.
	There exist a finite subset $W \subset \ker \pi$ and a critical exponent $N \in \N$ with the following property.
	For every odd integer $n \geq N$, for every homomorphism $\phi \colon G \to \Gamma / \Gamma^n$, if $W \subset \ker \phi$, then $\phi$ factors through $\pi$.
\end{coro}

\begin{coro}
	Let $G$ be a finitely generated group.
	Let $(L, \eta) \in \mathfrak G(G)$ be a limit group.
	If no quotient of $L$ admits an essential root splitting, then one of the following holds
	\begin{itemize}
		\item $L$ is an abelian group whose torsion part is not trivial
		\item $(L,\eta)$ is a limit group over $\Gamma$.
	\end{itemize}
\end{coro}

\begin{rema}
	If we restrict ourselves to limit groups over the periodic quotients of $\Gamma$ with \emph{prime} exponents, then the first case never happens and every limit group none of whose quotient admits an essential root splitting is a limit group over $\Gamma$.
\end{rema}

%
\section{Applications}
%
\label{sec: applications}

%
\subsection{Groups with no root splitting}
%
We focus here on freely indecomposable, hyperbolic groups $\Gamma$ without essential root splitting.


\begin{theo}
\label{res: hopf - rigid}
	Let $\Gamma$ be a torsion-free, hyperbolic group, none of whose quotients admit an abelian splitting.
	There exists $N \in \N$, such that for every odd integer $n \geq N$, the periodic quotient $\Gamma / \Gamma^n$ is Hopfian.
\end{theo}

\begin{proof}
	If $\Gamma$ is cyclic, then $\Gamma / \Gamma^n$ is finite, hence the result is straightforward.
	We now assume that $\Gamma$ is non-elementary.
	We apply \autoref{res: lifting morphism - rigid} with the group $G = \Gamma$.
	Let $N$ be the corresponding critical exponent and $\psi_1, \dots, \psi_m$ the associated list of morphisms $G \to \Gamma$.
	Let $n \geq N$ be an odd integer.
	For simplicity we write $\pi \colon \Gamma \to \Gamma/ \Gamma^n$ for the canonical projection.
	Let $\phi$ be a surjective morphism from $\Gamma/\Gamma^n$ onto itself.
	Assume that contrary to our claim $\phi$ is not injective.
	Hence $K_p = \ker(\phi^p)$ is an increasing sequence of subgroups of $\Gamma/ \Gamma^n$.
	Let $p \in \N$.
	The composition $\phi^p \circ \pi$ is a homomorphism from $\Gamma$ onto $\Gamma/\Gamma^n$.
	In particular, its image is not abelian.
	It follows from \autoref{res: lifting morphism - rigid} there is $i \in \intvald 1m$ such that a conjugate of $\psi_i$ lifts $\phi^p \circ \pi$.
	Hence there exist two integers $p > q$, together with $\gamma \in \Gamma/\Gamma^n$, such that $\phi^p = \iota_\gamma \circ \phi^q$.
	Hence $K_p = K_q$ which contradicts the fact that the sequence $(K_p)$ is increasing.
\end{proof}

\begin{theo}
\label{res: no splitting}
	Let $\Gamma$ be a torsion-free, hyperbolic group.
	If $\Gamma$ has no essential root splitting, then the following are equivalent.
	\begin{enumerate}
		\item \label{enu: no splitting - one-ended}
		The group $\Gamma$ is freely indecomposable.
		\item \label{enu: no splitting - no-splitting exist}
		For every $N \in \N$, there exists an odd integer $n \geq N$ such that the quotient $\Gamma/ \Gamma^n$ is freely indecomposable in $\mathfrak B_n$.
		\item \label{enu: no splitting - no-splitting forall}
		There exists $N \in \N$, such that for every odd exponent $n \geq N$, the quotient $\Gamma/ \Gamma^n$ is freely indecomposable in $\mathfrak B_n$.
	\end{enumerate}
\end{theo}

\begin{proof}
	Observe that \ref{enu: no splitting - no-splitting forall} $\Rightarrow$ \ref {enu: no splitting - no-splitting exist} is obvious.
	Let us prove \ref{enu: no splitting - no-splitting exist} $\Rightarrow$ \ref{enu: no splitting - one-ended}.
	Assume that \ref{enu: no splitting - one-ended} does not holds, that is $\Gamma$ splits as a free product $\Gamma = \Gamma_1 \freep \Gamma_2$.
	Note that $\Gamma_1$ and $\Gamma_2$ are quasi-convex subgroups of $\Gamma$.
	Hence they are either infinite cyclic or non-elementary torsion-free hyperbolic groups.
	In particular, there exists $N \in \N$, such that for every odd exponent $n \geq N$, their $n$-periodic quotients are non-trivial.
	Let $n \geq N$ be an odd integer.
	Let $B_1$ and $B_2$ be the respective $n$-periodic quotients of $\Gamma_1$ and $\Gamma_2$.
	One checks that $\Gamma/ \Gamma^n$ is isomorphic to $B_1 \ast^n B_2$.
	Hence \ref{enu: no splitting - no-splitting exist} does not hold.

	Observe first that if $\Gamma$ is cyclic, then $\Gamma / \Gamma^n$ is freely indecomposable in $\mathfrak B_n$ provided $n$ is sufficiently large.
	Indeed assume on the contrary that $\Gamma / \Gamma^n$ is isomorphic to $A_1 \ast^n A_2$ where $A_1$ and $A_2$ are two non-trivial $n$-periodic groups.
	If $n$ is sufficiently large then the projection $A_1 \ast A_2 \to \Gamma / \Gamma^n$ is one-to-one on the set of elements whose word length is at most four, see \autoref{res: approximating sequence}\ref{enu: approximating sequence - metric}.
	In particular if $a_1$ and $a_2$ are non-trivial element of $A_1$ and $A_2$ respectively, then the commutator $[a_1,a_2]$ is non trivial in $\Gamma/\Gamma^n$.
	This contradicts the fact that $\Gamma / \Gamma^n$ is cyclic.
	
	From now on we assume that $\Gamma$ is non-elementary.
	Let us now focus on \ref{enu: no splitting - one-ended} $\Rightarrow$ \ref {enu: no splitting - no-splitting forall}.
	We now apply \autoref{res: lifting morphism - killing element free product} with $G = \Gamma$.
	We write $N \in \N$ and $W \subset \Gamma\setminus\{1\}$ for the data given by this statement.
	Up to increasing the value of $N$ we can assume that for every odd integer $n \geq N$, the intersection $W \cap \Gamma^n$ is empty (see \autoref{rem: canonical proj asymp injective}).
	
	Let $n \geq N$ be an odd integer.
	For simplicity we write $\pi \colon \Gamma \to \Gamma / \Gamma^n$ for the canonical projection.
	Assume that contrary to our claim, $\Gamma/ \Gamma^n$ splits as a free product in $\mathfrak B_n$, i.e. there exist two non-trivial groups $B_1, B_2 \in \mathfrak B_n$ and an isomorphism $\theta \colon \Gamma/ \Gamma^n \to B_1 \ast^n  B_2$.
	In particular, $B_1$ and $B_2$ can be seen as subgroups of $\Gamma / \Gamma^n$, hence they are CSA.
	
	The projection $\pi$ induces a map $\aut{\Gamma} \to \aut{\Gamma/\Gamma^n}$ that we write $\alpha \mapsto \alpha_n$.
	In particular, $\theta \circ \pi \circ \alpha = \theta\circ \alpha_n \circ \pi$, for every $\alpha \in \aut{\Gamma}$.
	As $\theta$ and $\alpha_n$ are one-to-one, the kernel of the map $\theta\circ \alpha_n \circ \pi$ is $\Gamma^n$.
	Hence $W \cap \ker (\theta \circ \pi \circ \alpha)$ is empty, for every $\alpha \in \aut{\Gamma}$.
	According to  \autoref{res: lifting morphism - killing element free product}, the image of $\theta$ is contained in a conjugate of $B_1$ or $B_2$, which contradicts the fact that $\theta$ is onto.
\end{proof}

\begin{theo}
\label{res: lifting monomorphism}
	Let $\Gamma_1$ and $\Gamma_2$ be two torsion-free, hyperbolic groups.
	Assume that $\Gamma_1$ is freely indecomposable with no essential root splitting and $\Gamma_2$ is non-elementary.
	There exists a critical exponent $N\in \N$, such that for every odd integer $n \geq N$, the following holds.
	
	Given any monomorphism $\phi \colon \Gamma_1/ \Gamma_1^n \into \Gamma_2/\Gamma_2^n$, there exists a monomorphism $\tilde \phi \colon \Gamma_1 \into \Gamma_2$ lifting $\phi$, i.e. such that the following diagram commutes
	\begin{center}
			\begin{tikzpicture}
				\matrix (m) [matrix of math nodes, row sep=2em, column sep=2.5em, text height=1.5ex, text depth=0.25ex] 
				{ 
					\Gamma_1 & \Gamma_2 \\
					\Gamma_1/ \Gamma_1^n & \Gamma_2 / \Gamma_2^n	  \\
				}; 
				\draw[>=stealth, ->] (m-1-1) -- (m-1-2) node[pos=0.5, above]{$\tilde \phi$};
				\draw[>=stealth, ->] (m-2-1) -- (m-2-2) node[pos=0.5, below]{$\phi$};
				
				\draw[>=stealth, ->] (m-1-1) -- (m-2-1);
				\draw[>=stealth, ->] (m-1-2) -- (m-2-2);
			\end{tikzpicture}
		\end{center}
\end{theo}

\begin{rema*}	
	Note that we allow $\Gamma_1$ to be cyclic.
\end{rema*}

\begin{proof}
	Let $X_2$ be a Cayley graph of $\Gamma_2$.
	It is a $\tau$-bootstrap for some $\tau \in (0,1)$.
	We now apply \autoref{res: lifting morphism - killing element for hyperbolic groups} with $G = \Gamma_1$.
	We write $N \in \N$ and $W \subset \Gamma_1\setminus\{1\}$ for the data given by this proposition.
	Up to increasing the value of $N$ we can assume that for every odd integer $n \geq N$, the intersection $W \cap \Gamma_1^n$ is empty (see \autoref{rem: canonical proj asymp injective}).
	
	Let $n \geq N$ be an odd integer.
	For simplicity we write $\pi_i \colon \Gamma_i \to \Gamma_i / \Gamma_i^n$ for the canonical projection.
	Let $\phi \colon \Gamma_1/ \Gamma_1^n \into \Gamma_2/\Gamma_2^n$ be a monomorphism.
	Recall that $\pi_1$ induces a map $\aut{\Gamma_1} \to \aut{\Gamma_1/\Gamma_1^n}$ that we write $\alpha \mapsto \alpha_n$.
	In particular, $\phi \circ \pi_1 \circ \alpha = \phi\circ \alpha_n \circ \pi_1$, for every $\alpha \in \aut{\Gamma_1}$.
	As $\phi$ and $\alpha_n$ are one-to-one, the kernel of $\phi\circ \alpha_n \circ \pi_1$ coincides with $\Gamma_1^n$.
	Hence $W \cap \ker (\phi \circ \pi_1 \circ \alpha)$ is empty, for every $\alpha \in \aut{\Gamma_1}$.
	Applying \autoref{res: lifting morphism - killing element for hyperbolic groups}, there exists a monomorphism $\tilde \phi \colon \Gamma_1 \to \Gamma_2$ such that $\pi_2 \circ \tilde \phi = \phi \circ \pi_1$, whence the result.
\end{proof}

\begin{coro}
\label{res: lifting monomorphism - w free splitting}
	Let $\Gamma_1$ and $\Gamma_2$ be two non-elementary, torsion-free, hyperbolic groups.
	Assume that $\Gamma_1$ has no essential root splitting.
	There exists a critical exponent $N\in \N$, such that for every odd integer $n \geq N$, the following holds.
	
	Given any monomorphism $\phi \colon \Gamma_1/ \Gamma_1^n \into \Gamma_2/\Gamma_2^n$, there exists a morphism $\tilde \phi \colon \Gamma_1 \to \Gamma_2$  lifting $\phi$ whose restriction to every freely indecomposable, free factors of $\Gamma_1$ is one-to-one.
\end{coro}

\begin{proof}
	Consider the Grushko decomposition of $\Gamma_1$
	\begin{equation*}
		\Gamma_1 = A_1 \ast \dots \ast A_m \ast F,
	\end{equation*}
	where each $A_j$ is non-cyclic, freely indecomposable and $F$ is a free group.
	Since each $A_j$ is quasi-convex in $\Gamma_1$ it is also torsion-free hyperbolic.
	Observe also that $A_j$ has no non-essential root splitting, for otherwise $\Gamma_1$ would admit a non-essential root splitting.
	Hence we can apply for each $j \in \intvald 1m$ \autoref{res: lifting monomorphism} replacing $\Gamma_1$ by $A_j$.
	We write $N_j$ for the critical exponent given by this theorem.
	Let $N = \max \{ N_1, \dots, N_m\}$.
	
	Let $n \geq N$ be an odd integer and $\phi \colon \Gamma_1 / \Gamma_1^n \to \Gamma_2 / \Gamma_2^n$ a monomorphism.
	As we observed in the proof \autoref{res: def free product burnside} the periodic quotient $A_j /A_j ^n$ embeds into $\Gamma_1 / \Gamma_1^n$.
	Consequently $\phi$ induces an embedding of $A_j /A_j ^n$ into $\Gamma_2 / \Gamma_2^n$.
	According to \autoref{res: lifting monomorphism}  the latter lifts to a monomorphism $\tilde \phi_j \colon A_j \into \Gamma_2$.
	Since $F$ is a free group the map
	\begin{equation*}
		\mu \colon F \to \Gamma_1 \to \Gamma_1/\Gamma_1^n \xrightarrow{\ \phi \ } \Gamma_2 / \Gamma_2^n
	\end{equation*}
	also lifts to a map $\tilde \mu \colon F \to \Gamma_2$ (which may not be one-to-one though, unless $F$ is cyclic).
	We write $\tilde \phi \colon \Gamma_1 \to \Gamma_2$ for the map whose restriction to $A_j$ (\resp $F$) is $\tilde \phi_j$ (\resp $\tilde \mu$).
	By construction it is a lift of $\phi$.
\end{proof}

\begin{coro}
\label{res: auto/co-hopf - particular}
	Let $\Gamma$ be a non-elementary, torsion-free, freely indecomposable, hyperbolic group with no essential root splitting.
	There exists a critical exponent $N \in \N$ such that for every odd integer $n \geq N$, the following holds.
	\begin{enumerate}
		\item The map $\aut \Gamma \to \aut{\Gamma/\Gamma^n}$ is onto.
		\item The quotient $\Gamma/\Gamma^n$ is co-Hopfian
	\end{enumerate}
\end{coro}

\begin{proof}
	It suffices to prove that if $n$ is a sufficiently large odd integer then every monomorphism of $\Gamma/\Gamma^n \into \Gamma/\Gamma^n$ belongs to the image of the map $\aut \Gamma \to \aut{\Gamma/\Gamma^n}$.
	To that end, we apply \autoref{res: lifting monomorphism} with $\Gamma_1 = \Gamma_2 = \Gamma$.
	Hence every monomorphism $\Gamma/\Gamma^n \into \Gamma/\Gamma^n$ is the image of a monomorphism $\Gamma \into \Gamma$ (provided $n$ is large enough).
	Since $\Gamma$ is freely indecomposable, it is co-Hopfian. 
	Thus every monomorphism $\phi \colon \Gamma \into \Gamma$ is actually an automorphism, whence the result.
\end{proof}

\begin{lemm}
\label{res: non intersecting free factors}
	Let $\Gamma$ be a group splitting as a free product $\Gamma = A_1 \ast A_2$.
	For every $n \in \N$, the groups $A_1 / A_1^n$ and $A_2 /A_2^n$ embed in $\Gamma / \Gamma^n$ with trivial intersection.
\end{lemm}

\begin{proof}
	Let $n \in \N$.
	The natural embedding $A_1 \into \Gamma$ as well as the canonical projections $\Gamma \onto A_1 \times A_2$ and $A_1 \times A_2 \onto A_i$ induce maps on the corresponding periodic quotients so that the following diagram commutes.
	\begin{center}
			\begin{tikzpicture}
				\matrix (m) [matrix of math nodes, row sep=2em, column sep=2.5em, text height=1.5ex, text depth=0.25ex] 
				{ 
					A_i & \Gamma & A_1 \times A_2 & A_i \\
					A_i/ A_i^n & \Gamma / \Gamma^n & A_1/A_1^n \times A_2/A_2^n & A_i /A_i^n	  \\
				}; 
				\draw[>=stealth, ->] (m-1-1) -- (m-1-2);
				\draw[>=stealth, ->] (m-1-2) -- (m-1-3);
				\draw[>=stealth, ->] (m-1-3) -- (m-1-4);
				\draw[>=stealth, ->] (m-2-1) -- (m-2-2);
				\draw[>=stealth, ->] (m-2-2) -- (m-2-3);
				\draw[>=stealth, ->] (m-2-3) -- (m-2-4);

				\draw[>=stealth, ->] (m-1-1) -- (m-2-1);
				\draw[>=stealth, ->] (m-1-2) -- (m-2-2);
				\draw[>=stealth, ->] (m-1-3) -- (m-2-3);
				\draw[>=stealth, ->] (m-1-4) -- (m-2-4);
			\end{tikzpicture}
		\end{center}
		Note also the top row is just the identity of $A_i$.
		The result now follows from chasing diagram.
\end{proof}

\begin{coro}
\label{res: isom - particular}
	Let $\Gamma_1$ and $\Gamma_2$ be two torsion-free, hyperbolic groups with no essential root splitting.
	There exists $N \in \N$ such that the following are equivalent.
	\begin{enumerate}
		\item \label{enu: isom - particular - source}
		$\Gamma_1$ and $\Gamma_2$ are isomorphic.
		\item \label{enu: isom - particular - some}
		There exists an odd integer $n \geq N$, such that the groups $\Gamma_1/ \Gamma_1^n$ and $\Gamma_2/\Gamma_2^n$ are isomorphic.
		\item \label{enu: isom - particular - any}
		For every integer $n \in \N$, the groups $\Gamma_1/ \Gamma_1^n$ and $\Gamma_2/\Gamma_2^n$ are isomorphic.
	\end{enumerate}
\end{coro}

\begin{proof}
	Observe that \ref{enu: isom - particular - source} $\Rightarrow$ \ref{enu: isom - particular - any} and \ref{enu: isom - particular - any} $\Rightarrow$ \ref{enu: isom - particular - some} are obvious, whatever the value of $N$ is.
	Note that if $\Gamma_1$ or $\Gamma_2$ is cyclic, the result is straightforward.
	Indeed if $n$ is a sufficiently large odd integer, then the $n$-periodic quotient of a non-elementary torsion-free hyperbolic group is not cyclic.
	From now on, we assume that $\Gamma_1$ and $\Gamma_2$ are non-elementary.
	Let us now focus on \ref{enu: isom - particular - some} $\Rightarrow$ \ref{enu: isom - particular - source}.
	Denote by 
	\begin{equation*}
		\Gamma_i = A_{i,1} \ast \dots \ast A_{i,m_i} \ast F_i
	\end{equation*}
	the Grushko decomposition of $\Gamma_i$ where each $A_{i,j}$ is non-cyclic, freely indecomposable and $F_i$ is a free group. 
	We apply \autoref{res: lifting monomorphism} twice, first with $\Gamma_1$ and $\Gamma_2$, then with $\Gamma_2$ and $\Gamma_1$.
	Hence there exists $N \in \N$, such that for every odd integer $n \geq N$, every monomorphism $\Gamma_1/\Gamma_1^n \into \Gamma_2/\Gamma_2^n$ (\resp $\Gamma_2/\Gamma_2^n \into \Gamma_1/\Gamma_1^n$) comes from a morphism $\Gamma_1 \into \Gamma_2$ (\resp $\Gamma_2 \into \Gamma_1$) whose restriction to non-cyclic, freely indecomposable free factors is one-to-one.
	Up to increasing the value of $N$ we can assume that $A_{i,j}/ A_{i,j}^n$ is always non-trivial, for every $n \geq N$.
	
	Assume now that there exists an automorphism $\theta \colon \Gamma_1/\Gamma_1^n \to \Gamma_2/\Gamma_2^n$ for some $n \geq N$.
	We denote by $\pi_1 \colon \Gamma_1 \onto \Gamma_1 / \Gamma_1^n$ the canonical projection.
	According to our choice of $N$, there are two morphisms $\phi_1 \colon \Gamma_1 \into \Gamma_2$ and $\phi_2 \colon \Gamma_2 \into \Gamma_1$ as above respectively lifting $\theta$ and its inverse.
	Since $\phi_1$ and $\phi_2$ are one-to-one when restricted to freely indecomposable free factors there are two maps $\sigma_1 \colon \intvald 1{m_1} \to \intvald 1{m_2}$ and $\sigma_2  \colon \intvald 1{m_2} \to \intvald 1{m_1}$ such that 
	\begin{equation*}
		\phi_1(A_{1,j}) \subset A_{2, \sigma_1(j)} 
		\quad \text{and} \quad
		\phi_2(A_{2,j}) \subset A_{1, \sigma_2(j)},
		\quad \forall j.
	\end{equation*}
	We claim that $\sigma_1$ and $\sigma_2$ are bijections inverse of one another.
	For simplicity we let $\phi = \phi_2 \circ \phi_1$ and $\sigma = \sigma_2 \circ \sigma_1$ so that 
	\begin{equation*}
		\phi(A_{1,j}) \subset A_{1, \sigma(j)}, \quad \forall j \in \intvald 1{m_1}.
	\end{equation*}
	Let $j \in \intvald 1{m_1}$.
	By construction $\phi$ is a lift of the identity on $\Gamma_1 / \Gamma_1^n$.
	Hence 
	\begin{equation*}
		\pi_1\left(A_{1,j}\right) \subset \pi_1\left(A_{1, \sigma(j)}\right).
	\end{equation*}
	It follows from \autoref{res: non intersecting free factors} that either $\sigma(j) = j$ or $A_{1,j}/A_{1,j}^n$ is trivial. 
	We made sure when choosing $N$ that $A_{1,j}/A_{1,j}^n$ was not trivial.
	Consequently $\sigma = \sigma_2 \circ \sigma_1$ is the identity map.
	The same argument with $\phi_1 \circ \phi_2$ shows that $\sigma_1 \circ \sigma_2$ is also the identity, which completes the proof of our claim.
	In particular, $m_1 = m_2$ which we now denote by $m$.
	Moreover, up to reordering the free factors of $\Gamma_2$, we can assume without loss of generality that $\sigma_1$ (hence $\sigma_2$) is the identity.
	
	Let $j \in \intvald 1m$.
	According to the above discussion, $\phi_2 \circ \phi_1$ induces an embedding from $A_{1,j}$ into itself.
	However $A_{1,j}$ is a freely indecomposable hyperbolic group, hence co-Hopfian.
	It follows that $\phi_2 \circ \phi_1$ induces an automorphism of $A_{1,j}$.
	Similarly we prove that $\phi_1 \circ \phi_2$ induces an automorphism of $A_{2, j}$.
	Hence $\phi_1$ and $\phi_2$ respectively induce isomorphisms
	\begin{equation*}
		\phi_1 \colon A_{1,j} \to A_{2, j}, \quad \text{and} \quad
		\phi_2 \colon A_{2,j} \to A_{1, j},
	\end{equation*}
	for every $j \in \intvald 1m$.
	Let now us write
	\begin{equation*}
		A_i = A_{i,1} \ast \dots \ast A_{i,j}, \quad \forall i \in \{1, 2\}.
	\end{equation*}
	According to the above discussion $\phi_1$ (\resp $\phi_2$) induces an isomorphism from $A_1$ onto $A_2$ (\resp $A_2$ onto $A_1$).
	Consequently, if we consider the quotients of $\Gamma_1$ and $\Gamma_2$ by $A_1$ and $A_2$ respectively, we observe that the automorphism $\theta \colon \Gamma_1 /\Gamma_1^n \to \Gamma_2 / \Gamma_2^n$ induces an automorphism from $F_1/F_1^n$ onto $F_2/ F_2^n$.
	However $F_1$ and $F_2$ are free abelian groups.
	Looking at the abelianizations of $F_1/F_1^n$ and $F_2/ F_2^n$ we observe that $F_1$ and $F_2$ have the same rank, hence are isomorphic.
	It follows that $\Gamma_1$ and $\Gamma_2$ are isomorphic as well.
\end{proof}

\begin{rema*}
	Note that in the last step of the proof we do not claim that there is an isomorphism $\Gamma_1 \to \Gamma_2$ lifting $\theta$.
	Indeed \autoref{exa: lifting auto} can be adapted to show that such a statement is false in general.
\end{rema*}

%
\subsection{Root towers}
%
\label{sec: one-ended periodic groups}

As we observed in the introduction, all the above statements fail if one does not ask that $\Gamma$ has no essential root-splitting.
The notion of root tower defined below is designed to get rid of the root splittings.

\begin{defi}[Root tower]
	A \emph{root tower} is a finite sequence of groups $G_0, G_1, \dots, G_m$ where each $G_{i+1}$ is obtained from $G_i$ by adjoining roots.
	The \emph{orders} of this tower is the set $P = \{p_1, \dots, p_m\}$ where each $p_i$ is the order of the root splitting defining $G_i$.
	We respectively call $G_0$ and $G_m$ the \emph{floor} and the \emph{roof} of the tower.
\end{defi}

The next statement is a consequence of Louder and Touikan \cite{Louder:2017gy}.

\begin{theo}
	Let $\Gamma$ be a non-elementary, torsion-free, hyperbolic group.
	Then $\Gamma$ is the roof of a root tower whose floor does not admit any essential root splitting.
\end{theo}

Given a subset $P \subset \N$, we say that an integer $n \in \N$ is \emph{co-prime with $P$} if it is co-prime with every element in $P$.

\begin{lemm}
\label{res: floor tower props}
	Let $\Gamma$ be a non-elementary torsion-free hyperbolic group.
	Assume that $\Gamma$ is the roof of a root tower $\Gamma_0, \Gamma_1, \dots, \Gamma_m = \Gamma$.
	Let $P \subset \N$ be the orders of this tower.
	Then the following holds
	\begin{enumerate}
		\item \label{enu: floor tower props - qc}
		$\Gamma_0$ is a non-elementary, quasi-convex subgroup of $\Gamma$.
		\item \label{enu: floor tower props - isom}
		For every exponent $n \in \N$, which is co-prime with $P$ the inclusion $\Gamma_0 \into \Gamma$ induces an isomorphism $\Gamma_0/\Gamma_0^n \to \Gamma/\Gamma^n$.
	\end{enumerate}
\end{lemm}

\begin{proof}
	Observe that for each $i \in \intvald 1m$, the group $\Gamma_{i-1}$ is a quasi-convex subgroup of $\Gamma_i$ \cite[Lemma~3.5]{Kapovich:1997eb}.
	This is essentially due to the fact that the edge groups of the root splitting of $\Gamma_i$ are cyclic, hence quasi-convex in $\Gamma_i$.
	A proof by induction shows that $\Gamma_0$ is quasi-convex in $\Gamma_i$ for every $i \in \intvald 0m$, whence \ref{enu: floor tower props - qc}.
	Recall that $\Gamma$ is CSA.
	Hence $\Gamma_0$ cannot be cyclic, for otherwise $\Gamma$ would be cyclic as well.
	Point~\ref{enu: floor tower props - isom} is a simple induction using \autoref{res: root splitting isom burnside}.
\end{proof}

We can now study more generally all the $n$-periodic quotient of a torsion-free hyperbolic group which are freely indecomposable in the Burnside variety $\mathfrak B_n$.
We first give a characterization of those groups using root towers.

\begin{theo}
\label{res: no splitting - general}
	Let $\Gamma$ be a non-elementary, torsion-free, hyperbolic group.
	Assume that $\Gamma$ is the roof of a root tower whose floor $\Gamma_0$ has no essential root-splitting.
	Let $P$ be the orders of this tower.
	Then the following are equivalent.
	\begin{enumerate}
		\item \label{enu: no splitting - one-ended}
		The group $\Gamma_0$ is freely indecomposable.
		\item \label{enu: no splitting - no-splitting}
		There exists $N \in \N$, such that for every odd exponent $n \geq N$ which is co-prime with $P$, the quotient $\Gamma/ \Gamma^n$ does not split as a free product in $\mathfrak B_n$.
	\end{enumerate}
\end{theo}

\begin{proof}
	By \autoref{res: floor tower props}, the $n$-periodic quotients of $\Gamma$ and $\Gamma_0$ are isomorphic, for every $n \in \N$ that is co-prime with $P$.
	The result follows from \autoref{res: no splitting}.
\end{proof}

We now generalize the results of the previous section when $\Gamma$ is the roof of a root tower whose floor is freely indecomposable with no essential root splitting, i.e. when the $n$-periodic quotients $\Gamma / \Gamma^n$ are freely indecomposable in $\mathfrak B_n$.

\begin{theo}
\label{res: co-hopf/auto/etc - general}
	Let $\Gamma$ be a non-elementary, torsion-free, hyperbolic.
	Assume that $\Gamma$ is the roof of a root tower whose floor  $\Gamma_0$ is freely indecomposable, with no essential root splitting.
	Let $P$ be the orders of this tower.
	Then there exists $N \in \N$, such that for every $n \geq N$, which is co-prime with $P$, the following holds.
	\begin{enumerate}
		\item \label{enu: co-hopf/auto/etc - general - isom}
		The inclusion $\Gamma_0 \into \Gamma$ induces an isomorphism $\Gamma_0 /\Gamma_0^n \to \Gamma/\Gamma^n$.
		\item \label{enu: co-hopf/auto/etc - general - co-hopf}
		The group $\Gamma/ \Gamma^n$ is co-Hopfian
		\item \label{enu: co-hopf/auto/etc - general - aut}
		The canonical map from $\aut{\Gamma_0}$ to $\aut{\Gamma_0/\Gamma_0^n} = \aut{\Gamma/\Gamma^n}$ is onto.
	\end{enumerate}
\end{theo}

\begin{proof}
	It follows from \autoref{res: floor tower props}, that for every $n \in \N$ that is co-prime with $P$, the $n$-periodic quotients of  $\Gamma$ and $\Gamma_0$ are isomorphic, whence \ref{enu: co-hopf/auto/etc - general - isom}.
	Points~\ref{enu: co-hopf/auto/etc - general - co-hopf} and \ref{enu: co-hopf/auto/etc - general - aut} respectively follows from \autoref{res: auto/co-hopf - particular}.
\end{proof}

\begin{theo}
\label{res: isom - general}
	Let $\Gamma_1$ and $\Gamma_2$ be two non-elementary, torsion-free, hyperbolic groups.
	Assume that $\Gamma_i$ is the roof of a root tower whose floor $\Gamma_{i,0}$ has no essential root splitting.
	Let $P_i$ be the orders of this tower.
	Then there exists $N \in \N$ such that the following are equivalent
	\begin{enumerate}
		\item \label{enu: isom - particular - source}
		$\Gamma_{1,0}$ and $\Gamma_{2,0}$ are isomorphic.
		\item \label{enu: isom - particular - some}
		$\Gamma_1/ \Gamma_1^n$ and $\Gamma_2/\Gamma_2^n$ are isomorphic for some integer $n \geq N$ that is co-prime with $P_1 \cup P_2$.
		\item \label{enu: isom - particular - any}
		$\Gamma_1/ \Gamma_1^n$ and $\Gamma_2/\Gamma_2^n$ are isomorphic for every integer $n \in \N$, that is co-prime with $P_1 \cup P_2$.
	\end{enumerate}
\end{theo}

\begin{proof}
	It follows from \autoref{res: floor tower props}, that for every $n \geq N$ that is co-prime with $P_1 \cup P_2$, the $n$-periodic quotients of  $\Gamma_1$ and $\Gamma_{1,0}$ (\resp $\Gamma_2$ and $\Gamma_{2,0}$) are isomorphic.
	Hence the results follows from \autoref{res: isom - particular}.
\end{proof}

\appendix


%
\section{Approximation of periodic groups}
%
\label{sec: approx}

This appendix is dedicated to the proof of \autoref{res: approximating sequence}.
We first recall the construction of Delzant and Gromov \cite{Delzant:2008tu} to investigate small cancellation groups and then prove the existence of a class of approximations of a periodic group as in \autoref{sec: approx periodic groups}.
We follow Coulon \cite{Coulon:2014fr,Coulon:2016if,Coulon:2018vp}.

%
\subsection{Small cancellation theory}
%
\label{sec: appendix - sc}

%
\subsubsection{Cone and cone-off}
%

\paragraph{Cone over a metric space.}
Let $Y$ be a metric space.
The \emph{cone of radius $\rho$ over $Y$}, denoted by $Z_\rho(Y)$ or simply $Z(Y)$, is the quotient of $Y\times \left[0,\rho\right]$ by the equivalence relation that identifies all the points of the form $(y,0)$.	
The equivalence class of $(y,0)$, denoted by $c$, is called the \emph{apex} or \emph{cone point} of $Z(Y)$. 
By abuse of notation we still write $(y,r)$ for the equivalence class of $(y,r)$.
The cone over $Y$ is endowed with a metric characterized as follows \cite[Chapter I.5, Proposition 5.9]{Bridson:1999ky}.
If $x=(y,r)$ and $x'=(y',r')$ are two points of $Z(Y)$, then
\begin{equation}
\label{eqn: sc - metric cone}
	\cosh \dist x{x'} = \cosh r \cosh r' - \sinh r \sinh r' \cos \left(\min\left\{\pi, \frac {\dist y{y'}}{\sinh \rho}\right\} \right).
\end{equation}
This metric is modeled on the one of the hyperbolic place $\H^2$, see \cite{Coulon:2014fr} for the geometric interpretation.
If $Y$ is a length space, then so is $Z(Y)$.
The embedding $\iota \colon Y \to Z(Y)$ sending $y$ to $(y,\rho)$ satisfies
\begin{equation*}
	\dist {\iota(y)}{\iota(y')} = \mu \left(\dist y{y'}\right), \quad \forall y,y' \in Y,
\end{equation*}
where $\mu \colon \R_+ \to \R_+$ is the non-decreasing concave map characterized by 
\begin{equation*}
	\cosh \mu(t) = \cosh^2 \rho - \sinh^2 \rho \cos \left(\min\left\{\pi, \frac {t}{\sinh \rho}\right\} \right), \quad \forall t \in \R_+.
\end{equation*}
In addition, the cone comes with a \emph{radial projection} $p \colon Z(Y)\setminus\{c\} \to Y$ mapping $(y,r)$ to $y$.
The next proposition is an exercise using the geometric interpretation of the distance in a cone.
\begin{lemm}
\label{res: qi base implies qi cone}
	Assume that $f \colon Y_1 \to Y_2$ is a $(1, \ell)$-quasi-isometric embedding.
	Then $f$ extends to a $(1,\ell)$-quasi-isometric embedding $Z(Y_1)\to Z(Y_2)$ sending $(y,r)$ to $(f(y), r)$.
\end{lemm}

In particular, if $H$ is a group acting by isometries on $Y$, then this action extends to an action by isometries on $Z(Y)$ fixing its apex.
Moreover, if the original action of $H$ on $Y$ is proper, then the quotient $Z(Y)/H$ is a metric space isometric to $Z(Y/H)$.
	
\paragraph{Cone-off space.}
Let $\rho \in \R_+^*$.
Let $X$ be a $\delta$-hyperbolic length space and $\mathcal Y$ a collection of strongly quasi-convex subsets of $X$.
Given $Y \in \mathcal Y$, we denote by $\distV[Y]$ the length metric on $Y$ induced by the restriction of $\distV$ on $Y$.
We write $Z_\rho(Y)$ for the cone of radius $\rho$ over $(Y, \distV[Y])$.

The \emph{cone-off of radius $\rho$ relative to $\mathcal Y$}, denoted by $\dot X_\rho(\mathcal Y)$ or simply $\dot X$ is the space obtained by attaching for every $Y \in \mathcal Y$, the cone $Z_\rho(Y)$ on $X$ along $Y$ according to $\iota \colon Y \to Z_\rho(Y)$.
We write $\distV[\dot X]$ for the largest pseudo-metric on $\dot X$ such that the canonical maps $X \to \dot X$ as well as $Z(Y) \to \dot X$ are $1$-Lipschitz.
It turns out that $\distV[\dot X]$ is a metric on $\dot X$ \cite[Proposition~5.10]{Coulon:2014fr}.
The metrics of $X$ and $\dot X$ are related as follows \cite[Lemma~5.8]{Coulon:2014fr}.
For every $x,x' \in X$ we have
\begin{equation}
\label{eqn: appendix - lower bound dist dot X}
	\mu \left(\dist[X] x{x'}\right) \leq \dist[\dot X] x{x'} \leq \dist[X] x{x'}
\end{equation}
By construction, the embedding $Z(Y) \to \dot X$ is $1$-Lipschitz for every $Y \in \mathcal Y$.
Since we passed to the length metric on $Y$ before building the cones, this map is not an isometry in general.
Nevertheless we have the following statement.

\begin{prop}
\label{res: qi cone in cone-off}
	For every $Y \in \mathcal Y$, the map $Z(Y) \to \dot X$ is $(1,8\delta)$-quasi-isometric embedding.
\end{prop}

\begin{proof}
	The proof involves a variation on the cone-off construction.
	For every $Y \in \mathcal Y$, we write $Z'(Y)$ for the cone of radius $\rho$ over the space $(Y, \distV[X])$.
	Note that this time, the metric on $Y$ is the one from $X$ and not the induced length metric.
	In particular, $Z'(Y)$ may fail to be a length space, but this is not important here.
	As a set of points $Z(Y)$ and $Z'(Y)$ are the same though.
	We denote by $\dot X'$ the space obtained by attaching for every $Y \in \mathcal Y$, the cone $Z'(Y)$ on $X$ along $Y$ according to $\iota \colon Y \to Z'(Y)$.
	We endow $\dot X'$ with the largest (pseudo-)metric such that the canonical maps $X \to \dot X'$ as well as $Z'(Y) \to \dot X'$ are $1$-Lipschitz, see \cite[Section~3]{Coulon:2011il}.
	For every $Y \in \mathcal Y$, the identity map $(Y, \distV[Y]) \to (Y, \distV[X])$ is $1$-Lipschitz.
	It follows that the natural map $\dot X \to \dot X'$ is $1$-Lipschitz.
	Let $Y \in \mathcal Y$.
	For every $y,y' \in Y$, we have 
	\begin{equation*}
		\dist[Z'(Y)]{\iota(y)}{\iota(y')} \leq \mu \left(\dist[X] y{y'}\right) \leq \dist[\dot X'] y{y'},
	\end{equation*}
	see for instance \cite[Lemma~3.1.16]{Coulon:2011il}.
	It follows that $Z'(Y)$ is \emph{isometrically} embedded in $\dot X'$.
	Combining these facts, we get that for every $z,z' \in Z(Y)$ 
	\begin{equation*}
		\dist[Z'(Y)] z{z'} \leq \dist[\dot X]z{z'}.
	\end{equation*}
	By (\ref{eqn: def strongly qc}) the map $(Y, \distV[Y]) \to (Y, \distV[X])$ is a $(1, 8\delta)$-quasi-isometric embedding.
	The conclusion now follows from the fact that the map $Z(Y) \to Z'(Y)$ is a $(1,8\delta)$-quasi-isometric embedding (\autoref{res: qi base implies qi cone}).
\end{proof}

%
\subsubsection{The small cancellation theorem}
%

\paragraph{General settings.}
In this section $X$ is a $\delta$-hyperbolic length space endowed with a non-elementary action by isometries of a group $\Gamma$.
In addition we fix a collection $\mathcal Q$ of pairs $(H,Y)$ where $H$ is a subgroup of $\Gamma$ and $Y$ an $H$-invariant strongly-quasi-convex subset of $X$.
We assume that $\mathcal Q$ is invariant under the action of $\Gamma$ defined by $\gamma \cdot (H,Y) = (\gamma H\gamma^{-1},\gamma Y)$, for every $(H,Y) \in \mathcal Q$ and every $\gamma \in \Gamma$.
We denote by $K$ the (normal) subgroup of $\Gamma$ generated by all $H$ where $(H,Y)$ runs over $\mathcal Q$.
The goal is to study the quotient $\bar \Gamma = \Gamma/K$ and the corresponding projection $\pi \colon \Gamma \onto \bar \Gamma$.
More precisely we are going to build a hyperbolic space $\bar X$ endowed with a non-elementary action of $\bar \Gamma$, and relate the geometry of $\bar X$ to the one of $X$.
To that end, we define the following small cancellation parameters
\begin{eqnarray*}
	\Delta(\mathcal Q,X) & = & \sup \set{\diam\left(Y_1^{+3\delta} \cap Y_2^{+3\delta}\right)}{(H_1,Y_1) \neq (H_2,Y_2) \in \mathcal Q} \\
	T(\mathcal Q, X) & = & \inf\set{\norm h}{h \in H,\ (H,Y) \in \mathcal Q}.
\end{eqnarray*}
They respectively play the role of the length of the largest piece and the smallest relation.

\paragraph{Metric spaces.}
We fix a paramater $\rho \in \R_+^*$.
Its value will be made precise later (see \autoref{res: small cancellation}).
We make an abuse of notation and write  $\dot X_\rho(\mathcal Q)$ or simply $\dot X$ for the cone-off of radius $\rho$ over $X$ relative to the collection
\begin{equation*}
	\mathcal Y = \set{Y}{(H,Y) \in \mathcal Q}.
\end{equation*}
Let $\mathcal C \subset \dot X$ be the set of apices of all the cones.
Recall that $\mathcal Q$ is $\Gamma$-invariant.
Hence the action of $\Gamma$ on $X$ naturally extends to an action by isometries on $\dot X$: for every $g \in \Gamma$, $(H,Y) \in \mathcal Q$, and $x = (y,r)$ in $Z(Y)$, we define $gx$ to be the point of $Z(gY)$ given by $gx = (gy,r)$. 
The \emph{radial projection} $p \colon \dot X \setminus \mathcal C \to X$ is the map whose restriction to $X$ is the identity and whose restriction to any punctured cone $Z(Y) \setminus \{c\}$ is the radial projection defined above.

The \emph{quotient space}, denoted by $\bar X_\rho(\mathcal Q)$ or simply $\bar X$ is the quotient of $\dot X_\rho(\mathcal Q)$ by the normal subgroup $K$.
The metric of $\dot X$ induces a pseudo-metric on $\bar X$.
The group $\bar \Gamma$ naturally acts by isometries on $\bar X$ so that the canonical projection $f\colon \dot X \to \bar X$ is a $1$-Lipschitz, $\Gamma$-equivariant map.
If $x$ is a point in $X$ we write $\bar x = f(x)$ for its image in $\bar X$.
Similarly, we denote by $\bar {\mathcal C}$ the image of $\mathcal C$ in $\bar X$.
It is a $\bar \Gamma$ invariant subset of $\bar X$.
Moreover for every distinct $\bar c, \bar c' \in \bar {\mathcal C}$, we have $\dist{\bar c}{\bar c'} \geq 2 \rho$.
We let
\begin{equation*}
	\bar X^+ = \bar X \setminus \bigcup_{\bar c \in \bar{\mathcal C}} B(\bar c, \rho),
\end{equation*}
which is also the image of $X$ under the map $f$.
As explained in \autoref{sec: approx periodic groups}, we think of $\bar X^+$ as the thick part of thin-thick decomposition, where the thin parts correspond to the balls $B(\bar c,\rho)$ centered at a cone point $\bar c \in \bar{\mathcal C}$.
The radial projection $p \colon \dot X \setminus \mathcal C \to X$ is $\Gamma$-equivariant.
Hence it induces a $\bar \Gamma$-equivariant map, $\bar p \colon \bar X \setminus \bar{\mathcal C} \onto \bar X^+$, that we still call the \emph{radial projection}.

The next statement is a combination of Proposition~6.4, Proposition~6.7, Corollary~ 3.12 and Proposition~3.15 in \cite{Coulon:2014fr}

\begin{theo}
\label{res: small cancellation}
	There exist $\delta_0, \delta_1, \Delta_0, \rho_0 \in \R_+^*$, which do not depend on $X$, $\Gamma$ or $\mathcal Q$, with the following property.
	Assume that $\rho \geq \rho_0$.
	If $\delta \leq \delta_0$, $\Delta(\mathcal Q,X) \leq \Delta_0$ and $T(\mathcal Q,X) \geq 10\pi \sinh \rho$, then the following holds
	\begin{enumerate}
		\item \label{enu: small cancellation - hyp cone-off}
		The cone-off space $\dot X$ is $\dot \delta$-hyperbolic with $\dot \delta \leq \delta_1$.
		\item \label{enu: small cancellation - hyp}
		The quotient space $\bar X$ is $\bar \delta$-hyperbolic with $\bar \delta \leq \delta_1$.
		\item \label{enu: small cancellation - local embedding}
		Let $(H,Y) \in \mathcal Q$.
		Let $\bar c$ be the image in $\bar X$ of the apex $c$ of $Z(Y)$.
		The projection $\pi \colon \Gamma \onto \bar \Gamma$ induces an isomorphism from $\stab Y/H$ onto $\stab{\bar c}$.
		\item \label{enu: small cancellation - local isom}
		For every $r \in (0, \rho/20]$, and $x \in \dot X$, if $d(x, \mathcal C) \geq 2r$, then the map $f \colon \dot X \to \bar X$ induces an isometry from $B(x,r)$ onto $B(\bar x, r)$.
		\item \label{enu: small cancellation - translation kernel}
		For every $x \in \dot X$, and $g \in K \setminus\{1\}$, we have $\dist[\dot X]{gx}x \geq \min \{2r, \rho/5\}$, where $r = d(x, \mathcal C)$.
		In particular, $K$ acts freely on $\dot X \setminus \mathcal C$ and the map $f \colon \dot X \to \bar X$ induces a covering map $\dot X \setminus \mathcal C \to \bar X \setminus \bar{\mathcal C}$.
	\end{enumerate}
\end{theo}

\begin{rema*}
	Note that the constants $\delta_0$ and $\Delta_0$ (\resp $\rho_0$) can be chosen arbitrarily small (\resp large).
	From now on, we will always assume that $\rho_0 > 10^{20} \delta_1$ whereas $\delta_0, \Delta_0 < 10^{-10}\delta_1$.
	Similarly we suppose that $\tanh(\rho_0) \geq 1/2$.
	These estimates are absolutely not optimal.
	We chose them very generously to ensure that all the inequalities which we need later are satisfied.
	What really matters is their orders of magnitude recalled below
	\begin{equation*}
		\max\left\{\delta_0, \Delta_0\right\} \ll \delta_1  \ll \rho_0 \ll \pi \sinh \rho_0.
	\end{equation*}
	
	From now on and until the end of \autoref{sec: appendix - sc} we assume that $X$, $\Gamma$ and $\mathcal Q$ are as in \autoref{res: small cancellation}.
	In particular, $\dot X$ and $\bar X$ are respectively $\dot\delta$- and $\bar \delta$-hyperbolic.
	Up to increasing one constant or the other, we can actually assume that $\dot \delta = \bar \delta$.
	Nevertheless we still keep two distinct notations, to remember which space we are working in.
\end{rema*}

%
\subsubsection{Lifting properties}
%

\autoref{res: small cancellation}~\ref{enu: small cancellation - local embedding} and \ref{enu: small cancellation - local isom} state that any ``small scale picture'' in $\bar X$ can be lifted in $X$.
The next statement is an illustration of this phenomenon.

\begin{lemm}
\label{res: sc - lifting isometries}
	Let $\bar x \in \bar X$ such that $d(\bar x, \bar{\mathcal C}) \geq \rho /2$.
	Let $\bar U$ be a subset of $\bar \Gamma$.
	If $\dist{\bar \gamma\bar x}{\bar x} \leq \rho/100$, for every $\bar \gamma \in \bar U$, then there exists a subset $U$ of $\Gamma$ with the following properties:
	\begin{enumerate}
		\item the epimorphism $\pi \colon \Gamma \to \bar \Gamma$ induces a bijection from $U$ onto $\bar U$;
		\item for every $\gamma,\gamma' \in U$, if $\pi(\gamma\gamma')$ belongs to $\bar U$, then $\gamma\gamma' \in U$.
	\end{enumerate}
\end{lemm}

\begin{proof}
	Let $x \in \dot X$ be a pre-image of $\bar x$.
	According to \autoref{res: small cancellation}~\ref{enu: small cancellation - local isom}, for every $\bar \gamma \in \bar U$, there exists $\gamma \in \Gamma$, such that $\dist[\dot X]{\gamma x}x = \dist{\bar \gamma\bar x}{\bar x}$.
	We denote by $U$ the set of pre-images of all elements of $\bar U$ obtained in this way.	
	By \autoref{res: small cancellation}~\ref{enu: small cancellation - translation kernel} for every $\gamma \in K\setminus\{1\}$, we have $\dist[\dot X]{\gamma x}x \geq \rho/5$.
	It follows that the projection $\Gamma \twoheadrightarrow \bar \Gamma$ induces a bijection from $U$ onto $\bar U$.
	Let us focus now on the second part of the statement.
	Let $\gamma, \gamma' \in U$.
	Assume that $\pi(\gamma\gamma')$ belongs to $\bar U$.
	By construction, $\pi(\gamma\gamma')$ has a unique pre-image in $U$, say $\gamma_0$.
	Hence $\gamma_0^{-1}\gamma\gamma'$ is an element of $K$ which moves $x$ by at most $3\rho/100$.
	By \autoref{res: small cancellation}~\ref{enu: small cancellation - translation kernel}, $\gamma_0^{-1}\gamma\gamma'$ is trivial, hence $\gamma\gamma'$ belongs to $U$.
\end{proof} 

\begin{prop}
\label{res: sc - lifting morphism}
	Let $F$ be the free group generated by a finite set $U$.
	Let $\ell \in \N \setminus\{0\}$ and $V\subset F$ be the closed ball of radius $\ell$ (for the word metric relative to $U$).
	Let $\phi \colon F \to \bar \Gamma$ be a morphism whose image does not fix a cone point in $\bar {\mathcal C}$.
	Assume that  
	\begin{equation*}
		\lambda_\infty\left(\phi, U\right) < \frac {\rho}{100\ell},
	\end{equation*}
	(the energy is measured with the metric of $\bar X$).
	Then there exists a map $\tilde \phi \colon F \to \Gamma$ such that $\phi = \pi \circ \tilde \phi$ and $V \cap \ker \phi = V \cap \ker \tilde \phi$.
\end{prop}

\begin{proof}
    By definition of the energy, there exist $\bar x \in \bar X$ such that for every $g \in U$, 
    \begin{equation}
    \label{eqn: sc - lifting morphism}
    	\dist{\phi(g)\bar x}{\bar x} < \frac {\rho}{100\ell}.
    \end{equation}
    Note that 
	\begin{equation*}
		d(\bar x, \bar {\mathcal C}) \geq \left(1-\frac 1{200\ell}\right)\rho \geq \frac 12 \rho.
	\end{equation*}
	Indeed if it was not the case, then by the triangle inequality, the image of $\phi$ would fix a cone point in $\bar {\mathcal C}$.
	Set $\bar W = \phi(V)$.
	It follows from (\ref{eqn: sc - lifting morphism}) that for every $\bar \gamma$ in $\bar W$, we have 
	\begin{equation*}
    	\dist{\bar \gamma\bar x}{\bar x} < \frac {\rho}{100}.
    \end{equation*}
    By \autoref{res: sc - lifting isometries}, there exists a finite subset $W \subset \Gamma$ such that  
	\begin{enumerate}
		\item \label{enu: sc - lifting morphism - isom}
		the epimorphism $\pi \colon \Gamma \to \bar \Gamma$ induces a bijection from $W$ onto $\bar W$;
		\item \label{enu: sc - lifting morphism - morphism}
		for every $\gamma, \gamma' \in W$, if $\pi(\gamma\gamma')$ belongs to $\bar W$, then $\gamma\gamma' \in W$.
	\end{enumerate}
	We define a morphism $\tilde \phi \colon F \to \Gamma$ by sending $g \in U$ to the unique pre-image of $\phi(g)$ in $W$.
	In particular, $\pi \circ \tilde \phi = \phi$.
	If follows from \ref{enu: sc - lifting morphism - morphism} that $\tilde \phi(g) \in W$, for every $g \in V$.
	Combined with \ref{enu: sc - lifting morphism - isom} wet get $V \cap \ker \phi = V \cap \ker \tilde \phi$.
\end{proof} 

\autoref{res: small cancellation} can be used to lift small scale pictures.
More generally  any quasi-convex of $\bar X$ avoiding the cone points lifts in $\dot X$.
This is the purpose of the next statement which is a combination of Lemmas~4.17 and 4.18 of \cite{Coulon:2018vp}.
See also \cite[Proposition~3.21]{Coulon:2014fr}

\begin{lemm}
\label{res: sc - lifting quasi-convex}
	Let $\bar Z$ be a subset of $\bar X$ such that for every $\bar z, \bar z' \in \bar Z$, for every $\bar c \in \bar{\mathcal C}$, we have $\gro{\bar z}{\bar z'}{\bar c} > 13 \bar \delta$.
	Let $\bar z_0$ be a point of $\bar Z$ and $z_0$ a pre-image of $\bar z_0$ in $\dot X$.
	Then there is a unique subset $Z$ of $\dot X$ containing $z_0$ such that the map $\dot X \to \bar X$ induces an isometry from $Z$ onto $\bar Z$.
	Moreover the following holds.
	\begin{enumerate}
		\item \label{enu: sc - lifting quasi-convex - prestab}
		For every $z_1, z_2 \in Z$, for every $\bar \gamma \in \bar \Gamma$, if $\bar \gamma \bar z_1 = \bar z_2$, then there exists a unique pre-image $\gamma \in \Gamma$ of $\bar \gamma$ such that $\gamma z_1 = z_2$.
		Moreover for every $z,z' \in Z$, if $\bar \gamma \bar z = \bar z'$ then $\gamma z = z'$.
		\item \label{enu: sc - lifting quasi-convex - stab}
		The projection $\Gamma \onto \bar \Gamma$ induces an isomorphism from $\stab Z$ onto $\stab{\bar Z}$.
	\end{enumerate}
\end{lemm}

\autoref{res: sc - lifting quasi-convex} holds in particular if $\bar Z$ is an $\alpha$-quasi-convex subset of $\bar X$ such that $d(\bar c, \bar Z) > \alpha + 13 \bar \delta$, for every $\bar c \in \bar {\mathcal C}$.
To handle configurations lying in the neighborhood of the cone point, we use the following statement.

\begin{prop}
\label{res: quotient map isom around cone point}
	Let $(H,Y) \in \mathcal Q$.
	Let $c$ be the apex of $Z(Y)$ and $\bar c$ its image in $\bar X$.
	The map $f \colon \dot X \to \bar X$ induces an isometry from $B(c, \rho - 2\dot \delta)/H$ onto $B(\bar c, \rho -2 \bar \delta)$.
\end{prop}

\begin{rema}
	Recall that we chose the hyperbolicity constants so that $\dot \delta = \bar \delta$.
	For the experts, one can prove using the Greendlinger lemma that $B(c, \rho)/H$ and $B(\bar c, \rho)$ are isometric.
	However the proof below is much more elementary and illustrates the benefits of the geometry of $\bar X$.
\end{rema}

\begin{proof}
	Since $\mathcal C$ is $2\rho$-separated, it follows from the triangle inequality that 
	\begin{enumerate}
		\item $\dist cx = \dist{\bar c}{\bar x}$, for every $x \in B(c, \rho)$;
		\item $f$ maps $B(c, \rho)$ onto $B(\bar c,\rho)$;
		\item if $x,y \in B(c,\rho)$ have the same image in $\bar X$ then there exists $h \in H$ such that $y = hx$.
	\end{enumerate}
	In particular, $f$ maps $B(c, \rho - 2\dot\delta)$ onto $B(\bar c,\rho- 2\bar \delta)$.
	Let $x,y \in B(c, \rho - 2\dot \delta)$ and $\bar x, \bar y$ their respective images in $\bar X$.
	Assume first that $\gro{\bar x}{\bar y}{\bar c} = 0$, i.e.
	\begin{equation*}
		\dist{\bar x}{\bar y} = \dist{\bar x}{\bar c} + \dist{\bar c}{\bar y} = \dist xc + \dist cy.
	\end{equation*}
	Since $f$ is $1$-Lipschitz, it follows that $\dist xy = \dist{\bar x}{\bar y}$.
	Assume now that $\gro{\bar x}{\bar y}{\bar c} > 0$.
	For every sufficiently small $\eta > 0$, there exists a $(1, \eta)$-quasi-geodesic $\bar \sigma \colon \intval ab \to \bar X$ joining $\bar x$ to $\bar y$ such that $\bar c$ does not lie on $\bar \sigma$.
	Since balls in $\bar X$ are $2\bar\delta$-quasi-convex, we may assume also that $\bar \sigma$ is contained in $B(\bar c, \rho)$.
	Recall that $\dot X \setminus \mathcal C \to \bar X \setminus \bar {\mathcal C}$ is a covering map and a local isometry \autoref{res: small cancellation}~\ref{enu: small cancellation - translation kernel}.
	Let $\sigma \colon \intval ab \to \dot X$ be the path starting at $x$ and lifting $\bar \sigma$.
	By construction $\sigma$ lies in $B(c,\rho)$. 
	Hence its endpoint, which is also a pre-image of $\bar y$ can be written $hy$ for some $h \in H$.
	In addition $\sigma$ and $\bar \sigma$ have the same length, thus
	\begin{equation*}
		\dist x{hy} \leq L(\sigma) \leq L(\bar \sigma) \leq \dist {\bar x}{\bar y} + \eta.
	\end{equation*}
	In both cases we proved that for every $\eta > 0$, there exist $h \in H$ such that $\dist x{hy} \leq \dist{\bar x}{\bar y}+\eta$, hence $f$ induces an isometry from $B(c, \rho - 2\dot\delta)/H$ onto $B(\bar c,\rho- 2\bar \delta)$.
\end{proof} 

%
\subsubsection{Isometries of $\bar X$}
%
\label{sec: isom bar X}

In this section we study elementary subgroups of $\bar \Gamma$.
The goal is to get a control in the spirit of Axioms~\ref{enu: family axioms - elusive}-\ref{enu: family axioms - loxodromic} provided the same holds for $\Gamma$.
To that end we make the following additional assumptions (we refer the reader to \autoref{def: thin isom} for the definition of thin isometries).

\begin{assu}\
	\begin{itemize}
		\item Both $\nu(\Gamma, X)$ and $A(\Gamma, X)$ are finite.
		\item $\Gamma$ has no even torsion.
		\item Every elementary subgroup of $\Gamma$ which is not elliptic is loxodromic and cyclic.
		We say that a loxodromic element $\gamma \in \Gamma$ is \emph{primitive} if it generates the maximal loxodromic subgroup containing it.
		\item There exists $\alpha \in \R_+$ with the following property: for every $(H, Y) \in \mathcal Q$, there is an $\alpha$-thin primitive element $\gamma \in \Gamma$ generating $\stab Y$ and an odd integer $n \geq 100$ such that $H = \group{\gamma^n}$.
	\end{itemize}
\end{assu}

According to \autoref{res: small cancellation}~\ref{enu: small cancellation - local embedding} it implies that $\stab{\bar c}$ is cyclic for every $\bar c \in \bar{\mathcal C}$.
It also follows from these assumption that the action of $\Gamma$ on $X$ is weakly properly discontinuous (WPD), see \cite{Bestvina:2002dr} for the definition.
In particular, all the results of \cite[Section~5]{Coulon:2016if} apply here.

\paragraph{Elliptic subgroups of $\bar \Gamma$.}

\begin{lemm}[{\cite[Proposition~5.18]{Coulon:2016if}}]
\label{res: lifting elliptic subgroups}
	Let $\bar E$ be an elliptic subgroup of $\bar \Gamma$ (for its action on $\bar X$).
	One of the following holds.
	\begin{itemize}
		\item  There exists an elliptic subgroup $E$ of $\Gamma$ (for its action on $X$) such that the projection $\Gamma \onto \bar \Gamma$ induces an isomorphism from $E$ onto $\bar E$.
		\item There exists $\bar c \in \bar{\mathcal C}$ such that $\bar E \subset \stab {\bar c}$.
	\end{itemize}
\end{lemm}

\begin{lemm}
\label{res: sc - axis small elliptic}
	Assume that there exists $\beta \in \R_+$ such that every non-trivial elliptic element $\gamma \in \Gamma$ is $\beta$-thin (for its action on $X$).
	Let $\bar \gamma \in \bar \Gamma$ be an elliptic element.
	If $\bar \gamma$ does not fix a cone point in $\bar{\mathcal C}$, then $\bar \gamma$ is $\bar \beta$-thin, where  $\bar \beta = \beta + \pi\sinh(10 \bar \delta)$.
\end{lemm}

\begin{proof}
	For simplicity we let $\bar Z = \fix{\bar \gamma, 6\bar \delta}$.
	Since $\bar \gamma$ does not fix any apex in $\bar {\mathcal C}$, if follows from the triangle inequality that $d(\bar c,\bar Z) \geq \rho - 3\bar \delta$, for every $\bar c \in \bar{\mathcal C}$.
	By \autoref{res: sc - lifting quasi-convex}, there exists a subset $Z$ of $\dot X$ such that the map $f \colon \dot X \to \bar X$ induces an isometry from $Z$ onto $\bar Z$;
	moreover the projection $\pi \colon \Gamma \to \bar \Gamma$ induces an isomorphism from $\stab Z$ onto $\stab {\bar Z}$.
	We denote by $\gamma \in \Gamma$ the pre-image of $\bar \gamma$ in $\stab Z$.
	Let $\bar E = \group{\bar \gamma}$.
	Fix a point $\bar z_0 \in \fix{\bar E, 6\bar\delta}$.
	Note that $\bar E\bar z_0 \subset \bar Z$.
	Let $z_0$ be its pre-image in $Z$ and $y_0$ the radial projection of $z_0$ onto $X$ so that $\dist[\dot X] {z_0}{y_0} \leq 3 \dot \delta$.
	It follows from (\ref{eqn: appendix - lower bound dist dot X}) and the definition of $\bar z_0$, that for every $k \in \Z$, we have 
	\begin{equation*}
		\mu\left(\dist[X]{\gamma^ky_0}{y_0}\right) 
		\leq \dist[\dot X]{\gamma^ky_0}{y_0} 
		\leq \dist[\dot X]{\gamma^kz_0}{z_0}+6\dot\delta 
		\leq \dist{\bar \gamma^k \bar z_0}{\bar z_0} + 6\bar \delta 
		\leq 12 \bar \delta.
	\end{equation*}
	Thus 
	\begin{equation*}
		\dist[X]{\gamma^ky_0}{y_0}\leq \pi \sinh(6 \bar \delta),
	\end{equation*}
	see for instance \cite[Proposition~4.4]{Coulon:2014fr}.
	In particular, the $\group \gamma$-orbit of $y_0$ is bounded, hence $\gamma$ is elliptic (for its action on $X$).
	According to our assumption $\gamma$ is $\beta$-thin at some point $x \in X$.
		
	Let $\bar x$ be the image of $x$ in $\bar X$.
	We now claim that $\bar Z$ is contained in ${B(\bar x, \bar \beta - 2\bar \delta)}$.
	Let $\bar z \in Z$.
	Let $z$ be its pre-image in $Z$ and $y \in X$ the radial projection of $z$.
	The same computation as above shows that the point $y \in X$ is moved by $\gamma$ by at most $\pi\sinh(6\bar \delta)$.
	By the very definition of thinness we get $\dist[X] xy \leq \beta + \pi \sinh(6\bar \delta)$.
	Recall that the map $f \colon X \to \bar X$ is $1$-Lipschitz, we have
	\begin{equation*}
		\dist{\bar x}{\bar z}
		\leq \dist[X]xy + \dist[\dot X]yz
		\leq  \beta + \pi \sinh(6\bar \delta) + 6 \bar \delta \leq \bar \beta - 2\bar \delta.
	\end{equation*}
	This inequality holds for every $\bar z \in \bar Z$, which completes the proof of our claim.
	The result now follows from \autoref{res: local-to-global thinness}.
\end{proof} 

Let $(H,Y) \in \mathcal Q$.
Let $c$ be the apex of the cone $Z(Y)$ and $\bar c$ its image in $\bar X$.
We now focus on the behavior of $\stab{\bar c}$.
According to our assumptions, $\stab Y$ is generated by an $\alpha$-thin primitive element $h \in \Gamma$ and $H = \group {h^n}$ for some odd integer $n \geq 100$.
It follows from \autoref{res: small cancellation}~\ref{enu: small cancellation - local embedding} that $\pi \colon \Gamma \onto \bar \Gamma$ induces an isomorphism from $\stab Y/H \equiv \Z/n\Z$ onto $\stab{\bar c}$.
We think of $\stab{\bar c}$ acting on $B(\bar c, \rho)$ as a rotation group acting on a hyperbolic disc.
Let us make this idea more rigorous.

\paragraph{Angle cocycle.}
Let $(z_n)$ be a sequence points in $X$ converging to $h^+$ (the unique attractive point of $h$ in $\partial X$).
Fix an non-principal ultra-filter $\omega$.
We define a ``Busemann cocycle''  $b_0 \colon X \times X \to \R$ at $h^+$ by
\begin{equation*}
	b_0(x,x') = \limo \left(\dist x{z_n} - \dist {x'}{z_n}\right), \quad \forall x,x' \in X
\end{equation*}
Although $h^+$ is fixed by $\stab Y$, the cocycle $b_0$ may not be $\stab Y$-invariant.
Nevertheless $\stab Y$ is amenable.
Proceeding as in \cite[Section~4.3]{Coulon:2018vp} we can average the orbit of $b_0$ under $\stab Y$ to get a new cocycle $b \colon X \times X \to \R$, with the following properties
\begin{enumerate}
	\item $\abs{b(x,x') - b_0(x,x')} \leq 6\delta$, for every $x,x' \in X$,
	\item $b(\gamma x, \gamma x') = b(x,x')$, for every $x,x' \in X$ and $\gamma  \in \stab Y$,
	\item $\abs{b(x,\gamma  x)} = \snorm \gamma $, for every $x \in X$ and $\gamma  \in \stab Y$.
\end{enumerate}
We rescale $b$ to get a $\stab Y$-invariant cocycle $\theta_c \colon Y \times Y \to \R$ given by 
\begin{equation*}
	\theta_c(y,y') = \frac 1 {\sinh \rho} b(y,y'), \quad \forall y,y' \in Y,
\end{equation*}
which we extend radially to a cocycle 
\begin{equation*}
	\theta_c \colon Z(Y)\setminus\{c\} \times Z(Y)\setminus\{c\} \to \R
\end{equation*}
as follows:
if $x = (y,r)$ and $x' = (y',r')$ are two points of $Z(Y)\setminus\{c\}$, then $\theta_c (x,x') = \theta_c(y,y')$.
Let 
\begin{equation*}
	\Theta_{\bar c} = \frac{n\snorm h}{\sinh \rho}
\end{equation*}
It follows from the small cancellation assumption that $\Theta_{\bar c} \geq 9\pi$.
By construction $\theta_c$ induces a $\stab{\bar c}$-invariant cocycle 
\begin{equation*}
	\theta_{\bar c} \colon \mathring B(\bar c, \rho) \times \mathring B(\bar c, \rho) \to \R/ \Theta_{\bar c} \Z.
\end{equation*}
Recall that $\stab Y$ is assumed to be cyclic.
Hence the map $\stab{\bar c} \to  \R/ \Theta_{\bar c} \Z$ sending $\bar \gamma$ to $\theta_{\bar c}(\bar \gamma \bar x, \bar x)$ is an injective homomorphism, which does not depend on the point $\bar x \in \mathring B(\bar c, \rho)$.

\paragraph{Comparison map.}
Consider the action of $\group h$ on $\R$ by translation of length $\snorm h$.
The quotient $\R / H$ is isometric to a circle whose length is $n\snorm h$.
We denote by $\mathcal D$ the cone of radius $\rho$ over $\R/ H$ and by $o$ its apex.
Observe that $\mathcal D$ is a hyperbolic cone (i.e. locally isometric to $\H^2$ everywhere except maybe at the cone point $o$) whose total angle at $o$ is $\Theta_{\bar c}$.

Let $y_0$ be an arbitrary point in $Y$.
One checks that the map $\phi \colon Y \to \R$ sending $y$ to $b(y_0, y)$ is a  $\stab Y$-equivariant $(1, 150\delta)$-quasi-isometric embedding, see Coulon \cite[Section~4.3]{Coulon:2018vp}.
Hence it induces a $\stab {\bar c}$-equivariant $(1, 150\delta)$-quasi-isometric embedding
\begin{equation*}
	\bar \phi \colon Z(Y/H) \to \mathcal D.
\end{equation*}
Note that $Z(Y/H)$, which is also isometric to $Z(Y)/H$, is endowed here with the metric defined by (\ref{eqn: sc - metric cone}).
Although $Z(Y)/H$ can be identified, as a set of points, with the closed ball of radius $\rho$ centered at $\bar c$, the metric may be slightly different.
Nevertheless, in view of Propositions~\ref{res: qi cone in cone-off} and \ref{res: quotient map isom around cone point} the map $\bar \phi \colon B(\bar c, \rho) \to \mathcal D$ induced by $\bar \phi \colon Z(Y/H) \to \mathcal D$ is a $(1,10\bar \delta)$-quasi-isometric embedding.
Moreover, for every $\bar x \in B(\bar c, \rho)$, we have $\dist{\bar c}{\bar x} = \dist o {\bar \phi(\bar x)}$.
As a consequence we get the following statement.

\begin{prop}
\label{res: sc - angle cocycle}
	For every $\bar x, \bar x' \in \mathring B(\bar c, \rho)$, we have  $\abs{\dist {\bar x}{\bar x'} - \ell} \leq 10 \bar \delta$, where
	\begin{equation*}
		\cosh \ell = \cosh\dist {\bar c}{\bar x}\cosh\dist {\bar c}{\bar x'} - \sinh\dist {\bar c}{\bar x}\sinh\dist {\bar c}{\bar x'}  \cos\left( \min \left\{ \pi, \tilde \theta_{\bar c}(\bar x,\bar x') \right\}\right)
	\end{equation*}
	and $ \tilde \theta_{\bar c}(\bar x,\bar x')$ is the unique representative of $\theta_{\bar c}(\bar x,\bar x')$ in $(-\Theta_{\bar c}/2, \Theta_{\bar c}/2]$.
	Moreover, $\theta_{\bar c}(\bar \gamma \bar x, \bar x) \neq 0$, for every $\bar \gamma \in \stab {\bar c}\setminus \{1\}$, and $\bar x \in \mathring B(\bar c,\rho)$.
\end{prop}

\begin{coro}
\label{res: sc - unique fixed cone point}
	Let $\bar c \in \bar{\mathcal C}$.
	If $\bar \gamma \in \stab{\bar c}$ is non-trivial, then there exists $q \in \Z$, such that $\fix{\bar \gamma^q, 6\bar \delta}$ is contained in $B(\bar c, 8\bar \delta)$.
	In particular, $\bar c$ is the unique cone point fixed by $\bar \gamma$.
\end{coro}

\begin{proof}
	Since $\bar \gamma$ is non-trivial, there exists $p \in \Z \setminus n \Z$ such that $\bar \gamma = \bar h^p$.
	As $n$ is odd, the greatest common divisor $d$ of $p$ and $n$ satisfies $1 \leq d \leq n/3$.
	On the other hand $\Theta_{\bar c} \geq 9\pi$, so that
	\begin{equation*}
		\frac {\Theta_{\bar c}}2 - \pi  > \frac {\Theta_{\bar c}}3 \geq \frac{d\snorm h}{\sinh \rho}.
	\end{equation*}
	Therefore, there exists $q,m \in \Z$ such that 
	\begin{equation*}
		 \pi \leq (qp + mn) \frac{\snorm h}{\sinh \rho} \leq \frac {\Theta_{\bar c}}2.
	\end{equation*}
	In other words $\bar \gamma^q$ acts on $\mathcal D$ as rotation whose angle is at least $\pi$.
	By \autoref{res: sc - angle cocycle}, $\fix{\bar \gamma^q, 6\bar \delta} \cap B(\bar c, \rho)$  is contained in $B(\bar c, 8\bar \delta)$.
	However $\fix{\bar \gamma^q, 6\bar \delta}$ is $8\bar \delta$-quasi-convex.
	Thus it cannot contain a point outside $B(\bar c, \rho)$ without violating the previous conclusion.
	Consequently $\fix{\bar \gamma^q, 6\bar \delta}$  is contained in $B(\bar c, 8\bar \delta)$.
\end{proof} 

\begin{coro}
\label{res: sc - unique fixed cone point consequence}
	Let $\bar c \in \bar{\mathcal C}$.
	Let $\bar \gamma_0 \in \stab{\bar c} \setminus \{1\}$ and $\bar \gamma \in \bar \Gamma$.
	If $ \bar \gamma \bar \gamma_0 \bar \gamma^{-1}$ fixes $\bar c$, then so does $\bar \gamma$.
\end{coro}

\begin{proof}
	For simplicity we write $\bar \gamma_1 =  \bar \gamma \bar \gamma_0 \bar \gamma^{-1}$ so that $\bar \gamma_1$ fixes $\bar c$ and $\bar \gamma \bar c$.
	However by \autoref{res: sc - unique fixed cone point}, $\bar \gamma_1$ cannot fixed two distinct cone points hence $\bar \gamma \bar c = \bar c$.
\end{proof}

\begin{prop}
\label{res: sc - fix point set conical contained in a ball}
	Let $\bar c \in \bar{\mathcal C}$.
	Every non-trivial element $\bar \gamma \in \stab{\bar c}$ is $(\rho + \bar \beta)$-thin at $\bar c$, where $\bar \beta = \alpha + \pi \sinh(20 \bar \delta)$.
\end{prop}

\begin{proof}
	Let $\bar \gamma$ be a non-trivial element of $\stab{\bar c}$.
	In view of \autoref{res: local-to-global thinness} it suffices to prove that $\fix{\bar \gamma, 6\bar\delta}$ is contained in $B(\bar c, \rho + \bar \beta - 2\bar \delta)$.
	Let $\bar x \in \fix{\bar \gamma, 6\bar \delta}$.
	Without loss of generality we can assume that $\bar x$ does not belongs to $B(\bar c, \rho)$.
	Let $\bar \sigma \colon \intval 0a \to \bar X$ be a $(1, \bar \delta)$-quasi-geodesic from $\bar c$ to $\bar x$.
	Using the thinness of triangles we observe that $\dist{\bar \gamma\bar \sigma(t)}{\bar \sigma(t)} \leq 12\bar \delta$, for every $t \in \intval 0a$, see for instance \cite[Lemma~2.2(3)]{Coulon:2014fr}. 
	For simplicity we let $b = \rho/2$ and $\bar z = \bar \sigma(b)$.
	We write $\bar \sigma_0$ for the path $\bar \sigma$ restricted to $\intval ba$.
	Let $(H,Y)$ such that $\bar c$ is the image of the apex $c$ of $Z(Y)$.
	We fix a pre-image $z \in Z(Y)$ of $\bar z$.
	
	Since $\bar \sigma_0$ is a quasi-geodesic, it is contained in $\bar X \setminus B(\bar c, \rho/2 - \bar \delta)$.
	Observe also that $d(\bar c', \bar \sigma_0) \geq \rho - 6\bar \delta$ for every $\bar c' \in \bar{\mathcal C}\setminus\{\bar c\}$.
	Indeed $\bar \gamma$ hardly moves the points on $\bar \sigma_0$.
	By the triangle inequality, any other cone point closed to $\bar \sigma_0$ should be fixed by $\bar \gamma$, which contradicts \autoref{res: sc - unique fixed cone point}.
	We now apply \autoref{res: sc - lifting quasi-convex} with the $12\bar \delta$-neighborhood of $\bar \sigma_0$ as the set $\bar Z$.
	In particular, there exists a $(1,\dot \delta)$-quasi-geodesic $\sigma_0 \colon \intval ba \to \dot X$ starting at $z$ and lifting $\bar \sigma_0$ and a pre-image $\gamma \in \Gamma$ of $\bar \gamma$ such that $\dist[\dot X]{\gamma \sigma_0(t)}{\sigma_0(t)} \leq 12\dot \delta$, for every $t \in \intval ba$. 
	In particular, $\dist[\dot X]{\gamma z}z \leq 12\dot \delta$.
	It follows from the triangle inequality that $\gamma$ necessarily fixes $c$, thus is a non-trivial element from $\stab Y = \group h$.
	The point $x = \sigma_0(b)$ is pre-image of $\bar x$.
	We observed earlier that $d(\bar x, \bar {\mathcal C}) \geq \rho - 6\bar \delta$.
	Hence the radial projection $y$ of $x$ onto $X$ satisfies $\dist[\dot X] xy \leq 6\dot \delta$.
	Using (\ref{eqn: appendix - lower bound dist dot X}) we get
	\begin{equation*}
		\mu \left( \dist[X]{\gamma y}y \right)
		\leq \dist[\dot X]{\gamma y}y \leq \dist[\dot X]{\gamma x}x+ 2\dist xy \leq 24\dot \delta < 2\rho.
	\end{equation*}
	Hence $\dist[X]{\gamma y}y \leq D$ where $D = \pi \sinh (12\dot\delta)$, see for instance \cite[Proposition~4.4]{Coulon:2014fr}.
	Since $\gamma$ is $\alpha$-thin, $\gro{h^-}{h^+}y \leq D/2 + \alpha$.
	Hence $y$ lies in the $(D/2 + \alpha)$-neighborhood of $Y$.
	Recall that any point of $Y$ is at a distance $\rho$ from $c$ in $\dot X$.
	Since the maps $X \to \dot X$ and $\dot X \to \bar X$ are $1$-Lipschitz, we obtain with the triangle inequality that $\bar x$ belongs to $B(\bar c, \rho + \bar \beta - 2\bar \delta)$.
\end{proof} 

\paragraph{Parabolic subgroups of $\bar \Gamma$.}

\begin{lemm}
\label{res: sc - no parabolic}
	The group $\bar \Gamma$ has no parabolic subgroup.
\end{lemm}

\begin{proof}
	It follows from \cite[Proposition~5.25]{Coulon:2016if} that every parabolic subgroup of $\bar \Gamma$ is the image of a parabolic subgroup of $\Gamma$ for its action on $X$.
	However, according to our assumption, this action does not admit parabolic subgroups.
\end{proof}

\paragraph{Loxodromic subgroups of $\bar \Gamma$.}

\begin{lemm}[{\cite[Proposition~5.26]{Coulon:2016if}}]
\label{res: sc - lifting loxo sbgp}
	Let $\bar E$ be a loxodromic subgroup of $\bar \Gamma$.
	There exists a loxodromic subgroup $E$ of $\Gamma$ (for its action on $X$) such that the projection $\Gamma \onto \bar \Gamma$ induces an isomorphism from $E$ onto $\bar E$.
\end{lemm}

Combining the previous statement with our assumptions, we get that every loxodromic subgroup of $\bar \Gamma$ is cyclic.

\begin{lemm}
\label{res: approx - thin loxodromic}
	Assume that there exists $\beta \in \R_+$ such that every loxodromic element of $\Gamma$ is $\beta$-thin (for its action on $X$).
	Then every loxodromic element of $\bar \Gamma$ is $\bar \beta$-thin, where $\bar \beta = \beta + 10\pi \sinh (100\bar \delta)$.
\end{lemm}

\begin{proof}	
	Let $\bar \gamma$ be a loxodromic element of $\bar \Gamma$.
	If $\snorm{\bar \gamma} >  \bar \delta$, then $\bar \gamma$ is $200 \bar \delta$-thin, by \autoref{rem: thin isom - loxo}.
	Therefore we can assume that $\snorm{\bar \gamma}\leq  \bar \delta$, hence $\norm{\bar \gamma} \leq 10 \bar \delta$.
	For simplicity we let $\bar Z = \fix{\bar \gamma, 100\bar \delta}$.
	In addition we denote by $\bar \sigma \colon \R \to \bar X$ an $L$-local $(1, \bar \delta)$-quasi-geodesic from $\gamma^-$ to $\gamma^+$, with $L > 12\bar \delta$.
	The path $\bar \sigma$ is contained in $\bar Z$, see for instance \cite[Lemma~2.9]{Coulon:2018vp}.
	Compare also with \cite[Proposition~2.28]{Coulon:2014fr}.
	According to \autoref{res: local-to-global thinness}, it suffices to prove that $\bar Z$ lies in the $(\bar \beta - 6\bar\delta)$-neighborhood of $\bar \sigma$.
	
	Being loxodromic $\bar \gamma$ does not fix any apex in $\bar {\mathcal C}$.
	It follows from the triangle inequality that $d(\bar c, \bar Z) \geq \rho  - 50\bar \delta$, for every $\bar c \in \bar{\mathcal C}$.
	According to \autoref{res: sc - lifting quasi-convex}, there exists a subset $Z$ of $\dot X$ such that the map $f \colon \dot X \to \bar X$ induces an isometry from $Z$ onto $\bar Z$.
	Moreover the projection $\pi \colon \Gamma \to \bar \Gamma$ induces an isomorphism from $\stab Z$ onto $\stab {\bar Z}$.
	We denote by $\gamma \in \Gamma$ the pre-image of $\bar \gamma$ in $\stab Z$.
	Since $\bar \gamma$ is loxodromic, so is $\gamma$.
	It follows from the construction that $Z$ is contained in $\fix{\gamma, 100 \dot \delta}$.
	
	Let $\bar z \in \bar Z$ and $z$ its pre-image in $Z$.
	Since $\bar Z$ stays away from the cone points, the radial projection $y$ of $z$ is such that $\dist[\dot X]yz \leq 50 \dot \delta$.
	It follows that 
	\begin{equation*}
		\mu\left(\dist[X]{\gamma y}y\right)
		\leq \dist[\dot X]{\gamma y}y 
		\leq \dist[\dot X]{\gamma z}z +  100\dot \delta
		\leq \dist{\bar \gamma\bar z}{\bar z} +  100\bar \delta
		\leq 200\bar \delta <2\rho.
	\end{equation*}
	Thus $\dist[X]{\gamma y}y \leq \pi \sinh(100\bar \delta)$.
	Since $\gamma$ is $\beta$-thin we get
	\begin{equation*}
		\gro{\gamma^-}{\gamma^+}y \leq \beta + D/2, \quad \text{where} \quad D = \pi \sinh(100\bar \delta).
	\end{equation*}
	(the Gromov product is computed in $X$).
	In particular, there exists $m_0 \in \N$ such that for every integer $m \geq m_0$, we have  
	\begin{equation*}
		\gro{\gamma^{-m}y}{\gamma^my}y \leq \beta + D/2 +  3\delta,
	\end{equation*}		
	Geodesic triangles in a hyperbolic space are thin.
	We do not assumed here that $\dot X$ is geodesic though.
	Nevertheless it is a length space.
	Hence there is a point $x \in \dot X$ such that 
	\begin{equation*}
		\max\left\{\gro y{\gamma^{-m}y}x, \gro{\gamma^{-m}y}{\gamma^m y}x, \gro{\gamma^m y}yx \right\} \leq 2\dot \delta.
	\end{equation*}
	where all Gromov products are computed in $\dot X$.
	
	We claim that there exists a point $w \in X$ such that $\dist[\dot X] wx \leq 151\dot \delta$.
	Combining the inequality above with the triangle inequality, we have indeed $ \gro{\gamma^m z}zx \leq 102 \dot \delta$.
	The set $\fix{\gamma, 100 \dot \delta}$ which contains $z$ and $\gamma^mz$ is $8\dot \delta$-quasi-convex.
	Hence there exists $s \in \fix{\dot \gamma, 100\dot \delta}$ such that $\dist[\dot X] xs \leq 111\dot \delta$.
	Note that $\gamma$ cannot fix a cone point from $\mathcal C \subset \dot X$.
	Consequently $s$ necessarily lies in the $50\dot \delta$-neighborhood of $X$.
	The radial projection $w$ of $s$, satisfies our claim.
	Applying the triangle inequality, we get
	\begin{equation*}
		\max\left\{\gro y{\gamma^{-m}y}w, \gro{\gamma^{-m}y}{\gamma^m y}w, \gro{\gamma^m y}yw \right\} \leq 153\dot \delta
	\end{equation*}
	(the Gromov products are still computed in $\dot X$).
	It follows that the same Gromov products, computed in $X$ this time, are bounded above by 
	\begin{equation*}
		d = \frac {\sinh \rho}{\sinh(\rho - 157\dot \delta)} \ \pi \sinh ( 83\dot \delta)
	\end{equation*}
	see \cite[Proposition~2.16]{Coulon:2018ac}.
	In particular, $\dist[X]yw \leq \gro{\gamma^{-m}y}{\gamma^my}y + 2d$ where the Gromov product in computed in $X$. 
	Hence $\dist[X]yw \leq \beta + D/2 +  2d +  3\delta$.
		
	Using the triangle inequality, we get an upper bound for the following Gromov product computed in $\dot X$
	\begin{eqnarray*}
		\gro{\gamma^{-m}y}{\gamma^my}y
		\leq
		\dist[\dot X]yx + 2\dot \delta
		& \leq & \dist[\dot X] yw + 153\dot\delta \\
		& \leq & \dist[X] yw + 153\dot\delta \\
		& \leq & \beta + D/2 +  2d +  154\dot\delta.
	\end{eqnarray*}
	Applying one more time the triangle inequality we get (in $\dot X$)
	\begin{equation*}
		\gro{\gamma^{-m}z}{\gamma^mz}z \leq \gro{\gamma^{-m}y}{\gamma^my}y + 150\dot \delta \leq \beta + D/2 +  2d +  304\dot \delta.
	\end{equation*}
	Since $f \colon \dot X \to \bar X$ induces an isometry from $Z$ onto $\bar Z$, we get that for every $m \geq m_0$, 
	\begin{equation*}
		\gro{\bar \gamma^{-m}\bar z}{\bar \gamma^m\bar z}{\bar z} 
		\leq \beta + D/2 +  2d +  304\bar \delta.
	\end{equation*}
	Consequently
	\begin{equation*}
		\gro{\bar \gamma^-}{\bar \gamma^+}{\bar z} \leq  \beta + D/2 +  2d +  304\bar \delta.
	\end{equation*}
	Hence $\bar z$ lies in the $M$-neighborhood of $\bar \sigma$ where
	\begin{equation*}
		M =  \beta + D/2 +  2d +  311\bar \delta 
	\end{equation*}
	Recall that we chose $\rho \geq \rho_0$ with $\tanh(\rho_0) \geq 1/2$.
	It follows that $d \leq 4 \pi \sinh(83 \dot \delta)$.
	Hence $\bar z$ belongs to the $(\bar \beta - 6\bar\delta)$-neighborhood of $\bar \sigma$.
	This holds for every $\bar z \in \bar Z$, which completes the proof.
\end{proof} 

%
\subsubsection{The action of $\bar \Gamma$}
%

In the previous section we investigated the isometries of $\bar X$.
We now briefly review some properties of the action of $\bar \Gamma$ on $\bar X$.

\begin{prop}[{\cite[Proposition~5.15]{Coulon:2016if}}]
\label{res: approx - non elem action}
	The action of $\bar \Gamma$ on $\bar X$ is non-elementary.
\end{prop}

A careful analysis of the geometry of $\bar X$ allows us to control the invariants $A(\bar \Gamma, \bar X)$, $\nu(\bar \Gamma, \bar X)$, and $\inj[\bar X]{\bar \Gamma}$.

\begin{prop}[{\cite[Corollary~5.32]{Coulon:2016if}}]
\label{res: approx - inj radius}
	Let $\ell$ be the infimum over the stable translation lengths (measured in $X$) of loxodromic elements of $\Gamma$ which do not belong to $\stab Y$ for some $(H,Y) \in \mathcal Q$.
	Then $\inj[\bar X]{\bar \Gamma} \geq \min\{ \kappa \ell, \bar \delta\}$ where $\kappa = \bar \delta/\pi \sinh(26 \bar \delta)$.
\end{prop}

\begin{prop}[{\cite[Proposition~5.28]{Coulon:2016if}}]
\label{res: approx - nu inv}
	The $\nu$-invariant is $\nu(\bar \Gamma, \bar X) = 1$.
\end{prop}

\begin{prop}
\label{res: approx - A inv}
	The invariant $A(\bar \Gamma, \bar X)$ is at most 
	$A(\Gamma,X) + 5 \pi \sinh( 1000\bar \delta)$.
\end{prop}

\begin{proof}
	Although the definition of $A(\Gamma, X)$ is slightly different, the proof works verbatim as in \cite[Proposition~5.30]{Coulon:2014fr}.
	See also \cite[Proposition~4.47]{Coulon:2018vp}.
\end{proof}

Recall that a group $G$ is \emph{commutative transitive} if it satisfies the following property: for every $g_1, g_2,g_3 \in G \setminus\{1\}$, if $[g_1,g_2] = 1$ and $[g_2, g_3] = 1$, then $[g_1, g_3] = 1$.

\begin{prop}
\label{res: sc - commutative transitive}
	If $\Gamma$ is commutative transitive, then so is $\bar \Gamma$.
\end{prop}

\begin{proof}
	Let $\bar \gamma_1, \bar \gamma_2, \bar \gamma_3 \in \bar \Gamma \setminus\{1\}$ such that $[\bar \gamma_1, \bar \gamma_2] = [\bar \gamma_2, \bar \gamma_3] = 1$.
	Suppose first that $\bar \gamma_2$ is loxodromic.
	Hence  $\bar \gamma_1$  and $\bar \gamma_3$ are contained in the maximal elementary subgroup containing $\bar \gamma_2$, which is cyclic by \autoref{res: sc - lifting loxo sbgp}. 	Hence $[\bar \gamma_1, \bar \gamma_3]  = 1$.
	
	Assume now that $\bar \gamma_2$ fixes some cone point $\bar c \in \bar{\mathcal C}$.
	It follows from \autoref{res: sc - unique fixed cone point consequence} that $\bar \gamma_1$ and $\bar \gamma_3$ fix $\bar c$ as well.
 	As we observed earlier, $\stab{\bar c}$ is cyclic, hence $[\bar \gamma_1, \bar \gamma_3]  = 1$.
 	
	Assume now that $\bar \gamma_2$ is elliptic and does not fix any cone point.
	We let $\bar Z = \fix{\bar \gamma_2, 6\bar \delta}$.
	Since $\bar \gamma_2$ does not fix a cone point, it follows from the triangle inequality that $d(\bar c, \bar Z) \geq \rho - 3 \bar \delta$, for every $\bar c \in \bar{\mathcal C}$.
	By \autoref{res: sc - lifting quasi-convex}, there is a subset $Z \subset \dot X$ such that the map $\dot X \to \bar X$ induces an isometry from $Z$ onto $\bar Z$.
	Note that $\bar \gamma_i$ preserves $\bar Z$, for every $i \in \{1,2,3\}$.
	We denote by $\gamma_1, \gamma_2, \gamma_3 \in \Gamma \setminus\{1\}$ their pre-images in $\stab Z$.
	It follows from  \autoref{res: sc - lifting quasi-convex} that $[\gamma_1, \gamma_2] = [\gamma_2, \gamma_3] = 1$.
	Since $\Gamma$ is commutative transitive, we have $[\gamma_1, \gamma_3]  = 1$ and thus $[\bar \gamma_1, \bar \gamma_3]  = 1$.
	Recall that $\bar \Gamma$ has no parabolic isometries (\autoref{res: sc - no parabolic}), hence the proof is complete.
\end{proof}

\begin{prop}
\label{res: sc - csa}
	If $\Gamma$ is CSA, then so is $\bar \Gamma$.
\end{prop}

\begin{proof}
	According to \autoref{res: sc - commutative transitive}, $\bar \Gamma$ is commutative transitive.
	Hence it suffices to prove that for every $\bar \gamma, \bar \gamma_0 \in \bar \Gamma \setminus\{1\}$ if $\bar \gamma_0$ and $\bar \gamma \bar \gamma_0 \bar \gamma^{-1}$ commute, then so do $\bar \gamma_0$ and $\bar \gamma$.
	For simplicity we let $\bar \gamma_1 = \bar \gamma \bar \gamma_0 \bar \gamma^{-1}$.
	If $\bar \gamma_0$ is loxodromic, then $\bar \gamma$ belongs to the maximal elementary subgroup containing $\bar \gamma_0$, which is cyclic by \autoref{res: sc - lifting loxo sbgp}. 	Hence $[\bar \gamma, \bar \gamma_0] = 1$.
	
	Suppose that $\bar \gamma_0$ fixes some cone point $\bar c \in \bar{\mathcal C}$.
	Note that $\bar \gamma_1 \bar \gamma_0 \bar \gamma_1^{-1} = \bar \gamma_0$ fixes $\bar c$.
	Hence by \autoref{res: sc - unique fixed cone point consequence}, $\bar \gamma_1$ fixes $\bar c$.
	Applying again \autoref{res: sc - unique fixed cone point consequence}, we see that $\bar \gamma$ fixes $\bar c$.
	As we already observed earlier, $\stab{\bar c}$ is cyclic, hence $[\bar \gamma, \bar \gamma_0] = 1$.
	
	We are left to handle the case where $\bar \gamma_0$ is elliptic, but does not fix any cone point.
	The group $\bar \Gamma$ has no parabolic subgroup (\autoref{res: sc - no parabolic}), while its loxodromic subgroups are cyclic (\autoref{res: sc - lifting loxo sbgp}). 
	Hence $\bar U = \{\bar \gamma_0, \bar \gamma_1\}$ generates an elliptic subgroup.
	Thus we can find a point $\bar z \in \fix{\bar U, 6 \bar \delta}$.
	We let $\bar Z = \fix{\bar \gamma_1, 6\bar \delta}$.
	Note that $\bar z$ and $\bar \gamma \bar z$ belong to $\bar Z$.
	Since $\bar \gamma_0$ (and thus $\bar \gamma_1$) does not fix a cone point, it follows from the triangle inequality that $d(\bar c, \bar Z) \geq \rho - 3 \bar \delta$, for every $\bar c \in \bar{\mathcal C}$.
	Let $z \in \dot X$ be a pre-image of $\bar z$.
	By \autoref{res: sc - lifting quasi-convex}, there is a subset $Z \subset \dot X$ containing $z$ such that the map $\dot X \to \bar X$ induces an isometry from $Z$ onto $\bar Z$.
	Note that both $\bar \gamma_0$ and $\bar \gamma_1$ preserves $\bar Z$.
	We denote by $\gamma_0$ and $\gamma_1$ their respective pre-images in $\stab Z$.
	By  \autoref{res: sc - lifting quasi-convex}~\ref{enu: sc - lifting quasi-convex - stab} the projection $\Gamma \onto \bar \Gamma$ induces an isomorphism from $\stab Z$ onto $\stab{\bar Z}$, hence $\gamma_0$ and $\gamma_1$ commutes.
	In addition, we fix $\gamma \in \Gamma\setminus\{1\}$ such that $\gamma z$ is the pre-image  in $Z$ of $\bar \gamma \bar z$.
	Note that $\bar \gamma_1 \bar \gamma \bar z = \bar \gamma \bar \gamma_0 \bar z$ belongs to $\bar Z$.
	By construction $\gamma_1 \gamma z$ is its pre-image in $Z$.
	It follows from \autoref{res: sc - lifting quasi-convex}~\ref{enu: sc - lifting quasi-convex - prestab} applied with $\gamma$ that $\gamma \gamma_0z = \gamma_1 \gamma z$.
	Recall that $\bar \gamma_1 \bar \gamma = \bar \gamma \bar \gamma_0$.
	Hence by \autoref{res: small cancellation}~\ref{enu: small cancellation - translation kernel}, we obtain $\gamma_1 \gamma = \gamma \gamma_0$, that is $\gamma_1 = \gamma \gamma_0 \gamma^{-1}$.
	Since $\Gamma$ is CSA, it implies that $\gamma$ commutes with $\gamma_0$.
	Consequently $\bar \gamma$ commutes with $\bar \gamma_0$.
\end{proof}

%
\subsection{Approximation sequence of periodic group}
%
\label{sec: sc - approx}

Given a group $\Gamma$ acting on a suitable hyperbolic space $X$ we are going to use small cancellation theory to approximate the periodic quotients of $\Gamma$ by a sequence of negatively curved groups.
For our applications the group $\Gamma$ or the space $X$ may vary.
In order to make all the dependencies clear we first define some auxiliary parameters.

\subsubsection{The small cancellation parameters.}
\label{sec: sc parameters}

We start by defining the parameters involved in \autoref{res: approximating sequence}.
We write $\delta_0$, $\delta_1$, $\Delta_0$ and $\rho_0$ for the constants given by \autoref{res: small cancellation}.
For simplicity we let 
\begin{equation*}
	A = 10\pi \sinh\left(10^3\delta_1\right), \quad
	\kappa = \frac{\delta_1 }{\pi \sinh (26\delta_1)}, \quad
	\quad \text{and}\quad 
	\alpha = 20 \pi \sinh(100\delta_1).
\end{equation*}
We choose once for all a non-decreasing map $\rho \colon \N \to \R_+$ which is bounded below by $\rho_0$, diverges to infinity, and satisfies
\begin{equation*}
	\rho(n) =o\left(\ln n\right).
\end{equation*}
We define a rescaling parameter by
\begin{equation*}
	\epsilon(n) = \sqrt{\frac {10\pi \sinh \rho(n)}{n \kappa \delta_1}}.
\end{equation*}
Observe that 
\begin{equation*}
	\ln \epsilon(n) = - \frac 12 \ln n + \frac 12\rho(n) + O(1).
\end{equation*}
Thus, $\epsilon(n)$ and $\epsilon(n)\rho(n)$ converge to zero.
Let $\tau \in (0,1)$.
We set
\begin{equation*}
	\tau ' = \tau \min \{A, \delta_1, \alpha\} = \tau \delta_1.
\end{equation*}
There exists an critical exponent $N_\tau \geq 100$, such that for every integer $n \geq N_\tau$, 
\begin{align}
	\label{eqn: induction param - rescale}
	\epsilon(n) & < 1, \\
	\label{eqn: induction param - delta}
	\epsilon(n) \delta_1 & \leq \delta_0, \\
	\label{eqn: induction param - Delta}
	\epsilon(n)(A + 500 \delta_1)  & \leq  \min \left\{ \Delta_0, 5\pi \sinh\left(10^3\delta_1\right)\right\}, \\
	\label{eqn: induction param - inj}
	\epsilon(n) \kappa  \delta_1 & < \tau \delta_1 \leq \min\{ \delta_1, \tau'\} \\
	\label{eqn: induction param - thinness}
	\epsilon(n)  (\rho(n) +\alpha) & \leq \delta_1
\end{align}
We now fix an odd integer $n \geq N_\tau$.
For simplicity we write $\rho = \rho(n)$ and $\epsilon = \epsilon(n)$.
The parameter $\epsilon$ will serve as a rescaling factor for certain metric spaces. 
Note that it has no connection with one used in \autoref{sec: action on limit tree} though.

\subsubsection{Approximation groups.}

We fix a $\tau$-bootstrap $(\Gamma, X)$ for the exponent $n$ (see \autoref{def: bootstrap}).
We are going to define by induction a sequence of pairs
\begin{equation}
	(\Gamma_0, X_0) \onto (\Gamma_1, X_1) \onto \dots \onto (\Gamma_j,X_j) \onto (\Gamma_{j+1}, X_{j+1}) \onto \dots
\end{equation}
where $\Gamma_j$ is a group acting by isometries on a metric length space $X_j$.
Recall that the notation 
\begin{equation*}
(\Gamma_j,X_j) \onto (\Gamma_{j+1}, X_{j+1})
\end{equation*}
means that the pair $(\Gamma_{j+1}, X_{j+1})$ comes with a projection $\pi_j \colon \Gamma_j \onto \Gamma_{j+1}$ together with a (non necessarily onto) $\pi_j$-equivariant map $f_j \colon X_j \to X_{j+1}$.
In addition this sequence will satisfy, among others, the following properties.
\begin{labelledenu}[R]
	\item \label{enu - induction - hyp space}
	$X_j$ is $\delta_1$-hyperbolic and the action of $\Gamma_j$ is a non-elementary.
	\item \label{enu - induction - elem subgroup}
	Every elliptic element of $\Gamma_j$ has finite order dividing $n$.
	Every elementary subgroup of $\Gamma_j$, which is not elliptic is loxodromic and cyclic.
	\item \label{enu - induction - sc param}
	$\nu(\Gamma_j, X_j) = 1$, $\displaystyle A(\Gamma_j,X_j) \leq A$, and $\displaystyle \inj[X_j]{\Gamma_j} > \epsilon \kappa  \delta_1$.
	\item \label{enu - induction - lip map}
	The map $f_j \colon X_j \to X_{j+1}$ is $\epsilon$-Lipschitz.
	\item \label{enu - induction - torsion}
	The kernel of the map $\Gamma_j \onto \Gamma_{j+1}$ is generated (as a normal subgroup) by all the elements of the form $\gamma^n$ where $\gamma \in \Gamma_j$ is loxodromic primitive and satisfies $\norm[X_j]\gamma \leq 10\delta_1$.
\end{labelledenu}

\paragraph{The base of the induction.}
We assumed that $(\Gamma, X)$ is $\tau$-bootstrap.
Hence, we can rescale $X$ to get a new metric space $X_0$ which is $\delta_1$-hyperbolic and such that $A(\Gamma, X_0) \leq A$ while $\inj[X_0]{\Gamma} \geq \tau'$.
Moreover we can require that every element of $\Gamma$ is $\alpha$-thin for its action on $X_0$.
Note that this rescaling does not depend on $n$.
We start by letting $(\Gamma_0, X_0) = (\Gamma, X)$.
It follows from the very definition of bootstrap and our choice of $N_\tau$ that \ref{enu - induction - hyp space} -- \ref{enu - induction - sc param} holds.

\paragraph{The induction step.}
Let $j \in \N$ and assume that the pair $(\Gamma_j, X_j)$ satisfying \ref{enu - induction - hyp space}, \ref{enu - induction - elem subgroup} and \ref{enu - induction - sc param} has been previously defined.
The construction of $(\Gamma_{j+1}, X_{j+1})$ goes as follows.
Let $R_j$ be the set of all loxodromic primitive elements $\gamma\in \Gamma_j$ such that $\norm[X_j]\gamma \leq 10\delta_1$.
Let $K_j$ be the (normal) subgroup of $\Gamma_j$ generated by $\set{\gamma^n}{\gamma \in R_j}$.
The group $\Gamma_{j+1}$ is defined as the quotient $\Gamma_{j+1} = \Gamma_j/K_j$.
In particular, the canonical projection $\pi_j \colon \Gamma_j \onto \Gamma_{j+1}$ satisfies \ref{enu - induction - torsion}.

If $R_j$ is empty, then $K_j$ is trivial and $\Gamma_{j+1} = \Gamma_j$.
In this situation we choose for $X_{j+1}$ the rescaled space $X_{j+1} = \epsilon X_j$.
It satisfies \ref{enu - induction - hyp space}-\ref{enu - induction - lip map}.

If $R_j$ is not empty, then the space $X_{j+1}$ is obtained by mean of small cancellation theory.
More precisely, we focus on the action of $\Gamma_j$ on the rescaled space $\epsilon X_j$.
According to (\ref{eqn: induction param - delta}) the space $\epsilon X_j$ is $\delta_0$-hyperbolic.
Let $\mathcal Q_j$ be the family defined by
\begin{equation*}
\mathcal Q_j = \set{\left(\group{\gamma^n}, Y_\gamma\right)}{\gamma \in R_j }
\end{equation*}
Combining (\ref{eqn: induction param - Delta}) with  our control of $A(\Gamma_j, X_j)$ and $\inj[X_j]{\Gamma_j}$ one can estimate the small cancellation parameters of $\mathcal Q_j$.
More precisely we have the following statement.

\begin{lemm}[Compare with {\cite[Lemma~6.2]{Coulon:2016if}}]
	The family $\mathcal Q_j$ satisfies the following: $\Delta(\mathcal Q_j, \epsilon X_j) \leq \Delta_0$ and $T(\mathcal Q_j, \epsilon X_j) \geq 10\pi \sinh \rho$.
\end{lemm} 

Let $\dot X_j$ be the cone-off over the \emph{rescaled} space $\epsilon X_j$ relative to $\mathcal Q_j$.
The space $X_{j+1}$ is the quotient of $\dot X_j$ by the normal subgroup $K_j$.
Recall that $\rho \geq \rho_0$.
It follows from the small cancellation theorem (\autoref{res: small cancellation}) that $X_{j+1}$ is a $\delta_1$-hyperbolic length space.
The action of $\Gamma_{j+1}$ is non-elementary (\autoref{res: approx - non elem action}).
Hence $(\Gamma_{j+1}, X_{j+1})$ satisfies \ref{enu - induction - hyp space}.

The group $\Gamma_{j+1}$ has no parabolic subgroup (\autoref{res: sc - no parabolic}).
Hence its elementary subgroups which are not elliptic are loxodromic and cyclic (\autoref{res: sc - lifting loxo sbgp}).
Every elliptic subgroup of $\Gamma_{j+1}$ is either the isomorphic image of an elliptic subgroup of $\Gamma_j$ or contained in $\stab Y / H$ for some $(H,Y) \in \mathcal Q_j$ (\autoref{res: lifting elliptic subgroups}).
On the one hand, every elliptic element of $\Gamma_j$ has finite order dividing $n$.
On the other hand, $\stab Y / H$ is isomorphic to $\Z / n \Z$, for every $(H,Y) \in \mathcal Q_j$.
Thus every elliptic element of $\Gamma_{j+1}$ has finite order dividing $n$. 
Consequently $(\Gamma_{j+1}, X_{j+1})$ satisfies \ref{enu - induction - elem subgroup}.

Combining Proposition~\ref{res: approx - A inv} and \ref{res: approx - inj radius} combined with (\ref{eqn: induction param - Delta}) and (\ref{eqn: induction param - thinness}) we get the following statement corresponding to \ref{enu - induction - sc param}.
\begin{lemm}[Compare with {\cite[Lemmas~6.3 and 6.4]{Coulon:2016if}}]
	The pair $(\Gamma_{j+1}, X_{j+1})$ satisfies
	\begin{equation*}
		A(\Gamma_{j+1}, X_{j+1}) \leq A
		\quad \text{and} \quad
		\inj[X_{j+1}]{\Gamma_{j+1}} > \epsilon \kappa  \delta_1	
	\end{equation*}
	Moreover $\nu(\Gamma_{j+1}, X_{j+1}) = 1$.
\end{lemm}

Recall that the natural projection $\dot X_j \to X_{j+1}$ is $1$-Lipschitz.
Consequently $f_j \colon X_j \to X_{j+1}$ is $\epsilon$-Lipschitz, which proves \ref{enu - induction - lip map}.
Hence the pair $(\Gamma_{j+1}, X_{j+1})$ satisfies all the requirements listed above.
This completes the induction step.

\paragraph{Cone points.}
We extend the structure of the pairs $(\Gamma_j, X_j)$ built above by attaching to them a set of cone points.
Let $j \in \N \setminus\{0\}$.
Recall that $\dot X_{j-1}$ stands for the cone-off over the rescaled space $\epsilon X_{j-1}$.
In particular, it comes with a set of cone points corresponding to the apices of the attached cones.
We denote by $\mathcal C_j$ its image in $X_j$.
By convention $\mathcal C_0$ is empty.

\subsubsection{Properties of the approximation sequence}
\label{sec: sc - approx properties}

In this section we prove that the approximation sequence $(\Gamma_j, X_j, \mathcal C_j)$ built above satisfies the conclusion of \autoref{res: approximating sequence}.
As we observed before, the space $X_0$ is obtaining by rescaling $X$, where the rescaling factor does not depend on $n$.
Thus \autoref{res: approximating sequence}~\ref{enu: approximating sequence - init} holds.
The next statement corresponds to the first half of \autoref{res: approximating sequence}~\ref{enu: approximating sequence - cvg}.

\begin{prop}[Compare with {\rm \cite[Theorem~6.9]{Coulon:2016if}}]
\label{res: direct limit approx sequence}
	The direct limit of the sequence $\Gamma_0 \onto \Gamma_1 \onto \dots$ is isomorphic to $\Gamma/\Gamma^n$.
\end{prop}

\begin{lemm}
	If every finite subgroup of $\Gamma$ is cyclic, then the same holds for $\Gamma_j$, for every $j \in \N$.
\end{lemm}

\begin{proof}
	The proof is by induction on $j$.
	It follows from our assumption if $j = 0$.
	Fix $j \in \N$ and assume that every finite subgroup of $\Gamma_j$ is cyclic.
	Let $E$ be a finite subgroup of $\Gamma_{j+1}$.
	Assume first that $E \subset \stab c$ for some $c \in \mathcal C_{j+1}$.
	Since $\stab c$ is cyclic, so is $E$.
	Assume now that some element of $E$ does not fix a cone point.
	It follows from \autoref{res: lifting elliptic subgroups} that $E$ is isomorphic to a finite subgroup of $\Gamma_j$.
	Hence it is cyclic according to our induction hypotheses.
\end{proof}

Using the previous statement, we observe that if every finite subgroup of $\Gamma$ is cyclic, then the same holds for $\Gamma / \Gamma^n$.
This proves  \autoref{res: approximating sequence}~\ref{enu: approximating sequence - cvg}.

We now focus on \autoref{res: approximating sequence}~\ref{enu: approximating sequence - control}.
The goal is to prove that every triple $(\Gamma_j, X_j, \mathcal C_j)$ belongs to the class $\mathfrak H_\delta(\rho(n))$ given by \autoref{def: preferred class} for a suitable value of $\delta$ (which is independent of $n$ and $j$).

Let $j \in \N \setminus\{0\}$.
The cones attached to $\epsilon X_{j-1}$ to build the cone-off $\dot X_{j-1}$ have radius $\rho$.
Consequently the set $\mathcal C_j$ is $2\rho$-separated.
Recall the terminology from \autoref{sec: approx periodic groups}.
An element of $\Gamma_j$ is \emph{conical}, if it fixes a unique cone point.
It is \emph{visible}, if it is loxodromic or conical, and \emph{elusive} otherwise.
A subgroup of $\Gamma_j$ is \emph{elusive}, if all its elements are elusive.

Note that $\Gamma_0 = \Gamma$ is CSA by the very definition of bootstrap (see \autoref{def: bootstrap}).
A proof by induction using \autoref{res: sc - csa} shows that $\Gamma_j$ is CSA, for every $j \in \N$.

\begin{lemm}[Elementary subgroups]
\label{res: approx - axioms - elem sbgp}
	For every $j \in \N$, every elementary subgroup of $\Gamma_j$ which is not elusive is cyclic.
\end{lemm}

\begin{proof}
	Let $j \in \N$ and $E$ be a non-elusive elementary subgroup of $\Gamma_j$.
	In view of \ref{enu - induction - elem subgroup} we can assume that $E$ is elliptic.
	Since $E$ is not elusive, it fixes a unique cone point $c \in \mathcal C_j$.
	In particular, $j \geq 1$.
	Moreover $c$ is the image in $X_j$ of the apex of the cone $Z(Y)$ for some $(H,Y) \in \mathcal Q_{j-1}$.
	It follows from \autoref{res: small cancellation}~\ref{enu: small cancellation - local embedding} that $E$ embeds in $\stab Y/ H$.
	Nevertheless by \ref{enu - induction - elem subgroup} every loxodromic subgroup of $\Gamma_{j-1}$ is cyclic. 
	In particular, so is $\stab Y$ and thus $E$.
\end{proof}

Recall that $\alpha$ has been fixed with the other small cancellation parameters at the beginning of \autoref{sec: sc parameters}.

\begin{lemm}[Loxodromic isometries]
\label{res: approx - axioms - loxo isom}
	For every $j \in \N$, every loxodromic element of $\Gamma_j$ is $\alpha$-thin.
\end{lemm}

\begin{proof}
	The proof is by induction on $j \in \N$.
	According to our initial rescaling, every loxodromic element of $\Gamma_0$ is $\alpha$-thin.
	Assume now that the result holds for some $j \in \N$.
	By (\ref{eqn: induction param - thinness}) every loxodromic element of $\Gamma_j$ is $\delta_1$-thin, for its action on the rescaled space $\epsilon X_j$.
	It follows from \autoref{res: approx - thin loxodromic}, that every loxodromic element of $\Gamma_{j+1}$ is $\alpha$-thin.
	Hence the statement holds for $j+1$.
\end{proof}

\begin{lemm}[Elliptic isometries]
\label{res: approx - axioms - elliptic isom}
	For every $j \in \N$, the following holds.
	\begin{enumerate}
		\item Every non-trivial elusive element $\gamma \in \Gamma_j$, is $\alpha$-thin.
		\item Every non-trivial element $\gamma \in \Gamma_j$ fixing a cone point $c \in \mathcal C_j$ is $(\rho + \alpha)$-thin at $c$.
	\end{enumerate}
\end{lemm}

\begin{proof}
	The proof is again by induction on $j \in \N$.
	Recall that $\mathcal C_0$ is empty.
	According to our initial rescaling, every elliptic element of $\Gamma_0$ is $\alpha$-thin, whence the result.
	
	Assume now that the result holds for some $j \in \N$.
	According to \autoref{res: approx - axioms - loxo isom} combined with (\ref{eqn: induction param - thinness}) every loxodromic element of $\Gamma_j$ is $\delta_1$-thin for its action on the rescaled space $\epsilon X_j$.
	On the other hand it follows from our induction hypothesis that every non-trivial elliptic element of $\Gamma_j$ is $(\rho + \alpha)$-for its action on $X_j$.
	Thus by (\ref{eqn: induction param - thinness}) they are $\delta_1$-thin when acting on the rescaled space $\epsilon X_j$.
	Let $\gamma$ be a non-trivial elliptic element of $\Gamma_{j+1}$.
	If $\gamma$ fixes a cone point in $c \in \mathcal C_{j+1}$, it follows from \autoref{res: sc - fix point set conical contained in a ball} that $\gamma$ is that $(\rho + \alpha)$-thin at $c$.
	Assume now that $\gamma$ is elusive.
	It follows from \autoref{res: sc - unique fixed cone point} that $\gamma$ does not fix a cone point in $\mathcal C_{j+1}$.
	Consequently, $\gamma$ is $\alpha$-thin (\autoref{res: sc - axis small elliptic}).
	In particular, the statement holds for $j +1$.
\end{proof}

\begin{prop}
\label{res: approx - all axioms}
	There is $\delta \in \R_+^*$ which does not depend on $n$ such that for every $j \in \N$, the approximation triple $(\Gamma_j, X_j, \mathcal C_j)$ of $\Gamma/\Gamma^n$ belongs to $\mathfrak H_\delta(\rho(n))$.
\end{prop}

\begin{proof}
	We set $\delta = \max\{A, 10\delta_1, \alpha\}$.
	It follows from our induction that 
	\begin{equation*}
		\nu(\Gamma_j, X_j) = 1
		\quad \text{and} \quad
		A(\Gamma_j, X_j) \leq A \leq \delta.
	\end{equation*}
	Hence Axiom~\ref{enu: family axioms - acylindricity} holds.
	Axiom~\ref{enu: family axioms - elem subgroups} follows from \autoref{res: approx - axioms - elem sbgp}.
	Recall that $\alpha \leq \delta$.
	Hence Lemmas~\ref{res: approx - axioms - elliptic isom} and \ref{res: approx - axioms - loxo isom} prove \ref{enu: family axioms - elusive}, \ref{enu: family axioms - conical - 1}, and \ref{enu: family axioms - loxodromic}.
	Axiom \ref{enu: family axioms - conical - 2} is a reformulation of \autoref{res: sc - angle cocycle}.
	We are left to prove Axiom~\ref{enu: family axioms - radial proj}.
	Since $\mathcal C_0$ is empty, we can assume that $j \geq 1$.
	Let $c \in \mathcal C_j$ and $x \in X_j$ such that $d(x, \mathcal C_j) < \rho$.
	Recall that the radial projection is a map $p_j \colon X_j\setminus\mathcal C_j \to X_j^+$.
	If $x \neq c$, one checks that $y = p_j(x)$ satisfies $\dist xy = \rho - \dist xc$.
	If $x = c$, then the radial projection $y$ of any point in the punctured ball $\mathring B(c, \rho)$ works.
\end{proof}

The previous statement corresponds to \autoref{res: approximating sequence}~\ref{enu: approximating sequence - control}.
As we already observed for every $j \in \N$, the map $X_j \to X_{j+1}$ is $\epsilon(n)$-Lipschitz with $\epsilon(n) < 1$.
It follows from  \autoref{res: small cancellation}~\ref{enu: small cancellation - translation kernel} that for every $x \in X_j$, the projection $\Gamma_j \onto \Gamma_{j+1}$ is one-to-one when restricted to set
\begin{equation*}
	\set{\gamma \in \Gamma_j}{ \epsilon(n) \dist{\gamma x}x \leq \frac {\rho(n)}{100}},
\end{equation*}
which proves  \autoref{res: approximating sequence}~\ref{enu: approximating sequence - metric}.
We are left to prove  \autoref{res: approximating sequence}~\ref{enu: approximating sequence - lifting}.
This is a direct application of \autoref{res: sc - lifting morphism}.
Hence the proof of  \autoref{res: approximating sequence} is complete.


\bigskip
\noindent
\emph{R\'emi Coulon} \\
Université de Bourgogne, CNRS \\
IMB - UMR 5584 \\
F-21000 Dijon, France\\
\texttt{remi.coulon@cnrs.fr} \\
\texttt{http://rcoulon.perso.math.cnrs.fr}

\bigskip
\noindent
\emph{Zlil Sela} \\
Department of Mathematics \\
Hebrew University \\
Jerusalem 91904, Israel \\
\texttt{zlil@math.huji.ac.il} \\
\texttt{https://math.huji.ac.il/\~{}zlil}

\end{document}